\documentclass{amsart}[12pt]
\usepackage{graphicx,url,amsthm,xspace,enumitem,subfig,amssymb}
\usepackage[dvipsnames]{xcolor}
\usepackage{amscd, amsmath}
\usepackage{multirow}
\usepackage[normalem]{ulem}

\usepackage{xy}
\xyoption{all}

\usepackage{pinlabel}
\usepackage[linktocpage]{hyperref}

\usepackage[margin=3cm, marginpar=2.5cm]{geometry}

\newcommand{\C}{\mathbb{C}}
\renewcommand{\P}{\mathbb{P}}
\newcommand{\R}{\mathbb{R}}

\newcommand{\Z}{\mathbb{Z}}
\newcommand{\cptwo}{\C\P^2}
\newcommand{\cpone}{\C\P^1}
\newcommand{\cptwobar}{\overline{\C\P}\,\!^2}
\newcommand{\rptwo}{\R\P^2}

\newcommand{\Ac}{\mathcal{A}}
\newcommand{\Bc}{\mathcal{B}}
\newcommand{\Cc}{\mathcal{C}}
\newcommand{\Dc}{\mathcal{E}_6}
\newcommand{\Ec}{\mathcal{E}_3}
\newcommand{\Gc}{\mathcal{G}}
\newcommand{\Hc}{\mathcal{H}}
\newcommand{\Lc}{\mathcal{L}}
\newcommand{\Vc}{\mathcal{V}}

\newcommand{\bd}{\mathbf{d}}

\newcommand{\total}[1]{\overline{#1}}
\newcommand{\proper}[1]{\widetilde{#1}}

\newcommand{\HFplus}{\HF^+}

\newcommand{\spinc}{spin$^c$~}

\newcommand{\fs}{\mathfrak{s}}

\DeclareMathOperator{\HF}{HF}

\DeclareMathOperator{\PD}{PD}
\DeclareMathOperator{\Sing}{Sing}
\DeclareMathOperator{\Iso}{Iso}

\theoremstyle{plain}
\newtheorem{theorem}{Theorem}[section]

\newtheorem{lemma}[theorem]{Lemma}
\newtheorem{prop}[theorem]{Proposition}
\newtheorem{cor}[theorem]{Corollary}

\theoremstyle{definition}
\newtheorem{question}[theorem]{Question}
\newtheorem{definition}[theorem]{Definition}

\theoremstyle{remark}
\newtheorem{remark}[theorem]{Remark}
\newtheorem{example}[theorem]{Example}

\numberwithin{equation}{section}

\title{The symplectic isotopy problem for rational cuspidal curves}

\author{Marco Golla}
\email{marco.golla@univ-nantes.fr}
\address{CNRS, Laboratoire de Math\'ematiques Jean Leray, Nantes, France}

\author{Laura Starkston}
\email{lstarkston@math.ucdavis.edu}
\address{Department of Mathematics, UC Davis, One Shields Ave, Davis, CA 95616, U.S.A.}

\begin{document}

\maketitle

\begin{abstract}
We define a suitably tame class of singular symplectic curves in 4-manifolds, namely those whose singularities are modeled on complex curve singularities. We study the corresponding symplectic isotopy problem, with a focus on rational curves with irreducible singularities (rational cuspidal curves) in the complex projective plane.
We prove that every such curve is isotopic to a complex curve in degrees up to 5, and for curves with one singularity whose link is a torus knot.
Classification results of symplectic isotopy classes rely on pseudo-holomorphic curves together with a symplectic version of birational geometry of log pairs and techniques from $4$-dimensional topology.
\end{abstract}


\section{Introduction}
{
In this article, we take up an extensive study of singular curves in the symplectic category. We focus on rational (genus zero) curves, taking advantage of the singularities to obtain low-genus curves with high degree. We primarily study irreducible rational cuspidal curves, but also consider reducible configurations with rational cuspidal components. 
Rational cuspidal curves are a source of rich complexity in algebraic geometry~\cite{Palka6cusps,KorasPalka,PalkaPelka,PalkaPelka2}.
We use the term \emph{cusp} to refer to any locally irreducible singularity, but we focus on cusps locally modeled on $\{x^p=y^q\}$ where $p$ and $q$ are relatively prime (the link is a $(p,q)$-torus knot). 
To our knowledge, prior work on singular symplectic curves has been restricted to nodes, simple cusps ($(2,3)$-cusps), and tacnodes (simple tangencies between two branches). Rational cuspidal curves provide an effective class to work with in the symplectic category because pseudoholomorphic curves are most powerful in the rational (genus-$0$) case. We give obstruction results determining which rational cuspidal curves are realizable symplectically in the complex projective plane, as well as isotopy classification results proving uniqueness of realization up to symplectic isotopy. Both of these problems (existence and uniqueness) are difficult even in the complex algebraic category, and in the symplectic category, uniqueness has not even been proven for smooth curves of degree greater than $17$ (our results apply to singular curves of arbitrary degree).

Complex plane curves have a distinguished history in algebraic geometry: from Zariski's examples of singular sextics with distinct fundamental groups in their complements~\cite{Zariski}, to many more recent results continuing to the present from line arrangements to cuspidal curves (see for example~\cite{Hirzebruch-lines,Harris,ZaidenbergLin,Rybnikov,ArtalCarmonaCogolludo,ArtalCarmonaCogolludoBuznariz,KorasPalka}). 
Algebraic geometers have built up powerful tools to tackle these problems: braid monodromy~\cite{Moishezon,Libgober}, Alexander invariants~\cite{LibgoberA}, the log minimal model program~\cite{Miyanishi}, Miyaoka's inequalities~\cite{Miyaoka}.
For example it was conjectured that every rational cuspidal curve has at most four cusps. This has recently been proven by Koras and Palka~\cite{KorasPalka2} using the almost minimal model program~\cite{Palka6cusps} (see~\cite{Pion,ZaOrev,Tono} for previous progress on this problem and~\cite{PalkaPelka,PalkaPelka2} for more recent developments). However, a complete classification of singular plane curves is still far out of reach in this rich subject.
Beyond questions of which singular curves exist, serious study has been devoted to asking how many planar realizations there are of a given singular curve up to automorphisms of the plane, diffeomorphisms of their complements, or isomorphisms of the fundamental groups of their complements. Distinct realizations (under one of these equivalence relations) are often called ``Zariski pairs'' and it is unknown in general which curves have Zariski pairs.

Compared to algebraic geometry, the development of tools in symplectic geometry has occurred relatively recently. The classification of any type of symplectic planar curve, at first glance, is completely intractable because the space of symplectic curves is a relatively poorly understood infinite-dimensional moduli space.
By contrast, the space of complex algebraic curves of a particular degree with particular singularities is cut out by finitely many discriminant polynomials in finitely many complex variables.
However, Gromov's theory of pseudoholomorphic curves brought some hope that some information about moduli spaces of symplectic curves could be understood. In particular, Gromov classified smooth symplectic surfaces in the complex projective plane in degrees $1$ and $2$ up to symplectic isotopy. While further work extended this to classify smooth symplectic surfaces up to degree $17$~\cite{Sikorav, Shev, SiebertTian}, in degrees greater than or equal to $18$ the question remains open.

\begin{question}[Symplectic Isotopy Problem]\label{q:sympliso}
	Is every smooth symplectic surface in $\cptwo$ symplectically isotopic to a complex curve? 
	\end{question}

This question is equivalent to asking, is there a unique symplectic isotopy class of symplectic surfaces in each degree? The equivalence results from the fact that the moduli space of smooth complex curves of fixed degree is connected. (The space of all degree-$d$ curves is parameterized by the coefficients of the defining polynomial, and singular curves have positive complex co-dimension.)
Because of the difficulty of the classification for smooth symplectic curves, there has been somewhat limited work on singular symplectic curves~\cite{McDuff-immersed,IvashShev,Shev2,Barraud,Francisco}. Nevertheless, the study of singular symplectic curves in $\cptwo$ has significant ramifications for the study of all symplectic $4$-manifolds via branched covering constructions \cite{Auroux1,Auroux2}.


Our first result addresses curves with a single cusp singularity. Such curves exist in every degree $d \geq 3$, so we are not bound by the uniform degree constraints encountered in attempts to answer Question~\ref{q:sympliso}. Note that the degrees and singularities for possible complex curves of this type were classified in~\cite{FdBLMHN}.

\begin{theorem}\label{t:unicusp}
	In $\cptwo$, every symplectic rational unicuspidal curve whose unique singularity is the cone on a torus knot is symplectically isotopic to a complex curve, and has a unique symplectic isotopy class.
\end{theorem}

As a corollary to our results, we see there are no Zariski pairs of rational unicuspidal curves with one Puiseux pair.
To our knowledge, the literature does not contain the statement that there are no Zariski pairs for such complex algebraic curves. However, experts believe this might follow from the negativity conjecture of~\cite{PalkaPelka,PalkaPelka2} (which would prove the statement for arbitrary rational cuspidal curves of log general type), together with some analysis of direct classifications in the cases of log Kodaira dimension at most $1$.

A great deal of complexity arises when we consider rational curves with more than one cusp. Although there are basic restrictions relating the degree to the genus of the torus knots appearing as the links of the singularities, there are many combinations of cusp singularity types which cannot be realized by complex algebraic curves even though they satisfy these basic adjunctive requirements. We initiate exploration of multi-cuspidal curves in the symplectic setting with low-degree rational cuspidal curves.
The corresponding problem in algebraic geometry has a long history. Complex algebraic rational cuspidal curves in degrees at  most $5$ were completely classified~\cite{Namba} (see~\cite{Moe} for a more modern exposition). The classification of which rational cupsidal curves of degree $6$ can be realized complex algebraically was completed by Fenske~\cite{Fenske}.

Here we look at symplectic curves up to degree $5$. While we do show in Section~\ref{ss:sextics} that our methods can say something for curves of degree $6$, even those with significantly different properties than those appearing in degree $5$, we do not venture into the combinatorics to give a complete classification. In fact, already in degree $5$ there is a great deal of complexity which could indicate how the symplectic category relates to the complex one. The techniques we develop to classify the variety of curves in degree $5$ provide a model for how one could approach many other cases in higher degrees. In degree $5$, there are $19$ different possible combinations of different cusp types which satisfy the basic degree-genus restrictions. Of these $19$, we show $9$ cannot embed symplectically into any symplectic manifold, $8$ embed into $\cptwo$ with a unique symplectic isotopy class, and the remaining $2$ more only embed symplectically and relatively minimally into $\cptwo\#4\cptwobar$. Comparing to known results on such algebraic curves we obtain the following theorem.

\begin{theorem}\label{t:lowdeg}\label{thm:quintics}
	In $\cptwo$, every symplectic rational cuspidal curve of degree at most five is symplectically isotopic to a complex curve and has a unique symplectic isotopy class. In particular, there are no Zariski pairs for rational cuspidal curves of degree up to 5.
\end{theorem}

Since the first version of this paper appeared, the result has been extended to degrees $6$ and $7$ by the first author and K\"utle~\cite{GKutle}.

As mentioned above, our techniques in fact obtain much stronger results than classifications in $\cptwo$. For each of the rational cuspidal curves we consider, we in fact classify the existence and uniqueness of minimal symplectic embeddings into \emph{any closed symplectic $4$-manifold} (up to symplectomorphism and symplectic isotopy of the pair).
Classifications up to symplectomorphism of the pair reduce to ambient symplectic isotopy statements in the case that the ambient manifold is $\cptwo$ by Gromov's result that the space of symplectomorphisms of $(\cptwo,\omega_{\rm FS})$ is homotopy equivalent to $PU(3)$ \cite{Gr,McDuffSalamonBig}, so in particular is path-connected.

To each singular symplectic curve $C$, we associate a contact $3$-manifold $(Y_C,\xi_C)$ which appears on the boundary of a concave neighborhood of the curve. Our classifications of symplectic embeddings of the curves provide classifications of all symplectic fillings of the contact manifolds associated to our curves. Thus we get new examples of contact manifolds which have no (strong) symplectic fillings, unique symplectic fillings, and non-unique but finitely many symplectic fillings that are classified.
Previous complete classifications of symplectic fillings have primarily been restricted to lens spaces~\cite{Li,PlamenevskayaVHM}, Seifert fibered spaces~\cite{OhtaOno,StFill}, and torus bundles~\cite{GL}. Our contact $3$-manifolds are more general graph manifolds which do not fall in any of these previous classes.

One question we had early on in our investigations, was whether symplectic embeddability of a certain type of rational cuspidal curve could distinguish the complex projective plane from a potential exotic or fake copy of $\cptwo$. We found that this cannot be the case.

\begin{theorem} \label{thm:rational}
	If $X$ is a rational homology $\cptwo$ that contains a rational cuspidal curve, then $X$ is symplectomorphic to $\cptwo$.
\end{theorem}

A similar question was raised by Chen~\cite{Chen}, who asked which symplectic 4--manifolds with the same rational homology as $\cptwo$ could be split along a contact-type hypersurface into two pieces, one of which was a rational homology ball.


We also prove a number of existence and uniqueness results for reducible configurations whose components are rational. These results play a role in proving the results for irreducible rational cuspidal curves. We summarize here some examples of these results.

\begin{theorem}\label{t:smallconfigs}
	Any configuration of symplectic conics and lines of total degree at most five either has a unique symplectic isotopy class or is obstructed in Section \ref{s:reducibleobstructions}. There is a unique symplectic isotopy class for each of the infinitely many configurations composed of one rational degree-$d$ curve $C$ with a singular point with multiplicity sequence $[d-1]$ together with a line which is tangent to $C$ of order $d$.
\end{theorem}

Two more complicated configurations of two conics with additional lines appear in Figures \ref{fig:H} and \ref{fig:L} and with further techniques we prove they also have unique symplectic isotopy classes.

To prove these results and relate them to each other, we develop a symplectic theory of birational geometry for pairs $(X^4,C)$ where $X$ is a symplectic $4$-manifold and $C$ is a (potentially singular) symplectic curve. We apply this, along with pseudoholomorphic curve techniques, to give symplectic classifications of many reducible configurations of curves in $\cptwo$. 
We define symplectic proper transforms of curves, and see that the outcome is only well-defined \emph{up to symplectic isotopy} (meaning a related through a smoothly varying equisingular family of symplectic curves, not necessarily related by an ambient isotopy).
Because we use symplectic blow-ups and blow-downs throughout the paper, most of our statements are true only up to symplectic isotopy.
The theory of symplectic birational geometry is based in fundamental work of McDuff on symplectic blow-up and blow-down. We develop this theory for ``log'' pairs $(X^4,C^2)$ to study singular symplectic curves and their isotopy classifications. While symplectic geometry does not have the full strength of algebraic geometry's log minimal model program, by combining this symplectic birational geometry with pseudoholomorphic techniques and topological tools, we are able to prove many new results that have not been addressed with pseudoholomorphic curves alone.


\subsubsection*{Organization}
The paper is organized as follows.
Section~\ref{s:singularcurves} defines singular symplectic curves and their equivalences, provides background on complex curve singularities and their resolutions, defines the contact manifold associated to a singular curve, and reviews previous topological obstructions to rational cuspidal curves.
Section~\ref{s:tools} develops the tools utilizing pseudoholomorphic curves and birational geometry that we constantly use throughout the paper.
In particular, Sections~\ref{ss:propertransform} and~\ref{ss:birationalequiv} contain the crucial definitions of symplectic proper transform, and of birational derivation and birational equivalence of configurations of curves, respectively.
Section~\ref{s:rational} proves Theorem~\ref{thm:rational}.
Section~\ref{s:rediso} studies symplectic isotopy problems for reducible configurations, which will be used to study the isotopy problem for rational cuspidal curves.
Sections~\ref{s:unicusp} and~\ref{s:lowdegrees} give the proofs of Theorems~\ref{t:unicusp} and~\ref{t:lowdeg}, respectively.
In Section~\ref{s:orevkov} we will give an example, communicated to us by Stepan Orevkov, of a symplectic rational (non-cuspidal) curve that is not isotopic to any complex curve.
The appendix contains a result about rational homology ball symplectic fillings of lens spaces that we use in the proof of Theorem~\ref{t:unicusp}.

\subsubsection*{Acknowlegments}
This project originated from conversations the two authors had during the intensive semester on \emph{Symplectic geometry and topology} at the Mittag-Leffler Institute; we acknowledge their hospitality and the great working environment. Further progress was made during the BIRS conference \emph{Thirty years of Floer theory for $3$-manifolds} in Oaxaca; we thank the organizers for this excellent conference and for the opportunity to collaborate.
The authors would like to thank the participants of the AIM workshop \emph{Symplectic four-manifolds through branched coverings} for their interest in the project.
MG acknowledges hospitality from Stanford University and UC Davis. LS is grateful for hospitality from Universit\'{e} de Nantes. 
LS was supported by NSF Grant No. 1501728 and 1904074.
We would like to thank J\'ozsi Bodn\'ar, Erwan Brugall\'e, Roger Casals, Anthony Conway, Paolo Ghiggini, Andr\'as N\'emethi, Tomasz Pe{\l}ka, Olga Plamenevskaya, Danny Ruberman, Andr\'as Stipsicz, and Chris Wendl.
Finally, the authors would like to thank Stepan Orevkov for his comments and for communicating us the example in Section~\ref{s:orevkov}.
}

\section{Symplectic singular curves}\label{s:singularcurves}

In this section, we will start by reviewing some basic facts about singular complex curves. A good reference for the material we cover here is~\cite{Wall}. Then we will give our definitions of the symplectic analogues.

Let $C$ in $\C^2$ be the zero set of an analytic function $F(x,y)$, such that $F(0,0) = 0$; we say that $C$ is singular at the origin if $\frac{\partial F}{\partial x}(0,0) = \frac{\partial F}{\partial y}(0,0) = 0$.
We suppose that $F$ is locally irreducible, i.e. irreducible in the ring $\C[[x,y]]$ of power series in two variables; in this case, we will say that the singularity is a \emph{cusp}.
(Note that certain authors call cusps only the singularities of type $(2,3)$; we will refer to the latter as \emph{simple cusps}.)

Up to diffeomorphism, $C$ can be parametrized (in a neighborhood of the origin) as the image of $\phi\colon\C \to \C^2$ defined by $\phi(t) = (t^m, a_1t^{b_1} + \dots + a_kt^{b_k})$ for some positive integers $m < b_1 < \dots < b_k$ and some $a_1,\dots,a_k \in \C^*$.
Moreover, if we require that the sequence $e_i$ defined by $e_1 = \gcd(m,b_1)$, $e_{i} = \gcd(e_{i-1},b_i)$ is strictly decreasing, the exponents $m, b_i$ are uniquely determined.
In this case, $\phi$ is the Puiseux parametrization of $C$, and $(m;b_1,\dots,b_k)$ are the \emph{Puiseux exponents} of $C$; $m$ is called the \emph{multiplicity} of the singularity.

\begin{remark}\label{r:tang}
Note that with the above coordinates from the Puiseux parametrization, the cusp curve $C$ has a unique tangent line $\{y=0\}$. This is the limit of the tangent lines of the nearby smooth points as $\frac{dx}{dt}=mt^{m-1}$, $\frac{dy}{dt} = b_1a_1t^{b_1-1}+\cdots$, so $\frac{dy}{dx}\to 0$ as $t\to 0$ since $b_1>m$. For a general complex curve singularity, there are a finitely many locally irreducible branches, each with a unique complex tangent line. The multiplicity of intersection of the tangent line at a cusp point is $b_1$ (the solutions to setting $y=0$), whereas a line which is not tangent to the cusp will intersect the cusp with multiplicity $m$ (the solution to setting $y=ax=at^m$ for $a\neq 0$). See~\cite[Section 2.3]{Wall} for more details.
\end{remark}

Recall that the \emph{link} of the singularity is the diffeomorphism type of $(S_\varepsilon \cap C, S_\varepsilon)$, where $S_\varepsilon = \partial B_\varepsilon$ is the boundary of a ball of radius $\varepsilon \ll 1$.
The fact that the singularity of $C$ at the origin is irreducible translates into the condition that the link is a knot (i.e. has one component).

The Puiseux exponents determine the topology of the singularity: two singularities have the same Puiseux exponents if and only if their links are diffeomorphic~\cite[Proposition 5.3.1]{Wall}.

In this paper we will be almost exclusively concerned with singularities whose Puiseux expansion has $k=1$, i.e. it is of the form $(m;b_1)$;
we will say that the \emph{singularity is of type $(m,b_1)$}.
Note that $m < b_1$; however, the link of a singularity of type $(m,b_1)$, which is the torus knot $T(m,b_1)$, is isotopic to the torus knot $T(b_1,m)$.

\subsection{Resolution of singularities}\label{ss:resolution}

Recall from~\cite{Wall} that every curve singularity can be resolved by blowing up (sufficiently many times), and the diffeomorphism type of the link determines the topology of the resolution.
If $\pi: \widetilde X \to X$ is a (multiple) blow-up at a point $x\in X$, and $C\subset X$ is a curve, we call $\overline{\pi^{-1}(C\setminus\{x\})}$ the \emph{proper transform} of $C$ under $\pi$, and we denote it by $\widetilde C$;
we also call $\pi^{-1}(C)$ the \emph{total transform} of $C$ under $\pi$, and we denote it by $\overline C$.

For every singular curve $C\subset X$, there exists a composition of blow-ups $\pi: \widetilde X \to X$ such that $\widetilde C$ is smooth; we call any such $\pi$ a \emph{resolution} of $C$.

Given a resolution $\pi$ of $C$, consider the restriction $n := \pi|_{\widetilde{C}}$ of $\pi$ to the proper transform $\widetilde{C}$ of $C$. We call $n: \widetilde{C} \to C$ (and, by abuse of notation, $\widetilde{C}$) the \emph{normalization} of $C$. Note that this is essentially independent of the choice of the resolution: if $\pi'$ is another resolution of $C$ and we construct $n'$ accordingly, then there is an isomorphism $\phi$ between the sources of $n$ and $n'$ such that $n = n' \circ \phi$.
We also observe that $C$ is cuspidal if and only if the normalization is one-to-one; in particular, in this case it is a homeomorphism onto its image.

There are two natural stopping points when resolving a singularity:
the \emph{minimal resolution} is the smallest resolution such that the proper transform $\widetilde C$ of $C$ is smooth;
the \emph{normal crossing resolution} is the smallest resolution such that the total transform $\overline C$ of $C$ is a normal crossing divisor, i.e. all singularities are double points
\footnote{The terminology in the log algebraic geometry community (e.g. in~\cite{KorasPalka, PalkaPelka,PalkaPelka2,KorasPalka2}) seems to by slightly different: they call \emph{minimal weak resolution} (or, in earlier papers, \emph{minimal embedded resolution}) what we call minimal resolution, and \emph{minimal log resolution} what we call normal crossing resolution.}.

Recall that the \emph{multiplicity} $m_p$ of a singularity $p$ of $C \subset X$ is the minimal intersection of a germ of a divisor $D$ at $p$ with $C$; that is, $m_p = \min_D (C\cdot D)_p$;
the multiplicity of a singularity has the following interpretation: blow up $X$ at $p$, obtaining $\widetilde X$, which contains the corresponding exceptional divisor $E$; then $m_p$ is the intersection number of $E$ and $\widetilde{C}$, i.e. $m_p = \widetilde{C}\cdot E$.
In terms of homology, there is an orthogonal decomposition $H_2(\widetilde X) = H_2(X) \oplus \Z[E]$, and we have $[\widetilde C] = [C]-m_p[E]$. In particular $[\widetilde C]^2 = [C]^2-m_p^2$.

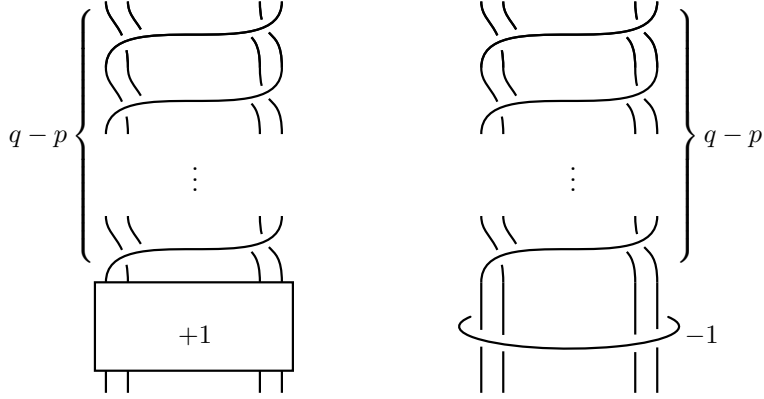
\begin{figure}[h!]
	\labellist
	\pinlabel $q-p\left\{\phantom{\begin{array}{c}a\\b\\c\\d\\e\\f\\g\\h\end{array}}\right.$ at -7 93
	\pinlabel $\left.\phantom{\begin{array}{c}a\\b\\c\\d\\e\\f\\g\\h\end{array}}\right\}q-p$ at 220 93
	\pinlabel $\vdots$ at 37 80
	\pinlabel $\vdots$ at 175 80
	\pinlabel $+1$ at 37 20
	\pinlabel $-1$ at 222 20
	\endlabellist
	\includegraphics[width=0.5\textwidth]{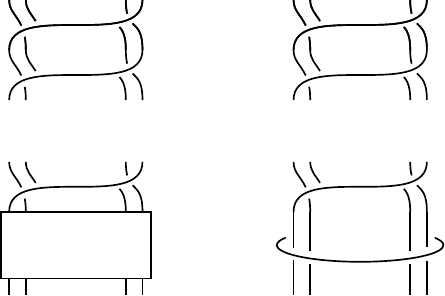}
	\caption{A handle diagram interpretation of the blow up of a singularity of type $(p,q)$. (Recall that $p<q$ in a singularity of type $(p,q)$.) There are $p$ strands on each side, and the rectangular box with the $+1$ denotes a full twist.}
\end{figure}

We record the singularity types at each step in the minimal resolution into a sequence of integers, called the \emph{multiplicity sequence}, as follows.
Suppose that the minimal resolution of a singularity of $C$ is a sequence of $k$ blow-ups, and call $C_j$ the proper transform of $C$ after the first $j-1$ blow-ups, so that $C_1 = C$, and $C_{k+1}$ is the proper transform of $C$ in the minimal resolution.
Then $C_j$ has a singularity of multiplicity $m_j$, and we write $[m_1,\dots,m_k]$ for the multiplicity sequence of the singularity of $C$.
Note that $m_{j+1} \le m_{j}$ for each $j$, but that not every sequence of integers correspond to the multiplicity sequence of a singularity; for instance $[3,2,2]$ is not the multiplicity sequence of any singularity.
The multiplicity sequence of an irreducible singularity determines the singularity~\cite[Theorem 3.5.6]{Wall}; however, the multiplicity sequence is defined both for irreducible and reducible singularities.

\begin{remark}\label{r:shortmultiseq}
	Algebraic geometers often consider the multiplicity sequence associated to the normal crossing divisors resolution of the singularity, rather than the one associated to the minimal resolution.
	In fact, one determines the other: if the latter ends in $m_k > 1$, the former coincides with it until the $k^{\rm th}$ entry, and then ends with a sequence of $1$s of length $m_k$.
	These last entries corresponds to the blow-ups needed to make the last exceptional divisor in the minimal smooth resolution disjoint from the proper transform of the curve.
\end{remark}

Given a cuspidal curve $C$, we say that the multiplicity multi-sequence of $C$ is the union of all multiplicity sequences of singularities of $C$.
This is only well-defined up to choosing an order of the cusps.
Instead of making the choice, we just sort all the entries decreasingly.
We use double-bracket notation to denote the multiplicity multi-sequence: $[[m_1,\dots, m_n]]$.

\begin{example}
	The multiplicity of a singularity whose link is the torus knot $T(p,q)$ is $\min\{p,q\}$. For a singularity of type $(p,p+1)$, one blow up already yields a smooth curve, and so the multiplicity sequence is $[p]$. In general, after blowing up once at a singularity of type $(p,q)$, the resulting curve has a singularity whose link is $T(p,q-p)$ (and therefore is of type $(p,q-p)$ or $(q-p,p)$, depending on whether $q$ is larger or smaller than $2p$). Iterating this process, one eventually reaches a smooth curve.
	
	The following are the singularities that will appear in our examples.
	\begin{itemize}
		\item For a singularity of type $(p,kp+1)$, the multiplicity is $p$, and singularity of the proper transform in the first blow-up is of type $(p,(k-1)p+1)$. It follows inductively that the multiplicity sequence is the string of length $k$, $[p,\dots,p]$.
		In particular, a curve with a singularity of type $(p,2p+1)$ and a curve with two singularities of type $(p,p+1)$ have the same multiplicity multi-sequence, $[[p,p]]$.
		\item A singularity of type $(p+1,2p+1)$ has multiplicity sequence $[p+1,p]$. In particular the singularity of type $(3,5)$ has multiplicity sequence $[3,2]$.
		\item A singularity of type $(p,4p-1)$ has multiplicity sequence $[p,p,p,p-1]$ ($p\geq 3$).
	\end{itemize}
\end{example}

Here we describe the normal crossing resolution of singularities of type $(p,q)$.
This is best described in terms of continued fraction expansions; given a rational number $r$, we write
\[
r = [a_1,\dots,a_k]^- = a_1-\frac{1}{a_2-\frac1{\dots-\frac{1}{a_k}}}
\]
for its (negative or Hirzebruch--Jung) continued fraction expansion, where $a_i \ge 2$ for each $i\ge 2$.

In what follows, we will adopt the notation $a^{[\ell]}$ to denote a string of $\ell$ entries, all equal to $a$; for instance $[2^{[\ell]}]$ will be the multiplicity sequence of the singularity of type $(2,2\ell+1)$.

It consists of a plumbing of spheres along a three-legged star-shaped graph decorated by Euler classes.
The Euler class on the central vertex is $-1$; that is, the central vertex is the exceptional divisor corresponding to the last blow-up in the resolution.
Two of the legs are linear chains whose decorations give the continued fraction expansions of 
\[
\frac{p}{p-q^*} \, \text{ and }\, \frac{q}{q-p^*},
\]
where $qq^*\equiv 1 \mod p$ and $0< q^*<p$, and $pp^* \equiv 1 \mod q$ and $0<p^*<q$.

More precisely, the resolution graph of the singularity of type $(p,q)$ looks like the following:
\[
\xygraph{
	!{<0cm,0cm>;<1cm,0cm>:<0cm,1cm>::}
	!~-{@{-}@[|(2.5)]}
	!{(0.3,1.5) }*+{\bullet}="x"
	!{(1.5,2) }*+{\bullet}="a1"
	!{(4.5,2) }*+{\bullet}="a2"
	!{(-1,1.5) }*+{\circ}="c1"
	!{(1.5,1) }*+{\bullet}="b1"
	!{(4.5,1) }*+{\bullet}="b2"
	!{(3,2) }*+{\dots}="am"
	!{(3,1) }*+{\dots}="bm"
	!{(0.3,1.9) }*+{-1}
	!{(1.5,2.4) }*+{-a_1}
	!{(4.5,2.4) }*+{-a_m}
	!{(1.5,1.4) }*+{-b_1}
	!{(4.5,1.4) }*+{-b_n}
	"x"-"c1"
	"x"-"a1"
	"x"-"b1"
	"a1"-"am"
	"b1"-"bm"
	"a2"-"am"
	"b2"-"bm"
}
\]
where $\frac{p}{p-q^*} = [a_1,\dots,a_m]^-$ and $\frac{q}{q-p^*} = [b_1,\dots,b_n]^-$, and the hollow dot represents the proper transform of $C$.
(It might be helpful to recall that $\frac{p}{p-q} = [a_m,\dots,a_1]^-$ and $\frac{q}{q-r(p,q)} = [b_n,\dots,b_1]^-$, where $r(p,q)$ is the remainder of the division of $p$ by $q$.)

The following lemma is well-known among algebraic geometers; we include a proof both for completeness and to give an example of the techniques for low-dimensional topologists.

\begin{lemma}\label{l:selfint-defect}
	Suppose $C$ is a singular curve in $X^4$ with a singularity of type $(p,q)$ at $x_0$.
	Consider the normal crossing divisor resolution of the singularity of $C$ at $x_0$ in $\widetilde X$. Then the proper transform $\widetilde C$ of $C$ in $\widetilde X$ satisfies $\widetilde C\cdot \widetilde C = C\cdot C - pq$.
\end{lemma}

\begin{proof}
	We prove the result by induction on the length $\ell$ of the multiplicity sequence of the singularity.
	
	If $\ell = 1$, the singularity is of type $(p,p+1)$ for some $p$, and the first blow-up already gives the minimal smooth resolution.
	The total transform of $C$ in the minimal resolution consists of the proper transform of $C$ and the exceptional divisor, and the two have a tangency of order $p$;
	as observed above, the proper transform has self-intersection $C\cdot C - p^2$.
	To get to the normal crossing divisor resolution, we blow up $p$ times at the tangency; since the proper transform of $C$ was already smooth, each of these blow-ups decrease the self-intersection by $1$, hence $\widetilde C\cdot \widetilde C = C\cdot C - p^2 - p\cdot 1 = C\cdot C - p(p+1)$, thus proving the base case.
	
	If $\ell > 1$, blow up once at $x_0$; the singularity of the proper transform of $C$ will be of type $(p,q-p)$ (or, possibly, $(q-p,p)$), and the length of its multiplicity sequence decreases by $1$ by definition;
	therefore $\widetilde C\cdot \widetilde C = C\cdot C - p^2 - p(q-p) = C\cdot C - pq$, as desired.
\end{proof}

If $C$ is a curve of self-intersection $s$ with a unique singularity of type $(p,q)$, then the total transform of $C$ is the full plumbing graph, where the hollow dot is replaced by a vertex with Euler class $s-pq$.
In our applications, in the case of a unique singularity, if $s-pq>1$, we will blow-up further times along the edge between the central vertex and this short leg to change the coefficient on the central vertex to $-2$ and add a chain of $s-pq-2$ more $(-2)$-vertices, a single $(-1)$-vertex, and a single $+1$-vertex.

In general, we can resolve each singularity of $C$ independently of the others;
furthermore, if all singularities are of type $(p_i,q_i)$, the total transform of $C$ will look like a plumbing tree, obtained by fusing the resolution graphs of the individual singularities along the hollow vertex, whose weight is $s-\sum_i p_iq_i$.

For instance, when $C$ is a rational curve with self-intersection $16$, with a singularity of type $(2,3)$ and one of type $(2,5)$, the total transform $\overline C$ of $C$ in the normal crossing resolution of $C$ is described by the following plumbing diagram:
\[
\xygraph{
	!{<0cm,0cm>;<1cm,0cm>:<0cm,1cm>::}
	!~-{@{-}@[|(2.5)]}
	!{(0,1) }*+{\bullet}="C"
	!{(1.5,1) }*+{\bullet}="e1"
	!{(3,1.75) }*+{\bullet}="a11"
	!{(3,0.25) }*+{\bullet}="b11"
	!{(4.5,0.25) }*+{\bullet}="b12"
	!{(-1.5,1) }*+{\bullet}="e2"
	!{(-3,1.75) }*+{\bullet}="a21"
	!{(-3,0.25) }*+{\bullet}="b21"
	!{(0,1.4) }*+{0}
	!{(1.5,1.4) }*+{-1}
	!{(3,2.15) }*+{-2}
	!{(3,0.65) }*+{-3}
	!{(4.5,0.65) }*+{-2}
	!{(-1.5,1.4) }*+{-1}
	!{(-3,2.15) }*+{-2}
	!{(-3,0.65) }*+{-3}
	"C"-"e1"
	"C"-"e2"
	"e1"-"b11"
	"e1"-"a11"
	"b11"-"b12"
	"a21"-"e2"
	"b21"-"e2"
}
\]
Here the vertex labelled with $0 = 16 - 2\cdot 3 - 2\cdot 5$ is the proper transform $\widetilde C$ of $C$, and is obtained by fusing the corresponding hollow vertices in the resolution diagrams for the two singularities of $C$.

\subsection{Singular symplectic curves}
We are now ready to define the main objects of interest in the paper.

\begin{definition}\label{def:sing}
A \emph{singular symplectic curve} $C$ in a symplectic $4$-manifold $(X,\omega)$ is a subset $C\subset X$
 such that:
\begin{itemize}
\item there exists a finite subset $\Sing(C) \in C$ such that $C\setminus \Sing(C)$ is a smooth symplectic submanifold of $X$;
\item every $q\in \Sing(C)$ has an open neighborhood $U\subset X$, such that there exists a symplectomorphism $(U,C\cap U) \to (V, D\cap V)$, where $V\subset \C^2$, and $D$ is a complex algebraic curve.
\end{itemize}
We call $\Sing(C)$ the \emph{set of singularities} of $C$, $C\setminus \Sing(C)$ the \emph{smooth part} of $C$.
We say that $C$ is \emph{cuspidal} if all its singularities are cusps (i.e. locally irreducible).
\end{definition}

Since the singularities are locally modeled on complex curves, we have the resolutions of Section~\ref{ss:resolution}.
In particular, $C$ can be parameterized as a locally injective smooth image of a parameterization $f:\Sigma\to X$ which is an embedding away from the singular points ($\Sigma$ is the normalization). When this normalization $\Sigma$ is connected, the curve $C$ is \emph{irreducible}. When $\Sigma$ is disconnected, $C$ is \emph{reducible} and we consider it as a curve with \emph{labeled components}.

When we want to refer to the abstract topological type of the singular curve, namely the labeled components and their homology classes, and the topological types of the singularities (encoded by the links of the singularities up to isotopy), we will call it a configuration and denote it with script font. A specific \emph{symplectic realization} $C$ of a configuration $\mathcal{C}$ will be a singular symplectic curve with the homology and singularity data specified by $\mathcal{C}$.

We restrict to singularities modeled by complex curves because these are the singularity types that can occur in pseudoholomorphic curves by two results of McDuff, that we summarize in the following theorem.

%
\begin{theorem}[McDuff~\cite{McDuffSing}; see also Micallef--White~\cite{MicWh}]
	Let $C$ be a $J$-holomorphic curve in an almost complex $4$-manifold $(X,J)$.
	\begin{enumerate}
	\item If $C$ is not multiply covered, there is a neighborhood $U$ of each of its singular points such that the pair $(U,U\cap C)$ is homeomorphic to the cone over $(S^3,K_x)$ where $K_x$ is an algebraic link in $S^3$ that depends only on the germ of $C$ at $x$.
	\item There are an almost complex structure $J_0$ and a $J_0$-holomorphic curve $C_0$, such that $J_0$ is integrable near each of the singular points of $C_0$ and $(X,C)$ is homeomorphic to $(X,C_0)$; moreover, $J_0$ can be chosen to be arbitrarily close to $J$ in the $C^0$-topology and $C_0$ can be chosen to be arbitrarily close to $C$ in the $C^1$-topology.
	\end{enumerate}
\end{theorem}

Since the symplectic condition is $C^1$-open, it is preserved by $C^1$-small isotopies. It follows that any $J$-holomorphic singular curve for any $\omega$-compatible $J$ is $C^1$-close to a curve $C_0$ whose singular points have complex models. Since $C^1$-small isotopies do not affect questions of existence or uniqueness, we model singular symplectic curves with complex Darboux charts as given in Definition~\ref{def:sing}.

There are two different notions of genus for a singular curve $C$.
Recall that every singularity can be perturbed; that is it can be replaced with a Milnor fiber, locally defined by a perturbation of the equation for the singularity.
The genus of the Milnor fiber is also the 3-genus of its link, and it is a topological invariant of the singularity.
If $q$ is a singular point of $C$, with $r$ branches, we define $\mu(q)$, the \emph{Milnor number} of $q$, as the first Betti number of the Milnor fiber of the singularity of $C$ at $q$;
we define $\delta(q)$ by $2\delta(q) = \mu(q) + r - 1$. 
If $q$ is a cusp point of $C$, then $\delta(q)$ is just the genus of the Milnor fiber of the singularity.

For the following definition, we borrow the terminology from algebraic geometry.
\begin{definition}
Let $C$ be a singular symplectic curve. We call $p_g(C) := g(\Sigma)$ where $\Sigma$ is the normalization of $C$, the \emph{geometric genus} of $C$, and we will say that $C$ is \emph{rational} if its geometric genus is 0.

We call
\begin{equation}\label{e:degree-genus}
p_a(C) := p_g(C) + \sum_{p \in \Sing(C)} \delta(p)
\end{equation}
the \emph{arithmetic genus} of $C$.
\end{definition}

Since $p_a(C)$ is the genus of a symplectic smoothing of $C$, $C$ satisfies the adjunction formula
\begin{equation}\label{e:adjunction}
K_X\cdot C + C\cdot C + 2-2p_a(C) = 0.
\end{equation}
It follows from the adjunction formula that for a singularity with multiplicity sequence $[m_1,\cdots, m_n]$
\begin{equation}\label{e:deltamulti}
\delta(p) = \frac12\sum m_j(m_j-1).
\end{equation}

Our focus will be on singular symplectic curves in $\cptwo$. In this case, because $H_2(\cptwo;\Z)\cong \Z$, the homology class of a curve is determined by an integer. We will use the convention that $[\cpone]$ with the complex orientation is the positive generator which we will often denote by $h$. Note that $\langle[\omega],[\cpone]\rangle = \int_{\cpone}\omega>0$. Since all symplectic surfaces have positive symplectic area, their integral homology classes will correspond to positive integers. Extending the terminology from the algebraic case, we have the following definition.

\begin{definition}
	The \emph{degree} of a (singular) symplectic curve $C$ in $\cptwo$ is the positive integer $d$ such that $[C]=d[\cpone]\in H_2(\cptwo;\Z)$.
\end{definition}

We will focus on classifying pairs $(X,C)$ which are relatively minimal, meaning there are no exceptional spheres in $X$ that are disjoint from $C$. The following definition will be convenient.

\begin{definition}
	A singular symplectic curve $C$ is said to be \emph{minimally embedded} in a symplectic $4$-manifold $(X,\omega)$ if $X\setminus C$ contains no exceptional symplectic $(-1)$-spheres.
\end{definition}

There is a natural equivalence relation between different singular curves.
\begin{definition}\label{d:singular-iso}
	A symplectic \emph{equisingular isotopy} of singular symplectic curves is a one-parameter family of singular symplectic curves $C_t\subset (X,\omega)$ such that for each $t, t'\in [0,1]$ the curves $C_{t}$ and $C_{t'}$ have topologically equivalent singularities. If $C_t$ is reducible, we label the components of $C_t$ and require the discrete labeling to vary continuously in $t$ (i.e. the labeling is preserved by the family). 
\end{definition}


We may sometimes drop the term equisingular for brevity and simply say that $C_0$ and $C_1$ are \emph{symplectically isotopic}, with the understanding that all of the isotopies we consider in this article are equisingular.

Note that each singularity is locally uniquely determined up to equisingular isotopy by the smooth isotopy class of its link. Our definition of a symplectic singular curve required the existence of a local symplectomorphism to a complex curve singularity. Here we verify that complex curve singularities with the same topological type have the same symplectomorphism type. Thus we can take one representative complex model for each topological type of curve singularity.

\begin{lemma} \label{l:localsing}
	Suppose $C_0$ and $C_1$ are two complex curves in $(\C^2,\omega_{\rm std})$ where $(0,0)$ is a singular point for $C_i$ such that the links of the singularities at $(0,0)$ in $C_i$ for $i=0,1$ are smoothly isotopic. Then there exist an equisingular family of curves $(C'_t)_{t\in[0,1]}$ such that $C'_t$ has a singularity at the origin for each $t$, and there are neighborhoods $U_0$ and $U_1$ of $(0,0)$ such that $C_i \cap U_i = C'_i\cap U_i$ for $i=0,1$.
\end{lemma}

\begin{proof}
	We can choose two complex charts around $(0,0)$ such that, for $i=0,1$, $C_i$ has a Puiseux expansions
	\[
	y = c^i_{b_1} x^{{b_1}/m} + c^i_{b_1+1} x^{(b_1+1)/m} + \dots
	\]
	Two singularities are topologically equivalent if and only if they have the same Puiseux exponents $(m;b_1,\dots,b_k)$. 
	In terms of the Puiseux expansion, this means that $c^i_{b_j}\neq 0$ and that $c_\ell = 0$ for every $b_j < \ell < b_{j+1}$ such that $e_j \nmid \ell$. (Recall from the beginning of the section that the sequence $\{e_i\}$ is defined recursively as $e_1 = \gcd(m,b_1)$ and $e_i = \gcd(e_{i-1},b_i)$ for $i>1$.)
	Since the space of sequences $\{c_\ell\}$ satisfying these constraints and the space of local charts are both connected, we can always isotope one singularity into the other without changing the singularity type.
\end{proof}

\begin{remark}
	For smooth symplectic curves in any symplectic $4$-manifold $(X,\omega)$, any 1-parameter family of curves $C_t$ is induced by an ambient symplectic isotopy. Namely there exist symplectomorphisms $\Psi_t:(X,\omega)\to (X,\omega)$ with $\Psi_0= \rm{Id}_X$ and $\Psi_t(C_0)=C_t$.
	
	On the other hand, singular curves carry local symplectomorphism invariants beyond their topological type. Therefore there can be a $1$-parameter equisingular family of symplectic (or even complex) curves which may not be related by an ambient symplectic isotopy.
	
	The differences stem from the preservation by symplectomorphisms of ``angles'' between collections of symplectic planes intersecting at the origin. For a cusp singularity, the angles are not as easily visible, but after blowing up to a normal crossing resolution, the position of the intersection points of the exceptional divisors can vary. After blowing down one exceptional divisor, different relative intersection positions of other curves with that exceptional divisor will change the symplectic angles between the resulting transversally intersecting curves. These angles will have an effect on the geometry of the cusp when you blow down completely to undo the resolution.
	
	This is the analogue of the situation in the complex category, where, for instance, there are families of ordinary quadruple points in $(\C^2,0)$ such that different elements are related by deformations but are not isomorphic. These are distinguished by the cross-ratio of the four points the four branches determine in $\mathbb{P}(T_{(0,0)}\C^2) \cong \cpone$.
	
	In order to see the symplectic analogue of the complex statement, we reduce to the linear case, and argue by dimension-counting.
	Consider two ordinary $n$-tuple points in $(\C^2,(0,0))$ for some $n>1$. If there is a global symplectomorphism $\psi: (U,(0,0)) \to (V,(0,0))$ between open sets containing them that sends one singularity to the other, then the differential $d\psi_{(0,0)}$ is a linear symplectomorphism of $(T_{(0,0)}\C^2,\omega)$ that identifies the $n$-tuple of tangent spaces to the branches of the two singularities.
	The space ${\rm Sp}(4)$ of linear symplectomorphisms of $(\R^4,\omega=\omega_{\rm std})$ has finite dimension $d_{\rm ls}$; the Grassmannian ${\rm Gr}_\omega(4,2)$ of symplectic planes in $(\R^4,\omega)$ has finite dimension $d_{\rm Gr}$, too.
	Since the diagonal action of ${\rm Sp}(4)$ on ${\rm Gr}_\omega(4,2)^{\times n}$ is smooth for each $n$, the dimension of the orbits is at most $d_{\rm ls}$.
	The space of $\omega$-planes intersecting positively and transversely is open and non-empty in ${\rm Gr}_\omega(4,2)^{\times n}$.
	This proves that, as soon as $d_{\rm ls} < nd_{\rm Gr}$, the action cannot be transitive on any connected component, and in particular that there are isotopic ordinary $n$-tuple points that are isotopic but not ambient-symplectomorphic.
	In fact, one can compute $d_{\rm ls}=10$ and $d_{\rm Gr} = 4$, since the symplectic condition is open, and therefore $d_{\rm Gr}$ is the same as the dimension of the Grassmannian of 2-planes in $\R^4$. This shows that $n=3$ already suffices.
\end{remark}

\subsection{The cuspidal contact structure and its caps}

\begin{theorem} Given a singular symplectic curve $C$, with specified singularity types and normal Euler number $s$.
Suppose $s=C\cdot C>0$. Then there exists a symplectic manifold $(N,\omega)$ with concave boundary such that $N$ is a regular neighborhood of $C$, such that $[C]^2=s$ in $N$. Moreover, every symplectic embedding of $C$ into a symplectic manifold $(X,\omega)$ has a concave neighborhood inducing the same contact boundary.
\end{theorem}

\begin{proof}
Let $\overline C$ denote the normal crossing resolution of $C$ which is a collection of transversally intersecting curves with specified genus, intersections, and normal Euler numbers. The intersections are generic double points.

A regular neighborhood $U$ of $\overline C$ can be constructed as a plumbing of surfaces.
Moreover, by undoing the normal crossing resolution, we see that $U$ is the blow-up of a regular neighborhood $N$ of $C$. In particular the intersection form of $U$ is $\langle s\rangle \oplus \langle -1\rangle^{\oplus n}$ for some $n\ge 0$.
	
Because $s = C\cdot C > 0$, the intersection form of the plumbing is not negative definite.
We claim that the inclusion $\partial U \to U$ induces the trivial map on the second cohomology group: $H^2_{\rm dR}(U) \to H^2_{\rm dR}(\partial U)$.
This follows from the long exact sequence of the pair $(U,\partial U)$:
in fact, the map $H^2(U,\partial U;\Z) \to H^2(U;\Z)$ is presented by the intersection form of $U$;
since the latter is non-degenerate, the map $H^2_{\rm dR}(U,\partial U) \to H^2_{\rm dR}(U)$ is an isomorphism.
It follows that $\omega_U$ is exact on $\partial U$.
(The exactness assumption is automatically satisfied if $C$ is rational and cuspidal, since in that case $\partial U$ is a rational homology sphere, and therefore $H^2_{\rm dR}(\partial U) = 0$.)

Construct a concave symplectic structure on $U$, such that $\overline{C}$ is a symplectic submanifold using the construction of \cite{GayStipsicz} adapted to the concave case~\cite[Theorem~1.3]{LiMak}. Blow down to obtain the symplectic structure on $N$.

The existence of an equivalent concave neighborhood in any symplectic embedding follows similarly from blowing up to the normal crossing resolution and using the analogous result of \cite{GayStipsicz,LiMak}.
\end{proof}

\begin{definition}
Given a singular symplectic curve $C$, defined by the collection of its singularities and its self-intersection, we define $(Y_C,\xi_C)$ by $Y_C = -\partial N$, and $\xi_C$ as the contact structure induced by the concave symplectic structure.
\end{definition}

Note that the contact structure depends only on the topological types of the singularities of $C$, its geometric genus, and its self-intersection.
Also, observe that, in order to define $(Y_C,\xi_C)$, we are using that there exists a normal crossing resolution.

This is a generalization of the contact structures introduced by Chen (see~\cite[Section 4]{Chen}) for rational curves. The focus in~\cite{Chen} was on curves in $\cptwo$ (or in symplectic manifolds with the same algebraic topology). Chen described the contact structure using transverse symplectic handle attachment, instead of using the normal crossing resolution. The two contact structures are actually the same because they are both the canonical contact structure on the boundary of a small concave neighbourhood of the curve, though we will not use this fact here.

%
%

\begin{remark}
We recall that, if $C$ is an algebraic curve in $\cptwo$ of degree $d$, then the contact structure $\xi_C$ is defined algebraically.
It is a hyperplane section of the Veronese embedding of the degree $d$, and hence its complement inherits a Stein structure from $\C^M = \P^M \setminus H$.
\end{remark}

\begin{remark}
	Note that in our setup, $C$ need not be realised as a symplectic curve \emph{embedded} in a closed K\"ahler surface. Instead its regular neighborhood can be defined using its self-intersection number, genus, and singularities. We will exhibit examples of such curves which do not embed in a closed symplectic manifold (see Section~\ref{s:lowdegrees}).
\end{remark}

\subsection{Topological obstructions}
	We mention here two techniques used previously to obstruct rational cuspidal curves in $\cptwo$ which have topological interpretations and can thus obstruct symplectic curves (not just complex algebraic curves).
	
	The semigroup condition was first discovered by Borodzik and Livingston~\cite{BorodzikLivingston}.
	Recall that to each curve singularity $x$ there is an associated semigroup $\Gamma_x$, that records the local multiplicities of intersections of germs of curves with the singularity;
	for instance, for a singularity of type $(p,q)$, the semigroup is generated by $p$ and $q$.
	By convention, we declare that $0$ belongs to $\Gamma_x$, so $0$ is the first element of the semigroup, and the multiplicity is the second.
	We define the counting function $R_x$ of $\Gamma_x$ as $R_x(n) = \#(\Gamma_x \cap (-\infty,n-1))$, and the minimum convolution $F\diamond G$ of two functions as $(F\diamond G)(n) = \min_{k\in \Z} F(k) + G(n-k)$;
	given a rational cuspidal curve $C$ of degree $d$, with cusps $x_1,\dots,x_k$, and associated $R$-functions $R_1,\dots, R_k$, let $R = R_1\diamond\dots\diamond R_k$.
	Then,~\cite[Theorem 6.5]{BorodzikLivingston} asserts that, for each $j=-1,0,\dots,d-2$,
	\begin{equation}\label{e:semigroup}
	R(jd+1) = \frac{(j+1)(j+2)}2.
	\end{equation}
	Note that this is a smooth obstruction: it obstructs the existence of a rational homology 4-ball which glues by a diffeomorphism of the boundaries to a regular neighbourhood of $C$.
	
	\begin{remark}
		In fact, the function $R$ only depends on the multiplicity multi-sequence, not the way it is split into multiplicity sequences for individual cusps~\cite[Theorem 1.3.12]{BodnarNemethi}.
		Moreover, one can see that the semigroup condition is a generalization of B\'ezout's theorem: see, for instance,~\cite[Proposition~3.2.1]{FdBLMHN2}.
	\end{remark}
	
	\begin{example}\label{ex:33}
		In the case of quintics, the only multiplicity multi-sequence that is excluded by the semigroup criterion is $[3,3]$, which corresponds to two rational cuspidal curves, one with a singularity of type $(3,7)$ and one with two singularities of type $(3,4)$.
	\end{example}
	
	\begin{remark}\label{r:involutive}
		There is a strengthening of~\cite{BorodzikLivingston} to curves of odd degree, using involutive Heegaard Floer homology, due to Borodzik, Hom, and Schinzel~\cite{BorodzikHom};
		in particular, by~\cite[Section 5.1]{BorodzikHom}, the curves with cusps of types $\{[2,2,2], [2,2,2]\}$ and $\{[2,2,2,2,2], [2]\}$ are also obstructed.
	\end{remark}
	
	We call \emph{spectrum semicontinuity} an inequality on the Levine--Tristram signature and nullity functions of the links, $\sigma_\bullet, \eta_\bullet: S^1\to \Z$.
	It is related to the algebro-geometric spectrum semicontinuity, formulated in terms of Hodge-theoretic data, and then recast in more topological terms by Borodzik and N\'emethi~\cite[Corollary 2.5.4]{BorodzikNemethi}.
	From a curve $C$, by choosing a generic line $\ell$, we obtain a genus-$p_g(C)$ cobordism in $S^3\times [0,1]$ from the connected sum $K$ of links of all singularities of $C$, to the torus link $T(d,d)$;
	this cobordism is obtained from removing a neighborhood of a path connecting all singularities of $C$, and a neighborhood of the line $\ell$.
	Suppose that $C$ is rational and its singularities are $k$ cusps. Then $K$ is a knot, and for every $\zeta \in S^1_! \subset S^1$ we have:
	\begin{equation}\label{e:spectrum}
		|\sigma_{T(d,d)}(\zeta) - \sigma_K(\zeta)| + |\eta_{T(d,d)}(\zeta)- \eta_{K}(\zeta)| \le d-1.
	\end{equation}
	Here $S^1_!$ is the unit circle with all \emph{Knotennullstellen} (roots of integer polynomials such that $p(1) = 1$) removed; that is,
	\[
	S^1_! = S^1 \setminus \{\alpha \mid \exists p(t) \in \Z[t], p(1) = 1, p(\alpha) = 0\}.
	\]
	The inequality above was essentially proved by Nagel and Powell~\cite{NagelPowell}, who also introduced the notation and terminology for $S^1_!$ (see also~\cite[Theorem~2.12]{Conway-survey}).
	
	To be concrete: all transcendental complex numbers of norm $1$ belong to $S^1_!$, so the set of \emph{Knotennullstellen} has measure $0$. A root of unity belongs to $S^1_!$ if and only if its order is a prime power; for this it suffices to evaluate cyclotomic polynomials at $1$: $\Phi_{p^r}(1) = p$ for every prime $p$ and positive integer $r$, while $\Phi_n(1) = 1$ whenever $n$ has at least two distinct prime factors.
	
	The obstruction~\eqref{e:spectrum} is topological, once we assume that there exists a locally flatly embedded sphere $\ell$ in the homology class $h$ that intersects $C$ transversely and positively in $d$ points.
	This assumption is satisfied when $C$ is a symplectic curve, because we can choose an almost complex structure such that $C$ is $J$-holomorphic (by Lemma~\ref{l:Jcompatible}) and then choose a generic $J$-holomorphic line.
	Then each $J$-holomorphic line intersects $C$ positively, and the space of $J$-holomorphic lines has (real) dimension $4$. The $J$-holomorphic lines which intersect $C$ at a singular point or tangentially has positive codimension (real codimension $2$) in the space of all $J$-holomorphic lines. Therefore there is a $J$-holomorphic line which intersects $C$ positively and transversally in $d$ points.
	
	\begin{example} \label{ex:2727}
		The spectrum semicontinuity obstructs the existence of a quintic with two cusps of type $(2,7)$. As mentioned in Remark~\ref{r:involutive} above, this was also obstructed by using Floer-theoretic techniques.
		The spectrum semicontinuity is a stronger obstruction, because it holds in the topologically locally flat category, rather just the smooth category.
	\end{example}


\section{Pseudoholomorphic techniques}\label{s:tools}
{
\subsection{Proper transforms in symplectic blow-ups}\label{ss:propertransform}

	{
	In the symplectic category, blowing up and blowing down take on a slightly different character than they do in algebraic geometry \cite{McDuffBlowups}. To perform a symplectic blow-up at a point $p$ in a symplectic manifold $(X,\omega)$, we first choose a Darboux chart centered at $p$. Remove a ball inside this chart of radius $\lambda$, and collapse the boundary by the Hopf fibration to obtain the exceptional divisor. Alternatively, remove a ball of radius $\lambda+\varepsilon$ centered at $p$ and replace it with an $\varepsilon$-neighborhood of a copy of $\cpone$ of symplectic area $\lambda$ in $\mathcal{O}(-1)$. The symplectic form on the ring of the ball between radius $\lambda$ and $\lambda+\varepsilon$ agrees with the symplectic form on $\mathcal{O}(-1)$ when the zero section is a symplectic submanifold of area $\lambda$. A symplectic blow-down reverses the operation, replacing an $\varepsilon$-neighborhood of an exceptional divisor of symplectic area $\lambda$ with a ball of radius $\lambda+\varepsilon$ (or equivalently deleting the exceptional divisor of symplectic area $\lambda$ and replacing it by a closed ball of radius $\lambda$).
		
	\begin{remark}
	In the algebraic geometric (or smooth) category, the blow-up has a well-defined effect on both the variety and its subvarieties. The effect on the subvarieties gives rise to the notions of the total and proper transforms discussed in Section~\ref{ss:resolution}.
	The proper transform turns first order data of the subvariety at the point $p$ into zeroth order information (and second order information into first order information, etc). Because the symplectic blow-up deletes an entire ball instead of just a point, we need to define the total and proper transforms of (singular) symplectic curves in symplectic $4$-manifolds with a bit more care.
	\end{remark}
	
	When we perform a symplectic blow-up at a point, we will always choose the radius $\lambda+\varepsilon$ of the symplectic ball to be sufficiently small such that every sphere $S_r(p)$ of radius $0<r\leq \lambda+\varepsilon$ intersects the symplectic curves transversally. Note that using the radial vector field in the Darboux chart as a Liouville vector field, the spheres $S_r(p)$ are contact-type hypersurfaces.
	
	Suppose now that we have a symplectic curve $C$ in $X$, and that we blow $X$ up at a (possibly singular) point $p$ on $C$.
	The transverse intersections of $C$ with the contact spheres are transverse links $T_r$ in the spheres $S_r(p)$. Note that if we identify spheres of different radii by a rescaling contactomorphism, the transverse links $T_r$ in the spheres of different radii $0<r<\lambda+\varepsilon$ are all transversally isotopic to each other. If $p$ is a smooth point, the transverse link $T_r$ will be the standard unknot of maximal self-linking number. If $p$ is a singularity, $T_r$ will be transversally isotopic to an algebraic link.
	
	For the algebraic proper transform in the blow-up, the exceptional divisor $E$ replaces the point $p$, and the proper transform $\widetilde C$ intersects $E$ according to its tangent derivative information at that point.
	Thus the proper transform is a finite-point compactification in the algebraic blow-up of the family of transverse links $T_r\subset S_r(p)$. We would like to have a similar proper transform in the symplectic blow-up, but in this case we delete a ball of radius $\lambda$ instead of only a point. 
	Therefore, the links $T_r$ for $0<r\leq \lambda$ would be cut out. We cannot guarantee that the transverse links $T_r$ for $\lambda<r<\lambda+\varepsilon$ will finite-point compactify as $r\to \lambda$ with the same diffeomorphism type as the family $T_r$ for $0<r<\lambda+\varepsilon$ as $r\to 0$. 
	
	To overcome this issue, we will squeeze the entire transverse isotopy $T_r$ for $0<r<\lambda+\varepsilon$ into the collar where $\lambda<r<\lambda+\varepsilon$ by reparametrizing the radial coordinate by a smooth, strictly increasing function $\rho$ such that $\rho(\widetilde{r})=\widetilde{r}-\lambda$ near $\widetilde{r}=\lambda$ and $\rho(\widetilde{r})=\widetilde r$ near $\lambda+\varepsilon$.
	In other words, in the symplectic blow-up, for $\lambda<\widetilde{r}<\lambda+\varepsilon$, the proper transform intersects the sphere $S_{\widetilde{r}}$ in a link $\widetilde{T}_{\widetilde{r}}$ defined by  
	\[
	\widetilde{T}_{\widetilde{r}}:=T_{\rho(\widetilde{r})}.
	\]
	As $\widetilde{r}\to \lambda$, $\rho(\widetilde{r})\to 0$. 
	Therefore this family extends via a finite-point compactification in the exceptional divisor in the same manner as the algebraic proper transform.
	Note that we do \emph{not} change the symplectic structure on the ambient symplectic $4$-manifold, we only modify the symplectic curve.
	
	\begin{lemma}
	The surface $\bigcup_{\lambda < \widetilde{r} < \lambda+\varepsilon} {\widetilde T}_{\widetilde r}$ is a symplectic curve.
	\end{lemma}
	
	\begin{proof}
	We want to prove that the tangent spaces to the surface are still symplectic subspaces; the idea is that this is true because the bases for the tangent spaces only differ by a positive scaling factor $\rho'(r)$.
	Namely, if $\phi(r,t)$ denotes a polar parametrization of the original surface such that $\phi(r,t)=(T_r(t),r)$, then $\widetilde{\phi}(r,t) = (T_{\rho(r)}(t),\rho(r))$ is a parametrization for the new surface and
	\[
	\frac{\partial{\widetilde \phi}}{\partial r}(r,t) = \rho'(r)\frac{\partial \phi}{\partial r}(\rho(r),t), \qquad \frac{\partial{\widetilde \phi}}{\partial t}(r,t) = \frac{\partial \phi}{\partial t}(\rho(r),t)
	\]
	so the value of $\omega$ must be positive on an oriented basis for the curve parametrized by $\widetilde\phi$ since it is positive on an oriented basis for the curve parametrized by $\phi$.
	\end{proof}

	A symplectic blow-down reverses this procedure. Deleting an $\varepsilon$-neighborhood of an exceptional sphere of weight $\lambda$ (i.e. the symplectic area is $\pi\lambda^2$), we replace it with a symplectic ball of radius $\lambda+\varepsilon$. To define the image of a symplectic curve which intersects the exceptional divisor over $B_{\lambda+\varepsilon}(p)$ we reverse the squeezing procedure above to stretch the transverse link family back out. Namely we take the surface defined by the union of the center $p$ of the ball together with the transverse links $T_r\subset S_r(p)$ where
	\[
	T_r:=\widetilde{T}_{\rho^{-1}(r)}.
	\]
	
	Note that the definition of the proper transform depends on the choice of the Darboux chart, of $\varepsilon$ and of $\rho$. However, the \emph{symplectic isotopy class} is \emph{independent} of all these choices.
	
	\begin{prop}
		The symplectic isotopy class of the proper transform of a singular symplectic curve $C\subset (X,\omega)$ is independent of the choice of the Darboux ball, of $\varepsilon$ and $\rho$, and is related to the algebraic proper transform by a diffeomorphism supported in a neighborhood of the exceptional divisor.
	\end{prop}
	
	\begin{proof}
		We start by showing independence of $\varepsilon$.
		Without loss of generality assume that $\varepsilon_1>\varepsilon_0$. If one blow-up/down is performed using $\varepsilon_0$ with function $\rho_0$ and the other using $\varepsilon_1$ with function $\rho_1$, we may extend $\rho$ to the interval $[\lambda,\lambda+\varepsilon_1]$ by extending by the identity on the interval $[\lambda+\varepsilon_0,\lambda+\varepsilon_1]$. Therefore, without loss of generality we may assume $\varepsilon_0=\varepsilon_1$ and simply show independence of $\rho_i$.
		
		To find the symplectic isotopy connecting the proper transforms for different choices of $\rho_i$, define a $1$-parameter family of proper transforms using $\rho_s(r) = (1-s)\rho_0(r)+s\rho_1(r)$. The required characteristics of $\rho$ are convex conditions so $\rho_s(r)$ defines a symplectic proper transform for $s\in[0,1]$ interpolating between the two choices of proper transform.
		
		We now note that the space of embeddings of Darboux balls (with varying radius) in any connected symplectic manifold is itself connected.
		Thus, in order to prove independence (up to isotopy) of the Darboux ball, we can suppose that we have a 1-parameter family $D_t$ of Darboux balls centered at $p$, parametrized by symplectomorphisms $(\phi_t\colon D^4_r \to X)_{t\in[0,1]}$ connecting two Darboux balls $D_0$ and $D_1$. (We can assume that the balls have the same volume by passing to a subfamily.)
		Then, by choosing $\rho + \varepsilon < r$, and applying the recipe above, we obtain a 1-parameter family $\widetilde{C}_t$ of symplectic proper transforms that varies smoothly with $t$, i.e. a symplectic isotopy between $\widetilde{C}_0$ and $\widetilde{C}_1$.
		
		To see that the symplectic proper transform is related to the algebraic proper transform by a diffeomorphism, simply re-stretch out the ring between radius $\lambda$ and $\lambda+\varepsilon$ in the blow-up by reparametrizing the radial coordinate by $\rho^{-1}$ (and shrink by reparametrizing by $\rho$ in the blow-down). Note, this diffeomorphism will not be a symplectomorphism.
	\end{proof}
	
	}

\subsection{Pseudoholomorphic curves}
	{
	Here we prove two technical lemmas that we will use throughout.
			\begin{lemma}\label{l:Jcompatible}
				Let $(M,\omega)$ be a symplectic $4$-manifold and $C\subset M$ be a singular symplectic surface.
				The space $\mathcal{J}^{\omega}(C)$ of almost complex structures $J$ which are compatible with the symplectic structure $\omega$ such that $C$ is $J$-holomorphic is non-empty and contractible.
			\end{lemma}
				
				
				\begin{proof}
					
					The proof is a mild upgrade of a standard proof that the space of almost complex structures compatible with a given symplectic structure is contractible. There are multiple ways to prove this classical fact, and here we follow one given in~\cite[Proposition 2.63]{McDuffSalamon}. 
					
					The first observation to make is that the condition that $J\in \mathcal{J}^{\omega}(C)$ is a point-wise condition. Namely, at each point $x\in M$, $J_x$ must be an $\omega_x$-compatible almost complex structure (at the vector space level), and whenever $x\in C$, we also require that $J$ preserves all tangent spaces to $C$ at $x$. Note that at a singular point in $C$, there may be more than one branch defining finitely many distinct tangent spaces, however each branch does have a well-defined tangent space (Remark~\ref{r:tang}).
					Note that at a critical point of a $J$-holomorphic map $u:\Sigma\to M$, $du=0$ so the $J$-holomorphic condition $du\circ j=J\circ du$ is actually vacuous at that point. However, at nearby smooth points, the $J$-holomorphic condition does require that $J$ preserves the tangent space to $C$, so if $J$ varies continuously over points in $M$, since the tangent space at a singular point is the limit of nearby tangent spaces, the requirement that $J$ preserve the tangent spaces at singular points will be automatically satisfied.
					
					We want to choose $J$ as a continuous section of the bundle $\pi: E\to M$ whose fiber $E_x$ over $x$ is the space of $\omega_x$ compatible almost complex structures on $T_xM$, such that the section is constrained over points $x\in C$ to the subset $E_x(C)\subset E_x$ of compatible almost complex structures on $T_xM$ which preserve the tangent space(s) to $C$. We will show that the contraction of $E_x$ to a point from~\cite[Proposition 2.50]{McDuffSalamon} preserves the subset $E_x(C)$. Therefore, the space $\mathcal{J}^\omega(C)$ is contractible.
					
					One way to show that $E_x$ is contractible is by showing that $E_x$ is homeomorphic to the space of symmetric positive definite symplectic matrices~\cite[Proposition 2.50 (i), (iii)]{McDuffSalamon}. Fixing a standard symplectic basis on $T_xM$, let $J_0$ be the matrix for the standard complex structure. Then we can identify any almost complex structure $J$ on $T_xM$ with a symmetric, positive definite, symplectic matrix $P$ by $P = -J_0J$ (and $J=J_0P$). Here, we will choose the basis on $T_xM$ with some care along points of $C$. Near singular points of $C$, there is by definition, a symplectic identification with a subset of $\C^2$ such that $C$ is identified with a complex curve (with respect to the standard complex structure). We will choose the basis of $T_xM$ compatibly with the standard coordinates on $\C^2$ under this identification so that at singular points, $J_0$ preserves all tangent spaces to $C$ in $T_xM$. Along smooth points of $C$, the basis for $T_xM$ should extend a symplectic basis for $T_xC$. Therefore $J_0$ preserves the tangent space(s) of $C$ at every point. Away from $C$ Darboux coordinates can be extended arbitrarily.
					
					A deformation contraction from the space of symmetric positive definite symplectic matrices to a point is given by sending $P$ to $P^{1-t}$ for $t\in [0,1]$. It is proven in~\cite[Lemma 2.21]{McDuffSalamon} that $P^{1-t}$ is a (symmetric, positive definite) symplectic matrix whenever $P$ is. The contraction deforming $J=J_0P$ to $J_0P^{1-t}$ gives the contraction of $E_x$. Therefore, it suffices to check that this contraction preserves the subset $E_x(C)$ of almost complex structures $J$ on $T_xM$ which preserve the tangent spaces of $C$.
					
					Since $J_0$ preserves the tangent spaces of $C$ by assumption on the choice of frame for $TM$ along $C$, $J$ preserves the tangent spaces of $C$ if and only if the corresponding matrix $P$ does. Since $P$ is symmetric, it is diagonalizable. A subspace of $T_xM$ is an invariant subspace of $P$ if and only if it is spanned by eigenvectors of $P$. Similarly a subspace is invariant under $P^{1-t}$ if and only if it is spanned by eigenvectors of $P^{1-t}$. Since the eigenvectors of $P$ are the same as the eigenvectors of $P^{1-t}$, their invariant subspaces are the same. Therefore if $P$ preserves the tangent spaces of $C$, $P^{1-t}$ does as well. Thus the contraction preserves the subset $E_x(C)$.
				\end{proof}
						
		The strength of using $J$-holomorphic curves is that we have much greater control over their geometric intersections. Two general symplectic surfaces may intersect with a canceling pair of positive and negative intersections, but in this case they could not be realized as $J$-holomorphic simultaneously for the same almost complex structure $J$. When $(M,\omega)$ is a 4-manifold with a compatible almost complex structure $J$, the orientation induced by $\omega$ and $J$ agree. Two transversally intersecting $J$-holomorphic curves can be easily seen to have positive intersections because the almost complex structure induces complex orientations on each of the curves at an intersection point, which adds up to the positive orientation on the $4$-manifold induced by $J$. More generally, any (not necessarily transverse) intersection between simple $J$-holomorphic curves (possibly with singularities) contributes positively (see~\cite[Section 2.6, Appendix E.2]{McDuffSalamonBig}).
		
		\begin{lemma}[\cite{Mc,Li}] \label{l:blowdown}
			Let $h,e_1,\dots, e_N$ be the standard orthogonal basis for $H_2(\cptwo \# N \cptwobar)$ with $h^2=1$ and $e_i^2=-1$. Suppose $\mathcal{C}$ is a configuration of positively intersecting symplectic surfaces in $\cptwo \# N\cptwobar$ {including a sphere $C_0$ representing the class $h$}. Let $e_{i_1},\dots,e_{i_\ell}$ be the exceptional classes which have non-negative algebraic intersections with each of the symplectic surfaces in the configuration $\mathcal{C}$. Then there exists {a sequence of} exceptional spheres representing the classes $e_{i_1},\dots, e_{i_\ell}$ respectively such that any geometric intersections of the exceptional spheres with $\mathcal{C}$ are positive.
		\end{lemma}
		
		Blowing down exceptional spheres using this lemma, together with exceptional spheres appearing in a configuration, eventually we will blow down spheres in all $e_i$ classes and reach a configuration in $\cptwo$.
		
		\begin{proof}
			{Fix an almost complex structure $J\in \mathcal{J}^\omega(\mathcal{C})$. By the proof of~\cite[Lemma 4.5]{Li}, there exists an embedded $J$-holomorphic sphere $E$ with self-intersection $-1$ with algebraic intersection $0$ with $C_0$. Since $E$ and $\mathcal{C}$ are both $J$-holomorphic, they satisfy positivity of intersections. Since $[C_0]=h$, $[E]\cdot h=0$. The only $H_2$ classes represented by symplectic spheres with square $-1$ which are orthogonal to $h$ are the $e_i$ classes. Since $E$ intersects $\mathcal{C}$ non-negatively, it must represent $e_{i_j}$ for some $1\leq j\leq \ell$. Repeat this until all such classes have been blown down to obtain the sequence.}
		\end{proof}

		The order of the $e_i$ which we choose to $J$-holomorphically blow down does depend to some extent on the configuration $\mathcal{C}$ because if one of the surfaces in $\mathcal{C}$ represents the class $e_{i_0}-e_{i_1}-\dots-e_{i_k}$, then we cannot blow down a $J$-holomorphic exceptional sphere in the class $e_{i_0}$ until we have blown down exceptional spheres in the classes $e_{i_1},\dots, e_{i_k}$ first. This is because $e_{i_0}$ has negative algebraic intersection number with $e_{i_0}-e_{i_1}-\dots -e_{i_k}$. If $k=0$ then the surface in $\mathcal{C}$ represents the class $e_{i_0}$ so it can itself act as the $J$-holomorphic exceptional sphere. {Additionally, in some cases, the order of the $e_i$ which are blown down depends on the choice of $J$. This is because for some non-generic $J$, it may be that there is a $J$-holomorphic sphere representing a class $e_{i_0}-e_{i_1}-\cdots - e_{i_k}$. (Note by part~\ref{e:exceptional-disjoint} of Lemma~\ref{l:hom} below, this is the only other class that can be represented by a symplectic sphere which is orthogonal to $h$.) Then it we would not be able to find a $J$-holomorphic sphere representing $e_{i_0}$ until we had blown down $J$-holomorphic spheres in classes $e_{i_1},\dots, e_{i_k}$ first. Furthermore, the points in the image of the blow-downs of the spheres representing $e_{i_0}, e_{i_1},\dots, e_{i_k}$ would all coincide. However, in practice, many times this situation can be ruled out based on the homology classes of the surfaces in the configuration $\mathcal{C}$. This is because any $J$-holomorphic curve must have non-negative algebraic intersection with each component of $\mathcal{C}$ unless it is itself a component of $\mathcal{C}$. (Note this is true for any $J$ for which $\mathcal{C}$ is $J$-holomorphic and does not depend on any genericity assumption.) Therefore, in many but not all instances, we can ensure that exceptional spheres for distinct classes $e_i$ can actually be realized and blown down disjointly.}

	}

\subsection{Embeddings of plumbings into $\cptwo\#N\cptwobar$}	
	{\label{s:caps}
	Suppose $P$ is a plumbing of symplectic spheres, such that one of the spheres has self-intersection $+1$. In our context, this will typically be a neighborhood of the normal crossing resolution of a rational cuspidal curve (or possibly a further blow-up). A theorem of McDuff strongly restricts the closed symplectic manifolds in which  $P$ can symplectically embed.
	
	\begin{theorem}[\cite{Mc}] \label{thm:mcduff}
		If $(X,\omega)$ is a closed symplectic $4$-manifold and $C_0\subset X$ is a smooth symplectic sphere of self-intersection number $+1$, then there is a symplectomorphism of $(X,\omega)$ to a symplectic blow up of $(\cptwo,\lambda\omega_{\rm{FS}})$ for some $\lambda > 0$, such that $C_0$ is identified with $\cpone$.
	\end{theorem}
	
	More generally, McDuff's result shows that if the plumbing contains a sphere of square $p\neq 4$, $p>0$ then the symplectic manifold is a blow-up of the $p^{\rm th}$ Hirzebruch surface and the sphere is identified with the associated $0$-section.
	A sphere of square $4$ might be the conic in $\cptwo$ or the sphere representing $2[S^2\times \{*\}]+[\{*\}\times S^2]$ in $S^2\times S^2$.
	A sphere of square $0$ is a fiber in a ruled surface.
	We will typically work in cases where we can guarantee that our symplectic 4-manifold is a blow-up of $\cptwo$ (or occasionally in $S^2\times S^2$).
	(Note that $S^2\times S^2\#\cptwobar \cong \cptwo\#2 \cptwobar$.)
	
	We will now discuss how to classify all symplectic embeddings of $P$ into $\cptwo\#N\cptwobar$. A symplectic embedding of a rational cuspidal curve is equivalent (by a sequence of blow-ups supported in a neighborhood of the rational cuspidal curve) to a symplectic embedding of its normal crossing resolution plumbing $(U,\omega_U)$. 
	
	We classify embeddings of $P$ into $\cptwo\#N \cptwobar$ in two steps (similar arguments appear in many classification of fillings results, starting with Lisca~\cite{Li}). First, we determine the possibilities for the map on second homology induced by the embedding. Since the core spheres of the plumbing form a basis for $H_2(P)$, we just need to classify the possible classes in $H_2(\cptwo\#N\cptwobar)$ that these symplectic spheres can represent. The restrictions here are given by the adjunction formula and the intersection form on $P$, as well as the constraint that the $+1$-sphere in $P$ is identified with $\cpone\subset \cptwo\# N\cptwobar$. Next, for each possible adjunctive embedding $H_2(P)\to H_2(\cptwo\# N \cptwobar)$, we will classify the geometric realizations of embeddings $P\to \cptwo\# N\cptwobar$ up to symplectic isotopy. Typically, we will show that such a geometric embedding is unique or that it cannot exist. This will be done by supposing that we have such a geometric embedding into $\cptwo\#N\cptwobar$, keeping track of the configuration formed by the core symplectic spheres of the plumbing, and then blowing down using Lemma~\ref{l:blowdown}. We then look at the possible images of this configuration after blowing down to $\cptwo$. This will generally be a reducible configuration, which we will then try to classify up to symplectic isotopy. This classification of reducible symplectic configurations in $\cptwo$ is the subject of Section~\ref{s:rediso}.
	

	Now we analyze the possible homology classes represented by the spheres in $P$. A symplectic surface $\Sigma$ in a symplectic 4-manifold $(X,\omega)$ satisfies the adjunction formula~\eqref{e:adjunction}. When $X=\cptwo\#N\cptwobar$, $K_X=-c_1(X) = -3H+E_1+\dots +E_N$. Here our convention is that $h$ is the class of $\cpone$, with dual $H$ and $e_i$ is represented by the $i^{\rm th}$ exceptional sphere with dual $E_i$. The following lemma is a generalization of~\cite[Propositions~4.4]{Li}
	
	\begin{lemma}\label{l:adjclass}\label{l:hom}
		Suppose $\Sigma$ is a smooth symplectic sphere in $\cptwo\#N\cptwobar$ intersecting $\cpone$ non-negatively. Then writing $[\Sigma]=a_0h+a_1e_1+\dots+a_Ne_N$ (so $a_0\geq 0$), we have:
		\begin{enumerate}
			\item \label{adj} $\sum (a_i^2+a_i)=2+a_0^2-3a_0$.
			\item \label{e:exceptional-disjoint} If $a_0=0$, there is one $i_0$ such that $a_{i_0}=1$ and all other $a_i\in \{0,-1\}$.
			\item \label{sgn} If $a_0\neq 0$, then for all $i\geq 1$, $a_i\leq 0$.
		\end{enumerate}
		Some particular cases which we will use often are:
		\begin{enumerate}\setcounter{enumi}{3}
			\item If $a_0=1$ or $a_0=2$, $a_i\in \{0,-1\}$ for all $1\leq i\leq N$.
			\item If $a_0=3$, then there exists a unique $i_0$ such that $a_{i_0}=-2$, and $a_i\in \{0,-1\}$ for all other $i$.
		\end{enumerate}
		The self-intersection number of $\Sigma$ can be used to compute how many $a_i$ have coefficient $0$ versus $-1$.
	\end{lemma}
	
	\begin{proof}
		Since $\Sigma$ is a sphere ($g=0$), the adjunction formula gives
		\[3a_0+a_1+\dots +a_N = 2+a_0^2-a_1^2-\dots-a_N^2\]
		which can be rearranged to item~\ref{adj}. Note that $a_i^2+a_i\geq 0$ since $a_i\in \Z$, so the coefficient $a_0$ determines a bound on the possible coefficients $a_i$.
		
		{For item~\ref{e:exceptional-disjoint}, when $a_0=0$, we have $\sum (a_i^2+a_i)=2$. Note that $a_i^2+a_i\geq 0$ for all integers $a_i$, with equality if and only if $a_i\in \{0,-1\}$. Furthermore, $0<a_i^2+a_i\leq 2$ only when $a_i\in\{1, -2\}$. Therefore there there is one index $i_0$ with $a_{i_0}=1$ or $-2$ and the rest of the coefficients are $0$ or $-1$. If $a_{i_0}=-2$, then the symplectic area of this class would be negative (since we choose the $e_i$ basis to be classes with positive symplectic area), and thus could not be represented by a symplectic surface.}
		
		To prove item~\ref{sgn}, we use positivity of intersections. Fix an index $0 < i_0 \le N$ and suppose first that $a_{i_0} > 0$. Since $\Sigma$ and $\cpone$ are symplectic with non-negative intersections, there is an almost complex structure on $\cptwo\#N\cptwobar$ such that $\Sigma$ and $\cpone$ are $J$-holomorphic by Lemma~\ref{l:Jcompatible}. By Lemma~\ref{l:blowdown}, we can blow-down a $J$-holomorphic exceptional sphere in the class $e_i$ for some $i$.
		If we blow down a $J$-holomorphic sphere corresponding to $e_i$ for $i\neq i_0$, there is an induced $J'$ on the blow-down such that $\Sigma$ is a singular $J'$-holomorphic curve.
		Therefore eventually, we have a situation where both $\Sigma$ and $e_{i_0}$ are represented by $J$-holomorphic spheres, so if $a_{i_0} > 0$ then $[\Sigma]\cdot e_{i_0}<0$. Since $J$-holomorphic spheres must intersect non-negatively, this can only occur if $\Sigma$ is exactly equal to the $J$-holomorphic exceptional sphere representing $e_{i_0}$.
		In particular, if $a_{i_0} > 0$, then $a_0 = 0$, thus proving item~\ref{sgn}. If $a_0=0$, it implies $a_{i_0}=1$ and $\Sigma$ is a blow-up of the exceptional class representing $e_{i_0}$. In this case, for ${i_j}\neq {i_0}$ the same argument shows that $a_{i_j}$ cannot be positive so we get item~\ref{e:exceptional-disjoint}.
		
		The particular cases follow from combining items~\ref{adj} and~\ref{sgn} and observing which integers give low values of $n^2+n$.
	\end{proof}
		
	In the following lemmas, $C_i$ and $C_j$ are smooth symplectic spheres in a positive plumbing in $\cptwo\#N\cptwobar$ such that $[C_i]\cdot h=[C_j]\cdot h=0$.
	
	\begin{lemma}\label{l:consecutive}
		If $[C_i]\cdot[C_j]=1$ (and $[C_i]\cdot h=[C_j]\cdot h=0$), there is exactly one exceptional class $e_i$ which appears with non-zero coefficient in both $[C_i]$ and $[C_j]$. The coefficient of $e_i$ is $+1$ in one of $[C_i],[C_j]$ and $-1$ in the other.
	\end{lemma}
	
	\begin{proof}
		It follows from Lemma~\ref{l:hom} item~\ref{e:exceptional-disjoint}, that $[C_i]$ and $[C_j]$ have the form $[C_i]=e_{n_0}-e_{n_1}-\dots -e_{n_k}$, $[C_j]=e_{m_0}-e_{m_1}-\dots - e_{m_\ell}$. Since the algebraic intersection is $+1$, either $e_{n_0}=e_{m_q}$ or $e_{m_0}=e_{n_p}$ for some $q,p$. To rule out the possibility that both of these occur, we consider the symplectic areas. 
		Since $C_i$ and $C_j$ are both symplectic spheres, $\langle \omega, [C_i]\rangle, \langle \omega, [C_j]\rangle >0$. Each of the exceptional classes also has positive symplectic area. Let $ a := \langle \omega, e_{n_0}\rangle >0$ and $b := \langle \omega, e_{m_0}\rangle >0$
		. Then
		\[
		0<\langle \omega, [C_i] \rangle < a-b \qquad \text{ and } \qquad 0<\langle \omega, [C_j] \rangle < b-a
		\]
		which is a contradiction.
	\end{proof}
	
	\begin{lemma}\label{l:pos}
		If $e_m$ appears with coefficient $+1$ in $[C_i]$ then it does not appear with coefficient $+1$ in the homology class of any other sphere in the plumbing.
	\end{lemma}
	
	\begin{proof}
		This follows from positivity of intersection and Lemma~\ref{l:hom}.
	\end{proof}
	
	\begin{lemma}\label{l:share2}
		If $[C_i]\cdot [C_j]=0$, then either there is no exceptional class which appears with non-zero coefficients in both, or there are exactly two exceptional classes $e_m$ and $e_n$ appearing with non-zero coefficients in both. One of these classes $e_m$ has coefficient $-1$ in both $[C_i]$ and $[C_j]$ and the other $e_n$ appears with coefficient $+1$ in one of $[C_i]$ or $[C_j]$ and coefficient $-1$ in the other.
	\end{lemma}
	
	\begin{proof}
		This follows from a similar argument as in Lemma~\ref{l:consecutive}, but with an additional exceptional class appearing with coefficient $-1$ in both $[C_i]$ and $[C_j]$ to cancel out the positive intersection.
	\end{proof}
	
	When there is a linear chain of such symplectic spheres with self-intersection $-2$ (consecutive spheres intersect once), there are few options for the homology classes of the spheres in that chain.
	
	\begin{lemma}\label{l:2chain}
		Suppose $\Sigma_1,\dots, \Sigma_k$ are a chain of symplectic spheres of self-intersection $-2$ disjoint from $\cpone$ in $\cptwo\# N\cptwobar$. Then the homology classes are given by one of the following two options up to re-indexing the exceptional classes:
		\begin{enumerate}[label=(\Alph*)]
			\item \label{i:to} $[\Sigma_i]=e_i-e_{i+1}$ for $i=1,\dots, k$.
			\item \label{i:from} $[\Sigma_i]=e_{i+1}-e_i$ for $i=1,\dots, k$.
		\end{enumerate}
		The homology class of any surface disjoint from the chain has the same coefficient for $e_1,\dots, e_{k+1}$.
	\end{lemma}
	
	\begin{proof}
		By Lemma~\ref{l:hom} item~\ref{e:exceptional-disjoint}, each $[\Sigma_i]=e_{k^i_1}-e_{k^i_2}$. Since consecutive spheres intersect once positively, by Lemma~\ref{l:share2} either 
		\begin{enumerate}
			\item \label{i:left} $k^i_2=k^{i+1}_1$ or 
			\item \label{i:right} $k^i_1 = k^{i+1}_2$.
		\end{enumerate} 
		Each exceptional class can appear with positive coefficient at most once by lemma~\ref{l:pos} so all $k^i_1$ are distinct. Therefore if for any $i_0$, $k^{i_0}_1=k^{i^0+1}_2$, we must have that $k^i_1=k^{i+1}_2$ for all $i_0\leq i \leq k$.  We can never switch from choosing option~\ref{i:right} to choosing option~\ref{i:left} as we go down the chain.
		
		If $k^{i_0}_2=k^{i_0+1}_1$ and $k^{i_0+1}_1=k^{i_0+2}_2$ (i.e. if we switch from option~\ref{i:left} to option~\ref{i:right}), then we have $[\Sigma_{i_0}]=e_a-e_b$, $[\Sigma_{i_0+1}]=e_b-e_c$, and $[\Sigma_{i_0+2}]=e_d-e_b$. Since $[\Sigma_{i_0}]\cdot[\Sigma_{i_0+2}]=0$ we must have $(e_a-e_b)\cdot(e_d-e_b)=0$ but $(e_a-e_b)\cdot(e_d-e_b)=e_a\cdot e_d -1<0$, so we cannot switch from option~\ref{i:left} to option~\ref{i:right} either.
		
		Therefore if we start with option~\ref{i:left}, we get choice~\ref{i:to}, and if we start with option~\ref{i:right} we get choice~\ref{i:from}. For  last statement, let $a_i$ denote the coefficient of $e_i$ in $[\Sigma]$. Then the statement follows from $0=[\Sigma]\cdot[\Sigma_i]=\mp(a_i-a_{i+1})$.
	\end{proof}
	
	Often, only one of these options can occur. The following particular case will appear frequently in our homological classifications.
	
	\begin{lemma}\label{l:2chainfix}
		 If $\Sigma_1,\dots, \Sigma_k$ form a chain of $(-2)$-spheres as in Lemma~\ref{l:2chain}, such that the chain is attached to another symplectic sphere $\Sigma_0$ which does intersect $\cpone$, option~\ref{i:from} can only occur if $e_2,\dots, e_{k+1}$ all appear with coefficient $-1$ in $[\Sigma_0]$. In particular if $[\Sigma_0]\cdot h=1$, option~\ref{i:from} can only occur if $[\Sigma_0]^2\leq 1-k$.
	\end{lemma}
	
	\begin{proof}
		By Lemma~\ref{l:hom}, all coefficients of exceptional classes in $[\Sigma_0]$ are negative. Let $[\Sigma_0]=a_0h+\sum_i a_ie_i$. Since $[\Sigma_0]\cdot [\Sigma_1]=1$ because they are joined in the chain, the exceptional class which appears with coefficient $+1$ in $[\Sigma_1]$ must appear with coefficient $-1$ in $[\Sigma_0]$. If the chain $\Sigma_1,\dots, \Sigma_k$ has homology classes as in option~\ref{i:from}, this means $a_{2}=-1$. Since $[\Sigma_0]\cdot[\Sigma_i]=0$ for $i=2,\dots, k$, we find that $-a_{i+1}+a_{i}=0$ for $i=2,\dots, k$ so $a_2=\dots = a_{k+1}=-1$.
		
		When $a_0=1$, $a_i\in \{0,-1\}$ for all $i\geq 1$ by Lemma~\ref{l:hom} so $[\Sigma_0]^2=1-\#\{a_i\neq 0\}$. If we have the chain of $-2$ spheres with homology classes as in option~\ref{i:from}, we must have at least $k$ non-zero coefficients so $[\Sigma_0]^2\leq 1-k$.
	\end{proof}

These lemmas, together with some arithmetic considerations, will generally suffice to allow us to classify all possibilities for the homology classes of the spheres in a normal crossing resolution. Given these homology classes we can apply Lemma~\ref{l:blowdown} to blow down to a configuration in $\cptwo$. The way that the exceptional classes appeared in the homology classes of the plumbing spheres affects the intersections between the proper transforms. We record the data of these intersections, including the degree of tangency between surfaces at each intersection and when intersections between different components coincide. There may also be singularities in a single component which we record as well. Fixing this data of the singularities and intersections, we then try to classify symplectic configurations of surfaces in $\cptwo$ up to equisingular symplectic isotopy. In the next section we solve this classification for families of singular reducible configurations of symplectic surfaces in $\cptwo$ that we will need.

	}
	
\subsection{Birational transformations}\label{ss:birationalequiv}
{
	In complex dimension $2$, a birational transformation is a sequence of blow-ups and blow-downs. Since blow-ups and blow-downs can be done symplectically~\cite{Mc}, these transformations from algebraic geometry can be imported into the symplectic context. When we have a singular surface in $\cptwo$, the birational transformation may change the singularities and self-intersection number of the surface. The blow-ups can begin to resolve singularities, and blow-downs may create new singularities.
	
	There are two ways that we will relate singular symplectic surfaces in $\cptwo$ using birational transformations. The first notion is weaker, but for two surfaces related in this way, the existence of one type of singular surface will imply the existence of another type of singular surface.
	
	\begin{definition}
	A configuration $\mathcal{C}_2$ in $(M_2,\omega_2)$ is \emph{birationally derived} from another configuration $\mathcal{C}_1\subset (M_1,\omega_1)$ if for \emph{every} symplectic realization $\Sigma_1$ of $\mathcal{C}_1$ in $(M_1,\omega_1)$, there is a sequence of blow-ups of the pair $(M_1,\Sigma_1)$ to the total transform $(M_1\#N\cptwobar,\total{\Sigma}_1)$, followed by a sequence of blow-downs of exceptional spheres $\pi: M_1\#N\cptwobar \to M_2$, such that $\Sigma_2 = \pi(\total{\Sigma}_1)$ is a realization of $\mathcal{C}_2$.
	\end{definition}

	From this definition, if we birationally invert the blow-down $\pi$ by blowing-up, we see that the total transform of $\Sigma_2$ in $M_1\#N\cptwobar$, contains the total transform of $\Sigma_1$, i.e. $\total{\Sigma}_1\subseteq \total{\Sigma}_2$. This follows from the fact that a set is contained in the preimage of its image, so $\total{\Sigma}_2 = \pi^{-1}(\Sigma_2) = \pi^{-1}(\pi(\total{\Sigma}_1)) \supseteq \total{\Sigma}_1$. For this reason, we can also refer to the birational derivation relation as saying that $\mathcal{C}_1$ is \emph{birationally contained} in $\mathcal{C}_2$.
	
	Note that this relation is directional. If $\mathcal{C}_2$ is birationally derived from $\mathcal{C}_1$, it is not typically true that $\mathcal{C}_1$ is birationally derived from $\mathcal{C}_2$. To strengthen this notion, we impose a more restrictive condition that the exceptional spheres that are blown down by $\pi$ are actually contained in $\total{\Sigma}_1$.
	
	\begin{definition}
		$\mathcal{C}_1$ in $(M_1,\omega_1)$ is \emph{birationally equivalent} to $\mathcal{C}_2$ in $(M_2,\omega_2)$ if for every symplectic realization $\Sigma_1$ of $\mathcal{C}_1$ there is a sequence of blow-ups of the pair $(M_1,\Sigma_1)$ to the total transform $(M_1\#N\cptwobar, \total{\Sigma}_1)$, followed by a sequence of blow-downs of exceptional spheres $\pi:M_1\#N\cptwobar \to M_2$ such that the exceptional locus of $\pi$ is contained in $\total{\Sigma}_1$ and $\Sigma_2 = \pi(\total{\Sigma}_1)$ is a symplectic realization of $\mathcal{C}_2$.
	\end{definition}
	
	Equivalently, a birational equivalence is a birational derivation where in the blow-up $\total{\Sigma}_1 = \total{\Sigma}_2$. 
	Reversing the sequence of blow-ups and blow-downs, we see that any realization $\Sigma_2$ of $\mathcal{C}_2$ in $(M_2,\omega_2)$ will have exceptional spheres contained in its total transform $\total{\Sigma}_2$ which can be blown down to obtain a realization of $\mathcal{C}_1$ in $(M_1,\omega_1)$. Therefore birational equivalence is a symmetric relation. 
	If $\mathcal{C}_1$ and $\mathcal{C}_2$ are birationally equivalent, they are each birationally derived from the other. 
	
	Without the extra condition for birational equivalence, a birational derivation is not an equivalence relation. More precisely, suppose that a configuration $\mathcal{C}_2$ that is birationally derived from a configuration $\mathcal{C}_1$. Then for any realization $\Sigma_1$ of $\mathcal{C}_1$ it has a resolution such that $\total{\Sigma}_1\subseteq \total{\Sigma}_2$ for some realization $\Sigma_2$ of $\mathcal{C}_2$. However, $\total{\Sigma}_2$ may contain some additional exceptional spheres that were not contained in $\total{\Sigma}_1$. Therefore when we blow down the exceptional spheres in $\total{\Sigma}_1$, the image of $\total{\Sigma}_2$ will only \emph{contain} $\Sigma_1$ instead of being \emph{equal} to $\Sigma_1$. A weaker equivalence relation could be obtained from the definition of birational derivation by replacing the condition $\Sigma_2=\pi(\total{\Sigma}_1)$ by $\Sigma_2\subseteq \pi(\total{\Sigma}_1)$ however this relation will not be particularly useful to us.
	
	Note that the number of components of the configuration is preserved by a birational equivalence but not by a birational derivation.
	
	\begin{example}\label{ex:birat}
		Let $\mathcal{C}_1$ be a configuration in $(\cptwo,\omega_{\rm{FS}})$ whose realizations consist of a single symplectic conic $C$ with three symplectic lines $L_1,L_2,L_3$ that intersect the conic tangentially at three distinct points $p_1,p_2,p_3$ respectively, and intersect each other transversally at three distinct points. Blow up at each of the three tangential intersection points, and denote the resulting proper transforms by $\proper{C}$ and $\proper{L}_i$. The resulting self-intersection numbers satisfy $[\proper{C}]^2=1$ and $[\widetilde{L_i}]^2=0$. The intersections between $\proper{C}$ and $\proper{L}_i$ are transverse and there are three new exceptional spheres $E_1, E_2, E_3$ which pass through those intersection points. See Figure~\ref{fig:biratderiv}. Because $[\proper{C}]^2=+1$ and $\proper{C}$ is a smooth sphere, we can find a symplectomorphism of $\cptwo\# 3\cptwobar$ which identifies $\proper{C}$ with $\cpone$. We calculate the possible homology classes of the other surfaces in the picture in terms of a standard basis which identifies $[\proper{C}]=h$ using Lemma~\ref{l:hom} and intersection numbers. We find the only option is
		\[ [\proper{C}]=h, \qquad [\proper{L}_1] = h-e_1, \qquad [\proper{L}_2] = h-e_2, \qquad [\proper{L}_3] = h-e_3 \]
		\[ [E_1] = h-e_2-e_3, \qquad [E_2] = h-e_1-e_3, \qquad [E_3] = h-e_1-e_2 \]
		Lemma~\ref{l:blowdown} implies that there exist exceptional spheres in classes $e_1,e_2,e_3$ which intersect the labeled surfaces non-negatively. Blowing down such exceptional spheres results in a configuration of seven symplectic lines intersecting in six triple points (the three triple points that already existed between $C,L_i,E_i$ for $i=1,2,3$, and the three triple points that are created when blowing down $e_1,e_2,e_3$). Therefore the configuration of seven lines with six triple points can be \emph{birationally derived} from the configuration of a conic with three tangent lines. Note, the exceptional spheres which are blown down in the last step of the transformation are not included in the configuration (because none of the surfaces $\proper{C},\proper{L}_i$ or $E_i$ represented a class $e_i$). Therefore this is not a \emph{birational equivalence}.
		
		\begin{figure}
			\centering
			\includegraphics[scale=.3]{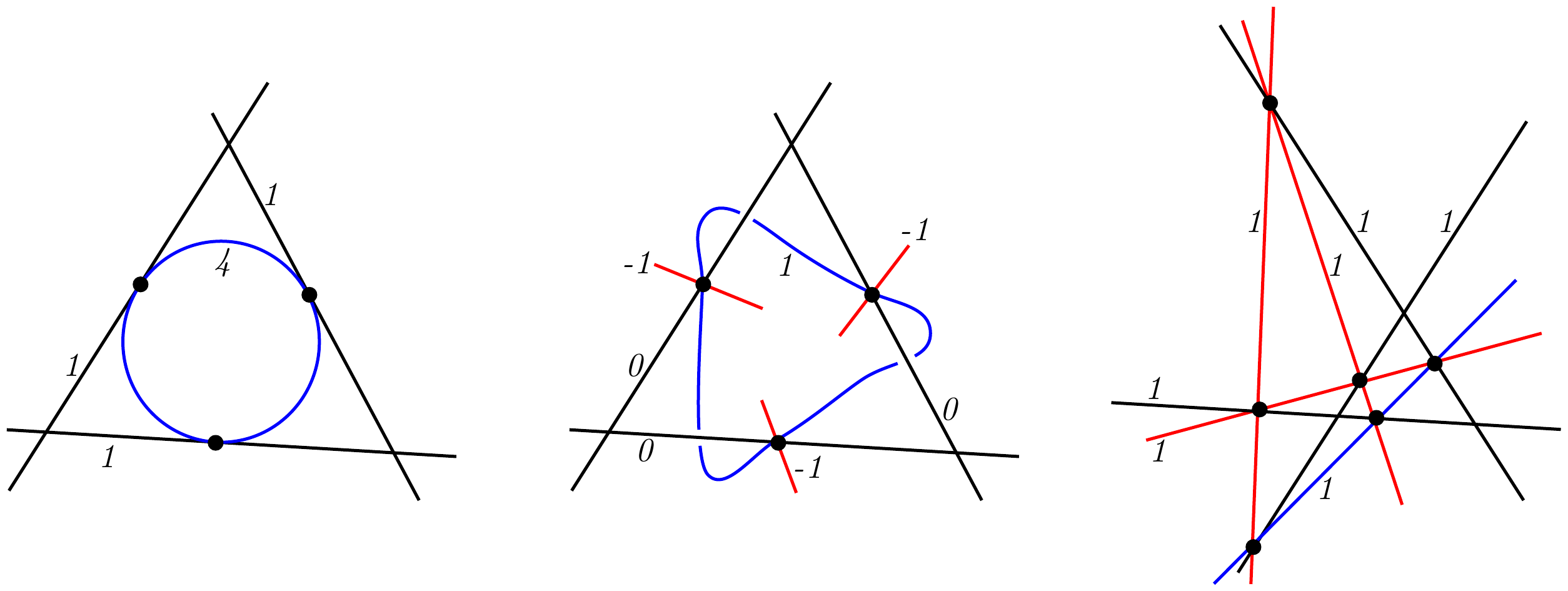}
			\caption{Birational derivation. The black lines forming the triangle on the left are $L_1,L_2,L_3$, and the blue circle represents the conic $C$. After blowing up as indicated in the center figure, we obtain three exceptional spheres $E_1,E_2,E_3$ indicated in red. There exist three other exceptional spheres which are not visible in the center diagram that can be blown down to result in the configuration shown on the right.}
			\label{fig:biratderiv}
		\end{figure}
		
		In order to get a birational equivalence, we would need to augment the original configuration by adding in components which will become the exceptional spheres that eventually will be blown down. We can reverse engineer to predict what must be added to our configuration to obtain a birational equivalence to the same final configuration. The exceptional sphere $\proper{S}_1$ in $\cptwo\#3\cptwobar$ representing $e_1$ intersects $E_2$ and $E_3$. Reversing the birational transformation by blowing down $E_2$ and $E_3$, we find that the image $S_1$ in $\cptwo$ will be a sphere of self-intersection $+1$ which must pass transversally through the tangential intersections between $C$ with $L_2$ and $C$ with $L_3$. The other exceptional spheres descend similarly to $+1$-spheres passing through two of the tangential intersections. 
		Replace our starting configuration by a configuration consisting of the original conic $C$ and lines $L_1,L_2,L_3$ with $L_i$ tangent to $C$ at $p_i$ together with three additional symplectic lines $S_1,S_2,S_3$ such that $S_i$ intersects $C$ transversally at the points $p_{i+1}$ and $p_{i+2}$ (where indices are taken mod $3$). Now the same birational transformation becomes a birational equivalence between a configuration consisting of a conic with three tangent lines and three lines through the three pairs of tangent points with a configuration of seven lines with six triple points and three double points. See Figure~\ref{fig:biratequiv}.
		
		\begin{figure}
			\centering
			\includegraphics[scale=.3]{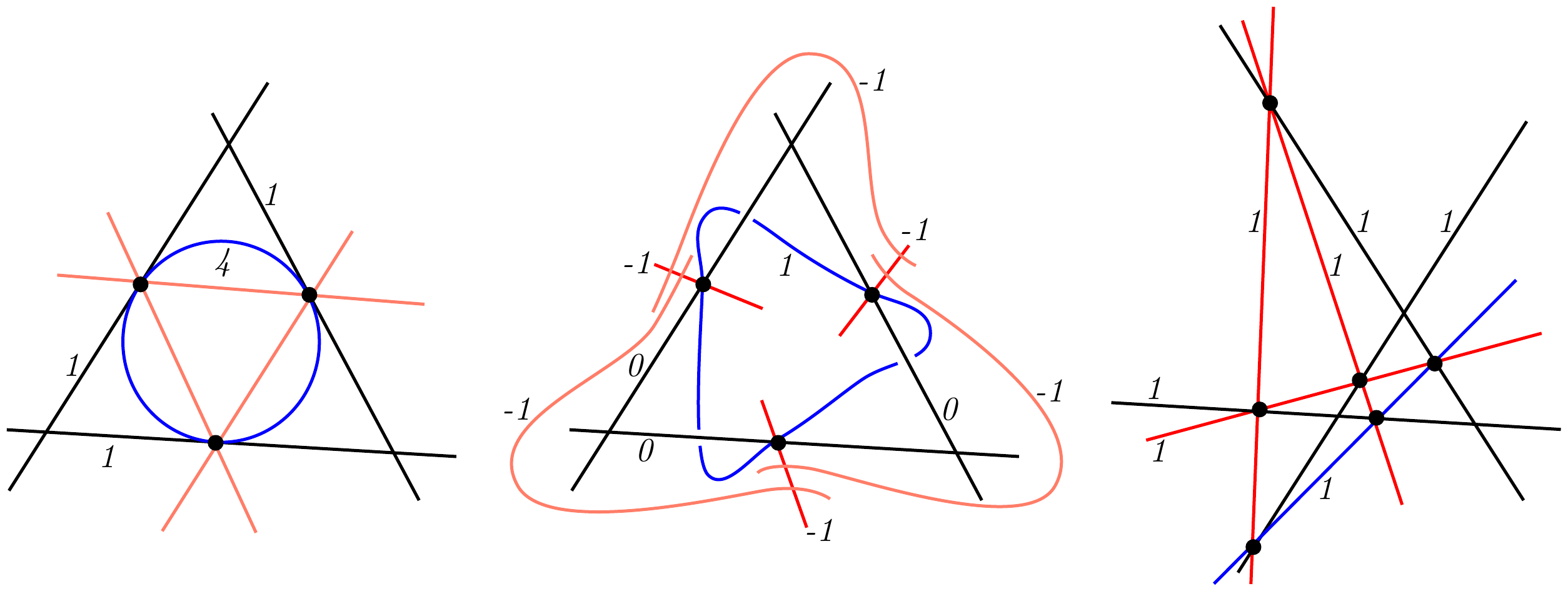}
			\caption{An example of a birational equivalence. The black lines forming the outer triangle in the left figure represent $L_1,L_2,L_3$, the blue circle represents the conic $C$, and the orange lines forming the inscribed triangle represent $S_1,S_2,S_3$. After blowing up at the marked points to get the center figure, the proper transforms $\proper{S}_1,\proper{S}_2,\proper{S}_3$ are exceptional spheres. The red curves in the center figure represent the exceptional spheres $E_1,E_2,E_3$. The right figure is obtained by blowing down $\proper{S}_1,\proper{S}_2,\proper{S}_3$. The images of $E_1,E_2,E_3$ after this blow-down are lines in the right figure.}
			\label{fig:biratequiv}
		\end{figure}
	\end{example}
	
	\begin{remark}
		In this article, we will prove the existence of birational derivations and birational equivalences using pseudoholomorphic curves. As demonstrated in Example~\ref{ex:birat}, we will primarily infer the existence of a particular birational derivation from a singular curve using Theorem~\ref{thm:mcduff} and Lemma~\ref{l:blowdown}. Alternatively, we can look for birational equivalences. As seen in Example~\ref{ex:birat}, to upgrade a birational derivation to a birational equivalence, we generally need to augment the original configuration $\Sigma_1$ by adding extra components. If these components are symplectic lines (degree one) in $\cptwo$ with sufficiently simple intersections with the other components of $\Sigma_1$, we will see that we can use pseudoholomorphic curves to infer the existence of such augmenting curves through Proposition~\ref{p:addline}. In practice, we will typically use Theorem~\ref{thm:mcduff} and Lemma~\ref{l:blowdown} along with the homological analysis from Section~\ref{s:caps} to discover a birational derivation. If desired, we can then reverse engineer to find an augmented configuration yielding a birational equivalence. Then in proving our results, we may justify the existence of the birational derivation \emph{or} justify the existence of the augmentation using the pseudoholomorphic curve results mentioned above. Note that once a configuration is augmented to produce a birational equivalence, no pseudoholomorphic curve result is necessary to justify the existence of the birational equivalence since all of the exceptional curves which one needs to blow-down are visibly included in the total transform of the configuration. Instead, pseudoholomorphic curves are used for augmenting the configuration to get the birational equivalence. By contrast, to state the existence of a birational derivation from one configuration to another can require pseudoholomorphic curve machinery to imply the existence of appropriate exceptional divisors which may not be visible.
	\end{remark} 
		
	Our use of birational derivations and equivalences arises from their implications to symplectic isotopy problems, which we state next. These implications are somewhat immediate from the definitions, so the mathematical power goes into proving such birational derivations exist as described in the previous remark.
			
	 Now we give the relations between symplectic isotopy problems and birational derivations and equivalences.
	
	\begin{prop} \label{l:biratexist}
		If $\mathcal{C}_1$ is a configuration in $(M_1,\omega_1)$, and $\mathcal{C}_2$ in $(M_2,\omega_2)$ is birationally derived from $\mathcal{C}_1$, then any subconfiguration of $\mathcal{C}_2$ symplectically embeds into $(M_2,\omega)$.
	\end{prop}
	
	This statement is immediate from the definition, but the utility of the statement comes from its contrapositive. Namely, we will show that certain configurations $\mathcal{C}_1$ cannot be symplectically realized in a closed symplectic manifold $(M_1,\omega_1)$, using the non-existence of a subconfiguration of symplectic curves that can be birationally derived from $\mathcal{C}_1$.
	
	\begin{prop}\label{l:biratderiveunique}
		Suppose a configuration $\mathcal{C}_2$ in $(M_2,\omega_2)$ is birationally derived from $\mathcal{C}_1$ in $(M_1,\omega_1)$, and suppose $\mathcal{C}_2$ has a unique (non-empty) equisingular symplectic isotopy class in $(M_2,\omega_2)$. Then $\mathcal{C}_1$ also has a unique (non-empty) symplectic isotopy class in $(M_1,\omega_1)$. If $(M_1,\omega_1)=(M_2,\omega_2)=(\cptwo,\omega_{\rm FS})$ and $\mathcal{C}_2$ can be realized by a complex curve, then $\mathcal{C}_1$ can also be realized by a complex curve.
	\end{prop}
	
	\begin{proof}
		Suppose any two symplectic embeddings of $\mathcal{C}_2$ into $(M_2,\omega_2)$ are symplectically isotopic.
		Let $Q_0$ and $Q_1$ be two symplectic embeddings of $\mathcal{C}_1$ into $(M_1,\omega_1)$. 
		By definition of birational derivation, for each $i=0,1$ there is a sequence of blow-ups of $Q_i$ to $\total{Q}_i$ and a sequence of blow-downs that contract $\total{Q}_i$ to $R_i$, where $R_i$ is a symplectic embedding of $\mathcal{C}_2$ into $(M_2,\omega_2)$. 
		Then there exists a family $R_t$ of equisingular symplectic embeddings of $\mathcal{C}_2$ into $(M_2,\omega_2)$ for $t\in[0,1]$ which connects $R_0$ and $R_1$. 
		For each $R_t$, perform the sequence of blow-ups along the appropriate smooth or singular points in $R_t$ to obtain $\total{R}_t$. By definition of birational derivation, there is a distinguished subset of the components $\total{Q}_t\subseteq \total{R}_t$ for $t\in[0,1]$, agreeing $\total{Q}_0$ and $\total{Q}_1$ for $t=0,1$. There are exceptional spheres in $\total{Q}_t$ for each $t\in [0,1]$ which can be blown down to give equisingular symplectic embeddings $Q_t$ of $\mathcal{C}_1$ into $(M_1,\omega_1)$. 
		This gives a symplectic family connecting $Q_0$ and $Q_1$.
		
		Since complex curves are preserved under birational transformations, the last statement follows from the same proof.
	\end{proof}
	
	If $\mathcal{C}_1$ and $\mathcal{C}_2$ are birationally equivalent, they are each birationally derived from the other yielding the following corollary.	
	
	\begin{cor} \label{l:biratiso}
		Suppose $\mathcal{C}_1$ in $(M_1,\omega_1)$ and $\mathcal{C}_2$ in $(M_2,\omega_2)$ are birationally equivalent. There is a unique equisingular symplectic isotopy class for $\mathcal{C}_1$ in $(M_1,\omega_1)$, if and only if there is a unique equisingular symplectic isotopy class for $\mathcal{C}_2$ in $(M_2,\omega_2)$. If $(M_1,\omega_1)=(M_2,\omega_2)=(\cptwo,\omega_{\rm FS})$ and if the equisingular symplectic isotopy class contains complex representatives for one configuration, it contains complex representatives for the other.
	\end{cor}
}

\subsection{Riemann--Hurwitz}
{
	The Riemann--Hurwitz obstruction uses symplectic information in a more global way.
	Fix an almost complex structure $J$ on $\cptwo$ such that $C$ is $J$-holomorphic.
	Fix a point $p\in \cptwo$, and consider the pencil of $J$-holomorphic lines through $p$, and the associated projection $\pi_p: \cptwo \setminus \{p\} \to\cpone$. Restricting this projection to $\pi: C\to \cpone$, and pre-composing with the normalization map $n:\widetilde C \to C$, gives a ramified covering map $\pi\circ n : \widetilde C \to \cpone$.
	The fact that the only singularities are ramification points follows from positivity of intersections between $J$-holomorphic curves. Ramification points arise from tangencies between lines in the pencil with $C$ and from singular points of $C$. 
	
	For a point $q\neq p$, the ramification index of $n^{-1}(q)$ is equal to the multiplicity of intersection of the $J$-line through $q$ with $C$. For a smooth point, this is just the order of tangency between the $J$-line and $C$ at $q$. At a singular point $q\neq p$ with multiplicity sequence starting with $[m_q,m_{q,2}]$ the ramification index of $x=n^{-1}(q)$ is $m_q$ if the $J$-line from $p$ to $q$ is transverse to $C$ at $q$. If the $J$-line from $p$ to $q$ is tangent to $C$ at $q$ then $x$ will have ramification index at least $m_q+m_{q,2}$.
	
	A priori, if $p\in C$, the map is not well-defined at $p$, but in fact, $\pi$ has a unique continuous extension defined by sending $p$ to the image of the $J$-holomorphic line through $p$ which is tangent to $C$ at $p$. In this case, the ramification index of $p$ is equal to the multiplicity of intersection of the $J$-line tangent to $C$ at $p$ \emph{minus the multiplicity of $C$ at $p$}. Therefore, if $p$ is a singular point whose multiplicity sequence starts with $[m_p,m_{p,2}]$ (where we set $m_{p,2}=1$ if the multiplicity sequence has length $1$) then the ramification index of $\pi\circ n$ at $p$ is at least $m_{p,2}$.
	
	The Riemann--Hurwitz formula is the calculation of the Euler characteristic of the branched covering in terms of the ramification indices and degree of the cover. If $\widetilde{\pi}:=\pi \circ n: \widetilde C \to \cpone$ is a $k$-fold ramified cover with ramification points $x_1,\dots, x_\ell$ with corresponding ramification indices $e_{\widetilde{\pi}}(x_j)$ then
	\[
	\chi (\widetilde C) = k(2-\ell ) + \sum_{j=1}^\ell (k+1-e_{\widetilde{\pi}}(x_j)) = 2k-\sum_{j=1}^\ell (e_{\widetilde{\pi}}(x_j)-1).
	\]
	In our case, $\widetilde C$ is a $2$-sphere so $\chi(\widetilde C)=2$. Suppose $d$ is the degree of $C$. If we choose $p\notin C$, then a generic line through $p$ intersects $C$ $d$ times, so the degree of the cover is $d$. Therefore the above equation specializes to
	\[
	2d-2=\sum_{j=1}^\ell (e_{\widetilde{\pi}}(x_j)-1).
	\]
	If instead, we choose $p\in C$, where $p$ has multiplicity $m_p$ (where $m_p$ is the first entry of the multiplicity sequence if $p$ is a singular point and is $1$ if $p$ is a smooth point of $C$), then a generic line through $p$ intersects $C$ at $d-m_p$ other points.
	Therefore $\widetilde{\pi}$ is a $(d-m_p)$-fold cover. This gives the following equation and inequality. The inequality is particularly useful as an obstruction to symplectically realizing certain cuspidal curves in $\cptwo$.
	\[
	2(d-m_p)-2=\sum_{j=1}^\ell (e_{\widetilde{\pi}(x_j)}-1) \geq \sum_{q\neq p} (m_q-1) +(m_{p,2}-1),
	\]
	which we re-write as follows:
	\begin{equation}\label{e:riemannhurwitz}
	2d-2m_p \ge 2 + \sum_{q\neq p} (m_q-1) + (m_{p,2} - 1).
	\end{equation}
	
	\begin{example}
		\label{ex:RH}
	We can apply Riemann--Hurwitz to exclude the following configurations of cusps for a quintic: $\{[3,2], [2], [2]\}$, $\{[3], [2], [2], [2]\}$, $\{[2,2], [2], [2], [2], [2]\}$, and $\{[2], [2], [2], [2], [2], [2]\}$.
	In the first case, we project from the $[3,2]$-cusp, and we obtain the following contradiction.
	\[2\cdot 5-2\cdot 3 \geq 2+1+1+1.\]
	In the second, we project from the $[3]$-cusp:
	\[2\cdot 5-2\cdot 3 \geq 2+1+1+1.\]
	In the third case we project from the $[2,2]$-cusp:
	\[2\cdot 5-2\cdot 2 \geq 2+1+1+1+1+1,\]
	and in the last from any of the cusps, to obtain:
	\[2\cdot 5-2\cdot 2 \geq 2+1+1+1+1+1.\]
	Thus, in each of the four cases, we get a contradiction.
	\end{example}
	
%
%
%

}

}

\section{The necessity of rationality}\label{s:rational}
	{
	An interesting question to ask is the following: can we distinguish exotic symplectic 4-manifolds (for example a potentially exotic $\cptwo$) in terms of the singular symplectic submanifolds they contain?
	For example, we will show that there are certain rational cuspidal curves which admit no symplectic embedding in $\cptwo$, but a priori such curves could admit symplectic embeddings in an exotic symplectic $\cptwo$. Though this is an alluring hope, we explain here that rational cuspidal curves we consider here cannot exist in any exotic or even homology $\cptwo$.
	As mentioned in the introduction, a similar question was raised, and partially addressed, by Chen~\cite{Chen}.
	
	We recall that Taubes proved that any symplectic structure on the standard, smooth $\cptwo$ is in fact symplectomorphic, up to rescaling, to the Fubini--Study form~\cite[Theorem~0.3]{Taubes}.
	
	We consider triples $(X,\omega,C)$, where $X$ is a rational homology $\cptwo$, $\omega$ is a symplectic form on $X$, and $C \subset X$ is a symplectic rational cuspidal curve, and show that $X$ is necessarily $\cptwo$.
	
	This question is answered in the algebraic setting by dividing into two cases, depending on the sign of the canonical divisor. If $X$ is an algebraic surface that is a rational homology $\cptwo$, either $K_X$ is not nef, which proves that $X$ is rational (and hence $\cptwo$, by our homological assumption) or $K_X$ is nef, in which case Yau proved that $X$ is a ball quotient~\cite{Yau,Aubin}.
	These are known as \emph{Mumford surfaces}, or \emph{fake projective planes}; the first example was given by Mumford~\cite{mumford}, and were classified by Cartwright and Steger~\cite{fakenews}, building on work of Prasad and Yeung~\cite{PrasadYeung}.
	Since they are all ball quotients, they have no symplectic rational curves, for maps from $S^2$ lift to the universal cover, and there are no compact complex (or symplectic) curves in $B^4$.
	
	In the proof of Theorem~\ref{thm:rational} we, too, will distinguish between the two cases, according to the sign of the canonical divisor $K_X$ (or, equivalently, of the first Chern class); we will use techniques inspired by gauge theory in both cases, albeit in two different ways.
	The proof in the case where $K_X < 0$ is essentially known to experts, and was already proved by Chen~\cite[Corollary 2.3]{Chen}; we give a slightly different proof here. The proof in the case $K_X > 0$ uses tools from Heegaard Floer homology, and is the part of the proof that is genuinely novel.

	\begin{proof}[Proof of Theorem~\ref{thm:rational}]		
	Recall that, for links of curve singularities and their connected sums one has that the invariant $\nu^+$ coming from Heegaard Floer homology~\cite{HomWu}, the slice genus, and the 3-genus all agree~\cite[Proposition~3]{HomWu}.
	Since $p_g(C) = 0$, the arithmetic genus $p_a(C)$ of $C$ is equal to the sum of the $\delta$-invariants of the singular points of $C$.
	In particular, if we let $K$ denote the connected sum of all links of singularities of $C$, then $\nu^+(K) = p_a(C)$.
		
	Since $X$ admits a compatible almost-complex structure $J$, $c_1^2(J)-2\chi(X)-3\sigma(X) = 0$.
	As $X$ is assumed to be a homotopy $\cptwo$, this implies that $c_1^2(J) = 9$, and thus $c_1(J) = \pm 3\PD(h)$ (where $h=[\cpone]\in H_2(\cptwo;\Z)$).
		
	If $C$ has degree $d$ and $c_1(J) = 3h$, as it is for $(\cptwo,\omega_{\rm FS})$, the adjunction formula~\eqref{e:adjunction} yields $p_a(C) = \frac12(d-1)(d-2)$, and Proposition~\ref{p:rationalsurface} below guarantees that $X$ is the standard $\cptwo$.
		
	We turn now to the case when $c_1(J) = -3h$; as in the algebraic case, we argue that there are no symplectic $4$-manifolds $X$ with $c_1(J) = -3h$ admitting rational cuspidal curves.
	In fact, now the adjunction formula yields:
	\[
		p_a(C) = \frac{(d+1)(d+2)}2.
	\]
	Consider the boundary of a regular neighborhood $N$ of $C$, which is homeomorphic to $S^3_{d^2}(K)$.
	Then the complement of the interior of $N$, taken with the opposite orientation, is a rational homology ball whose boundary is $S^3_{d^2}(K)$.
	Since $p_a(C) = \nu^+(K)$ we have that
	\[
		\frac{d(d-1)}2 - \nu^+(K) =  \frac{d(d-1) - (d+1)(d+2)}2 = -2d-1 < 0,
	\]
	contradicting the bound obtained in~\cite[Theorem~5.1]{AG}.
	\end{proof}
	We conclude the proof with the following proposition, which is mostly a corollary of~\cite[Corollary~1.5]{McS-survey}. This, in turn, follows from work of Liu~\cite{Liu96}, building on Taubes' work on Seiberg--Witten theory. (See also~\cite[Theorem~7.36]{rationalruled} for a more modern and self-contained treatment.)
	The theorem asserts that, if $C$ is a smooth symplectic curve in a closed symplectic 4-manifold $(X,\omega)$ such that $\langle c_1(\omega), C \rangle \geq 1$, and $C$ is not a $(-1)$-sphere, then $(X,\omega)$ is rational or ruled.
	Using the adjunction formula, we can rephrase the hypothesis on the first Chern class into an assumption on the self-intersection. Namely $\langle c_1(\omega),C\rangle = 2-2g(C)+[C]^2$, so the hypothesis is equivalent to $[C]^2 \geq 2g(\Sigma)-1$ and $C$ is not an exceptional sphere.
	(Note that the case $g(\Sigma) = 0$ is Theorem~\ref{thm:mcduff}.)
		
	\begin{prop}\label{p:rationalsurface}
		Suppose $C$ is a rational cuspidal curve with $C\cdot C \ge 2p_a(C)-1 > 0$.
		If $C$ symplectically embeds into a closed symplectic manifold $(X,\omega)$, then $(X,\omega)$ is a rational surface.
	\end{prop}
		
	\begin{proof}
		By deforming the curve $C$ in a regular neighborhood using the Milnor fibration model, we can find a \emph{smooth}, symplectic surface $C'\subset (X,\omega)$ in the same homology class as $C$, with genus $g(C')=p_a(C)$.
		By~\cite[Corollary~1.5]{McS-survey}, $(X,\omega)$ must be either a rational surface or an irrational ruled surface $(X',\omega')$.
		However, Lemma~\ref{l:genusboundsruled} below shows that $X'$ cannot be irrational ruled.
	\end{proof}
		
	\begin{lemma}\label{l:genusboundsruled}
		Let $C$ be a (possibly singular) curve of positive self-intersection in a (possibly non-minimal) symplectic $4$-manifold $(X',\omega')$, ruled over a Riemann surface $\Sigma$.
		Then $p_g(C) \ge g(\Sigma)$.
	\end{lemma}
		
	\begin{proof}
		Let $\widetilde{J}\in \mathcal{J}^{\omega'}(C)$. By~\cite[Theorem 3.4]{Mc} we can find a maximal collection of disjoint $\widetilde{J}$-holomorphic exceptional spheres, to blow down to a minimal surface $b:(X',\omega')\to (X,\omega)$ that is still ruled over $\Sigma$.
		Note that $b(C)$ is still a singular symplectic surface and it is $J$-holomorphic where $J$ is the almost complex structure on $X$ induced from $\widetilde{J}$. 
		
		Now, $(X,\omega)$ is symplectomorphic to a ruled surface. Let $B$ be the homology class of a smooth fiber. Consider the moduli space $\mathcal{M}(J',B)$. As argued in \cite[Proposition 4.1]{Mc}, if $J$ is regular for the associated Fredholm operator, there is a unique $J$-holomorphic $B$ curve from $\mathcal{M}(J,B)$ through each point $p\in X$. Moreover, because $B^2=0>-2$, $J$ is automatically a regular value of the associated Fredholm operator $P_B$ (\cite{Gr}, \cite[Lemma 2.8]{Mc}, \cite{HLS}). Therefore we have a projection map $\pi:X\to \Sigma$ whose fibers intersect $b(C)$ positively.
		

		By compactness and positivity of intersection, $(\pi\circ b)|_C$ is onto $\Sigma$.
		Let $n: C' \to C$ be the normalization map, and note that, by definition, $g(C') = p_g(C)$.
		Pre-composing $\pi\circ b$ with $n$ yields a surjective map $\pi' = (\pi\circ b)\circ n: C' \to \Sigma$, hence $g(C') \ge g(\Sigma)$, as desired.
	\end{proof}
	
}		

\section{Symplectic isotopy problems for reducible configurations}
\label{s:rediso}

In order to obstruct and classify symplectic isotopy classes of rational cuspidal curves and symplectic fillings of their associated contact manifolds, we will perform birational transformations to relate the original problem to a classification of a reducible configuration consisting of multiple lower degree smooth components intersecting in a particular way. Depending on the intersection pattern, such a configuration may be verified to have a unique symplectic isotopy class, or shown not to exist in $\cptwo$. Examples of such configurations shown not to exist symplectically in $\cptwo$ appeared in~\cite{RuSt}. The example from that article that we will use most is the Fano plane---a configuration of seven lines intersecting in seven transverse triple points.

Typically, we will use script letters to denote abstract configuration types, and non-script letters to denote actual realizations.

\subsection{Existence and uniqueness} Our first method of extending symplectic isotopy results to a larger collection of reducible configurations is through the following lemma that allows us to add a line with limited constraints to an existing configuration. A single smooth component is known to have a unique symplectic isotopy class in $\cptwo$ if its degree is at most $17$~\cite{SiebertTian}.

The following proposition gives a way of obtaining many configurations in $\cptwo$ with unique isotopy classes by adding degree $1$ components sufficiently generically.


All configurations in the statements have labelled components and labelled singular points, and isotopies preserve the labellings. That is to say, if $\mathcal{C}$ is a configuration of $n$ curves, then its components are labelled $B_1,\dots, B_n$, and its singular points are labelled $P^1,\dots,P^m$. When we say that two realizations of $\mathcal{C}$ are isotopic, we mean that the isotopy preserves the labelling. Viewed differently, an isotopy from a labelled realization of $\mathcal{C}$ to an unlabelled one induces a labelling of $\mathcal{C}$.

\begin{prop}\label{p:addline}
	Suppose $\mathcal{C}_1$ is a configuration of singular symplectic curves in $(\cptwo,\omega_{\rm{FS}})$ obtained from $\mathcal{C}_0$ by adding a single symplectic line $L$.
	Suppose that in the configuration $\mathcal{C}_1$ either:
	\begin{enumerate}
		\item \label{i:trans} $L$ intersects the curves of $\mathcal{C}_0$ transversally and the intersection points of $L$ with $\mathcal{C}_0$ contain at most two singular points, $P^i$ and $P^j$, of $\mathcal{C}_0$, or
		
		\item \label{i:tang} $L$ has a simple tangency to a subset of the components of $\mathcal{C}_0$, $\{B_{i_1},\dots, B_{i_k}\}$ at a single point $P^i$ in $\mathcal{C}_0$ ($P^i$ may be either a smooth or singular point of $B_\ell$ and it can be a singular point of $\mathcal{C}_0$ in which case it uses the existing label, or a smooth point of $\mathcal{C}_0$ in which case it takes a new label index) and all other intersections of $L$ with $\mathcal{C}_0$ are transverse double points. Further assume in this case that in $\mathcal{C}_1$, there are no other intersections of $L$ with the components $B_{i_1},\dots, B_{i_k}$ outside of the tangent point (but $L$ may intersect other components in transverse double points).
	\end{enumerate}
	Then there is a bijection between the isotopy classes of realizations of $\mathcal{C}_0$ and those of $\mathcal{C}_1$. In particular, $\mathcal{C}_0$ has a unique equisingular symplectic isotopy class if and only if $\mathcal{C}_1$ does.
\end{prop}

The hypothesis in item (\ref{i:tang}) fixes the multiplicity of the intersection of $L$ with the components of $\mathcal{C}_0$ at the point $P$. The requirement that the tangency is simple means that the multiplicity of the intersection is as small as possible for a symplectic curve with tangential components. More specifically when $B_\ell$ is smooth at $P$, a line has a simple tangency with $B_\ell$ if and only if the multiplicity of intersection is $2$. If $B_\ell$ has a singularity at $P$, a line has a simple tangency to $B_\ell$ if and only if the multiplicity of intersection is equal to the third element of the semigroup of $(B_\ell,P)$. The second requirement in the hypothesis in item (\ref{i:tang}) says that the multiplicity of intersection of $L$ with $B_\ell$ at $P$ is as large as it possibly can be for global degree reasons. Thus, these two conditions ensure that the multiplicity of intersection between $L$ and $B_\ell$ is automatically as it should be whenever $L$ is tangent to $B_\ell$ and their union is a singular symplectic configuration.

Note that the second constraint in item (\ref{i:tang}) is automatically satisfied when $B_\ell$ is a smooth conic (degree-$2$ curve). It also holds when the tangency occurs at a singular point of $B_\ell$ with multiplicity sequence $[d-1]$ when $B$ has degree $d$. This will be sufficient to cover our applications in this article. We hope to generalize this isotopy statement for reducible configurations and in particular, remove the second constraint from the hypothesis in item (\ref{i:tang}) through future work.

In most cases, in this article we will not need to worry about the labels: in most cases, the \emph{type} of the points $P^i$ and $P^j$ (or, in case (2), of $P^i$) determine the labelling. (For instance, there could be exactly two triple points in the configuration $\mathcal{C}_0$, so that ``the line passing through the two triple points'' is well defined, or a unique conic in $\mathcal{C}_0$, so that ``the line tangent to the conic at a generic point'' is well defined.)
However, in general the labels are important when applying this proposition. Specifically there may be more than one point in the configuration $\mathcal{C}_0$ of the same type (even lying on the same collection of curves), which one may choose as candidates for $P^i$ (or $P^j$). In this case, the configuration $\mathcal{C}_1$ obtained from one choice of candidate for $P^i$ is different than the configuration $\mathcal{C}_1'$ obtained from a different choice of candidate for $P^i$, and there need not be an equisingular isotopy which takes a realization of $\mathcal{C}_1$ to a realization of $\mathcal{C}_1'$, even if $\mathcal{C}_0$ has a unique equisingular isotopy class.

\begin{example}
	This phenomenon is made explicit in~\cite{ASST}. The configurations concerned, $\mathcal{B}_1$ and $\mathcal{B}_2$ are made up of two conics $C_1$ and $C_2$ which are tangent at two points, two generic tangents $L_1$ and $L_2$ to $C_1$, and a third line $L$ passing through a point in $L_1 \cap C_2$ and through a point in $L_2 \cap C_2$. Depending on the choice of the point in $L_2\cap C_2$ we obtain two possible configurations as in Figure~\ref{f:whynolabels} (the line $L$ can be either the diagonal line or the horizontal line). As unlabeled abstract configurations, $\mathcal{B}_1$ and $\mathcal{B}_2$ are the same; as a labeled configuration, there are (a priori) four choices for $L$. It is shown in~\cite{ASST} that two of the four choices admit realizations $B_1$ and $B_2$ whose complements have non-isomorphic fundamental groups. (In fact it is clear from the figure that choices come in symmetric pairs.)
	
	We now can see that an unlabeled version of Proposition~\ref{p:addline} would lead to a contradiction.
	$\mathcal{B}_i\setminus L$ has a unique isotopy class: this can be seen, for instance, using the techniques we will develop in this section to show that the two conics have a unique isotopy class and then applying Proposition~\ref{p:addline} twice. If we tried to apply an unlabeled version of Proposition~\ref{p:addline}, we would treat $\mathcal{B}_1$ and $\mathcal{B}_2$ as the same configuration, and thus their realizations $B_1$ and $B_2$ of Figure~\ref{f:whynolabels} (which is taken from~\cite[Figure~1]{ASST}) would be symplectically isotopic. Since we know by~\cite{ASST} that these are non-isotopic, we see that this unlabeled version of Proposition~\ref{p:addline} is false.
\end{example}

\begin{figure}
	\centering
	\includegraphics[width=0.4\textwidth]{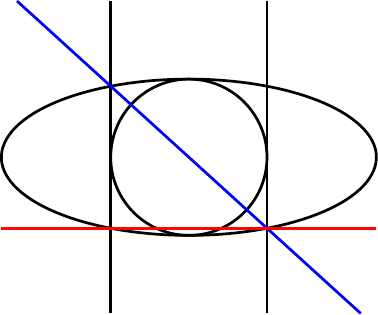}
	\caption{From this figure we can see two configurations $\mathcal{B}_1$ and $\mathcal{B}_2$ which are equivalent as unlabeled configurations. The black circle and oval represent conics $C_1$ and $C_2$. The black vertical lines represent lines $L_1$ and $L_2$. A third line $L$ is the blue diagonal line in $B_1$ and the red horizontal line in $B_2$. The configurations $\mathcal{B}_1$ and $\mathcal{B}_2$ are combinatorially equivalent as unlabeled configurations.}\label{f:whynolabels}
\end{figure}

\begin{proof}
	Let $\Iso(\Cc_0)$ and $\Iso(\Cc_1)$ denote the sets of symplectic isotopy classes in $(\cptwo,\omega_{\rm{FS}})$ of the configurations $\Cc_0$ and $\Cc_1$ respectively. There is a natural map
	\[
	\Psi: \Iso(\Cc_1) \to \Iso(\Cc_0)
	\]
	defined by $\Psi([C'])=[C]$ where $C$ is the realization of $\Cc_0$ obtained from the realization $C'$ of $\Cc_1$ by deleting the realization in $C'$ of the line $L$. It is clear that this map is well-defined, because if $C'_1$ and $C'_2$ are symplectically isotopic realizations of $\Cc_1$, there exists a symplectic isotopy $C'_t$ between them and dropping the realization of the line $L$ from each $C'_t$ yields a symplectic isotopy between the realizations $C_1$ and $C_2$ of $\Cc_0$. We will show that the map $\Psi$ is surjective and injective to obtain the stated result.
	
	Recall that Gromov proved that for any almost complex structure $J$ compatible with the symplectic form, and any two distinct points $(P,Q)$, there is a unique $J$-holomorphic line through $P$ and $Q$ \cite{Gr}. Similarly, for any almost complex structure $J$ compatible with the symplectic form, and any point and tangent vector $(P,T)$, there is a unique $J$-holomorphic line through $P$ with tangent vector $T$. This is shown in the proof of~\cite[Theorem 6.1]{Wendl} (see also \cite{McDuffBlowups}).
	
	\textbf{Outline:} Before diving further into the proof, we provide an outline of the arguments that follow. For the proofs of surjectivity and injectivity, we will augment a realization $C$ of $\mathcal{C}_0$ with an additional line to a realization of $\mathcal{C}_1$. For injectivity, (the harder direction), we will augment a $1$-parameter family of realizations $C_t$ of $\mathcal{C}_0$ with lines $L_t$ to form a family of realizations of $\mathcal{C}_1$. We will obtain these augmented configurations in three steps (performing these three steps first in the case of discrete realizations, and then in the $1$-parametric version). We summarize the three steps here (stated for a single realization $C$ of $\mathcal{C}_0$). In the first step, choose an almost complex structure $J$ which makes $C$ $J$-holomorphic, and use the above results from~\cite{Gr,McDuffBlowups,Wendl} to find the unique $J$-holomorphic line $L$ passing through two points $P,Q$, or through a single point $P$ with tangency $T$, so that $L$ intersects $C$ in the two singular points or one tangency on $C_0$ required by the hypotheses. It may seem like at this point we are done, but in fact $C\cup L$ may not yet realize $\mathcal{C}_1$ because $L$ may intersect $C$ more degenerately than required. We address this issue in the second and third steps. In step two, case~\ref{i:trans}, we adjust the line $L$ in a $C^1$ small manner locally near the points $P$ and $Q$ so that it intersects $C$ transversally at those points (in case it was originally accidentally tangent at those points). This step is unnecessary in case~\ref{i:tang} because the multiplicity of the tangency is fixed by the hypotheses. For either case, we perform step three, where we look at all other intersections of $L$ with $C$ outside of $P$ (and $Q$), and adjust $L$ keeping it fixed in small neighborhoods of $P$ (and $Q$) so that at the end, away from $P$ and $Q$, $C$ and $L$ only intersect in generic transverse double points. This provides the augmentation to a realization of $\mathcal{C}_1$. Now we provide the details of this argument and apply it to prove each direction of the statement.
	
	\textbf{Surjectivity:} We begin with the easier direction: that for each symplectic isotopy class $[C]\in \Iso(\Cc_0)$, there exists a symplectic isotopy class $[C']\in \Iso(\Cc_1)$ such that $\Psi([C'])=[C]$.
	
	\textbf{Step 1:} Suppose $C$ is a symplectic realization of $\mathcal{C}_0$. Let $J$ be an almost complex structure making $C$ $J$-holomorphic. In case (\ref{i:trans}), let $P$ and $Q$ be the special points in $C$ which $L$ is required to pass through to form the configuration $\mathcal{C}_1$ (if there are less than two singular points, one or both of these points can be chosen generically). Similarly, in case (\ref{i:tang}), let $(P,T)$ be the point and tangent direction at which $L$ must be placed to form the configuration $\Cc_1$ (again, the point can be chosen generically if it is not a singular point of $\mathcal{C}_0$). Now using the results above from~\cite{Gr,McDuffBlowups,Wendl}, let $L$ be the unique $J$-holomorphic line through $P$ and $Q$ or through $P$ tangent to $T$. It is possible that at this stage, $L$ passes through additional singular or tangency points of $C$ which is not desirable for the configuration $\mathcal{C}_1$. It also is possible that the intersections at $P$ or $Q$ fail to be transverse in case (\ref{i:trans}). In case (\ref{i:tang}), the multiplicity of the tangency at $P$ with the components of $i$ are fixed by the hypotheses so this latter problem is not relevant. We deal with the latter problem in case (\ref{i:trans}) first, and then deal with degenerate intersections away from $P$ and $Q$ in both cases.
	
	\textbf{Step 2:} For case (\ref{i:trans}), we now explain how to make any required adjustments to ensure that the intersections of $L$ with $C$ at $P$ and $Q$ are transverse. Let $\widehat U_P$ and $\widehat U_Q$ be open neighborhoods of $P$ and $Q$ with disjoint closures, such that $P$ and $Q$ are the only singular points of $L\cup C$ in the closures of $\widehat U_P$ and $\widehat U_Q$. Let $U_P$ and $U_Q$ be open neighborhoods of $P$ and $Q$ such that the closures satisfy $\overline{U}_P\subset \widehat U_P$ and $\overline{U}_Q\subset \widehat U_Q$. Let $\mathcal{M}_P$ be the moduli space of $J$-holomorphic lines which pass through $P$, and $\mathcal{M}_Q$ the corresponding moduli space of lines through $Q$. Each of these moduli spaces is diffeomorphic to $\cpone$. The subset of $\mathcal{M}_P$ of lines which are tangent to $C$ at $P$ is a compact $0$-dimensional subset (similarly for $Q$). Because this subset has codimension $2$, there exists a $J$-holomorphic line $L_P$ (respectively $L_Q$) which is $C^\infty$-close to $L$ and intersects $C$ transversally at $P$ (respectively $Q$).  Because $L_P$ and $L_Q$ are $C^\infty$-close to $L$, and the closures of the neighborhoods $\widehat U_P$ and $\widehat U_Q$ are disjoint, these lines can be spliced together symplectically. Our splicing will be a symplectic line which agrees with $L_P$ inside $U_P$, agrees with $L_Q$ inside $U_Q$, and agrees with the original $L$ outside of $\widehat U_P \cup \widehat U_Q$. Because $L$ has no intersections with $C$ in $\widehat U_P\setminus U_P$ or $\widehat U_Q\setminus U_Q$, by choosing the appropriate $C^{\infty}$-closeness for $L_P$ and $L_Q$ in terms of the fixed neighborhoods, we can ensure that the spliced line has no intersections with $C$ in $\widehat U_P\setminus U_P$ or $\widehat U_Q\setminus U_Q$. Therefore any intersections of the spliced line with $C$ occur in regions where it is $J$-holomorphic, thus creating singularities of $C\cup L$ with allowable models for singular symplectic curves (in particular the intersections are positive). Abusing notation, we rename this spliced line as $L$. Note that $C$ has not been changed.
	
	Now the new $L\cup C$ satisfies the singularity requirements of $\mathcal{C}_1$ near $P$ and $Q$ (in case (\ref{i:trans}) via the modification and in case (\ref{i:tang}) by the assumption constraining its intersections). However, it is still possible that $L$ passes through additional singular or tangency points of $C$ which cause it to differ from the configuration $\mathcal{C}_1$, because it should be intersecting $C$ generically away from $P$ and $Q$.
	
	\textbf{Step 3:} To make the final adjustment to $L$ so that $C\cup L$ will represent the configuration $\mathcal{C}_1$, let $U$, $\widehat U$  be open sets with $\overline{U}\subset \widehat U$ such that $P$ and $Q$ are not in the closure of $\widehat{U}$ and any singular point of $C\cup L$ which is not one of the designated points $P$ or $Q$ is contained in $U$. For a given $J$, the space of $J$-holomorphic lines in $\cptwo$ has real dimension four. The subspace of lines which intersect $C$ at a particular singular point is stratified with real co-dimension at least two, and similarly the subspace of lines which intersects $C$ tangentially is stratified with real co-dimension at least two. In particular, there exists a $J$-holomorphic line $\tilde L$ which is $C^\infty$-close to $L$ that intersect $C$ generically inside of $U$, and has no intersections with $C$ in $\widehat U\setminus U$. (Again, we use the assumption that there are no intersections of $L$ with $C$ in the closure of $\widehat{U}\setminus U$.) Although $\tilde L$ likely does not satisfy the required intersection properties at $P$ or $Q$, we can splice together $L$ with $\tilde L$ inside of $\widehat U$. Because each $\tilde L$ is chosen $C^\infty$-close to $L$, we can construct a symplectic line which agrees with $\tilde L$ inside of $U$, agrees with $L$ outside of $\widehat U$, and has no intersections with $C$ in $\widehat U\setminus U$. Again, we rename this spliced line as $L$. Note that because $P$ and $Q$ lie outside of the $\widehat U$, the new $L$ still intersects $C$ in the special points there as required.
		
	\textbf{Conclusion:} Since $C$ has not been modified, and $L$ has been constructed so that $C\cup L$ is a a symplectic realization of $\mathcal{C}_1$, we have found that $\Psi([C\cup L])=[C]$, as desired so $\Psi$ is surjective.
	
	\textbf{Injectivity:} To show that $\Psi$ is injective, suppose that $\Psi(A)=\Psi(B)$ for $A,B\in \Iso(\Cc_1)$. This means that $A$ is realized by a configuration $L^0\cup C^0$ of $\Cc_1$ and $B$ is realized by a configuration $L^1\cup C^1$ of $\Cc_1$ such that $C^0$ and $C^1$ are symplectically isotopic realizations of $\Cc_0$. Using the symplectic isotopy from $C^0$ to $C^1$ as realizations of $\Cc_0$, we will extend this to a symplectic isotopy from $L^0\cup C^0$ to $L^0\cup C^1$ as realizations of $\mathcal{C}_1$. Let $C^t$ denote the equisingular symplectic isotopy from $C^0$ to $C^1$. 
	
	\textbf{Step 1 ($1$-parametric):} By Lemma~\ref{l:Jcompatible}, there exists a family of compatible almost complex structures $J_t$ such that $C^t$ is $J_t$-holomorphic for $t\in[0,1]$. 
	Since $C^i\cup L^i$ is a singular symplectic curve, the space of compatible almost complex structures $\mathcal{J}^{\omega}(C^i\cup L^i)$ is a non-empty, contractible subspace of $\mathcal{J}^{\omega}(C^i)$, which is also contractible. Since $\mathcal{J}^{\omega}(C^i\cup L^i) \subseteq \mathcal{J}^{\omega}(C^i)$, there are almost complex structures $J_t$, $t\in [-1,2]$ agreeing with the previously defined $J_t$ for $t\in [0,1]$ such that $J_{-1}\in \mathcal{J}^{\omega}(C^0\cup L^0)$, $J_2\in \mathcal{J}^{\omega}(C^1\cup L^1)$, $J_t\in \mathcal{J}^{\omega}(C^0)$ for $t\in [-1,0]$ and $J_t\in \mathcal{J}^{\omega}(C^1)$ for $t\in [1,2]$. 
	
	Let $C^t=C^0$ for $t\in [-1,0]$, and $C^t=C^1$ for $t\in [1,2]$.
	Now $C^t$ for $t\in [-1,2]$ is an equisingular symplectic isotopy which is $J_t$-holomorphic. Define $L^{-1}:=L^0$ and $L^2:=L^1$. We will construct an equisingular symplectic isotopy connecting $L^{-1}\cup C^{-1}=L^0\cup C^0$ to $L^2\cup C^2=L^1\cup C^1$.
	
	Let $(P^t,Q^t)$ (respectively $(P^t,T^t)$) be the two points in $C^t$ (respectively point and tangent direction in $C^t$) which the line $L^t$ must pass through in order to satisfy the constraints of the configuration $\mathcal{C}^1$. Even if some of these points are chosen generically, we make sure to choose them to vary smoothly with $t$, and such that $P^i$ and $Q^i$ lie on $L^i$ for $i=-1,2$. As in the first direction, we will use Gromov's and McDuff's theorems to initially define $L^t$ for $t\in [-1,2]$ as the unique $J_t$-holomorphic line passing through $P^t$ and $Q^t$ (respectively passing through $P^t$ with tangent direction $T^t$). Because $L^{-1}$ and $L^2$ are the unique lines satisfying the point (respectively point and tangency) conditions, this provides a 1-parameter family of $J_t$-holomorphic symplectic lines connecting $L^{-1}$ to $L^2$.
	
	Again we need to deal with the possibility that $L^t$ may develop non-generic intersections away from $P^t$ and $Q^t$, or that the intersections at $P^t$ or $Q^t$ may become tangencies between $L^t$ and components of $C^t$ in case (\ref{i:trans}). We will perform a similar $C^\infty$-small adjustment of the line near such problematic points, but now fitting this into the $1$-parameter family relative to the endpoints of the family. Let $\mathbb{I}=[-1,2]$ denote the parameter interval.
	
	\textbf{Step 2 ($1$-parametric):} We again start by ensuring the intersection behavior of $L^t$ is correct near $P^t$ and $Q^t$. (Note this stage is unnecessary in case (\ref{i:tang}) because of the stronger hypothesis.) For this, fix open sets $\widehat U_P^t$ and $\widehat U_Q^t$ with disjoint closures with $P_t\in U_P^t\subset \overline{U}_P^t \subset \widehat U_P^t$ and $Q_t \in U_Q^t\subset \overline{U}_Q^t\subset \widehat U_Q^t$ for all $t\in \mathbb{I}$, such that $L^t\cup C^t$ has no singular points in $\widehat U_P^t\setminus U_P^t$ or $\widehat U_Q^t\setminus U_Q^t$ and where
	$$ \widehat{\mathbb{U}}_P = \bigcup_{t\in \mathbb{I}} \widehat U_P^t\subset \cptwo\times \mathbb{I}$$
	and
	$$ \widehat{\mathbb{U}}_Q = \bigcup_{t\in \mathbb{I}} \widehat U_Q^t\subset \cptwo\times \mathbb{I}$$
	are open neighborhoods of the paths $P_t$ and $Q_t$ in $\cptwo\times \mathbb{I}$ (and similarly for the smaller neighborhoods denoted without hats).
	Let $\mathbb{M}_P$ be the moduli space of pairs $(K^t,t)$ where $t\in \mathbb{I}$ and $K^t$ is a $J_t$-holomorphic line in $\cptwo$ which passes through $P_t$. Similarly, let $\mathbb{M}_Q$ be the moduli space of pairs $(K^t,t)$ where $t\in \mathbb{I}$ and $K^t$ is a $J_t$-holomorphic line in $\cptwo$ which passes through $Q_t$. Both $\mathbb{M}_P$ and $\mathbb{M}_Q$ have natural maps $\pi_P:\mathbb{M}_P\to \mathbb{I}$ and $\pi_Q:\mathbb{M}_Q\to \mathbb{I}$ (sending $(K^t,t)$ to $t$), and these maps are fibrations with fiber diffeomorphic to $\cpone$. The subset $B_{P}$ of $\mathbb{M}_{P}$ of pairs $(K^t,t)$ which are tangent to $C^t$ at $P^t$ is a finite set of points in each $t$-slice, which form a section or multi-section of the projection $\mathbb{M}_P\to \mathbb{I}$. The same statement holds for the analogous subset $B_Q\subset \mathbb{M}_Q$. Since it is required to pass through $P^t$ and $Q^t$, $L^t$ gives a section of $\pi_P$ and a section of $\pi_Q$. Since $C^t\cup L^t$ realizes the configuration $\mathcal{C}_1$ for $t$ near $\partial \mathbb{I}$, $L^t$ avoids the bad subsets $B_{P}$ (respectively $B_{Q}$) of $\mathbb{M}_{P}$ (respetively $\mathbb{M}_Q$) near $\pi_{P}^{-1}(\partial\mathbb{I})$ (respectively $\pi_{Q}^{-1}(\partial\mathbb{I})$). For $X\in\{P,Q\}$, let $L^t_{X}$ be a section of $\pi_{X}: \mathbb{M}_{X}\to \mathbb{I}$ which agrees with $L^t$ near $\pi_{X}^{-1}(\partial \mathbb{I})$, is $C^{\infty}$-close to the section $L^t$ everywhere and which avoids the bad subset $B_{X}$. Then $L_P^t$ and $L_Q^t$ are $J_t$-holomorphic lines which are $C^{\infty}$-close to $L^t$. We will splice them together to a smoothly varying $\mathbb{I}$-family of symplectic lines which agree with $L_P^t$ in $U_P^t$, with $L_Q^t$ in $U_Q^t$ and with the original $L^t$ outside of $\widehat{U}_P^t\cup \widehat{U}_Q^t$. To make the splicing vary smoothly with $t\in \mathbb{I}$, we choose the cut-off functions to vary continuously with $t$ by letting them be the restrictions to $t$ slices of a smooth cut-off function supported on $\widehat{\mathbb{U}}_{X}\setminus \mathbb{U}_{X}$, for $X\in\{P,Q\}$. The splicing is trivial near $\partial \mathbb{I}$. We rename the spliced family $L^t$.
	
	\textbf{Step 3 ($1$-parametric):} Once the behavior near $P^t$ and $Q^t$ is correct, we next fix any overly degenerate intersections of $L^t$ with $C^t$ away from $P^t$ and $Q^t$ in either case (\ref{i:trans}) or (\ref{i:tang}). Let $\mathbb{U}\subset \overline{\mathbb{U}}\subset \widehat{\mathbb{U}}\subset \cptwo\times \mathbb{I}$ be open subsets whose closures do not contain $(P^t,t)$ or $(Q^t,t)$ for any $t \in \mathbb{I}$, such that every intersection point of $L^t$ with $C^t$ which is not $P^t$ or $Q^t$ is contained in $\mathbb{U}\cap \cptwo\times\{t\}$. Note that this is possible to achieve because no other intersection points can approach $P^t$ or $Q^t$ after the above modification which ensures that the multiplicities of the intersection of $L^t$ with $C^t$ at $P^t$ and $Q^t$ remain constant. 
	
	Let $\mathbb{M}$ be the moduli space of all pairs $(K^t,t)$ where $t\in \mathbb{I}$ and $K^t$ is any $J_t$-holomorphic line in $\cptwo$. There is a natural fibration $\pi: \mathbb{M}\to \mathbb{I}$ sending $(K^t,t)$ to $t$, whose fibers $\mathbb{M}_t$ are diffeomorphic to $\cptwo$ (in particular, the fibers are $4$-dimensional manifolds). The subset $B_t$ of $\mathbb{M}_t$ of $J_t$-lines $K^t$ which are tangent to $C^t$ at some point other than $P^t$ in case (\ref{i:tang}) or pass through a singular point of $C^t$ other than $P^t$ or $Q^t$ is a stratified subspace with strata of codimension at least $2$ in $\mathbb{M}_t$. Since these subspaces vary smoothly with $t$, the union $\mathbb{B}=\cup_{t\in \mathbb{I}}B_t$ forms a stratified subspace of $\mathbb{M}$ with strata of codimension at least $2$. The family $(L^t,t)$ provides a section of $\pi: \mathbb{M}\to \mathbb{I}$, which avoids $\mathbb{B}$ near its endpoints, but may intersect $\mathbb{B}$ at interior values of $t$. Let $\widetilde{L}^t$ be a section which agrees with $L^t$ near its end points, is $C^\infty$-close to $L^t$ everywhere and is disjoint from $\mathbb{B}$. We splice together $L^t$ and $\widetilde{L}^t$ to form a smoothly varying family of symplectic lines which agree with $\widetilde{L}^t$ inside of $\mathbb{U}$ and with $L^t$ outside of $\widehat{\mathbb{U}}$. This ensures that the intersections of $L^t$ with $C^t$ near $P^t$ and $Q^t$ remain as they should be outside of $\widehat{\mathbb{U}}$ and the other intersections of $L^t$ with $C^t$ remain inside of $\mathbb{U}$ (by $C^{\infty}$-closeness) and are made generic as required for the configuration $\mathcal{C}_1$. 
	
	\textbf{Conclusion:} Replacing $L^t$ with this splicing provides the required equisingular symplectic isotopy from $C^{-1}\cup L^{-1}$ to $C^2\cup L^2$. Since $C^{-1}\cup L^{-1}=C^0\cup L^0$ and $C^2\cup L^2=C^1\cup L^1$, this shows that any two realizations of $\mathcal{C}_1$ are equisingularly symplectically isotopic.
\end{proof}

We can immediately recover unique isotopy classifications for small line arrangements.
\begin{cor}\label{c:linearr}
A symplectic line arrangement with at most six lines has a unique symplectic isotopy class. \end{cor}

\begin{cor}\label{c:oneconicthreelines}
	Any configuration of one symplectic conic with three symplectic lines has a unique symplectic isotopy class unless it is the configuration $\mathcal{G}$ of Figure \ref{fig:G} (which we obstruct in Proposition \ref{p:GFano}).
\end{cor}

\begin{proof}
	Gromov proved there is a unique symplectic isotopy class of a conic $Q$ in $\cptwo$ \cite{Gr}. We are adding at most three lines, and we will add lines tangent to the conic first. Therefore we will add the first line, $L_1$ which intersects this conic either tangentially at a point or transversally at two points and this does not change the uniqueness of the isotopy classification by Proposition \ref{p:addline}. If the second line $L_2$ is tangent to the conic, it must intersect $L_1$ generically off of $Q$. Similarly, if all three lines are tangent to $Q$, they must intersect generically unless they form the configuration $\mathcal{G}$ where they intersect at a triple point away from $Q$. Therefore tangent lines can be added with no further constraint than their tangency to the conic at a point, and thus Proposition~\ref{p:addline} suffices in these cases. For lines being added with no tangency condition to $Q$, the only combinatorial constraints can be that $L_i$ should pass through an intersection point of $Q\cap L_j$ or $L_j\cap L_k$ for $j,k<i$. Since $L_i$ can intersect $L_j$ at most once, the combinatorics requires $L_i$ to pass through at most $i-1\leq 2$ points, so Proposition \ref{p:addline} suffices to prove there is a unique symplectic isotopy class of such a configuration.
\end{proof}

Our second strategy to extend the known list of unique symplectic isotopy classes of reducible curves is to use birational transformations to modify configurations to collections of curves where at most one of the components is degree greater than one. We use a birational equivalence to relate the reducible configuration of interest to a reducible configuration we can understand through Proposition~\ref{p:addline}. We give below a collection of examples using this technique, which we will use later on in the paper.

\subsubsection{Two conics with a common tangent line}
Let $\Gc_3$ denote the configuration consisting of two conics $Q_1$ and $Q_2$ and a line $L_1$ tangent to both $Q_1$ and $Q_2$ (at different points) such that $Q_1$ and $Q_2$ intersect at one point with multiplicity $3$ and at another point transversally. See Figure~\ref{fig:G3}.

\begin{figure}
	\centering
	\includegraphics[scale=.4]{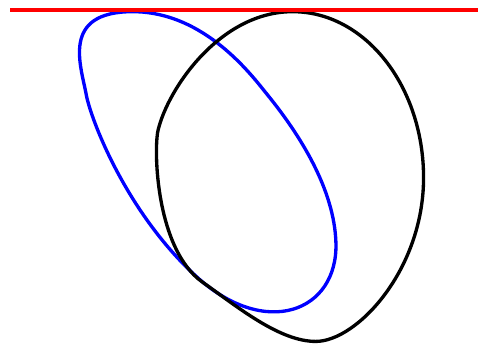}
	\caption{Configuration $\Gc_3$.}
	\label{fig:G3}
\end{figure}

\begin{prop} \label{p:G3} $\Gc_3$ has a unique equisingular symplectic isotopy class. \end{prop}

Let $\Ac$ denote the augmented configuration obtained from $\Gc_3$ by adding two additional lines $L_2$ and $L_3$ where $L_2$ is tangent to $Q_1$ and $Q_2$ at their multiplicity $3$ tangency point and $L_3$ intersects $Q_1$ transversally at its tangent point with $L_1$ and its tangent point with $Q_2$. See upper left of Figure~\ref{fig:G3birat}.

Let $\Bc$ denote the configuration consisting of a conic $R$ and four lines $M_1,M_2,M_3,M_4$ such that $M_1$ and $M_2$ are tangent to $R$, the intersection of $M_3$ with $M_2$ lies on $R$, the intersection of $M_3$ with $M_4$ lies on $R$ and $M_1$, $M_2$ and $M_4$ intersect at a triple point. See lower right of Figure~\ref{fig:G3birat}.

\begin{lemma}\label{l:biratAB}
	There is a birational equivalence between $\Ac$ and $\Bc$.
\end{lemma}

\begin{proof}
	Starting with a realization of $\Ac$ blow up twice at the multiplicity three tangential intersection of $Q_1$ with $Q_2$ and once at the tangential intersection of $L_1$ with $Q_1$. The resulting configuration in $\cptwo\#3\cptwobar$ is shown on the upper right of Figure~\ref{fig:G3birat}. There are three exceptional divisors $D_1,D_2,D_3$ where $D_2$ and $D_3$ have self-intersection $-1$ and $D_1$ has self-intersection $-2$. The proper transforms of $L_2$ and $L_3$ have self-intersection $-1$ so they can be symplectically blown down. After this the proper transform of $D_1$ becomes a $(-1)$-exceptional sphere which can be blown down as well. The resulting curve has configuration type $\Bc$ as in the bottom row of Figure~\ref{fig:G3birat}.	
\end{proof}
	
	\begin{figure}
		\centering
		\includegraphics[scale=.5]{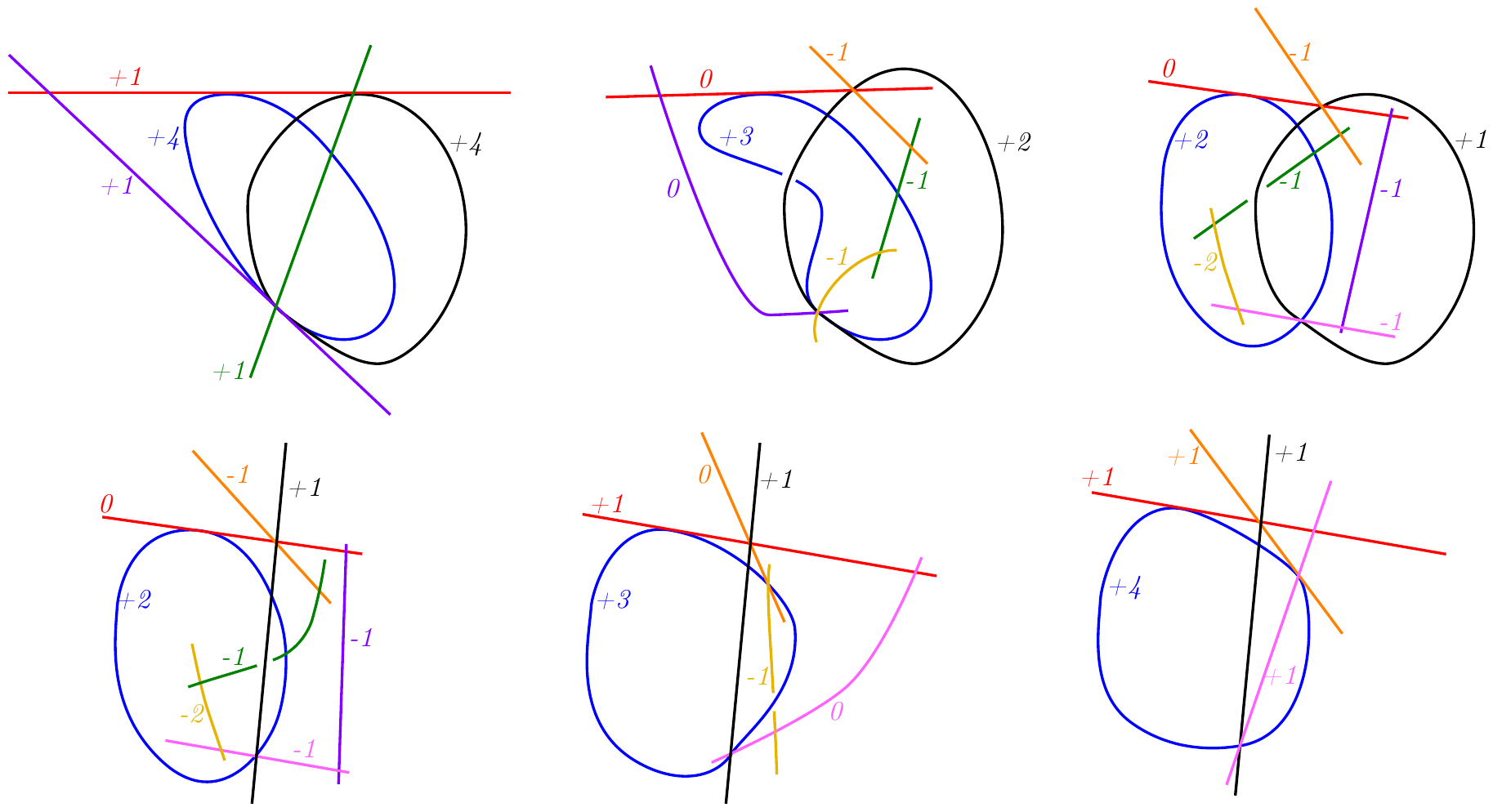}
		\caption{Birational equivalence from $\Ac$ to $\Bc$. The upper right and lower left configurations are the same, just redrawn indicating a symplectomorphism of $\cptwo\#3\cptwobar$ identifying the black $+1$-curve with $\cpone$.}
		\label{fig:G3birat}
	\end{figure}

\begin{proof}[Proof of Proposition~\ref{p:G3}]
	Since $\Ac$ is obtained from $\Gc_3$ by adding one line with a point tangency condition, and one line with two transverse intersection conditions, $\Gc_3$ has a unique equisingular symplectic isotopy class if and only if $\Ac$ does by Proposition~\ref{p:addline}. By Lemma~\ref{l:biratAB}, and Corollary~\ref{l:biratiso}, $\Ac$ has a unique equisingular symplectic isotopy class if and only if $\Bc$ does. We prove $\Bc$ has a unique symplectic isotopy class by iteratively applying Proposition~\ref{p:addline} to add lines to the configuration of a single conic $R$ which has a unique symplectic isotopy class by Gromov~\cite{Gr}. We apply Proposition~\ref{p:addline} four times to add the four lines, starting with the two tangent lines $M_1$ and $M_2$, then adding $M_3$ transversally through the tangent intersection of $M_2$ with $R$, and finally adding $M_4$ through $M_1\cap M_2$ and the other intersection of $M_3$ with $R$. Therefore $\Bc$ has a unique equisingular symplectic isotopy class so $\Gc_3$ does as well.
\end{proof}

\begin{remark}
	The configurations $\Ac$ and $\Bc$ were not pulled out of thin air, but rather come from augmenting a birational derivation obtained using Theorem~\ref{thm:mcduff} and Lemma~\ref{l:blowdown}. To find these configurations, start with $\Gc_3$ and blow up twice at the tangency of $Q_1$ and $Q_2$ and once at the tangency of $Q_1$ and $L_1$ to make these intersections transverse and to bring the self-intersection number of $Q_1$ to $+1$. Keep track of the exceptional divisors in the total transform. Now apply theorem~\ref{thm:mcduff} to identify the proper transform of $Q_1$ with $\cpone$ and determine the possible homology classes of the other curves in the configuration in terms of the standard basis (they are uniquely determined up to re-indexing). Next use Lemma~\ref{l:blowdown} to blow down exceptional curves in classes $e_1,e_2,e_3$ and observe the effect on the configuration (it will descend to $\Bc$). To change this from a birational derivation to a birational equivalence, we need to add in curves representing the $e_i$ which are not represented by curves in the configuration. Back-tracking these curves to $\Gc_3$ we find the two lines we need to augment $\Gc_3$ by to get $\Ac$.
\end{remark}

We can prove similar uniqueness statements for configurations of two conics with a common tangent line when the conics intersect more generically. Let $\Gc_2$ denote the configuration of two conics simply tangent to each other at one point with a line tangent to both conics away from their intersections. Let $\Gc_1$ denote two transversally intersecting conics with a line tangent to both.

\begin{prop}\label{p:G2}
$\Gc_2$ has a unique equisingular symplectic isotopy class.
\end{prop}

\begin{proof}
	Let the conic components of $\Gc_2$ be denoted by $Q_1, Q_2$ and the line by $L_1$. By Proposition~\ref{p:addline}, the symplectic isotopy classification of $\Gc_2$ is equivalent to the classification for the augmented configuration $\Ac_2$ obtained from $\Gc_2$ by adding two more lines $L_2$ through the tangential intersection of $Q_1$ with $Q_2$ and the tangency point $L\cap Q_2$. There is a birational equivalence of $\Ac_2$ to a configuration $\Bc_2$ consisting of a single conic, two tangent lines, a line passing through one of the tangency points (and otherwise generic), and a line passing through the intersection of the two tangent lines (and otherwise generic). (See Figure~\ref{fig:AB2}.) $\Bc_2$ can be built from the single conic configuration by repeated applications of Proposition~\ref{p:addline} so it has a unique equisingular symplectic isotopy class.
\end{proof}

\begin{figure}
	\centering
	\includegraphics[scale=.35]{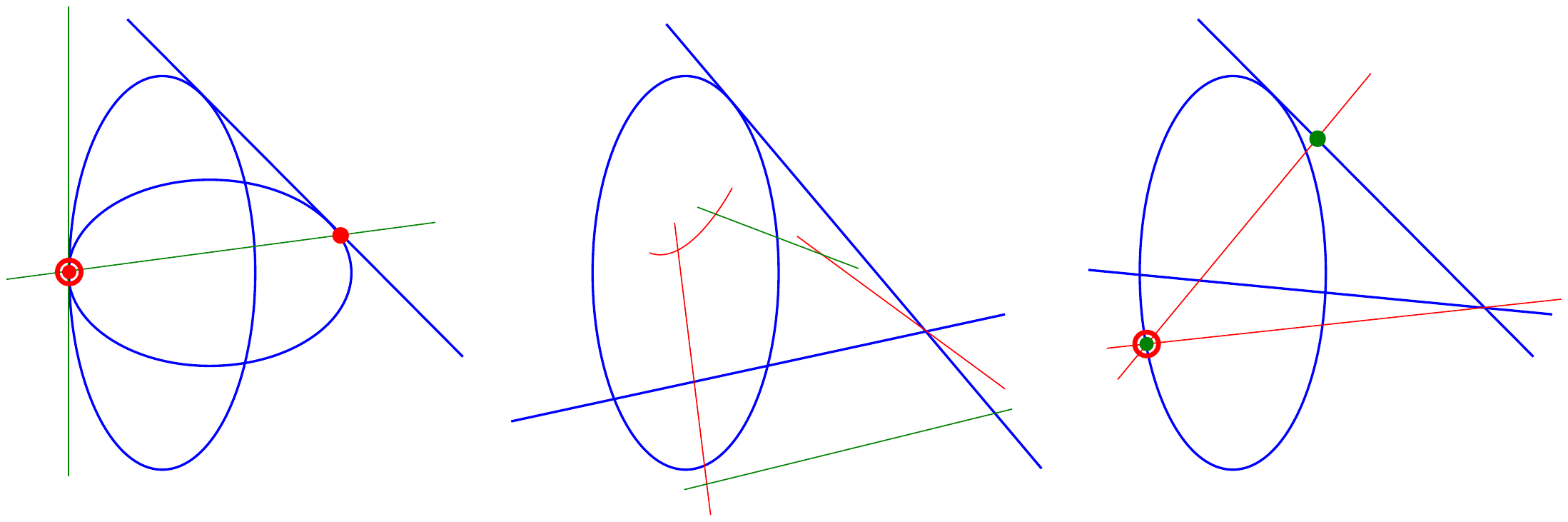}
	\caption{Birational transformation between $\Ac_2$ (left) and $\Bc_2$ (right).}
	\label{fig:AB2}
\end{figure}

\begin{prop}\label{p:G1}
$\Gc_1$ has a unique equisingular symplectic isotopy class.
\end{prop}

\begin{proof}
	Let the conic components of $\Gc_1$ be denoted by $Q_1, Q_2$ and the line by $L_1$. Fix two of the four intersection points $q_1,q_2$ of $Q_1$ with $Q_2$ and let $p=L_1\cap Q_2$. By Proposition~\ref{p:addline}, the symplectic isotopy classification of $\Gc_1$ is equivalent to the classification for the augmented configuration $\Ac_1$ obtained from $\Gc_1$  by adding three lines passing transversally through the three pairs of the points $p,q_1,q_2$. $\Ac_1$ is birationally equivalent to $\Bc_1$, the configuration built from one conic with one tangent line, and four other lines intersecting as in Figure~\ref{fig:AB1}. $\Bc_1$ can be built from the single conic configuration by repeated applications of Proposition~\ref{p:addline} so it has a unique equisingular symplectic isotopy class.
\end{proof}

\begin{figure}
	\centering
	\includegraphics[scale=.35]{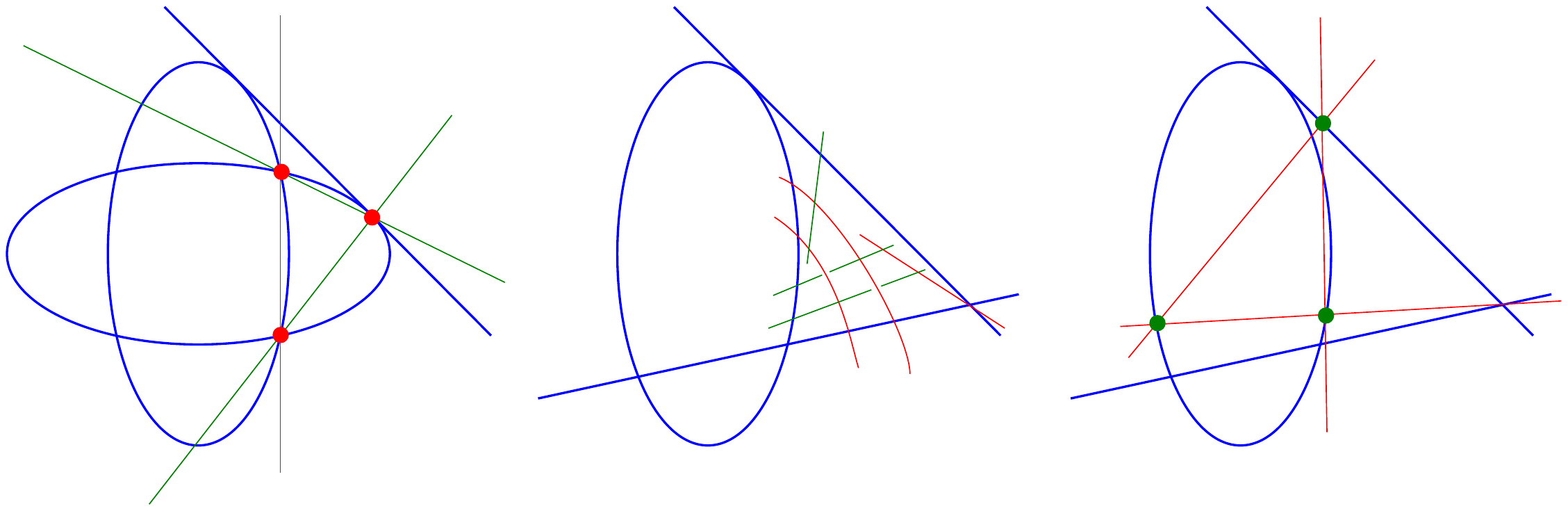}
	\caption{Birational transformation between $\Ac_1$ (left) and $\Bc_1$ (right).}
	\label{fig:AB1}
\end{figure}

By a similar set of birational transformations, we can prove uniqueness of the symplectic isotopy class of two conics without the additional tangent line. There are five ways that two conics can intersect each other: transversally at four points, with one simple tangency and two transverse points, with one triple tangency and one transverse point, with two simple tangencies, or with one quadruple tangency. Note that the last two configurations cannot be realized with an additional common tangent line (see Proposition~\ref{p:G422}).

\begin{prop} \label{p:twoconics}
	Any of the five configurations of two positively intersecting symplectic conics has a unique equisingular symplectic isotopy class.
\end{prop}

\begin{proof}
	Blow up the configuration three times at intersection points between the two conics. This transforms each of the conics to a $+1$-sphere, and includes three exceptional divisors (whose intersection configuration depends on the configuration of conics we started with). Regardless of the starting configuration, the pairwise intersections between curves in the configuration (including the exceptional divisors and proper transforms of the conics) are transverse and each pair of curves intersects at most once. By McDuff's Theorem~\ref{thm:mcduff}, there is a symplectomorphism of the resulting $\cptwo\#3\cptwobar$ which identifies the proper transform of one of the conics with $\cpone$. By Lemma~\ref{l:adjclass}, the proper transform of the other conic represents the class $h=[\cpone]$, and the exceptional divisors represent classes of the form $h-e_i-e_j$ or $e_i-e_j$. By Lemma~\ref{l:blowdown}, we can realize new exceptional spheres in the classes $e_i$ which intersect the configuration positively, and blow these down. Afterwards, we can blow down any exceptional spheres which then appears in the configuration representing a class $e_i$. Repeating this if needed, we reach a configuration of at most five curves in $\cptwo$, each representing the class $h=[\cpone]$. Therefore a configuration of at most five symplectic lines can be birationally derived from the original configuration of two conics. Since any configuration of at most five lines has a unique nonempty symplectic isotopy class by Proposition~\ref{p:addline}, we conclude that the configuration of two conics has a unique nonempty symplectic isotopy class by Proposition~\ref{l:biratderiveunique}.
\end{proof}

\begin{cor}\label{c:twoconics}
	Any configuration of two conics with one line has a unique equisingular symplectic isotopy class (which is empty only if it is obstructed in Proposition~\ref{p:G422}).
\end{cor}

\begin{proof}
	This follows from Proposition~\ref{p:twoconics} and Proposition~\ref{p:addline} when the line is not required to be tangent to both conics. This is because the line must intersect each conic exactly twice (with multiplicity), and the only special points occur at the intersection of \emph{both} conics (so a tangent line could not go through any special points, and a transverse line could go through at most two special points). When the line is tangent to both conics, this follows from Propositions~\ref{p:G3},~\ref{p:G2},~\ref{p:G1}, and~\ref{p:G422}.
\end{proof}

\subsubsection{Unisingular curves with a maximally tangent line}
Suppose $(C,0)$ is the germ of a (not necessarily locally irreducible) singularity with multiplicity sequence $[d-1]$.
Up to topological equivalence, $(C,0)$ is determined by a partition $\bd = (d_1,\dots,d_m)$ of $d-1$, i.e. an un-ordered $m$-tuple of positive integers $d_1,\dots,d_m$ such that $d_1+\dots+d_m = d-1$;
here $m$ is the number of branches of $(C,0)$.
Therefore, if $(C,0)$ has multiplicity sequence $[d-1]$, we say that it is of type $\bd$ if the corresponding partition is $\bd$.
Note that link of a singularity of type $\bd$ is obtained by cabling the link given by $m$ fibers of the Hopf fibration with cabling parameters $(d_1,d_1+1), \dots, (d_m, d_m+1)$.

\begin{prop}\label{p:d-with-tangent}
	Let $d$ be a positive integer and $\bd$ a partition of $d-1$.
	Let $\mathcal{C}$ denote the configuration with the following two irreducible components:
	\begin{itemize}
		\item $C$ is a rational curve of degree $d$ in with a singular point $P$ with multiplicity sequence $[d-1]$ and of type $\bd$;
		\item $L$ is a line with a tangency of order $d$ with $C$ at a point $Q$, where $P\neq Q$.
	\end{itemize}
	Then $\mathcal{C}$ has a unique equisingular symplectic isotopy class, and this unique isotopy class contains a complex curve.
\end{prop}


\begin{proof}
	Suppose that $C$ has $m$ branches at $P$; number them so that each of them has a singularity of multiplicity $d_i$, where $\bd = [d_1,\dots,d_m]$ is the given partition of $d-1$.
	In particular, $d_1+\dots+d_m = d-1$.

	By iteratively applying Proposition~\ref{p:addline}, the uniqueness of the symplectic isotopy class of the configuration $\mathcal{C}$ is equivalent to the uniqueness of the symplectic isotopy class of a configuration $\mathcal{C}'$ obtained from $\mathcal{C}$ by adding $m$ distinct lines tangent to $C$ at $P$ along the $m$ distinct branches, together with one line that passes transversally through $P$ and $Q$. Note that for degree reasons the only other intersections are transverse double points between $L$ and each of the $m$ tangent lines. Note that for the tangent lines, to satisfy the hypotheses of Proposition~\ref{p:addline}, it is important that they have no other intersections with $C$. This is because the intersection at $P$ has multiplicity $d$: a generic line through $P$ has intersection multiplicities $d_1,\dots,d_m$ with its branches; if $L$ is tangent to the $i^{\rm th}$ branch of $C$ at $P$, then its intersection with $C$ at $P$ satisfies:
	\[
	(C\cdot L)_P \ge d_1 + \dots + d_{i-1} + (d_i+1) + d_{i+1} + \dots d_m = d,
	\]
	which simultaneously implies that $(C\cdot L)_P = d$ and that the intersection multiplicity of $L$ with the $i^{\rm th}$ branch is exactly $d_i+1$.
	
	By Lemma~\ref{l:dbirat}, there is a birational equivalence from $\mathcal{C}'$ to a configuration of $m+3$ lines with a single $(m+1)$-fold point and a single triple point. This latter configuration has a unique symplectic isotopy class by Proposition~\ref{p:addline}. Therefore by Corollary~\ref{l:biratiso}, the configurations $\mathcal{C}'$ has a unique symplectic isotopy class containing complex curve representatives, and thus $\mathcal{C}$ does as well.
\end{proof}

\begin{lemma}\label{l:dbirat}
	Let $\mathcal{C}'$ denote the configuration above, obtained from $\mathcal{C}$ by adding $m$ lines tangent to $C$ at $P$ as well as one line through $P$ and $Q$.
	Then the configuration $\mathcal{C}'$ is birationally equivalent to a configuration of $m+3$ lines where $m+1$ intersect at a single point and the other two intersect the $m+1^{st}$ at a triple point. (See Figure~\ref{fig:dbirat}.)
\end{lemma}

\begin{figure}
	\labellist
	\pinlabel $p$ at 72 49
	\pinlabel $q$ at 90 99
	\pinlabel $C$ at 95 6
	\pinlabel $L$ at 139 28
	\pinlabel $\vphantom{a}_{m+1^{\rm st}}$ at 295 6
	\pinlabel $\vphantom{a}_{m+2^{\rm nd}}$ at 345 40
	\pinlabel $\vphantom{a}_{m+3^{\rm rd}}$ at 340 73
	\endlabellist
	\centering
	\includegraphics[scale=.8]{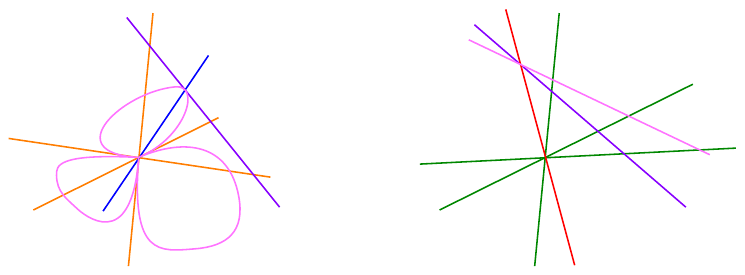}
	\caption{The two birationally equivalent configurations from Lemma~\ref{l:dbirat}. On the left, the configuration $\mathcal{C}'$ coming from the configuration $\mathcal{C}$ of Proposition~\ref{p:d-with-tangent}. On the right, the configuration of $m+3$ lines with one $(m+1)$-fold point and one triple point.}
	\label{fig:dbirat}
\end{figure}

\begin{proof}
	Blowing up at the singular point of $C$ yields a smooth rational curve, the proper transform of $C$, with $m$ points of tangency with the exceptional divisor $E_0$, of order $d_1,\dots,d_m$ respectively.
	Blow up $d_i$ times at the $i$-th tangency, so that the proper transform of $E_0$ is disjoint from that of $C$. Next, blow up $d-1$ times at the tangency of $C$ with $L$, so that their proper transforms intersect transversely.
	
	We obtain the configuration of rational curves represented in Figure~\ref{f:d-with-tangent}.
	
	The proper transform of $C$ has self-intersection $+1=d^2-(d-1)^2-2(d-1)$, and is a smooth sphere. We identify it with $\cpone$ using Theorem~\ref{thm:mcduff} and use this to determine the homology classes of the other curves in the configuration.
	
	By Lemma~\ref{l:hom}, and their pairwise disjointness, the vertical $(-1)$-curves intersecting $C$ must represent classes $h-e_0-e_1^i$ for $i=1,\dots, m$ with the final vertical $(-1)$-curve which also intersects $L$ in the class $h-e_0-e_1^0$. The chains of $(-2)$-curves emanating from these $(-1)$-curves are fully determined by Lemma~\ref{l:2chainfix} to be $e_1^i-e_2^i,\dots, e_{d_i-1}^i-e_{d_i}^i$. The last chain of $(-2)$-curves is similarly $e_1^0-e_2^0,\dots, e_{d-2}^0-e_{d-1}^0$. The intersection relations imply $[E_0]=e_0-e_1^0-\dots - e_{d-1}^0$ and $[L]=h-e_1^1-\dots-e_{d_1}^1-\dots-e_1^m-\dots - e_{d_m}^m$.
	The proper transforms of the tangent lines represent the classes $e_{d_i}^i$ for $1\leq i \leq m$, and the proper transform of the line through the two singular points represents the class $e^0_{d-1}$. Therefore we can blow down these exceptional curves, and consequently blow down the entire chain of $(-2)$-curves as the end curve becomes an exceptional divisor. 
	
	The resulting configuration consists of $m+3$ symplectic lines coming from the proper transforms of $C$, $L$, and the $m+1$ vertical $(-1)$-curves. The first $m$ vertical $(-1)$-curves will have a common intersection point which is the image of the blow-down of $e_0$. $C$, $L$ and the last vertical $(-1)$-curve have a common triple intersection point before blowing down, so this is preserved in the proper transform.
%
%
%
\end{proof}

\begin{figure}
	\labellist
	\pinlabel $C$ at 20 72
	\pinlabel $+1$ at 20 53
	\pinlabel $E_0$ at 42 30
	\pinlabel $-d$ at 42 12
	\pinlabel $-1$ at 68 5
	\pinlabel $-1$ at 105 5
	\pinlabel $-1$ at 198 5
	\pinlabel $-1$ at 235 5
	\pinlabel $-2$ at 61 110 
	\pinlabel $\vdots$ at 61 185
	\pinlabel $-2$ at 43 215
	\pinlabel $-2$ at 43 135
	\pinlabel $-2$ at 43 156
	\pinlabel $-2$ at 98 110 
	\pinlabel $\vdots$ at 99 185
	\pinlabel $-2$ at 80 215
	\pinlabel $-2$ at 80 135
	\pinlabel $-2$ at 80 156
	\pinlabel $-2$ at 187 110 
	\pinlabel $\vdots$ at 189 185
	\pinlabel $-2$ at 173 215
	\pinlabel $-2$ at 173 135
	\pinlabel $-2$ at 173 156
	\pinlabel $-2$ at 225 110
	\pinlabel $\vdots$ at 226 185
	\pinlabel $-2$ at 210 215
	\pinlabel $-2$ at 210 135
	\pinlabel $-2$ at 210 156
	\pinlabel $-2$ at 340 110 
	\pinlabel $\vdots$ at 343 185
	\pinlabel $-2$ at 327 215
	\pinlabel $-2$ at 327 135
	\pinlabel $-2$ at 327 156
	\pinlabel $-1$ at 352 70
	\pinlabel $2-d$ at 423 100
	\pinlabel $L$ at 413 120
	\pinlabel \rotatebox[origin=c]{90}{$\vphantom{a}_{d_1-2}$} at 18 170
	\pinlabel $\left\{\vphantom{\begin{matrix}{l}1\\1\\1\\1\\1\\1\\1\\1\end{matrix}}\right.$ at 30 170
	\pinlabel \rotatebox[origin=c]{-90}{$\vphantom{a}_{d_2-2}$} at 125 170
	\pinlabel $\left.\vphantom{\begin{matrix}{l}1\\1\\1\\1\\1\\1\\1\\1\end{matrix}}\right\}$ at 113 170
	\pinlabel \rotatebox[origin=c]{90}{$\vphantom{a}_{d_{m-1}-2}$} at 147 170
	\pinlabel $\left\{\vphantom{\begin{matrix}{l}1\\1\\1\\1\\1\\1\\1\\1\end{matrix}}\right.$ at 159 170
	\pinlabel \rotatebox[origin=c]{-90}{$\vphantom{a}_{d_m-2}$} at 253 170
	\pinlabel $\left.\vphantom{\begin{matrix}{l}1\\1\\1\\1\\1\\1\\1\\1\end{matrix}}\right\}$ at 241 170
	\pinlabel \rotatebox[origin=c]{90}{$\vphantom{a}_{d-3}$} at 303 170
	\pinlabel $\left\{\vphantom{\begin{matrix}{l}1\\1\\1\\1\\1\\1\\1\\1\end{matrix}}\right.$ at 315 170
	\endlabellist
	\centering
	\includegraphics[width=0.8\textwidth]{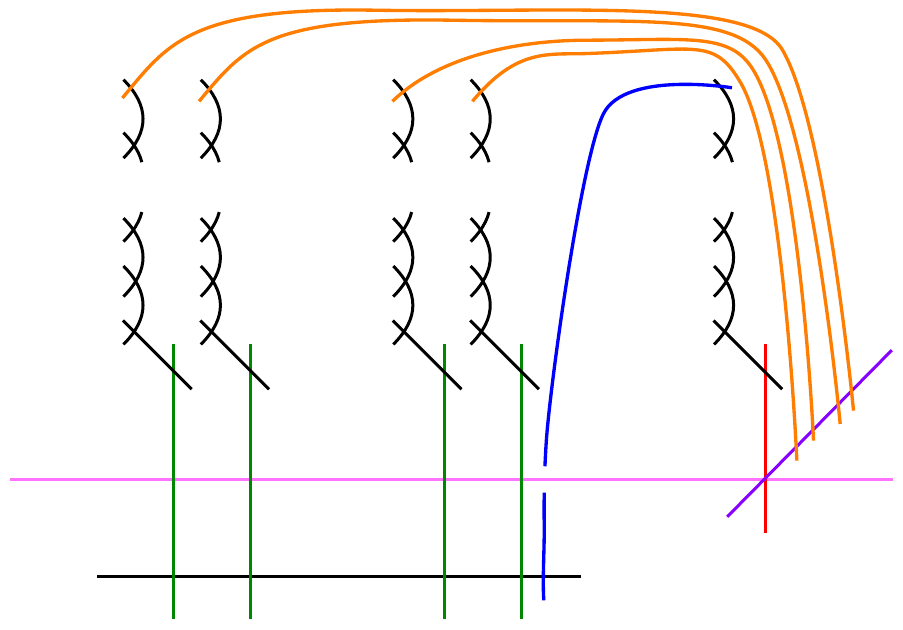}
	\caption{The configuration of curves in the proof of Lemma~\ref{l:dbirat}. The long curved orange and blue curves are $(-1)$-spheres which will be blown down to obtain the line arrangement on the right of Figure~\ref{fig:dbirat}.}\label{f:d-with-tangent}
\end{figure}

	This completes the proof of Theorem~\ref{t:smallconfigs}.
	\begin{proof}[Proof of Theorem~\ref{t:smallconfigs}]
	Line arrangements of degrees up to six are proven to have a unique and non-empty isotopy class in Corollary~\ref{c:linearr}
	Configurations of one conic and up to three lines are dealt with in Corollary~\ref{c:oneconicthreelines}, while configurations of two conics, or two conics and a line are taken care by Corollary~\ref{c:twoconics}, while Proposition~\ref{p:d-with-tangent} settles the case of a rational degree-$d$ curve with a unique singularity of multiplicity $[d-1]$ and a line with an order-$d$ tangency.
	\end{proof}

	\subsubsection{Two more configurations}
	
	Here we show two additional configurations have a unique symplectic isotopy class. The proofs require some new ideas along with the pseudoholomorphic techniques and birational transformations we have been relying on. For the first, we utilize a fixed point argument. For the second, we define certain branched covering maps from a singular curve to $\cpone$ utilizing pseudoholomorphic tangent lines and pencils. We will need the uniqueness of these configurations to prove there are unique symplectic isotopy classes of certain cuspidal quintics in Propositions~\ref{p:252525} and~\ref{p:27232323}.
	
	\begin{figure}
			\labellist
			\pinlabel $A$ at 87 128
			\pinlabel $B$ at 16 23
			\pinlabel $C$ at 140 23
			\pinlabel $\ell_1$ at 58 3
			\pinlabel $\ell_2$ at 139 40
			\pinlabel $\ell_3$ at 46 100
			\endlabellist
		\centering
		\includegraphics[scale=.8]{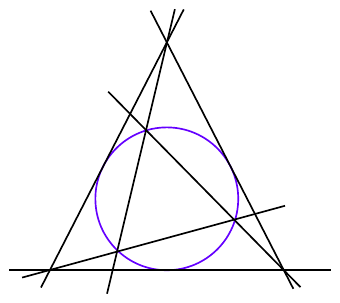}
		\caption{The configuration $\Hc$ built from a conic with three tangent lines and three additional lines intersecting in triple points as shown.}
		\label{fig:H}
	\end{figure}
	
	The first configuration, that we call $\Hc$, is comprised of a conic inscribed in a triangle of tangent lines meeting in vertices $A$, $B$, and $C$, and three lines $\ell_1$, $\ell_2$,  $\ell_3$ through $A$, $B$, and $C$ respectively, such that their pairwise intersections are on the conic. See Figure~\ref{fig:H}.
	
	\begin{prop}\label{p:H} There is a unique equisingular symplectic isotopy class of the configuration $\Hc$.
	\end{prop}
	
	\begin{proof}
		Consider the configuration $\mathcal{H}_0$ of a single symplectic conic with three distinct tangent lines. This configuration has a unique equisingular symplectic isotopy class by Corollary~\ref{c:oneconicthreelines}. Moreover, for each symplectic realization of $\mathcal{H}_0$ in $\cptwo$, the space of almost complex structures $J$ on $\cptwo$ for which the realization is $J$-holomorphic is non-empty and contractible by Lemma~\ref{l:Jcompatible}. Therefore the space of pairs $(H_0, J_0)$ where $H_0$ is a $J_0$-holomorphic realization of $\mathcal{H}_0$ is a path-connected space.
		
		For a fixed pair $(H_0, J_0)$, we will prove that we can add $J_0$-holomorphic lines to $H_0$ to get a realization of $\mathcal{H}$ in exactly two ways, and the map $p:\{H\in \mathcal{H} \text{ symplectic}\} \to \{(H_0,J_0)\in \mathcal{H}_0\times \mathcal{J}^{\omega}(H_0) \}$ is a double covering map. Then, to show that it is the connected cover rather than the trivial covering, we will show that for a particular pair $(H_0^{\rm std},i)$, the two extensions to complex realizations of $\mathcal{H}$ are symplectically isotopic. Therefore any two symplectic realizations $H_1, H_2$ of $\mathcal{H}$ can be equisingularly symplectically isotoped each to one of the two extensions of $(H_0^{\rm std},i)$ and these two extensions are equisingularly symplectically isotopic, so $H_1$ and $H_2$ are as well.
		
		For a given $B=(H_0, J_0)$, let $Q$ denote the conic in $H_0$, and let $p_1,p_2,p_3$ denote the three points of intersection of the three tangent lines. Consider the three pencils of lines $\pi_i: \cptwo\setminus \{p_i\} \to \ell_i$ through the points $p_i$, $i=1,2,3$, where $\ell_i$ are $J_0$-holomorphic lines not containing $p_i$. Restricting these pencils to the conic $Q$, gives a branched double covering map $\pi_i|_Q: Q\to \ell_i$ with exactly two branch points. The map is degree two because a generic line intersects the conic $Q$ in two points. There are two branch points, each corresponding to a point $q\in Q$ where the $J_0$-holomorphic line through $q$ and $p_i$ is tangent to $Q$, by the Riemann--Hurwitz formula: $2=2(2)-\sum_{q\in Q} (e_{\pi_i}(q)-1)$.
		
		Let $f^B_i:Q\to Q$ denote the unique involution such that $\pi_i\circ f^B_i = \pi_i|_Q$. Note that $f^B_i$ is independent of the choice of the auxiliary line $\ell_i$. Informally, $f^B_i(q)$ is the other intersection of the line $\ell_{p_i,q}$ through $p_i$ and $q$ with $Q$; if $\ell_{p_i,q}$ is tangent to $Q$, then $f^B_i(q) = q$. 
		
		Consider now the composition $f^B_3\circ f^B_2 \circ f^B_1: Q\to Q$: this is a complex automorphism of $Q$, which therefore has two fixed points $a_1$ and $a_2$. Note that if $q$ is one of the tangent points between $Q$ and a line in $H_0$, then $q$ is not a fixed point of $f^B_3\circ f^B_2 \circ f^B_1$ because each tangency point is fixed by exactly two of the $f^B_i$ and is sent to a different point by the third $f^B_i$. 
		
		Add to $H_0$ $J_0$-holomorphic lines $L_1$ through $p_1$ and $a$, $L_2$ through $p_2$ and $f^B_1(a)$, and $L_3$ through $p_3$ and $f^B_2(f^B_1(a))$. Note that $L_3$ intersects $L_1$ at a point on $Q$ if and only if $a\in \{a_1,a_2\}$. Therefore each realization of $\mathcal{H}$ which is $J_0$-holomorphic and whose conic and tangent lines agree with $H_0$, is given by $H\cup L_1\cup L_2\cup L_3$ where the $L_i$ are defined using $a\in \{a_1,a_2\}$. This tells us that the preimage of a point in the map $p:\{H\in \mathcal{H} \text{ symplectic}\} \to \{(H_0,J_0)\in \mathcal{H}_0\times \mathcal{J}^{\omega}(H_0) \}$ is two realizations of $\mathcal{H}$. Moreover, the constructions of $a_1,a_2$ and $L_1,L_2,L_3$ depend continuously on $J_0$, the conic $Q$, and the points $p_1,p_2,p_3$, which depend continuously on $H_0\in\mathcal{H}_0$, so $p$ is a covering map.
		
		Finally, we consider the relation between two explicit complex realizations of $\mathcal{H}$ with the same underlying $H_0$. Let $H_0$ be the following realization of $\mathcal{H}_0$. Start with three lines in $\R^2$ intersecting to form an equilateral triangle $ABC$ with side length $\overline{AB} = 1$ in the real Euclidian plane, and its inscribed circle $Q$, and then we complexify the configuration.
		For $\theta\in(0,\frac\pi3)$, consider the half-line $\ell^\theta_A$ starting from $A$, interior to the angle $BAC$, and such that the angle between $\ell_A$ and $AB$ is $\theta$. Analogously, define $\ell^\theta_B$ and $\ell^\theta_C$.
		Let $X = \ell^\theta_B\cap\ell^\theta_C$, $Y = \ell^\theta_C\cap\ell^\theta_A$, and $Z = \ell^\theta_A\cap \ell^\theta_B$ (here we drop the dependence on $\theta$ for convenience).
		We claim that there is a unique $\theta_0\in(0,\frac\pi6)$ such that $X$, $Y$, and $Z$ all lie on $Q$ and that for $\frac\pi3 - \theta_0$ the corresponding points all lie on $Q$, too.
		The area of the triangle $XYZ$ is given by:
		\begin{align*}
		A_{XYZ} &= A_{ABC} - A_{ABZ} - A_{BCX} - A_{CAY} = A_{ABC} - 3A_{ABZ}\\
		&= \frac{\sqrt3}4 - 3\cdot\frac12 \overline{AX}\cdot\overline{AB}\cdot\sin\theta
		= \frac{\sqrt3}4 - 3\cdot\frac12\cdot\frac{\overline{AB}\cdot\sin\big(\frac\pi3-\theta\big)}{\sin\frac{2\pi}{3}}\cdot \overline{AB}\cdot\sin\theta\\
		&= \frac{\sqrt3}4 - \sqrt3\cdot\sin\big(\textstyle\frac\pi3-\theta\big)\sin\theta
		= \frac{\sqrt3}4\big(1- 2\cos\big(\textstyle\frac\pi3-2\theta\big) + 2\cos\big(\textstyle\frac\pi3\big)\big)\\
		&= \frac{\sqrt3}2\big(1-\cos\big(\textstyle\frac\pi3-2\theta\big)\big)
		\end{align*}
		and the latter is a continuous, decreasing function from $[0,\frac\pi6]$ to $[0,+\infty)$, taking values $\frac{\sqrt3}4$ at $\theta = 0$ and $0$ at $\theta = \frac\pi6$.
		This means that there is a unique value of $0<\theta_0<\frac\pi6$ such that the area of $XYZ$ is the area of the inscribed triangle in $Q$. 
		By rotational symmetry, at $\theta_0$ the triangle $XYZ$ is in fact inscribed in $Q$, i.e. the configuration of $H = H_0 \cup \ell^{\theta_0}_A \cup \ell^{\theta_0}_B \cup \ell^{\theta_0}_C$ is a realization of $\mathcal H$.
		We observe that a reflection across one of the axes of the triangle preserves the triangle and the inscribed circle. 
		Since it preserves incidences, it sends the configuration $H$ to a configuration $H'$ that realizes $\mathcal H$ (indeed, $H'$ corresponds to the solution $\theta_1 = \frac\pi3-\theta_0$; note that $\theta_1 > \frac\pi6$).
		We can extend the reflection to a complex linear isometry in $PU(3)$. Since $PU(3)$ is a path-connected subset of ${\rm Symp}(\cptwo)$, there is a family of symplectomorphisms carrying $H$ to $H'$ and tracing out a symplectic isotopy between them.
	\end{proof}
	
	The second configuration, that we call $\Lc$, is made up of two conics $\mathcal{Q}_1,\mathcal{Q}_2$, tangent at two points, together with three lines such that each line is tangent to $\mathcal{Q}_1$ and the pairwise intersections of the lines are three distinct points on $\mathcal{Q}_2$. See Figure~\ref{fig:L}.
	
	\begin{figure}
		\centering
		\labellist
		\pinlabel $\mathcal{Q}_1$ at 128 98
		\pinlabel $\mathcal{Q}_2$ at 63 200
		\endlabellist
		\includegraphics[scale=.5]{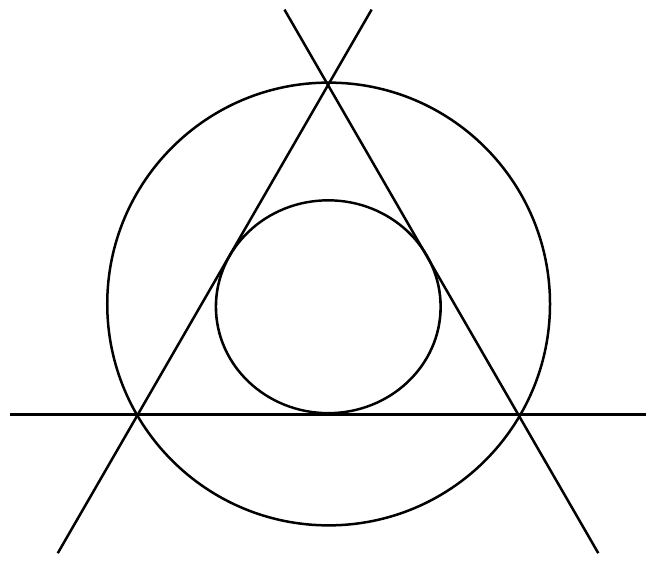}
		\caption{The configuration $\Lc$ of two conics and three lines. There are two simple tangency points between the conics not visible in this real picture.}
		\label{fig:L}
	\end{figure}
	
	\begin{prop}\label{p:L}
	There is a unique non-empty equisingular symplectic isotopy class of the configuration $\Lc$.
	\end{prop}
	
	\begin{proof}
	We define two auxiliary configurations, $\Lc'$ and $\Lc'_0$.
	The latter, $\Lc'_0$, consists of a tricuspidal quartic $\mathcal{V}$ and the triangle of lines passing transversally through its singularities.
	The former, $\Lc'$, is obtained from $\Lc'_0$ by adding a bitangent to $\mathcal{V}$, i.e. a line that is tangent to $\mathcal{V}$ at two distinct points, each of which is a smooth point of $\Vc$.
			
	First, observe that $\Lc$ is birationally equivalent to the configuration $\Lc'$.
	This is seen by blowing up once at each vertex of the triangle formed by the three lines, and then blowing down the proper transforms of the three lines (which are $(-1)$-spheres).

	Next, we see that $\Lc'_0$ has a unique symplectic realization, up to isotopy.
	Using the inverse birational equivalence just described, we see that $\Lc'_0$ is birationally equivalent to the configuration $\mathcal{H}_0$ of a conic inscribed in a triangle of lines, which has a unique realization, up to isotopy, by Corollary~\ref{c:oneconicthreelines}.
	
	What we really want to show is that $\Lc'$ has a unique symplectic isotopy realization, which will follow from showing that there is a unique way, up to symplectic isotopy, to add the bitangent line to $\Lc'_0$ to get $\Lc'$.
	This is shown in the following proposition. Then because of the birational equivalence above, $\Lc$ itself will have a unique non-empty equisingular isotopy class.
	\end{proof}		

	\begin{prop}\label{p:bitangent}
	In any symplectic realization $V\cup L_1\cup L_2\cup L_3$ of $\Lc'_0$, the tricuspidal quartic $V$ has a unique bitangent up to isotopy, and this bitangent necessarily intersects the lines $L_i$ in generic transverse double points.
	Therefore, $\Lc'$ has a unique symplectic realization.
	\end{prop}
	
%
	\begin{proof}
	Fix a tricuspical quartic $V\subset \cptwo$, and call $V_0$ the set of smooth points of $V$.
	Let $n:\cpone \to V$ be the normalization map.
	Fix an almost complex structure $J$ compatible with $\omega$ such that $V$ is $J$-holomorphic. Note that the set of such choices is contractible by Lemma \ref{l:Jcompatible}. 
	We will prove that there is a unique $J$-holomorphic bitangent line to $V$. Since this will hold for every $J$	compatible with $V$, we will have a unique $J$-holomorphic realization of $\Lc'$ whose tricuspidal quartic agrees with $V$ for each $J\in \mathcal{J}^\omega(V)$. Varying $J$ through this contractible (and thus path-connected) space yields equisingular isotopies fixing $V$ and isotoping the lines (the triangle of lines $L_1,L_2,L_3$ through the cusps and the bitangent line) in the configuration of $\Lc'$ by definining the lines $L_i$ and $T$ as the unique $J_t$-holomorphic lines with the specified intersection properties. Therefore any two symplectic extensions of $V$ to a realization of $\Lc'$ will be related by an equisingular symplectic isotopy.
	
	From now on, we fix a $J$ which makes $V$ $J$-holomorphic, and all of our choices of curves will be $J$-holomorphic.	
	Fix a generic $J$-holomorphic line $\ell$ to be the target of a pencil-like map.
	
	First, we associate to each point $p\in V_0$, a $J$-holomorphic line $\ell_p$ through $p$ with the property that $\ell_p$ is tangent to $V$ at another point $q\neq p$, $q\in V$. We prove that such $\ell_p$ exists and is unique in Lemma~\ref{l:defellp}.
	By tangency at $q$, we mean that either $q\in V_0$ and $\ell_p$ is tangent to $V$ at $q$, or that $q$ is a cusp of $V$, and $\ell_p$ is the tangent to the cusp.
	We also note that, in fact, the line $\ell_p$ might also be tangent to $V$ at $p$, in which case it is a bitangent to $V$.
	
	Next, define a map $\phi_0: V_0 \to \ell$ by $p\mapsto \ell_p\cap \ell$. 
	Extend this map to $V$ by letting, for each cusp $p$ of $V$, $\ell_p$ be the tangent to $V$ at the cusp.
	We can show that this extension gives a continuous map $\phi_1: V\to \ell$ as follows.
	For any point $q$ in a neighborhood of a cusp $p$, let $T_q$ be the unique $J$-holomorphic tangent line to $V$ at $q$. Then the family of submanifolds $T_q$ varies continuously with $q$. Therefore the intersection points with multiplicities $T_q\cap V$ also vary continuously with $q$. When $q\neq p$, $T_q\cap V=\{q,p_q,r_q\}$ where $q$ has multiplicity two and $p_q$ and $r_q$ have multiplicity one. When $q=p$, $T_p\cap V = \{p,r\}$ where $p$ has multiplicity three and $r$ has multiplicity one. Therefore we must have $r_q\to r$ as $q\to p$ and $p_q\to p$ as $q\to p$. Now, let $N$ be a small closed neighborhood of $p$ in $V$, and consider the map $f: N\to V$ defined by $f(q)=p_q$ (we choose $N$ sufficiently small to be able to distinguish $p_q$ (the points which converge to $p$ as $q\to p$) from $r_q$ (the points which converge to $r$ as $q\to p$). Then $f$ is a continuous map and $f(p)=p$. Additionally, $f$ is injective because if $f(q)=f(q')=p_q$ then the $J$-holomorphic pencil based at $p_q$ would include a tangent line to $V$ at $q$ and a tangent line to $V$ at $q'$, but such a pencil can only include a single tangent line to $V$ by Lemma~\ref{l:defellp}. Since $N$ is compact, and $f$ is continuous and injective, $f$ is a homeomorphism onto its image. Therefore $f(N)\subset V$ is homeomorphic to a disk centered at $p$. In particular, there is a neighborhood $U$ of $p$ such that for all $p_t'\in U$, $\ell_{p_t'}$ is tangent to $V$ inside $N$ so $\ell_{p_t'} = T_{q_t}$ for some $q_t\in N$. Therefore if $p_t'\to p$, $\ell_{p_t'}=T_{q_t}\to T_p$.
	
	Let $\phi = \phi_1 \circ n:\cpone\to\ell$. We will show that $\phi:\cpone \to \ell$ is a branched covering of degree $6$ and analyze all the ways in which ramification points can arise.
	
	Given a point $s\in \ell\setminus V$, $\phi^{-1}(s)$ consists of points $p$ such that the unique $J$-holomorphic line through $s$ and $p$ is tangent to $V$ at a point different from $p$ (i.e. $\ell_p$ passes through $s$). Therefore, we need to understand the $J$-holomorphic lines in the pencil based at $s$ which have tangencies to $V$. Since $s\notin V$, the Riemann--Hurwitz formula gives us that for this pencil, the sum of the ramification numbers corresponding to precomposing this pencil with the normalization is $\sum (e(x_j)-1)=6$. Each of the three cusps will be a ramification point $n(x_j)$ for some $j$. We will have $e(x_j)=2$ if the line in the pencil through the cusp $n(x_j)$ is transverse to $V$ at $n(x_j)$, and $e(x_j)$ is at least $3$ if the line in the pencil through the cusp $n(x_j)$ is tangent to $V$ at $n(x_j)$.
	Therefore we could have in the pencil through $s$ either
	\begin{enumerate}[label=(\alph*)]
		\item three distinct tangent lines to $V$, each intersecting $V$ transversally at two other points (or if the tangency is at a cusp, intersecting $V$ only at one other point transversally, but counting the cusp as the other point),
		\item one bitangent line to $V$ and one tangent line with two transverse intersections to $V$, or
		\item at least one line in the pencil has a higher order tangency to $V$.
	\end{enumerate}
	
	(Note that there cannot be a line with three distinct tangencies to $V$ since $V$ only has degree $4$ and three tangencies would contribute $6$ to the intersection.)
	
	In the first case (the generic case), we have $6$ points in $\phi^{-1}(s)$ coming from the six transverse intersections of the tangent lines with $V$. (In the case that one of the tangent lines is tangent at the cusp, the two points are the cusp point and the transverse intersection.) Since transversality is an open condition, any such $s$ has an open neighborhood $U$ such that the pencil through $s'\in U$ falls into the first case. After possibly shrinking $U$, we can show $\phi^{-1}(U)$ is homeomorphic to $U\times \{1,2,3,4,5,6\}$. To show this, let $q\in V$ be a point such that the line from $q$ to $s$ is tangent to $V$ at $q$. For $q_t$ converging to $q$, the tangent line to $V$ at $q_t$ intersects $\ell$ at points $s_t$ convering to $s$, and intersects $V$ at points $p_t^1$ and $p_t^2$ converging to $p^1,p^2\in \phi^{-1}(s)$. Because the number of tangent lines to $V$ in any given pencil is finite, we can ensure that if $q_t$ is sufficiently close to $q$, then $p_t^i\neq p_{t'}^i$ for $i=1,2$, and $s_t\neq s_{t'}$ for $t\neq t'$. Therefore the maps $q_t\mapsto s_t$ and $q_t\mapsto p_t^i$ are continuous and injective maps, so they locally have continuous inverses. Since $\phi(p_t)=s_t$, we see that $\phi$ is a covering map of degree $6$ at generic points.	
	
	In the second case, the bitangent line intersects $V$ at the two tangent points, $p_0,p_1$. 
	Let $\tau_0$ denote the $J$-line in the pencil through $s$ which is bitangent to $V$ and let $\tau_1$ denote the $J$-line in the pencil through $s$ which is simply tangent to $V$.
	Therefore by definition of $\phi$, $\phi^{-1}(s)$ consists of the two transverse intersections of $\tau_1$ with $V$ and the two tangential intersections $\tau_0\cap V$. Thus there are only four preimage points instead of $6$ so the contribution to the Euler characteristic reduction is $\sum_{x\in \phi^{-1}(s)}(e_\phi(x)-1)=2$.
	
	The third case is ruled out as a possibility in Lemmas~\ref{l:noinfl} and~\ref{l:no4tan}.
	
	Next, consider a point $s\in \ell\cap V$ (there are four of them). It follows from Lemma \ref{l:defellp} that there is a unique line $\ell_s$ which passes through $s$ and is tangent to $V$ at a different point $q\neq s$. Its two transverse intersections with $V$ ($s$ and one other point $p$) will be regular points in $\phi^{-1}(s)$ as in the generic case. The only other way that a line tangent to $V$ could pass through $s$ is if the tangency occurs at $s$. Let $T_s$ be the tangent $J$-line to $V$ at $s$. Then the only other points in $\phi^{-1}(s)$ are the two points where $T_s$ intersects $V$ transversally. (By genericity of $\ell$, we ensure that $T_s$ intersects $V$ at two points transversally.) Therefore $\phi^{-1}(s)$ consists of four points instead of six, so at each intersection $s\in \ell\cap V$, $\sum_{s\in \phi^{-1}(s)}(e_\phi(x)-1)=2$.
	
	Finally, we apply the Riemann--Hurwitz formula to $\phi$:
	\[
	2 = \chi(\cpone) = 6\chi(\cpone) - \sum_{x\in\cpone} (e_\phi(x)-1) = 12 - 8 - \sum_{x\in\cpone\setminus n^{-1}(\ell)} (e_\phi(x)-1).
	\]
	Therefore $\sum_{x\in \cpone\setminus n^{-1}(\ell)} (e_\phi(x)-1)=2$.
	Since the only contribution to ramification indices for points in $\phi^{-1}(\ell\setminus V)$ is a bitangent, and each bitangent contributes $2$ to this sum, we conclude that there is a unique $J$-holomorphic bitangent line to $V$ as claimed.
	\end{proof}	
		
	\begin{lemma} \label{l:defellp}
	For every point $p\in V_0$ there is a unique line $\ell_p$ through $p$ that has a tangency to $V$ at another point $q\neq p$.
	\end{lemma}
	
	\begin{proof}
	Consider the $J$-linear pencil based at $p$: $\pi_0: \cptwo\setminus\{p\} \to \cpone$.
	Restrict $\pi_0$ to $V\setminus \{p\}$, and extend it to $\pi_1: V \to \cpone$ by defining $\pi_1(p) = T_p$, where $T_p$ is the tangent to $V$ at $p$.
	The map $\pi_1$ is continuous: the tangent $T_p$ is the only curve that has local multiplicity of intersection $2$ with $C$ at $p$, and the lines $\ell_{pq}$ through $p$ and $q$ converge to such a curve as $q$ limits to $p$.
	Finally, define $\pi: \cpone \to \cpone$ by $\pi = \pi_1\circ n$, where $n: \cpone \to V$ is the normalization.
	By positivity of intersections, $\pi$ is a branched cover $\cpone \to \cpone$ of degree $4-1 = 3$.
	Moreover, if $n(x)$ is a cusp of $V$, then $x$ is a ramification point of $x$.
	The index of $x$ as a ramification point is $2$ if $p$ is not the intersection of the tangent to $\pi(x)$ with $V$, and $3$ otherwise.
	
	In the former case, the Riemann--Hurwitz formula yields:
	\[
	2 = \chi(\cpone) = 3\chi(\cpone) - \sum_{x\in\cpone} (e_\pi(x)-1) = 6 - 3 - \sum_{n(x)\in V_0} (e_\pi(x)-1),
	\]
	which implies that there is exactly one point $x$ in $\cpone$ such that $\pi$ has ramification of index $2$ at $x$.
	This means exactly that $\ell_p$ is tangent to $V$ at $n(x)$.
	
	In the latter case, Riemann--Hurwitz gives:
	\[
	2 = \chi(\cpone) = 3\chi(\cpone) - \sum_{x\in\cpone} (e_\pi(x)-1) = 6 - 2 - 1 - 1 - \sum_{n(x)\in V_0} (e_\pi(x)-1),
	\]
	which means that the only ramification points of $\pi$ are at the preimages of cusps of $V$.
	This means that the only tangent to $V$ through $p$ is the tangent to the cusp.
	
	We exclude the possibility that $n^{-1}(p)$ itself is a ramification point of $\pi$ in Lemma~\ref{l:noinfl}.
	If this occurred, $T_p$ would be a simple inflection line of $V$, which is exactly what the lemma obstructs.
	\end{proof}
	
%
	
		\begin{lemma} \label{l:noinfl}
			The tricuspidal quartic $V$ has no simple inflection $J$-lines at non-singular points.
		\end{lemma}
		
		\begin{proof}
			Suppose $p\in V_0$ is a simple inflection point, and let $q$ be the unique other intersection of $T_p$ and $Q$.
			Observe that, by positivity of intersections, $q\in V_0$.
			Now define the branched covering map $\pi': \cpone \to \cpone$ as in Lemma~\ref{l:defellp}, but use the point $q$ as the base of the pencil instead of $p$. 
			This is a degree $3$ cover.
			Each of the three cusps has ramification index $2$, and the inflection point has ramification index at least $3$. However, using the Riemann--Hurwitz formula, we get a contradiction:
			\[ 2= 3\cdot 2 -\sum_{x\in \cpone} (e_{\pi'}(x)-1) \leq 6-1-1-1-2=1.\qedhere  \]
		\end{proof}
		
		\begin{lemma} \label{l:no4tan}
			The tricuspidal quartic $V$ has no multiplicity-$4$ tangent $J$-lines.
		\end{lemma}
		
		\begin{proof}
		Suppose that there was such a $J$-line and call $p$ its (unique) intersection point with $V$. $p$ could either be a non-singular point, i.e. $p\in V_0$, or a cusp of $V$. In either case, we consider the pencil of $J$-lines through $p$ and project from $p$. Precomposing with the normalization of $V$, we get a branched cover $\pi' \colon \cpone \to \cpone$.
		
		In the first case, when $p$ is a non-singular point of $V$, $\pi'$ is a $3$-fold branched covering, with ramification points of index at least $2$ at each of the three cusps and of index $4$ at $p$ itself. This violates the Riemann--Hurwitz formula because we would have
			\[2 =3\cdot 2 -\sum_{x\in \cpone} (e_{\pi'}(x)-1) \leq 6-3-1-1-1=0.\]
			
		In the second case, when $p$ is a cusp, the considering the same pencil, $\pi'$ is a $2$-fold branched covering which has ramification points of index at least $2$ at the two remaining cusps and exactly $2$ at $p$, yielding a similar contradiction:
			\[2=2\cdot 2 -\sum_{x\in \cpone}(e_{\pi'}(x)-1) \leq 4-2-1-1=1.\qedhere\]
		\end{proof}

%

%
%

	\subsection{Obstructions} \label{s:reducibleobstructions}
	
	Next we will show that certain singular configurations do not symplectically embed into $\cptwo$ using birational derivations, and the following result.
	
	\begin{theorem}[{\cite[Theorem~1.1]{RuSt}}] \label{thm:Fano}
		There is no symplectic embedding in $\cptwo$ of the Fano plane: seven lines intersecting positively at seven triple intersection points (see Figure~\ref{fig:Fano}).
	\end{theorem}
	
	\begin{figure}
		\centering
		\includegraphics[scale=.5]{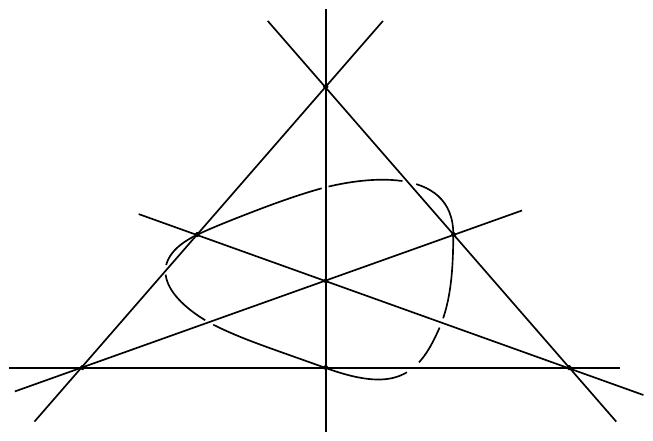}
		\caption{The Fano configuration of seven lines intersecting in seven triple points.}
		\label{fig:Fano}
	\end{figure}
	
	The next configuration we will consider is made up of three lines each tangent to a conic such that the three lines intersect each other at a triple point. We will call this configuration $\mathcal{G}$ (Figure~\ref{fig:G}).
	
	\begin{figure}
		\centering
		\includegraphics[scale=.4]{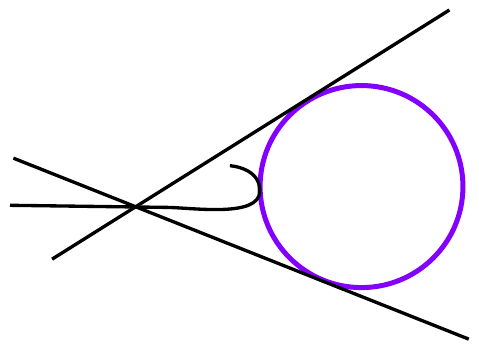}
		\caption{Configuration $\mathcal{G}$ consists of three lines intersecting at a common triple point, each tangent to a conic.}
		\label{fig:G}
	\end{figure}

	\begin{prop}\label{p:GFano} There is no symplectic embedding of $\mathcal{G}$ into $\cptwo$. \end{prop}

	\begin{proof}
		Let $Q$ be any smooth symplectic sphere in $\cptwo$ with $[Q]=2h$. Let $T_1$, $T_2$, and $T_3$ be three symplectic tangent lines at points $p_1$, $p_2$, and $p_3$ on $Q$. Suppose $T_1$, $T_2$, and $T_3$ all intersect at a common triple point.
		
		Follow Figure~\ref{fig:GtoFano} with the rest of the proof. Blow up once at each of $p_1$, $p_2$, and $p_3$. Then the proper transform of $Q$ is a symplectic sphere of self-intersection $+1$ so by Theorem~\ref{thm:mcduff} it can be identified with $\cpone$ with homology class $h$ via a symplectomorphism $\Psi$ of $\cptwo\# 3\cptwobar$. Under this identification, Lemma~\ref{l:hom} and the intersections of the components determines the homology classes of the remaining curves as follows:
		\[
		\Psi_*[\overline{Q}]=h, \qquad \Psi_*[\overline{T_i}] =h-e_i, \qquad \Psi_*[E_i]=h-e_1-e_2-e_3+e_i.
		\]
		
		Next, by Lemma~\ref{l:blowdown}, we can blow down exceptional spheres in the homology classes $e_i$ such that any intersection with $\Psi(\overline{Q})$, $\Psi(\overline{T_i})$, or $\Psi(\overline{E_i})$ is positive. After blowing down three times, the proper transform in $\cptwo$ of each of the seven surfaces is in the homology class $h$ and there are the following triple intersections: $P_i=\Psi({\overline{Q}})\cap \Psi(E_i)\cap \Psi(\overline{T_i})$ for $i=1,2,3$, $R_1=\Psi(\overline{T_1})\cap \Psi(E_2) \cap \Psi(E_3)$, $R_2=\Psi(\overline{T_2})\cap \Psi(E_1) \cap \Psi(E_3)$, $R_3=\Psi(\overline{T_3})\cap \Psi(E_1) \cap \Psi(E_2)$. If there were a seventh triple intersection between $\Psi(\overline{T_1})$,  $\Psi(\overline{T_2})$, and $\Psi(\overline{T_3})$ this would give an embedding of the Fano configuration into $\cptwo$.
		
		Thus we have shown that the Fano plane is birationally derived from $\mathcal{G}$, so by the contrapositive of Proposition~\ref{l:biratexist}, there is no symplectic realization of $\mathcal{G}$ in $\cptwo$. (In fact, the same argument using the full strength of Theorem~\ref{thm:mcduff} shows that there is no symplectic realization of a configuration $\mathcal{G}$ with the same self-intersection numbers into any closed symplectic $4$-manifold.)
	\end{proof}
	
	Note that the exceptional spheres representing $e_1,e_2,e_3$ are not contained in the total transform of the embedding of $\mathcal{G}$, so the Fano plane is birationally derived from $\mathcal{G}$, but they are not birationally equivalent.

	\begin{figure}
		\centering
		\includegraphics[scale=.5]{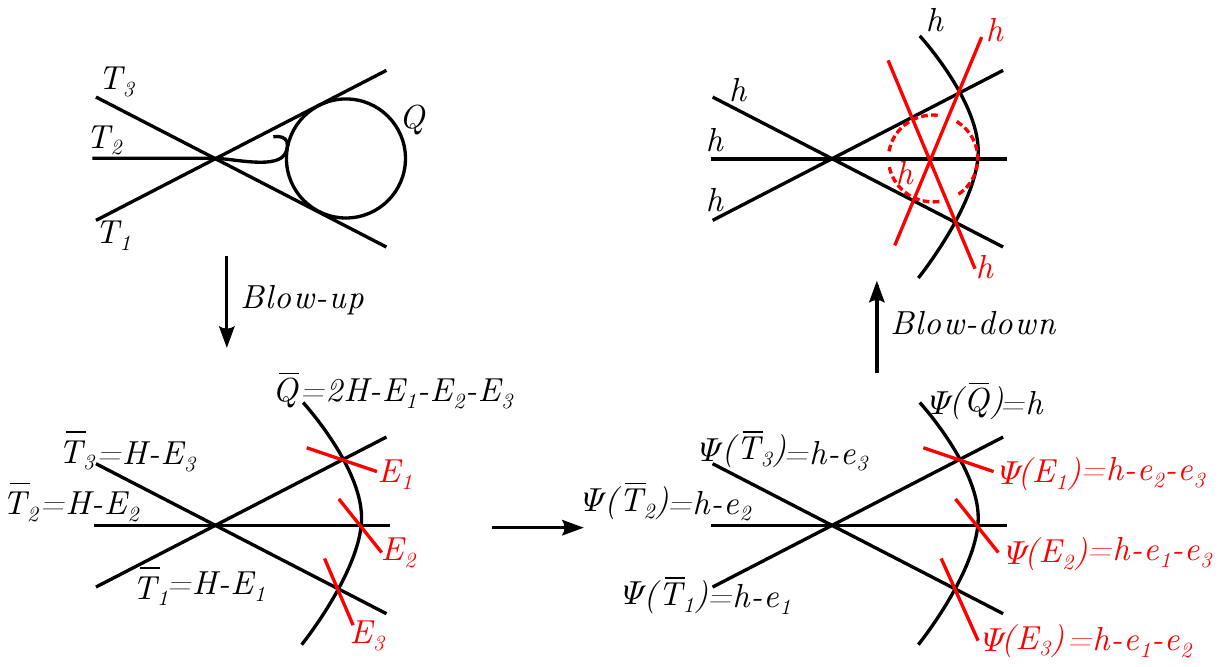}
		\caption{Birational derivation from $\Gc$ to the Fano configuration.}
		\label{fig:GtoFano}
	\end{figure}

$\mathcal{G}$ is dual to a configuration $\mathcal{G}^\star$, comprised of three conics in a pencil and a line tangent to all three of them. {We will actually use $\mathcal{G}^\star$ to denote a few different configurations which include the different possibilities for the pencil where the conics may intersect transversally or tangentially to different degrees. Specifically, the three conics may intersect in four transverse triple points, or one simple triple tangency and two transverse triple points, or one degree three triple tangency and one transverse triple point, or one degree four triple tangency, or two simple triple tangencies. (The latter two cases will also be obstructed as a consequence of Proposition~\ref{p:G422}.)}
We can also obstruct $\mathcal{G}^\star$ (in all its versions).

\begin{prop}\label{p:Gstar}
There is no symplectic embedding of the configuration $\mathcal{G}^\star$ in $\cptwo$.
\end{prop}

\begin{proof}
We will show that there is a birational derivation from $\mathcal{G}^\star$ to a configuration containing $\mathcal{G}$.
Given a symplectic realization of $\mathcal{G}^\star$, blow up at three of the four basepoints of the pencil {(counted with multiplicity)}. The proper transforms of the conics are symplectic +1-spheres, intersecting at a single point (namely, the fourth basepoint). 
We can apply Theorem~\ref{thm:mcduff} to identify one of the three conics with $\cpone$. Then Lemma~\ref{l:hom} implies the other two conics are also symplectic lines, and the image of the tangent line has class $2h-e_1-e_2-e_3$. {The classes represented by the three exceptional spheres depends on the type of pencil we started with (and the choice of 3 pencil points to blow-up). For the generic pencil with four transverse triple points,} the three exceptional spheres represent $h-e_1-e_2$, $h-e_2-e_3$, and $h-e_1-e_3$. {In other cases, they could represent $h-e_1-e_2$, $h-e_2-e_3$, and $e_3-e_4$ or $h-e_1-e_2$, $e_2-e_3$, and $e_3-e_4$.} Blowing down positively intersecting exceptional spheres using Lemma~\ref{l:blowdown}, the total transform of the realization of $\mathcal{G}^\star$ blows down to a conic simultaneously tangent to each of the three concurrent lines, together with some additional lines coming from the images of the exceptional spheres. Since this configuration contains $\mathcal{G}$, the configuration $\mathcal{G}^\star$ is obstructed by Proposition~\ref{l:biratexist}.
\end{proof}

We also obstruct the following related configurations:
\begin{itemize}
\item[$\Gc_4$:] comprising two conics with a single point of tangency of order 4, and a line tangent to both;
\item [$\Gc_{2,2}$:] comprising two conics with two simple tangencies, and a line tangent to both.
\end{itemize}

\begin{prop}\label{p:G422}
There is no symplectic embedding of any of the configurations $\Gc_4$ or $\Gc_{2,2}$ in $\cptwo$.
\end{prop}

\begin{proof}
Suppose that there existed an embedding of $\Gc_4$ in $\cptwo$.
After blowing up three times at the tangency point of the two conics, the proper transforms of the two conics are two $+1$-spheres. See Figure~\ref{fig:G4toG}.
Applying Theorem~\ref{thm:mcduff}, we obtain a symplectomorphism of $\cptwo$ blown up three times that sends the proper transform of one of the conics to $\cpone$.
Using Lemma~\ref{l:hom} and intersection numbers, the proper transforms of all the conics are in the homology class $h$, the proper transforms of the three exceptional divisors (in the order in which we made the blow-ups) are mapped to spheres in homology classes $e_1-e_2$, $e_2-e_3$, and $h-e_1-e_2$, and the proper transform of the line is sent to the homology class $2h-e_1-e_2-e_3$.
Using Lemma~\ref{l:blowdown} to blow down a positively intersecting sphere in the class $e_3$ followed by the $(-1)$-spheres in the configuration representing $e_2$ and $e_1$, we get a birational derivation from $\Gc_4$ to $\Gc$.

\begin{figure}
	\labellist
	\pinlabel $+4$ at 50 160
	\pinlabel $+4$ at 0 160
	\pinlabel $+1$ at 43 70
	\pinlabel $h$ at 340 190
	\pinlabel $h$ at 387 190
	\pinlabel $+1$ at 315 110
	\pinlabel $+1$ at 395 110
	\pinlabel $+1$ at 355 70
	\pinlabel $2h-e_1-e_2-e_3$ at 355 0
	\pinlabel $-1$ at 410 160
	\pinlabel $h-e_1-e_2$ at 430 140
	\pinlabel $-2$ at 275 135
	\pinlabel $e_2-e_3$ at 285 190
	\pinlabel $-2$ at 225 135
	\pinlabel $e_1-e_2$ at 215 190
	\pinlabel $h$ at 540 190
	\pinlabel $h$ at 562 220
	\pinlabel $h$ at 590 190
	\pinlabel $2h$ at 567 50
	\endlabellist
	\centering
	\includegraphics[width=0.8\textwidth]{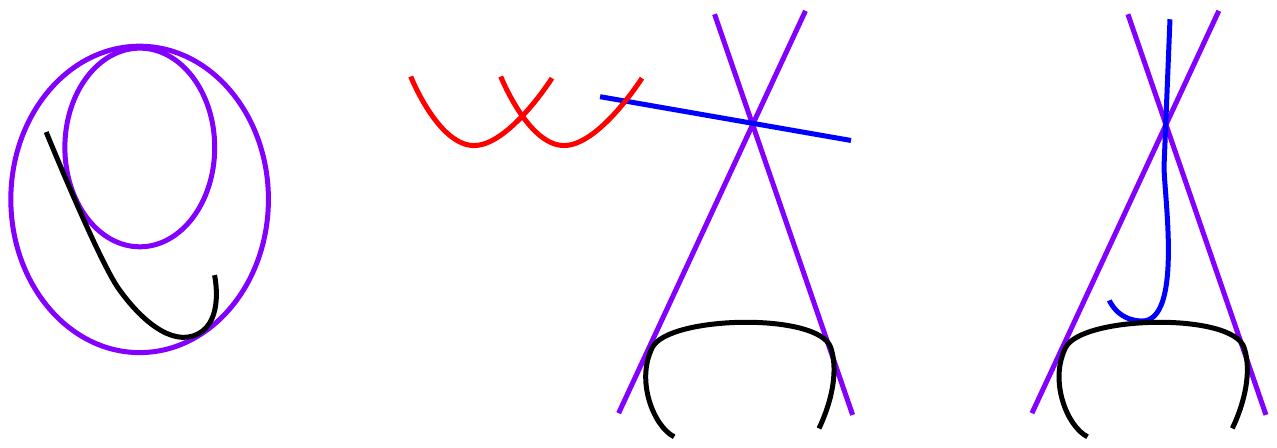}
	\caption{Birational derivation of $\mathcal{G}$ from $\mathcal{G}_4$.}
	\label{fig:G4toG}
\end{figure}

For $\Gc_{2,2}$ we blow up twice at one of the tangencies, and once at the other. See Figure~\ref{fig:G22toG}. Identifying the proper transforms of the conics with $\cpone$, the homology classes of the other curves in the total transform are uniquely determined by Lemma~\ref{l:hom} and intersection numbers. The proper transform of the line is again identified with a sphere in the class $2h-e_1-e_2-e_3$, and exceptional sphere at the intersection where we blew up once represents $h-e_1-e_2$. The two exceptional curves from the other blow-ups are identified with spheres representing the classes $h-e_2-e_3$ and $e_2-e_1$. Using Lemma~\ref{l:blowdown} to blow-down exceptional spheres in classes $e_1, e_2, e_3$, the configuration blows down to a configuration containing $\Gc$ (with an additional line). Therefore by Proposition~\ref{l:biratexist}, a symplectic embedding of $\Gc_{2,2}$ into $\cptwo$ is obstructed.
\end{proof}

\begin{figure}
	\labellist
	\pinlabel $+4$ at 70 180
	\pinlabel $+4$ at 0 180
	\pinlabel $+1$ at 150 130
	\pinlabel $h$ at 325 220
	\pinlabel $h$ at 375 220
	\pinlabel $+1$ at 325 243
	\pinlabel $+1$ at 375 243
	\pinlabel $+1$ at 355 70
	\pinlabel $2h-e_1-e_2-e_3$ at 355 20
	\pinlabel $-1$ at 415 160
	\pinlabel $h-e_2-e_3$ at 445 135
	\pinlabel $-1$ at 400 200
	\pinlabel $h-e_1-e_2$ at 440 180
	\pinlabel $-2$ at 235 125		
	\pinlabel $e_2-e_1$ at 240 165
	\pinlabel $h$ at 560 215
	\pinlabel $h$ at 582 245
	\pinlabel $h$ at 610 215
	\pinlabel $2h$ at 555 70
	\pinlabel $h$ at 530 145
	\endlabellist
	\centering
	\includegraphics[width=0.8\textwidth]{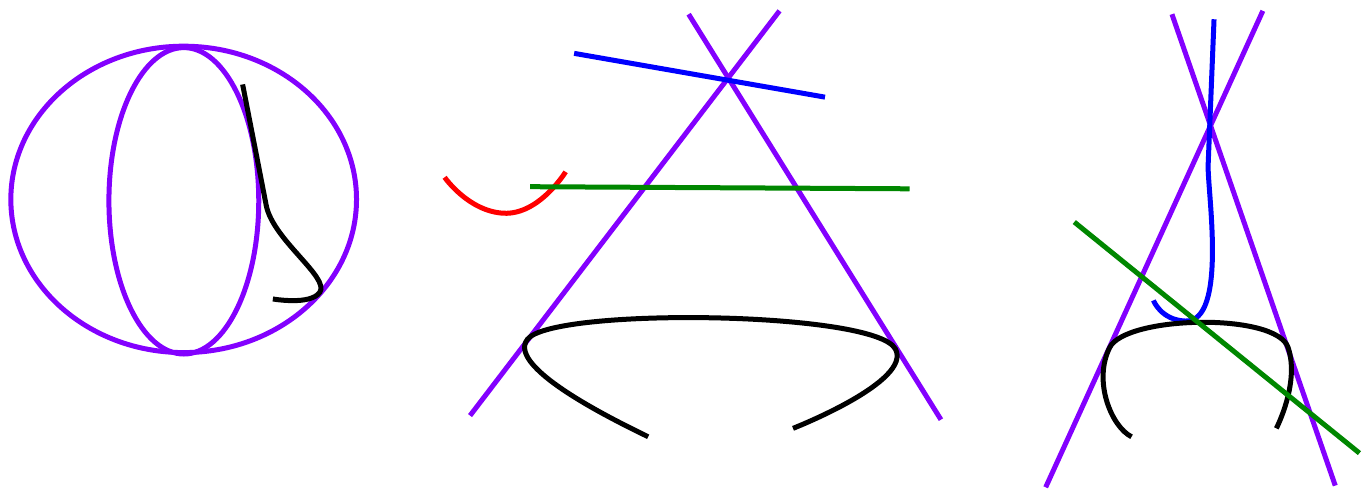}
	\caption{Birational derivation of $\mathcal{G}$ from $\mathcal{G}_{2,2}$.}
	\label{fig:G22toG}
\end{figure}

\begin{remark}
	The existence of a symplectic embedding of $\mathcal{G}$ into $\cptwo$ can also be obstructed using a $J$-holomorphic linear pencil based at the triple point $p$, and finding a contradiction using Riemann--Hurwitz.
	Indeed, consider the projection $\pi: Q \to \cpone$ induced by the pencil.
	By positivity of intersections, this is a branched covering map.
	Each tangent to $Q$ through $p$ gives a branching point of $\pi$ of index $2$, and Riemann--Hurwitz yields:
	\[
	2\cdot 2 = 2 \chi(Q) = \chi(\cpone) + \sum_{q\in Q} e_\pi(q) \ge  5,
	\]
	a contradiction.
	
	A similar argument also works for $\mathcal{G}^\star, \mathcal{G}_4$ and $\mathcal{G}_{2,2}$.
\end{remark}

\section{Unicuspidal curves with one Newton pair}\label{s:unicusp}
{
	The goal of this section is to prove Theorem~\ref{t:unicusp}, that every symplectic unicuspidal curve in $\cptwo$ whose singularity has one Puiseux pair is isotopic to a complex curve. 	In fact, for each of such curve we will classify \emph{all} of its symplectic embeddings into \emph{any} closed symplectic manifold. Correspondingly, we classify all the strong symplectic fillings of the corresponding contact manifolds.
%

	According to~\cite[Theorem~2.3]{Liu} (see also~\cite[Remark~6.18]{BCG}), if a rational cuspidal curve with a unique singularity of type $(p,q)$ satisfies the adjunction formula, that is $(p-1)(q-1) = (d-1)(d-2)$, and the semigroup distribution property~\eqref{e:semigroup}, then $(p,q)$ belongs to the list of~\cite[Theorem 1.1]{FdBLMHN}; namely, $(p,q)$ is one of:
	\begin{itemize}
		\item $(p,p+1)$, with $p\ge 2$, and the curve has degree $p+1$;
		\item $(p,4p-1)$, with $p\ge 2$, and the curve has degree $2p$
		\item $(F_{j-2},F_{j+2})$, with $j \ge 5$ and odd, and the curve has degree $F_j$;
		\item $(F_j^2; F_{j+2}^2)$ with $j \ge 3$ and odd, and the curve has degree $F_jF_{j+2}$;
		\item $(3,22)$, and the curve has degree $8$;
		\item $(6,43)$ and the curve has degree $16$.
	\end{itemize}
	
	Here, $F_j$ denotes the $j^{\rm th}$ Fibonacci number; recall that the sequence $\{F_j\}$ is defined by the recursion $F_{j+1} = F_j + F_{j-1}$, starting from $F_0 = 0$, $F_1 = 1$.
	
	\begin{remark}
	To be more precise, in~\cite{BCG} the result is only stated asymptotically, i.e. for ``sufficiently large'' degrees. (One can explicitly compute what ``sufficiently large'' and verify that the result holds for all degrees. This is not mentioned in~\cite{BCG}.) In~\cite{Liu}, Liu explains in the paragraphs above Theorem~1.1 and above Theorem~2.3 that his results are obtained in a purely combinatorial way from the semigroup distribution property, which is a consequence of the Heegaard Floer restriction together with the adjunction formula. A more general result than what is needed here was obtained in~\cite{AGLL}, where the requirement on the adjunction formula is dropped. Sifting through the lists of~\cite[Theorem~1.1]{AGLL}, one verifies that the only ones that satisfy the adjunction formula are the ones mentioned above.
	\end{remark}


	Since symplectic curves satisfy the adjunction formula~\eqref{e:adjunction} and the semigroup distribution property~\eqref{e:semigroup}, we can restrict to the six cases above.
	For each of them, we will use a resolution to classify symplectic fillings of the associated contact 3-manifold, and correspondingly classify the symplectic embeddings of these cuspidal curves in closed symplectic manifolds up to symplectic isotopy.
	
	In Section~\ref{ss:res}, we will describe the choice of resolutions that we use to classify the embeddings.
	Each such resolution will contain a smooth symplectic $+1$-sphere as the proper transform of the cuspidal curve.
	In Section~\ref{ss:homclass}, we use McDuff's theorem and the lemmas of Section~\ref{s:caps} to classify all homological embeddings of each of the resolutions.
	In this story, the two Fibonacci families will play a different role;
	we will treat them in Section~\ref{ss:fibonaccis}.
	In Section~\ref{ss:isofills}, we will look at geometric realization of these homological embeddings, and will prove a strengthening of Theorem~\ref{t:unicusp}.
	Finally, in Section~\ref{ss:QHBd}, we will talk about rational blow-down relations.

	\subsection{The resolutions}\label{ss:res}
	
	Here we describe our preferred resolutions for the singular curves in each of the six cases listed above, except the Fibonacci families (for which we will have a different argument in Section~\ref{ss:fibonaccis} below).
	In general, our preferred resolution will not be either the minimal or the normal crossing divisor resolution of the singularity;
	the goal will be to find a smooth symplectic $+1$-sphere.
	
	We start with the first family: a degree-$(p+1)$ curve with a unique singularity of type $(p,p+1)$.
	We first find the normal crossing resolution. 
	In the notation of Section~\ref{ss:resolution}, we have $q=p+1$, $p^* = p$ (since $p^2 \equiv 1 \pmod{p+1}$) and $q^* = 1$.
	We are interested in the continued fraction expansions of $\frac{p}{p-1}$ and $\frac{p+1}{p+1-p}$, which are $[2^{[p-1]}]^-$ and $[p+1]^-$, respectively.
	Since the curve has degree $p+1$, its self-intersection is $(p+1)^2$, and therefore the self-intersection of the proper transform in the normal crossing divisor resolution has self-intersection $(p+1)^2-p(p+1) = p+1$.
	From this normal crossing resolution, we blow up $p$ additional times along this third leg\footnote{What we call ``leg'' here is usually called a ``twig'' in algebraic geometry. We stick to the terminology of low-dimensional topology here.} (i.e. at the intersection between the latest exceptional divisor and the proper transform of the curve).
	The net effect of this operation is that the central vertex is decorated by Euler class $-2$ and the third leg consists of a chain of $(p-1)$ $(-2)$-vertices, a $(-1)$-vertex, and a $(+1)$-vertex.
	We will refer to this symplectic plumbing as $\mathcal{A}_p$; see Figure~\ref{f:Ap}.
	
	\begin{figure}
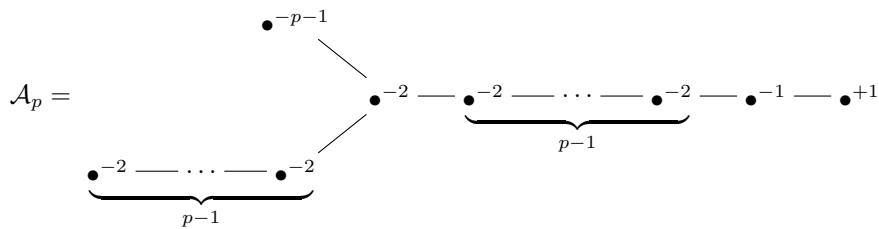

	\[
	\Ac_p = \xygraph{
	!{<0cm,0cm>;<1.25cm,0cm>:<0cm,1cm>::}
	!{(0,0) }*+{\bullet^{-2}}="c"
	!{(-1,1) }*+{\bullet^{-p-1}}="l1"
	!{(1,0) }*+{\bullet^{-2}}="l3a"
	!{(2,0) }*+{\dots}="l3b"
	!{(2,-0.5) }*+{\underbrace{\hphantom{aaaaaaaaaaaaaaaa}}_{p-1}}="brace1"
	!{(3,0) }*+{\bullet^{-2}}="l3c"
	!{(4,0) }*+{\bullet^{-1}}="l3d"
	!{(5,0) }*+{\bullet^{+1}}="l3e"
	!{(-1,-1) }*+{\bullet^{-2}}="l2a"
	!{(-2,-1) }*+{\dots}="l2b"
	!{(-2,-1.5) }*+{\underbrace{\hphantom{aaaaaaaaaaaaaaaa}}_{p-1}}="brace2"
	!{(-3,-1) }*+{\bullet^{-2}}="l2c"
	"c"-"l1"
	"c"-"l2a"
	"c"-"l3a"
	"l2a"-"l2b"
	"l2b"-"l2c"
	"l3a"-"l3b"
	"l3b"-"l3c"
	"l3c"-"l3d"
	"l3d"-"l3e"
	}
	\]
	\caption{The top and bottom legs correspond to the continued fraction expansions $\frac{p+1}{1} = [p+1]^-$ and $\frac{p}{p-1} = [2^{[p-1]}]^-$, respectively. The chain on the right is artificially longer than needed to make a $+1$-sphere appear.}\label{f:Ap}
	\end{figure}

	We now turn to the second case: a degree-$2p$ curve with a singularity of type $(p,4p-1)$.
	The normal crossing resolution of this unicuspidal curve has two legs with expansions of $\frac{p}{1}=[p]^-$ and $\frac{4p-1}{4p-5}=[2^{[p-2]},3,2,2]^-$, and the third leg has one vertex, which is initially decorated by $(2p)^2-p(4p-1)=p$.
	We blow up along this leg $p-1$ additional times, so that the central vertex is decorated by $-2$, and the third leg becomes a chain of $p-2$ vertices decorated by $-2$, one by $-1$, and one by $+1$.
	We will refer to this symplectic plumbing as $\mathcal{B}_p$.
	\begin{figure}
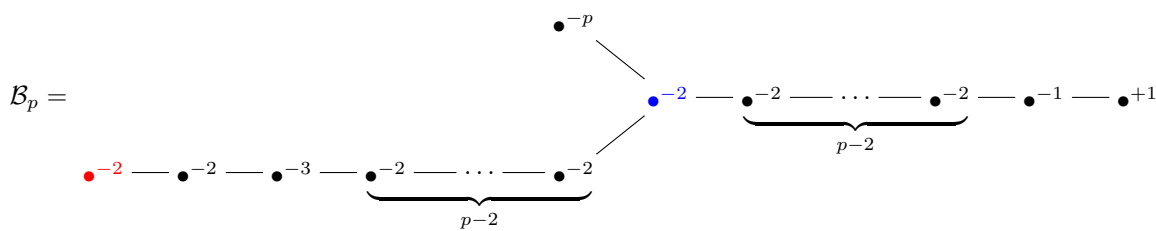

	\[
	\Bc_p = \xygraph{
		!{<0cm,0cm>;<1.25cm,0cm>:<0cm,1cm>::}
		!{(0,0) }*+[blue]{\bullet^{-2}}="c"
		!{(-1,1) }*+{\bullet^{-p}}="l1"
		!{(1,0) }*+{\bullet^{-2}}="l3a"
		!{(2,0) }*+{\dots}="l3b"
		!{(2,-0.5) }*+{\underbrace{\hphantom{aaaaaaaaaaaaaaaa}}_{p-2}}="brace1"
		!{(3,0) }*+{\bullet^{-2}}="l3c"
		!{(4,0) }*+{\bullet^{-1}}="l3d"
		!{(5,0) }*+{\bullet^{+1}}="l3e"
		!{(-1,-1) }*+{\bullet^{-2}}="l2a"
		!{(-2,-1) }*+{\dots}="l2b"
		!{(-2,-1.5) }*+{\underbrace{\hphantom{aaaaaaaaaaaaaaaa}}_{p-2}}="brace2"
		!{(-3,-1) }*+{\bullet^{-2}}="l2c"
		!{(-4,-1) }*+{\bullet^{-3}}="l2d"
		!{(-5,-1) }*+{\bullet^{-2}}="l2e"
		!{(-6,-1) }*+[red]{\bullet^{-2}}="l2f"
		"c"-"l1"
		"c"-"l2a"
		"c"-"l3a"
		"l2a"-"l2b"
		"l2b"-"l2c"
		"l2c"-"l2d"
		"l2d"-"l2e"
		"l2e"-"l2f"
		"l3a"-"l3b"
		"l3b"-"l3c"
		"l3c"-"l3d"
		"l3d"-"l3e"
	}
	\]
	\caption{The top and bottom legs correspond to the continued fraction expansions $\frac{p}{1} = [p]^-$ and $\frac{4p-1}{4p-5} = [2^{[p-2]},3,2,2]$, respectively. The chain on the right is artificially longer than needed to make a $+1$-sphere appear.}\label{f:Bp}
	\end{figure}

	The normal crossing resolution for the singularity of type $(3,22)$ has two legs expanding $\frac{3}{2}=[2,2]^-$ and $\frac{22}{7}=[4,2^{[6]}]^-$, with the third leg decorated by $8^2-3\cdot 22=-2$.
	Since we want a sphere of square $+1$, we blow back down to the minimal resolution, whose neighborhood will be called $\Ec$. The triple edge indicates a multiplicity-three tangency between the $(+1)$-sphere and the $(-1)$-sphere.

	\begin{figure}
		\[
			\Ec = \xygraph{
				!{<0cm,0cm>;<1.25cm,0cm>:<0cm,1cm>::}
				!{(0,0) }*+{\bullet^{-2}}="c"
				!{(1,0) }*+{\bullet^{-2}}="l3a"
				!{(2,0) }*+{\bullet^{-2}}="l3b"
				!{(3,0) }*+{\bullet^{-2}}="l3c"
				!{(4,0) }*+{\bullet^{-1}}="l3d"
				!{(5,0) }*+{\bullet^{+1}}="l3e"
				!{(-1,0) }*+{\bullet^{-2}}="l2a"
				!{(-2,0) }*+{\bullet^{-2}}="l2b"
				"c"-"l2a"
				"c"-"l3a"
				"l2a"-"l2b"
				"l3a"-"l3b"
				"l3b"-"l3c"
				"l3c"-"l3d"
				"l3d"-"l3e"
				"l3d"-@/_.1cm/"l3e"
				"l3d"-@/^.1cm/"l3e"
			}
		\]
		\caption{}\label{f:322}
	\end{figure}

	The normal crossing resolution for $(6,43)$ has two legs corresponding to the continued fraction expansions $\frac{6}{5}=[2^{[5]}]^-$ and $\frac{43}{7}=[7,2^{[6]}]^-$, and the third leg is labeled by $16^2-6\cdot 43=-2$.
	Again we must blow down three times to get a $(+1)$-sphere.
	We end up between the minimal smooth resolution and the minimal normal crossing resolution, with a configuration indicated by the graph of Figure~\ref{f:643}.
	We call the neighborhood $\Dc$.
	Here the $(+1)$- and $(-4)$-spheres intersect tangentially with multiplicity $3$, and the $(-1)$-sphere intersects these two at the same point transversally.
	\begin{figure}
	\[
	\Dc = \xygraph{
		!{<0cm,0cm>;<1.25cm,0cm>:<0cm,1cm>::}
		!{(0,0) }*+{\bullet^{-2}}="c"
		!{(1,0) }*+{\bullet^{-2}}="l3a"
		!{(2,0) }*+{\bullet^{-2}}="l3b"
		!{(3,0) }*+{\bullet^{-2}}="l3c"
		!{(4,0) }*+{\bullet^{-4}}="l3d"
		!{(5,0) }*+{\bullet^{+1}}="l3e"
		!{(-1,0) }*+{\bullet^{-2}}="l2a"
		!{(-2,0) }*+{\bullet^{-2}}="l2b"
		!{(2,1) }*+{\bullet^{-2}}="l1a"
		!{(3,1) }*+{\bullet^{-2}}="l1b"
		!{(4,1) }*+{\bullet^{-1}}="l1c"
		"c"-"l2a"
		"c"-"l3a"
		"l2a"-"l2b"
		"l3a"-"l3b"
		"l3b"-"l3c"
		"l3c"-"l3d"
		"l3d"-"l3e"
		"l3d"-@/_.1cm/"l3e"
		"l3d"-@/^.1cm/"l3e"
		"l1a"-"l1b"
		"l1b"-"l1c"
		"l1c"-"l3e"
		"l1c"-"l3d"
	}
	\]
	\caption{}\label{f:643}
	\end{figure}

	\subsection{The homological embeddings}\label{ss:homclass}
	
	In this subsection, we will apply McDuff's theorem to each of the resolutions of the previous subsection.
	In order to shorten up the statements in this section, using Theorem~\ref{thm:mcduff} we will \emph{implicitly} identify the $(+1)$-sphere in each of the configurations with a line in a blow-up of $\cptwo$, and correspondingly its homology class will be identified with $h$. We work with a standard basis for $H_2(\cptwo\#N\cptwobar)=\langle h, e_1,\dots, e_N\rangle$. In each of the following statements and proofs describing the possible homology classes represented by components of a resolution, the uniqueness will always be up to a reordering of the $e_i$'s. 
	
	We start with the configuration $\Ac_p$, corresponding to curves with a singularity of type $(p,p+1)$.
	
	\begin{lemma}\label{l:homAp}
		In the configuration $\Ac_p$ of Figure~\ref{f:Ap}, the homology classes of the curves in the chain which excludes the $(-p-1)$-sphere are:
		\[
		(h, h-e_0-e_1, e_1-e_2,\dots, e_{p-1}-e_p, e_p-e_{p+1}, e_{p+1}-e_{p+2},\dots, e_{2p-1}-e_{2p})
		\]
		and the symplectic $(-p-1)$-sphere represents
		\[
		e_0-e_1-\dots-e_{p}.
		\]
	\end{lemma}
	
	\begin{proof}
		The classes of the components of the chain, excluding the $(-p-1)$-sphere, are uniquely determined by Lemma~\ref{l:2chainfix}.
		The remaining $(-p-1)$-sphere intersects once positively with the class $e_{p}-e_{p+1}$ and zero with the other classes.
		Since the class $e_{p+1}$ appears with positive coefficient already, by Lemmas~\ref{l:consecutive},~\ref{l:pos}, and~\ref{l:share2} the class of this last sphere is uniquely determined as stated.
	\end{proof}

	The cap corresponding to the rational cuspidal curve with singularity of type $(p,4p-1)$ has similar restrictions on its possible homology classes.

	\begin{lemma} \label{l:homBp}
		 In the configuration $\Bc_p$ of Figure~\ref{f:Bp}, the long chain starting with the $+1$-sphere and including the entire configuration except the $(-p)$-sphere has two possible homology configurations differing only in the last sphere of the chain when $p\geq 3$. The first possibility is:
		
		\[
		(h,h-e_0-e_1,e_1-e_2,\dots, e_{p-2}-e_{p-1},{\color{blue}e_{p-1}-e_p},e_p-e_{p+1},\dots, e_{2p-3}-e_{2p-2},
		\]
		\[
		e_{2p-2}-e_{2p-1}-e_{2p},e_{2p}-e_{2p+1},{\color{red}e_{2p+1}-e_{2p+2}}).
		\]
		
		For the other possibility, the last sphere can represent the class ${\color{red} e_{2p-1}-e_{2p}}$.
		
		In both cases, there is a unique possibility for the class represented by the $(-p)$-sphere: 
		\[
		e_0-e_1-\dots - e_{p-1}.\qedhere
		\]
		When $p=2$, there are the above two options as well as the following additional case where the long chain starting at the $+1$-sphere represents
		\[ (h,h-e_0-e_1, {\color{blue}e_1-e_2}, e_0-e_1-e_3, e_3-e_4, {\color{red}e_4-e_5})  \]
		and the length-one arm is a $(-p)$-sphere in the class $e_2-e_6$.
	\end{lemma}
	
	Note that the second option is a homological embedding into $\cptwo\# (2p+2)\cptwobar$ whereas the first option is a homological embedding into $\cptwo \# (2p+3)\cptwobar$.
	Since the cap has $b_2^+=1$ and $b_2^-=2p+2$, the symplectic filling complementary to the first embedding will have $b_2=1$ and the one complementary to the second embedding will be a rational homology ball filling.
	In particular, the filling complementary to the second homology embedding is the only one that could give the complement of the rational cuspidal curve in $\cptwo$.
	
	\begin{proof}
		The $\Bc_p$ configuration is identical to the $\Ac_{p-1}$ configuration with three additional vertices of weights $-3,-2,-2$ on the lower left leg.
		Therefore all of the homology classes excluding these last three, are determined by Lemma~\ref{l:homAp}. By Lemmas~\ref{l:consecutive} and~\ref{l:share2}, the $(-3)$-sphere must represent $e_{2p-2}-e_{2p-1}-e_{2p}$ unless $p=2$ (we will come back to this case).
		When $p\geq 3$, the subsequent $(-2)$-sphere must represent $e_{2p}-e_{2p+1}$, and the last $(-2)$-sphere can either represent $e_{2p+1}-e_{2p+2}$ or $e_{2p-1}-e_{2p}$ by Lemma~\ref{l:2chain}.

	When $p=2$, the $(-3)$-sphere can also represent $e_0-e_1-e_3$, and the remaining classes are determined as in the statement.
	\end{proof}

	\begin{lemma}\label{l:homC}
		There are three homological embeddings of the symplectic configuration $\Ec$ into blow-ups of $\cptwo$ given as follows, listed linearly starting from the $+1$-sphere.
		\begin{equation}\label{eqn:homC1}
		(h,3h-2e_0-e_1-e_2-e_3-e_4-e_5-e_6, e_1-e_7,e_7-e_8,e_8-e_9,e_9-e_{10},e_{10}-e_{11},e_{11}-e_{12})
		\end{equation}
		\begin{equation}\label{eqn:homC2}
		(h,3h-2e_0-e_1-e_2-e_3-e_4-e_5-e_6, e_1-e_7, e_2-e_1, e_3-e_2, e_4-e_3, e_5-e_4, e_6-e_5) 
		\end{equation}
		\begin{equation}\label{eqn:homC3}
		(h,3h-2e_0-e_1-e_2-e_3-e_4-e_5-e_6, e_0-e_1, e_1-e_2, e_2-e_3, e_3-e_4, e_4-e_5, e_5-e_6) 
		\end{equation}
	\end{lemma}
	
	\begin{proof}
		The class of the $(-1)$-sphere which intersects $\cpone$ with multiplicity three is determined by Lemma~\ref{l:adjclass}.
		The remaining chain of $(-2)$-spheres has one of two forms determined by Lemma~\ref{l:2chain}. If it has form~\ref{i:from}, except for the first sphere in the chain, all of the exceptional classes appearing in the chain must have coefficient $-1$ in the previous sphere (Lemma~\ref{l:2chainfix}), which uniquely yields option~\eqref{eqn:homC2}. If the chain has the form in Lemma~\ref{l:2chain} for option~\eqref{i:to}, the exceptional class with positive coefficient for the first $(-2)$-sphere is either $e_0$ or $e_1$ (without loss of generality). To ensure the correct intersection numbers, this uniquely determines options~\eqref{eqn:homC3} and~\eqref{eqn:homC1} respectively.
	\end{proof}
	
	The possibilities for embeddings of $\Dc$ are determined in the same way for the lower chain.
	The intersection numbers then limit the possibilities for the indices of the exceptional classes on the upper chain as follows.
	
	\begin{lemma}\label{l:homD}
		There are six homological embeddings of the symplectic configuration $\Dc$ into blow-ups of $\cptwo$.
		
		Like in Lemma~\ref{l:homC}, there are three possibilities for the lower chain:
		\begin{enumerate}
			\item \label{eqn:homD1}$(h,3h-2e_0-e_1-\dots -e_9, e_1-e_{10},e_{10}-e_{11},e_{11}-e_{12},e_{12}-e_{13},e_{13}-e_{14},e_{14}-e_{15})$
			\item \label{eqn:homD2} $(h,3h-2e_0-e_1-\dots-e_9, e_1-e_{10}, e_2-e_1, e_3-e_2, e_4-e_3, e_5-e_4, e_6-e_5)$
			\item \label{eqn:homD3} $(h,3h-2e_0-e_1-\dots-e_9, e_0-e_1, e_1-e_2, e_2-e_3, e_3-e_4, e_4-e_5, e_5-e_6)$
		\end{enumerate}
		When the lower chain is option~\eqref{eqn:homD1}, the upper chain can be
		\begin{equation}\label{eqn:homD4}
		(h-e_8-e_9, e_8-e_7, e_7-e_6)
		\end{equation}
		When the lower chain is option~\eqref{eqn:homD1} or, respectively,~\eqref{eqn:homD2}, the upper chain can be
		\begin{equation}\label{eqn:homD5}
		(h-e_0-e_{16}, e_{16}-e_{17}, e_{17}-e_{18}) \hspace{.25cm} \text{ respectively } \hspace{.25cm}
		(h-e_0-e_{11}, e_{11}-e_{12}, e_{12}-e_{13})
		\end{equation}
		When the lower chain is any of the three options, the upper chain can be
		\begin{equation}\label{eqn:homD6}
		(h-e_8-e_9, e_8-e_7, e_9-e_8)
		\end{equation}
	\end{lemma}

	\subsection{The Fibonacci families}\label{ss:fibonaccis}

		We start with the third family, $(F_{j-2},F_{j+2})$ with $j$ odd.
		We recall that Fibonacci numbers satisfy the identity $F_j^2 = F_{j-2}F_{j+2}-1$, so that by Moser~\cite{Moser} the boundary $Y_j$ of a regular neighbourhood of a curve in this family, with the orientation induced by the cuspidal contact structure $\xi_j$, is
		\[
			Y_j = -S^3_{F_j^2}(T(F_{j-2},F_{j+2})) = -L(F_j^2,-F_{j-2}^2) = L(F_j^2,F_{j-2}^2).
		\]
		Note that the identity $F_{j-2}^2 = F_{j-4}F_{j}-1$ implies that $F_{j-2}^2 \equiv -1 \pmod {F_{j}}$, and that $\gcd(F_{j-2},F_{j}) = F_{\gcd(j-2,j)} = F_1 = 1$, since $j$ is odd.
		In particular, $Y_j = L(m^2,mk-1)$ for some $m,k$ with $\gcd(m,k)=1$.
		
		Since there is a symplectic realization of a rational cuspidal curve in $\cptwo$ in this family, coming from algebraic geometry~\cite{FdBLMHN}, $\xi_j$ has a rational homology ball symplectic filling.
		By Proposition~\ref{p:QHBimpliesSTD}, $\xi_j$ is the canonical contact structure on $Y$ (perhaps up to conjugation), and by Proposition~\ref{p:uniqueQHBfilling}, this filling is unique up to symplectic deformation.
		
		We now turn to the fourth family, $(F_j^2, F_{j+2}^2)$, with $j$ odd.
		Call $(Y'_j,\xi'_j)$ the cuspidal contact manifold; by Moser~\cite{Moser},
		\[
			Y'_j = -S^3_{F_j^2F_{j+2}^2}(T(F_{j}^2,F_{j+2}^2)) = -(L(F_j^2,-F_{j-2}^2)\#L(F_{j+2}^2,-F_{j}^2)) = Y_j \# Y_{j+2}.
		\]
		As above, by~\cite{FdBLMHN}, $\xi'_j$ has a rational homology ball (strong) symplectic filling $W_0$;
		moreover, $(Y',\xi'_j)$ is the contact connected sum $(Y_j,\eta_j)\#(Y_{j+2},\eta_{j+2})$.
		Note that $\xi'_j$ is planar, since all contact structures on lens spaces are;
		we recall that, by~\cite{Wendl}, all strong symplectic fillings of planar contact structures can be deformed to Stein fillings, so that $W_0$ is in fact the boundary connected sum of two rational homology ball fillings of $\eta_j$ and $\eta_{j+2}$, by Eliashberg~\cite{Eliashberg}.
		This implies that, up to conjugation, $\eta_j = \xi_j$ and $\eta_{j+2} = \xi_{j+2}$, and that they both have a unique rational homology ball symplectic filling;
		again, by~\cite{Eliashberg}, so does $\xi'_j$.

	\subsection{Isotopy classification}\label{ss:isofills}

	\begin{theorem}\label{t:isofills}
		In $\cptwo$, every symplectic rational unicuspidal curve whose unique singularity is the cone on a torus knot is symplectically isotopic to a complex curve. 
		\begin{itemize}
			\item The only minimal symplectic embedding of a rational unicuspidal curve with a $(p,p+1)$-singularity with normal Euler number $(p+1)^2$ into a closed symplectic manifold is the unique embedding into $\cptwo$.
			\item When $p>2$, there are exactly two minimal symplectic embeddings of a rational unicuspidal curve with a $(p,4p-1)$-singularity with normal Euler number $4p^2$ into closed symplectic manifolds. One into $\cptwo$ and another into $S^2\times S^2$, each unique up to symplectomorphism and symplectic deformation. When $p=2$, in addition to these two minimal symplectic embeddings into $\cptwo$ and $S^2\times S^2$, there is also one into $\cptwo\#\cptwobar$.
			\item There are exactly three minimal symplectic embeddings of a rational unicuspidal curve with a $(3,22)$-singularity with normal Euler number $64$. They are embeddings into $\cptwo$, $\cptwo\#\cptwobar$, and $\cptwo\#6\cptwobar$ with each unique up to symplectomorphism and symplectic deformation.
			\item There are exactly six minimal symplectic embeddings of a rational unicuspidal curve with a $(6,43)$-singularity with normal Euler number $256$. There is a single symplectic embedding into $\cptwo$, $\cptwo\#\cptwobar$, $\cptwo\#4\cptwobar$, and $\cptwo\#9\cptwobar$, up to symplectomorphism. There are two symplectic embeddings into $\cptwo\#6\cptwobar$, up to symplectomorphism and symplectic deformation.
			\item For the rational unicuspidal curves in the Fibonacci families, there is a unique symplectic embedding in $\cptwo$, and there is always at least one other minimal symplectic embedding into another rational surface with larger $b_2$.
		\end{itemize}
		In each of these cases, each symplectic isotopy class of embeddings corresponds to a distinct minimal symplectic filling of the associated contact manifold, distinguished by their second homology and intersection forms.
	\end{theorem}
	
	\begin{proof}		
	We start with the non-Fibonacci families.
		
	Suppose we have a symplectic embedding of the cuspidal curve into $(M,\omega)$. Then we have an embedding of the resolution $\Ac_p$, $\Bc_p$, $\Ec$ or $\Dc$ into $M \#K \cptwobar$. By Theorem~\ref{thm:mcduff} we can identify $M\#K\cptwobar$ with $\cptwo\#N\cptwobar$ and the $(+1)$-sphere with a line. In particular, $M$ is either $\cptwo\#m\cptwobar$ for some $m\geq 0$ or $S^2\times S^2$ (the symplectic structures on these manifolds are unique up to symplectomorphism and symplectic deformation). The homological embeddings of the resolutions in $\cptwo\#N\cptwobar$ are classified in Lemmas~\ref{l:homAp},~\ref{l:homBp},~\ref{l:homC}, and~\ref{l:homD}. Next use Lemma~\ref{l:blowdown} to find exceptional spheres representing the classes $e_i$ to blow down to $\cptwo$, keeping track of the intersections of these exceptional spheres with the resolution to determine how it descends. The core spheres of $\Ac_p$ or $\Bc_p$ descend to two symplectic spheres each in the homology class $h$ of the line. Therefore any embedding of a cuspidal curve of type $\Ac_p$ or $\Bc_p$ into a closed symplectic manifold $(M,\omega)$ has a birational derivation to $(\cptwo,L_1\cup L_2)$ where $L_1$ and $L_2$ are symplectic lines. There is a unique symplectic isotopy class of two symplectic lines by Proposition~\ref{p:addline}.
	Therefore by Proposition~\ref{l:biratderiveunique}, there is a unique symplectic isotopy class of such curves for each possible $(M,\omega)$ that they embed into.
	
	\textbf{The configuration $\Ac_p$:} Lemma~\ref{l:homAp} gives a unique homological embedding of $\Ac_p$ (which is obtained from the cuspidal curve by blowing up $2p+1$ times) into $\cptwo\#(2p+1)\cptwobar$. Therefore the only relatively minimal embedding of these cuspidal curves is into $\cptwo$.

	\textbf{The configuration $\Bc_p$:} Lemma~\ref{l:homBp} gives two homological embeddings of $\Bc_p$ (which is obtained from the cuspidal curve by blowing up $2p+2$ times), one into $\cptwo\#(2p+3)\cptwobar$ and the other into $\cptwo\#(2p+2)\cptwobar$, plus an additional embedding when $p=2$ into $\cptwo\# 7\cptwobar$. Therefore the cuspidal curve of type $\Bc_p$ has one symplectic embedding into $\cptwo$ and one (two if $p = 2$) more into a manifold $X$ which is either $\cptwo\#\cptwobar$ or $S^2\times S^2$.
	We now see how to determine whether the manifold is $S^2\times S^2$ or $\cptwo\#\cptwobar$ in this latter case. $C$ is the rational cuspidal curve in $X$, so using our chosen resolution, the total transform of $C$ embeds in a blow-up $\overline{X}$ of $X$.
	The intersection form of $X$ is recovered from that of $\overline{X}$ by ``algebraically blowing down'' all the exceptional divisors in the total transform of $C$.
	This boils down to taking the orthogonal of the homology classes. For the homological embedding from Lemma~\ref{l:homBp} which applies to all $p\geq 2$ cases, we need the orthogonal of the classes
	\[
	h-e_0-e_1,\;\; e_{2p-2}-e_{2p-1}-e_{2p},\;\; e_0-e_1-\dots-e_{p-1},\;\; e_{i}-e_{i+1} \quad (i\in\{1,\dots,2p-3,2p,2p+1\})
	\]
	in $H_2(\overline{X}) = \langle h,e_0,\dots,e_{2p+2}\rangle$. Orthogonality to the classes $e_i-e_{i+1}$ forces classes in the orthogonal to be of the form
	\[
	ah-b_0e_0-b_1(e_1+\dots+e_{2p-2})-b_{2p-1}e_{2p-1}-b_{2p}(e_{2p}+e_{2p+1}+e_{2p+2}).
	\]
	The orthogonal is generated by $C_1 = ph-(p-1)e_0-e_1-\dots-e_{2p-1}$, $C_2 = ph-(p-1)e_0-e_1-\dots-e_{2p-2}-e_{2p}-e_{2p+1}-e_{2p+2}$.
	We calculate $C_1^2 = 0$ and that $C_2^2 = -2$, so that the intersection form of $X$ is even, hence necessarily $X = S^2\times S^2$. As another way that we could have seen this, we will see in Proposition~\ref{p:QHblowdown} below that the latter embedding is obtained from the former by doing a rational blow-up of $T^*\R\P^2$, thus showing that the ambient 4-manifold is indeed $S^2\times S^2$.
	
	For the additional homological embedding in the $p=2$ case, we claim this corresponds to an embedding of $C$ into $\cptwo\#\cptwobar$. In this case, the intersection form of $X$ is isometric to the subspace $\langle h-e_0-e_1, e_1-e_2, e_0-e_1-e_3, e_3-e_4,e_4-e_5,e_2-e_6\rangle^\perp \subset H_2(\cptwo\#7\cptwobar)$. This subspace is generated by the classes $h-e_0-e_3-e_4-e_5$ and $2h - e_0-e_1-e_2-e_6$. Since the self-intersection of the first class is $-3$, which is odd, $\overline X$ is $\cptwo\#\cptwobar$.
	
	\textbf{The configuration $\Ec$:} For $\Ec$, after blowing down the exceptional spheres in the $e_i$ classes using Lemma~\ref{l:blowdown}, the configuration descends to a symplectic nodal or cuspidal cubic (depending on whether the exceptional sphere of class $e_0$ intersects the configuration transverally or tangentially), together with a symplectic line which is tangent to the cubic with multiplicity $3$. Each such configuration has a unique symplectic isotopy class by Proposition~\ref{p:d-with-tangent} with $d=3$. To connect these two possibilities, observe that each of these possibilities is realizable in the complex setting. There is a deformation from a complex algebraic cuspidal cubic $\{x^3-y^2z=0\}$ with inflection line $\{z=0\}$, to a nodal cubic $\{x^3-y^2z-\varepsilon x^2z=0\}$ with the same inflection line $\{z=0\}$. Blowing up at the cusp or node in the $\varepsilon$ family provides an equisingular isotopy of the proper transforms in $\cptwo\#\cptwobar$. Therefore there is a unique symplectic isotopy class of the configuration $\mathcal{C}'$ in $\cptwo\#\cptwobar$ consisting of one smooth rational curve in the class $3h-2e_0$ and a line in class $h$, such that the two components intersect tangentially at a single point of multiplicity $3$. The birational transformation described above shows that the configuration $\Ec$ has a birational transformation to $\mathcal{C}'$ in $\cptwo\#\cptwobar$. Therefore by Proposition~\ref{l:biratderiveunique}, for each possible $(M,\omega)$ determined by a homological embedding, there is a unique symplectic isotopy class of embeddings of the cuspidal curve of type $\Ec$. The three homological embeddings of $\Ec$ (which is obtained from the cuspidal curve by $7$ blow-ups) are into $\cptwo\#7\cptwobar$, $\cptwo\#8\cptwobar$ and $\cptwo\#13\cptwobar$. Therefore the cuspidal curve has one minimal symplectic embedding in $\cptwo$, one in either $\cptwo\#\cptwobar$ or $S^2\times S^2$ and one into $\cptwo\#6\cptwobar$.
	To determine that the second embedding is into $\cptwo\#\cptwobar$ instead of $S^2\times S^2$ we find a homology basis complementary to the classes of the exceptional divisors in the embedding of $\Ec$ ($3h-2e_0-e_1-\dots-e_6$, $e_1-e_7$, $e_2-e_1$, $e_3-e_2$, $e_4-e_3$, $e_5-e_4$, and $e_6-e_5$). Such a basis is given by $10h-6e_0-3(e_1+\dots+e_7)$ and $6h-3e_0-2(e_1+\dots+e_7)$. The intersection form generated by this basis is the odd form for $\cptwo\#\cptwobar$:
	\[
	\left[\begin{array}{cc} 1&0\\0&-1 \end{array} \right].
	\]

	\textbf{The configuration $\Dc$:} By Lemma~\ref{l:homD}, there are six different homological embeddings of $\Dc$ into blow-ups of $\cptwo$. We will show that each such homological embedding has a unique symplectic isotopy class. In each case, we blow down exceptional spheres until we reach $\cptwo$ using Lemma~\ref{l:blowdown}, and consider the image of $\Dc$. 
	
	For either of the homological embeddings where the upper chain is described by~\eqref{eqn:homD5}, the image of $\Dc$ in $\cptwo$ is a nodal or cuspidal cubic $R$ (depending on whether the exceptional sphere representing $e_0$ intersects the component representing $3h-2e_0-e_1-\cdots-e_9$ transversally or tangentially), together with an inflection line $T$ tangent to the cubic with multiplicity $3$, and another line $L$ which passes through the singularity of $R$ and through the point $R\cap T$ (intersecting the other components transversally at each of these points). The configuration of $R\cup T$ has a unique symplectic isotopy class by Proposition~\ref{p:d-with-tangent}, so the configuration $R\cup T\cup L$ has a unique symplectic isotopy class by Proposition~\ref{p:addline}. Therefore any two embeddings of $\Dc$ which descend to the version of this configuration where $R$ is nodal are symplectically isotopic, and similarly any two which descend to the version of this configuration where $R$ is cuspidal are symplectically isotopic. We connect the nodal and cuspidal cases of $R\cup T$ as in the $\Ec$ case, and augment the deformation from the cusp to the node by the line $L$ to connect the nodal and cuspidal instances of $R\cup T\cup L$.
	
	For the homological embedding where the upper chain is described by~\eqref{eqn:homD4}, the embedding of $\Dc$ descends under blow downs to a configuration $R\cup T\cup L'$. $R$ is again a nodal or cuspidal cubic, and $T$ is again an inflection line (tangent of order 3 to $C$). This time $L'$ is a line which passes transversally through $R\cap T$ and intersects $R$ again transversally at two smooth points (the first point is the image of exceptional spheres representing $e_6$, $e_7$, and $e_8$ and the second point is the image of the exceptional sphere representing $e_9$). $R\cup T$ has a unique symplectic isotopy class by Proposition~\ref{p:d-with-tangent}, and we can add $L'$ maintaining a unique symplectic isotopy class using Proposition~\ref{p:addline}. The nodal and cuspidal cases are connected as in the $\Ec$ case.
	
	For the homological embeddings where the upper chain is described by~\eqref{eqn:homD6}, the embedding of $\Dc$ descends to $R\cup T\cup L''$ where $R$ and $T$ are the same as the previous cases and $L''$ passes transversally through $R\cap T$ and intersects $R$ at one other smooth point with a simple tangency (the image of the exceptional spheres representing $e_7$, $e_8$, and $e_9$).	
	In this case, we will show that there is a unique symplectic isotopy class of such a configuration when the cubic is nodal and no realization of the configuration if the cubic has a cusp.
	
	To prove this, choose an almost complex structure $J$ which makes the cubic $R$ and its inflection line $T$ $J$-holomorphic (note the space of such $J$ is contractible by Lemma~\ref{l:Jcompatible}). Let $p=R\cap T$. Consider the pencil of $J$-lines through $p$, $\pi: \cptwo\setminus\{p\}\to \cpone$ and restrict this to $R\setminus \{p\}$. This restriction extends over $p$ by sending $p$ to the image of the line tangent to $R$ at $p$. We pre-compose with the normalization of $R$.
	
	In the cusp case, projecting the cubic from the inflection point gives a degree-2 map $\pi\colon\cpone\to\cpone$ with at least two ramification points (corresponding to the inflection line and the cusp respectively).
	Therefore Riemann--Hurwitz reads: $2 = 2\cdot 2 - \sum (e_\pi(p)-1)$, which implies that these are the only two ramification points, from which we deduce that there is no other tangent drawn to the cubic from the inflection point.
	In the nodal case, the projection $\pi \colon \cpone \to \cpone$ of the cubic from $p$ has no ramification at the node, and ramification $2$ at the inflection line, so there is exactly another point of ramification $2$, which corresponds to a tangent line to the cubic.

	Therefore, for each fixed homological embedding of the resolution, there is a unique symplectic embedding of the cuspidal curve up to isotopy. The resolution is obtained from the cuspidal curve by performing $10$ blow-ups. Therefore the embeddings of the resolution into $\cptwo\#10\cptwobar$, $\cptwo\#11\cptwobar$, two into $\cptwo\#16\cptwobar$, one into $\cptwo\#19\cptwobar$, and $\cptwo\#14\cptwobar$ correspond to embeddings of the cuspidal curve into $\cptwo$, $\cptwo\#\cptwobar$ or $S^2\times S^2$, two embeddings into $\cptwo\#6\cptwobar$, and one into $\cptwo\#9\cptwobar$ and $\cptwo\#4\cptwobar$. 
	
	It remains to distinguish the two embeddings into $\cptwo\#6\cptwobar$ and prove that this cuspidal curve embeds into $\cptwo\#\cptwobar$ instead of $S^2\times S^2$. For the latter question, the answer is similar to the $\Ec$ case, but $e_7$ is replaced by $e_{10}$. Therefore a basis for the homology complementary to the exceptional divisors is given by $10h-6e_0-3(e_1+\dots+e_6+e_{10})$ and $6h-3e_0-2(e_1+\dots+e_6+e_{10})$ which have the odd intersection form of $\cptwo\#\cptwobar$.
	
	To distinguish the two embeddings into $\cptwo\#6\cptwobar$, we show that their complementary symplectic fillings have different intersection forms. The two homological embeddings we are considering from Lemma~\ref{l:homD} both use option~\eqref{eqn:homD1} for the lower chain, and differ in the upper chain, being either option \eqref{eqn:homD4} or \eqref{eqn:homD6}.
	
	A homology basis for the complement of the embedding of the cap (not just the exceptional spheres but also the proper transform representing $h$) in the first case (option~\eqref{eqn:homD4}) is given by the following.
	\[
	\{e_0-e_4-e_5,e_2-e_1-e_{10}-\dots-e_{15}, e_2-e_3,e_3-e_4,e_4-e_5,e_0+e_9-e_6-e_7-e_8\}.
	\]
	giving an intersection form represented by
	\[
	Q_1 = \left[\begin{array}{rrrrrr} -3&0&0&-1&0&-1\\ 0&-8&-1&0&0&0\\ 0&-1&-2&1&0&0\\ -1&0&1&-2&1&0\\ 0&0&0&1&-2&0\\ -1&0&0&0&0&-5\end{array}  \right]
	\]
	In the second case (option \eqref{eqn:homD6}) a basis is given by the following.
	\[
	\{e_2-e_1-e_{10}-\dots-e_{15}, e_2-e_3,e_3-e_4,e_4-e_5,e_5-e_6, e_0-e_5-e_6\}.
	\]
	giving an intersection form represented by
	\[
	Q_2 = \left[\begin{array}{rrrrrr}-8&-1&0&0&0&0\\-1&-2&1&0&0&0\\0&1&-2&1&0&0\\0&0&1&-2&1&-1\\0&0&0&1&-2&0\\0&0&0&-1&0&-3\end{array}  \right].
	\]
	Since $\det Q_1 = 256$ and $\det Q_2 = 64$, the two embeddings are not even homeomorphic.

	\textbf{The Fibonacci families:} In the previous subsection, we have shown that in each case the cuspidal contact structure has a unique rational homology ball filling up to symplectic deformation.
	In particular, there is a unique symplectomorphism class of curves in $\cptwo$ for each of the members of either family, and by Gromov
	there is a unique symplectic isotopy class, so the complex representatives are necessarily in that isotopy class.
	Additionally, we observe that each member in the first Fibonacci family has at least one other filling, while each member in the second Fibonacci family has at least two. Indeed, in the first Fibonacci family, each of the cuspidal contact 3-manifolds $(Y_C,\xi_C)$ is a universally tight lens space. Each universally tight contact lens space admits at least another filling (namely, a plumbing of symplectic spheres); gluing this filling to the cuspidal cap yields a closed symplectic 4-manifold into which $C$ embeds (with minimal complement).
	In the second Fibonacci family, each of the cuspidal contact 3-manifolds $(Y_C,\xi_C)$ is a connected sum of two distinct universally tight lens spaces (each admitting a rational homology ball symplectic filling), so it has at least three non-rational homology ball minimal symplectic fillings, obtained by boundary connected summing either a plumbing filling and a rational homology ball filling for each of the summand, or the two plumbings. To show that these fillings are distinct, we look at the order of their first homology: the two rational homology balls have $H_1 = \Z/F_j\Z$ and $H_1 = \Z/{F_{j+2}}\Z$, respectively, while the plumbings are simply-connected, so the four fillings we exhibited have $|H_1| = F_{j}F_{j+2}, F_j, F_{j+2}$, and $1$, respectively.
	\end{proof}
		
%
%
%
%
%
%
	
	\subsection{Rational blow-down relations}\label{ss:QHBd}
	
	In this section we study the relationships between the two fillings of the cuspidal contact structure associated to the rational unicuspidal curves in the second family, i.e. those with a singularity of type $(p,4p-1)$. The two fillings are complements of concave neighborhoods of the two different embeddings.
	
	We study the two sporadic cases $(3,22)$ and $(6,43)$ in a companion paper~\cite{companion}, by showing that the corresponding cuspidal contact structure is in fact the canonical contact structure on the link of a singularity (compare with Stipsicz--Szab\'o--Wahl~\cite{SSzW}) and studying the relationship between the different fillings.
	
	As argued in Section~\ref{ss:fibonaccis}, the two Fibonacci families correspond to lens spaces with their canonical contact structure, for which fillings and rational blow-down relationships have been extensively studied~\cite{BhupalOzbagci}.
	
	Recall that when $p>2$ each rational curve with self-intersection $4p^2$ and a singularity of type $(p,4p-1)$ comes with two fillings, $W_0, W_1$, with $b_2(W_i) = i$. (See Remark~\ref{r:p2} for the $p=2$ case.)
	
	\begin{prop}\label{p:QHblowdown}
	Let $C$ be a rational unicuspidal curve with a singularity of type $(p,4p-1)$ and self-intersection $4p^2$, and $\xi$ the corresponding cuspidal contact structure.
	Then two of the symplectic fillings of $\xi$ are related by a rational blow-down of a symplectic $(-4)$-sphere.
	\end{prop}
	
	\begin{proof}
		We need to look for a symplectic $(-4)$-sphere in the filling with $b_2=1$, which we will call $W_1$.
		Recall from Lemma~\ref{l:homBp} that the homology classes of the embeddings of the spheres of the plumbing $\Bc_p$ whose complement is $W_1$ are:
		\[h,h-e_0-e_1,e_1-e_2,\dots,e_{2p-3}-e_{2p-2},e_{2p-2}-e_{2p-1}-e_{2p}, e_{2p}-e_{2p+1}, e_{2p+1}-e_{2p+2}, e_0-e_1-\dots-e_{p-1}\]
		The homology class $e_{2p-1}-e_{2p}-e_{2p+1}-e_{2p+2}$ generates the orthogonal of these classes in $H_2(\cptwo\#(2p+2)\cptwobar)$.
		We can choose to realize the configuration above by performing the last three blow-ups on the exceptional sphere in class $e_{2p-1}$. Then, the proper transform of $e_{2p-1}$ is an embedded symplectic sphere in the class $e_{2p-1}-e_{2p}-e_{2p+1}-e_{2p+2}$ with self-intersection $-4$, in the complement of the embedding of $\Bc_p$.
		Since $\xi$ only has two strong symplectic fillings, the rational blow-down of this $(-4)$-sphere must give the unique rational homology ball filling.
	\end{proof}

	\begin{remark}\label{r:p2}
		The same rational blow-down relation holds in the $p=2$ case for the two fillings associated to the homological cases that generalize for $p>2$. The additional filling with $b_2=1$ that appears only in the $p=2$ case is not related to the rational homology ball filling via rational blow-down. This is because the intersection form of this exceptional filling is $\langle -16\rangle$, so it cannot contain a $(-4)$-sphere to rationally blow down. To calculate the intersection form, start with the homological embedding of Lemma~\ref{l:homBp} and observe that the orthogonal to the subspace generated by the components of the total transform of $C$ in $H_2(\cptwo\# 7\cptwobar)$ is generated by the class $e_0-e_1-e_2+2e_3+2e_4+2e_5-e_6$, which has self-intersection $-16$.
	\end{remark}
	

\section{Low-degree cuspidal curves}\label{s:lowdegrees}
		
	The goal of this section is to prove Theorem~\ref{t:lowdeg}; namely, we want to prove that every rational cuspidal curve of degree up to $5$ has a unique isotopy class, and that this class contains a complex representative.
%
	
	We split the proof degree by degree; there are no singular lines or conics, so we only need to start at degree 3.
	By the degree-genus formula~\eqref{e:degree-genus}, a cuspidal cubic can only have one singularity, which is necessarily a simple cusp.
	This case was already considered in Section~\ref{s:unicusp} and in~\cite{OhtaOno}.
	
	In the next two subsections, we will look at quartics and quintics.
	As in the previous section, we will actually provide classifications of symplectic embeddings of these cuspidal curves (with prescribed normal Euler number) into any closed symplectic manifold, equivalently classifying the symplectic fillings of the associated contact structures.
	
	\subsection{Quartics}
	
	By the degree-genus formula~\eqref{e:degree-genus}, a cuspidal quartic can only have multiplicity multi-sequence $[[3]]$ or $[[2,2,2]]$; correspondingly, the allowed types of configurations of singularities are the following:
	\begin{itemize}
		\item $[3]$: a single cusp of type $(3,4)$ which we showed has a unique symplectic isotopy class of embeddings into $\cptwo$ in Section~\ref{s:unicusp};
		\item $[2,2,2]$: a single cusp of type $(2,7)$ which was also shown to have a unique embedding into $\cptwo$ in Section~\ref{s:unicusp};
		\item $[2,2],[2]$: one cusp of type $(2,5)$ and one of type $(2,3)$ will be shown to have a unique relatively minimal symplectic embedding into each of $\cptwo$, $\cptwo\#\cptwobar$, and $S^2\times S^2$ in Proposition~\ref{p:2quartic};
		\item $[2], [2], [2]$: three simple cusps will be shown to have a unique symplectic embedding into $\cptwo$ in Proposition~\ref{p:3quartic}.
	\end{itemize}
		
	As it turns out, all these configurations are realized by rational cuspidal quartics in $\cptwo$, and each realization has a unique equisingular isotopy class in $\cptwo$. The bicuspidal quartic also has realizations in $\cptwo\#\cptwobar$ and $S^2\times S^2$.
		
	\begin{prop}\label{p:2quartic}
		Let $C$ be a curve with normal Euler number $16$ with one cusp of type $(2,5)$ and one of type $(2,3)$. Then the only relatively minimal symplectic embeddings of $C$ are into $\cptwo$, $\cptwo\#\cptwobar$, and $S^2\times S^2$. For each of these there is a unique non-empty symplectic isotopy class of embeddings up to symplectomorphism.
		Correspondingly, the associated contact structure $\xi_C$ has three fillings; one is a rational homology ball, and the other two have second Betti number $1$ and are distinguished by their intersection forms.
	\end{prop}
	
	\begin{proof}
		Suppose $C$ embeds minimally symplectically in $(X,\omega)$. Blow up at the $(2,3)$-cusp of $C$ to the normal crossing resolution, and blow up at the $(2,5)$-cusp once more than its minimal resolution (not quite normal crossing). See Figure~\ref{fig:2quartic}. The proper transform of $C$ becomes a smooth symplectic $+1$-sphere, so we apply McDuff's theorem to identify it with a line in some blow up of $\cptwo$.
		
		We determine the homology classes of the remaining divisors in the total transform of $C$ relative to this identification. 
		Using Lemma~\ref{l:hom}, and the intersection relations, the homology classes of the divisors have three possibilities differing from each other only in the $(-2)$- and $(-3)$-spheres in the normal crossing resolution of the $(2,3)$-cusp. 
		The divisors in the resolution of the $(2,5)$-cusp must be $h-e_1-e_2$, $h-e_3-e_4-e_5$ and $e_5-e_6$, and the $(-1)$-curve in the resolution of the $(2,3)$-cusp must be in the class $h-e_1-e_3$ (up to relabeling the $e_i$). 
		The remaining $(-2)$-sphere class must be either $e_1-e_2$ or $e_3-e_4$. If it is $e_1-e_2$ there are two possibilities for the $(-3)$-sphere class: $e_3-e_5-e_6$ and $e_3-e_4-e_7$. If the $(-2)$-class is $e_3-e_4$, the $(-3)$-class must be $e_1-e_2-e_7$.
		See Figure~\ref{fig:2quartic}.
		
		\begin{figure}[h]
			\centering
			\labellist
			\pinlabel $+1$ at 170 87
			\pinlabel ${\scriptstyle h}$ at 170 72
			\pinlabel $-1$ at 275 65
			\pinlabel ${\scriptstyle h-e_1-e_2}$ at 298 50
			\pinlabel $-2$ at 233 65
			\pinlabel ${\scriptstyle h-e_3-e_4-e_5}$ at 208 55
			\pinlabel $-2$ at 203 30
			\pinlabel ${\scriptstyle e_5-e_6}$ at 203 10
			\pinlabel $-1$ at 97 63
			\pinlabel $\scriptstyle h-e_1-e_3$ at 63 65
			\pinlabel $-3$ at 97 41
			\pinlabel ${\color{red}\scriptstyle A:\; e_3-e_5-e_6}$ at 95 25
			\pinlabel ${\color{red}\scriptstyle B:\; e_3-e_4-e_7}$ at 95 15
			\pinlabel ${\color{blue}\scriptstyle C:\; e_1-e_2-e_7}$ at 95 5
			\pinlabel $-2$ at 41 30
			\pinlabel ${\color{red}\scriptstyle A,B:\; e_1-e_2}$ at 41 10
			\pinlabel ${\color{blue}\scriptstyle C:\; e_3-e_4}$ at 41 0
			\endlabellist
			\includegraphics[scale=.85]{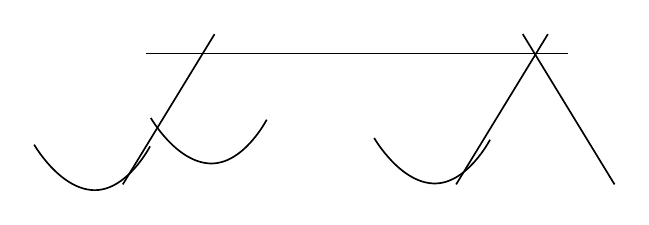}
			\caption{A resolution of a rational cuspidal quartic with one cusp of type $(2,3)$ and one of type $(2,5)$ with the possible homological embeddings.}
			\label{fig:2quartic}
		\end{figure}
		
		Observe that two of these three possibilities require $7$ exceptional classes and the third requires $6$ exceptional classes. Since the resolution was obtained from the cuspidal curve by blowing up $6$ times, the two embeddings of the resolution using $7$ exceptional classes correspond to potential embeddings of the cuspidal curve into \emph{either} $\cptwo\#\cptwobar$ \emph{or} $S^2\times S^2$. 
		
		In option B, where the $(-3)$-sphere represents $e_3-e_4-e_7$ and the $(-2)$-sphere represents $e_1-e_2$, the homology of the complement of the exceptional divisors is generated by $2h-e_1-e_2-e_3-e_4$ and $4h-2e_1-2e_2-2e_3-e_4-e_5-e_6-e_7$.
		The intersection form generated by these two classes is
		\[
		\left[ \begin{array}{cc} 0&1\\1&0\end{array}  \right]
		\]
		so this corresponds to an embedding of the bicuspidal quartic into $S^2\times S^2$.
		
		In option C, where the $(-3)$-sphere represents $e_1-e_2-e_7$ and the $(-2)$-sphere represents $e_3-e_4$, the homology of the complement of the exceptional divisors in the resolution is generated by $3h-2e_1-e_2-e_3-e_4-e_5-e_6-e_7$ and $5h-3e_1-2e_2-2e_3-2e_4-e_5-e_6-e_7$ which gives an intersection form
		\[
		\left[ \begin{array}{cc} -1 & 0 \\ 0 & 1\end{array}  \right]
		\]
		so this corresponds to an embedding of the bicuspidal quartic into $\cptwo\#\cptwobar$.
		
		The embedding of the resolution using six $e_i$ classes (option A) corresponds to a potential embedding of the cuspidal curve into $\cptwo$.
		
		To verify that there is a unique isotopy class for each homology embedding, we apply Lemma \ref{l:blowdown} to find exceptional spheres in the $e_i$ classes. In all three cases, the configuration blows down to four lines with one triple point and three double points. This has a unique symplectic isotopy class by Proposition \ref{p:addline}, so this cuspidal quartic has a unique symplectic isotopy class in $\cptwo$ by Proposition~\ref{l:biratderiveunique}.
		
		The complement of the embedding of the cuspidal curve into $\cptwo$ is a rational homology ball filling. We can further check that the fillings complementary to the embeddings of the cuspidal curve into $\cptwo\#\cptwobar$ and $S^2\times S^2$ are in fact differentiated by their intersection forms. Since each has $b_2=1$, we simply need to find the self-intersection number of a primitive class orthogonal to all of the divisors in the cap (the exceptional divisors together with the proper transform which represents $h$). 
		In option B, the embedding into $S^2\times S^2$, such a class is represented by $e_4-e_5-e_6-e_7$ which has self-intersection number $-4$ (and actually is represented by a symplectic sphere which can be blown down to get the rational homology ball filling).
		In option C, the embedding into $\cptwo\#\cptwobar$, such a class is given by $e_1-e_2-e_3-e_4+2e_5+2e_6+2e_7$ which has self-intersection number $-16$. Therefore the two fillings are not homeomorphic.
%
	\end{proof}
	
	Finally, we turn to the case of three cusps of type $(2,3)$.
	We call $E_1$, $E_2$, and $E_3$ the three exceptional divisors in the minimal resolution of the curve.
	
	\begin{prop}\label{p:3quartic}
		Let $C$ be a curve with normal Euler number $16$ with three cusps of type $(2,3)$. Then the only relatively minimal symplectic embedding of $C$ is into $\cptwo$ and this is unique up to symplectic isotopy.
		Correspondingly, the associated contact structure $\xi_C$ has a unique minimal symplectic filling, and it is a rational homology ball.
	\end{prop}
	
	\begin{proof}
		Suppose $C$ embeds minimally symplectically in $(X,\omega)$. Blow up at each cusp twice giving a resolution in between the minimal resolution and the normal crossing resolution. See Figure~\ref{fig:3quartic}. The proper transform of $C$ becomes a smooth symplectic $+1$-sphere, so we apply McDuff's theorem to identify it with a line in some blow up of $\cptwo$, and determine the homology classes of the exceptional divisors relative to this identification. The exceptional divisors from the first two cusps are uniquely determined up to relabeling the $e_i$, and there are three options for the exceptional divisors for the third cusp as shown in Figure~\ref{fig:3quartic}.
		
		\begin{figure}[h]
			\centering
			\labellist
			\pinlabel $+1$ at 130 87
			\pinlabel ${\scriptstyle h}$ at 130 72
			\pinlabel $-1$ at 35 65
			\pinlabel $\scriptstyle h-e_1-e_2$ at 3 30
			\pinlabel $-2$ at 75 65
			\pinlabel $\scriptstyle h-e_3-e_4-e_5$ at 102 50
			\pinlabel $-1$ at 160 65
			\pinlabel $\scriptstyle h-e_1-e_3$ at 130 30
			\pinlabel $-2$ at 200 65
			\pinlabel $\scriptstyle h-e_2-e_4-e_6$ at 230 50
			\pinlabel $-1$ at 290 65
			\pinlabel $-2$ at 328 65
			\pinlabel ${\color{red}\scriptstyle A:\; h-e_1-e_4}$ at 265 14
			\pinlabel ${\color{blue}\scriptstyle B,C:\; h-e_2-e_3}$ at 265 4
			\pinlabel ${\color{red}\scriptstyle A:\; h-e_2-e_3-e_7}$ at 382 25
			\pinlabel ${\color{blue}\scriptstyle B:\; h-e_1-e_4-e_7}$ at 382 15
			\pinlabel ${\color{blue}\scriptstyle C:\; h-e_1-e_5-e_6}$ at 382 5
			\endlabellist
			\includegraphics[scale=.85]{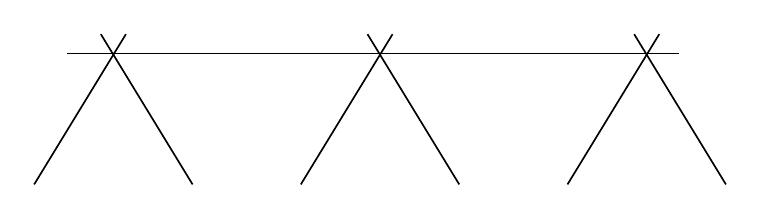}
			\caption{A resolution of a rational cuspidal quartic with three cusps of type $(2,3)$ with the possible homological embeddings.}
			\label{fig:3quartic}
		\end{figure}
		
		Next, we blow down exceptional spheres in the $e_i$ classes in each case using Lemma \ref{l:blowdown}. The two homological embeddings utilizing seven $e_i$ classes (options A and B) both blow down to a Fano configuration of seven lines with seven triple points (three are already in the resolution and four more come from blowing down $e_1,e_2,e_3,e_4$). The Fano configuration is obstructed from realization by symplectic (or even smooth) lines by Theorem \ref{thm:Fano}. Since both embeddings of the cuspidal curve into any symplectic manifold with $b_2=2$ have birational derivations to the Fano configuration, they cannot exist by Proposition \ref{l:biratexist}.
		
		The last homological embedding (option C) uses six $e_i$ classes so since the resolution was obtained from the cuspidal curve by performing $6$ blow-ups, this option corresponds to an embedding of the cuspidal curve into $\cptwo$. We now verify this exists uniquely, by blowing down exceptional spheres in $e_i$ classes using Lemma \ref{l:blowdown}. In this case, the configuration blows down to seven lines intersecting in $6$ triple points and two double points. This can be built up line by line using Proposition \ref{p:addline}. Therefore it has a unique symplectic isotopy class so there is a unique symplectic isotopy class of $C$ in $\cptwo$ by Proposition \ref{l:biratderiveunique}.
	\end{proof}


\subsection{Quintics}

	A quintic $C$ in $\cptwo$ has arithmetic genus $p_a(C) = \frac12(5-1)(5-2) = 6$. This allows for a large number of potential ways to absorb that genus into cusps.
	Some of these collections of singularity types are realized by rational cuspidal complex curves of degree five, and others cannot be realized. Here we prove that each equisingular type which is realized by complex curves, has a unique symplectic isotopy class of realizations, and that equisingular types which do not admit a complex realization do not admit a symplectic realization either. 
	
	
	To prove Theorem~\ref{thm:quintics} for quintics, and the classifications of embeddings into other closed symplectic manifolds, we first check what are the possible multiplicity multi-sequences which yield a rational cuspidal curve of arithmetic genus $p_a(C)=6$. 
	
	If $C$ is rational and cuspidal, then $p_a(C) = \sum \delta(p)$, where the sum is taken over all singular points of $C$.
	By Equation~\eqref{e:deltamulti}, summing $\frac12 m(m-1)$ over all multiplicity sequences of singular points, we have to obtain 6; this leaves the following possibilities to arise as multiplicity multi-sequences of the singularities of $C$: $[[4]]$, $[[3,3]]$, $[[3,2,2,2]]$, and $[[2,2,2,2,2,2]]$.
	
	As observed in Section~\ref{ss:resolution}, not all sequences of integers correspond to multiplicity sequences of singularities.
	In particular, among the possible sequences coming from the list above, $[3,2,2]$ and $[3,2,2,2]$ are not allowed.
	This leaves the following possibilities for each individual singularity:
	\begin{itemize}
		\item $[4]$, which corresponds to the singularity of type $(4,5)$;
		\item $[3^{[k]}]$, which corresponds to the singularity of type $(3,3k+1)$ ($k=1,2$);
		\item $[3,2]$, which corresponds to the singularity of type $(3,5)$;
		\item $[2^{[k]}]$, which corresponds to the singularity of type $(2,2k+1)$, $(k=1,\dots, 6$).
	\end{itemize}
	
	In particular, all singularities we encounter here have one Puiseux pair.
	Starting from degree 6 on, there will appear curves with multiple Puiseux pairs.
	For instance, in degree 6, there is a unicuspidal curve whose singularity has multiplicity sequence $[4,2,2,2,2]$ (this is curve $C_2$ in~\cite[Corollary~1.5]{Fenske});
	the link of this singularity is the iterated torus knot $T(2,3;2,17)$ (i.e the $(2,17)$-cable of $T(2,3)$).

	Theorem~\ref{thm:quintics} for degree 5 results from analyzing all the possibilities, which we do in the propositions making up the rest of this subsection. We provide a table summarizing the results and giving references (Table~\ref{table:quintics}).
	
	\begin{table}[h!]
	\centering

		\begin{tabular}{llllll}
	{\bf Sing. (MS)} & {\bf Sing. (L)} & {\bf In $\cptwo$} & {\bf Namba} & {\bf Fillable} & {\bf Reference}\\
	{$[4]$} & $(4,5)$ & Yes & 1 & No & Theorem~\ref{t:isofills}\\
	{$[3,3]$} & $(3,7)$ & No & & No & Proposition~\ref{p:noadjhom}\\
	{$[3],[3]$} & $(3,4),(3,4)$ & No & & No & Proposition~\ref{p:noadjhom}\\
	{$[3,2],[2,2]$} & $(3,5), (2,5)$ & Yes & 2 & No & Proposition~\ref{p:3222exist}\\
	{$[3,2],[2]^2$} & $(3,5), (2,3)^2$ & No & & No & Proposition~\ref{p:newfano}\\
	{$[3],[2,2,2]$} & $(3,4), (2,7)$ & Yes & 4 & No & Proposition~\ref{p:3222exist}\\ 
	{$[3],[2,2],[2]$} & $(3,4), (2,5), (2,3)$ & Yes & 3 & No & Proposition~\ref{p:3222exist}\\
	{$[3],[2]^3$} & $(3,4), (2,3)^3$ & No & & No & Proposition~\ref{p:newfano}\\
	{$[2^{[6]}]$} & $(2,13)$ & Yes & 5 & Yes & Theorem~\ref{t:isofills}\\
	{$[2^{[5]}], [2]$} & $(2,11), (2,3)$ & No & & Yes & Proposition~\ref{p:4blowups}\\
	{$[2^{[4]}], [2,2]$} & $(2,9), (2,5)$ & Yes & 7 & Yes & Proposition~\ref{p:4blowups}\\
	{$[2^{[4]}], [2]^2$} & $(2,9), (2,3)^2$ & No & & No & Proposition~\ref{p:obstrG}\\
	{$[2^{[3]}], [2^{[3]}]$} & $(2,7), (2,7)$ & No & & Yes & Proposition~\ref{p:4blowups}\\
	{$[2^{[3]}], [2,2], [2]$} & $(2,7), (2,5), (2,3)$ & No & & No & Proposition~\ref{p:obstrG}\\
	{$[2^{[3]}], [2]^3$} & $(2,7), (2,3)^3$ & Yes & 6 & No & Proposition~\ref{p:27232323}\\
	{$[2,2]^3$} & $(2,5)^3$ & Yes & 8 & No & Proposition~\ref{p:252525}\\
	{$[2,2]^2,[2]^2$} & $(2,5)^2, (2,3)^2$ & No & & No & Proposition~\ref{p:obstrG}\\
	{$[2,2], [2]^4$} & $(2,5), (2,3)^4$ & No & & No & Proposition~\ref{p:daisies}\\
	{$[2]^6$} & $(2,3)^6$ & No & & No & Proposition~\ref{p:daisies}
	\end{tabular}

	\caption{The first two columns contain the collections of singularities of each quintic, expressed as multiplicity sequences (MS) and as the type of their links (L). The third and fifth columns tell whether they are realized as quintics in $\cptwo$ or in any other rational surface, respectively. The fourth column contains the reference to~\cite[Theorem~2.3.10]{Namba} for the algebraic realization. Finally, the last column gives a reference to the relevant statement.}\label{table:quintics}
	\end{table}
%
		
		For the multi-sequence $[[3,3]]$, as mentioned in Example~\ref{ex:33}, the semigroup condition obstructs such quintic curves from being realized even smoothly in $\cptwo$.
		For the two possible cusp types with this multi-sequence, we prove with other techniques that there is no symplectic embedding of these curves in \emph{any} closed symplectic manifold. 
		Our proof actually shows that a resolution of such a curve admits no homological embedding whose components satisfy the adjunction formula into a blow-up of $\cptwo$ when we identify a $+1$-sphere in the resolution with the line.
		This is the crudest type of obstruction from the perspective of this paper, occurring already at the level of homology.
		
		\begin{prop} \label{p:noadjhom}
			There is no symplectic embedding into \emph{any} closed symplectic manifold of a rational cuspidal curve $C$ with normal Euler number $25$ and with a single cusp of type $(3,7)$ or with two cusps of type $(3,4)$. Thus, the corresponding contact manifolds have no symplectic fillings.
			In particular, there is no symplectic rational cuspidal quintic in $\cptwo$ with one singularity of type $(3,7)$ or two of type $(3,4)$.
		\end{prop}
		
		\begin{proof}
			For the case of a single cusp of type $(3,7)$, we blow up the normal crossing resolution three additional times, and obtain the configuration shown in Figure~\ref{fig:37}.
			
			\begin{figure}[h]
				\centering
				\labellist
				\pinlabel $+1$ at 45 87
				\pinlabel ${\scriptstyle h}$ at 45 72
				\pinlabel $-1$ at 90 87
				\pinlabel ${\scriptstyle h-e_0-e_1}$ at 90 72
				\pinlabel $-2$ at 135 87
				\pinlabel ${\scriptstyle e_1-e_2}$ at 135 72
				\pinlabel $-2$ at 185 87
				\pinlabel ${\scriptstyle e_2-e_3}$ at 185 72
				\pinlabel $-2$ at 240 87
				\pinlabel ${\scriptstyle e_3-e_4}$ at 240 72
				\pinlabel $-2$ at 285 87
				\pinlabel ${\scriptstyle e_4-e_5}$ at 285 72
				\pinlabel $-2$ at 330 87
				\pinlabel ${\scriptstyle e_5-e_6}$ at 330 72
				\pinlabel $-4$ at 198 50
				\pinlabel ${\scriptstyle e_0-e_1-e_2-e_3}$ at 242 50
				\pinlabel $-2$ at 220 28
				\pinlabel {\tiny no possibilities} at 220 12
				\endlabellist
				\includegraphics[scale=.85]{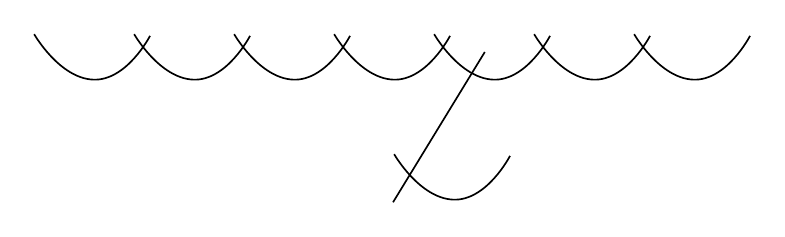}
				\caption{A resolution of a rational cuspidal quintic with one cusp of type $(3,7)$ showing that it admits no homological embedding.}
				\label{fig:37}
			\end{figure}

			
			We identify the $+1$-curve with a line in $\cptwo\#N\cptwobar$ by Theorem~\ref{thm:mcduff}. The homology classes of the upper chain are uniquely determined by Lemmas~\ref{l:hom} and~\ref{l:2chainfix} to be
			\[ (h, h-e_0-e_1, e_1-e_2, e_2-e_3, e_3-e_4, e_4-e_5, e_5-e_6).  \]
			The $(-4)$-sphere must have the form $e_{i_0}-e_{i_1}-e_{i_2}-e_{i_3}$. By Lemma~\ref{l:pos} the classes $e_1,\dots, e_5$ cannot appear with positive coefficient twice. Using this and the intersections with the other classes in the chain, the $(-4)$-sphere represents $e_0-e_3-e_2-e_1$. The remaining $(-2)$-sphere, $C$ must have the form $[C]=e_i-e_0$ for some $i > 3$, since $[C]\cdot(e_0-e_1-e_2-e_3)=1$ and $e_1,e_2,e_3$ cannot appear with positive coefficient again. Because $[C]\cdot(h-e_0-e_1)=0$, we must have $[C]=e_1-e_0$, but then $[C]\cdot(e_1-e_2)=-1$ when it should be $0$. Therefore there is no adjunctive embedding of this resolution configuration.

			For the curve with two cusps of type $(3,4)$, the normal crossing resolution graph is given by Figure~\ref{fig:3434}.

			\begin{figure}[h]
				\centering
				\labellist
				\pinlabel $-2$ at 45 87
				\pinlabel ${\scriptstyle e_5-e_6}$ at 45 70
				\pinlabel {{\tiny (not }${\scriptstyle e_0-e_4}${\tiny)}} at 45 60
				\pinlabel $-2$ at 90 87
				\pinlabel ${\scriptstyle e_4-e_5}$ at 90 70
				\pinlabel $-1$ at 145 87
				\pinlabel ${\scriptstyle h-e_0-e_4}$ at 148 70
				\pinlabel $+1$ at 185 87
				\pinlabel ${\scriptstyle h}$ at 188 70
				\pinlabel $-1$ at 228 87
				\pinlabel ${\scriptstyle h-e_0-e_1}$ at 228 70
				\pinlabel $-2$ at 285 87
				\pinlabel ${\scriptstyle e_1-e_2}$ at 285 70
				\pinlabel $-2$ at 330 87
				\pinlabel ${\scriptstyle e_2-e_3}$ at 330 70
				\pinlabel {{\tiny (not }${\scriptstyle e_0-e_1}${\tiny)}} at 330 60
				\pinlabel $-4$ at 100 50
				\pinlabel ${\scriptstyle e_0-e_1-e_2-e_3}$ at 145 50
				\pinlabel $-4$ at 248 50
				\pinlabel {\tiny no possibilities} at 300 35
				\endlabellist
				\includegraphics[scale=.85]{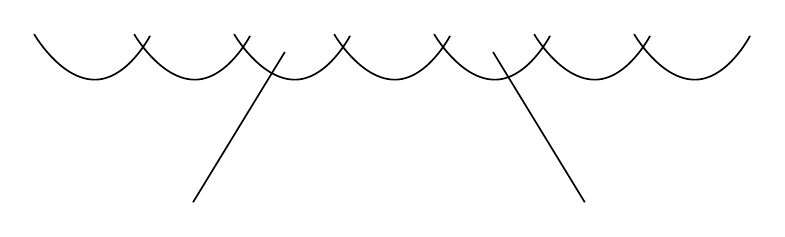}
				\caption{A resolution of a rational cuspidal quintic with two cusps of type $(3,4)$ showing that it admits no homological embedding.}
				\label{fig:3434}
			\end{figure}
			

			Again, identify the $+1$-sphere with a line in $\cptwo\#N\cptwobar$. 
			By Lemmas~\ref{l:2chain} and~\ref{l:2chainfix}, the chain emanating to the right from the $+1$-sphere initially has two possibilities
			\[ (h-e_0-e_1,e_1-e_2,e_2-e_3) \quad \text{ or } \quad (h-e_0-e_1, e_1-e_2, e_0-e_1) \]
			The other $(-1)$-sphere must have the form $h-e_0-e_4$ or $h-e_1-e_4$ based on its intersection with the $+1$ and other $(-1)$-spheres. This rules out the second possibility for the chain emanating to the right because $e_0-e_1$ intersects non-trivially with both options. The chain emanating to the left from the $+1$-sphere is uniquely determined similarly as shown in Figure~\ref{fig:3434}.

			But then the two $(-4)$-spheres must each have one exceptional class with coefficient $+1$, and that class must appear with coefficient $-1$ in the $(-1)$-sphere they are adjacent to. Since $e_1$ and $e_4$ have already appeared with positive coefficient, and no exceptional class can appear with positive coefficient twice by Lemma~\ref{l:pos}, there is no possible homological embedding.
		\end{proof}
		
		Now we consider three related cases of rational cuspidal quintics which symplectically embed uniquely into $\cptwo$. Each of these curves is known to have a complex algebraic realization in $\cptwo$ (cases 2, 4, and 3 in~\cite[Theorem~2.3.10]{Namba}, respectively).
		
		\begin{prop}\label{p:3222exist}
			If a rational cuspidal curve $C$ has normal Euler number $25$ and one of the following three cusp collections
			\begin{enumerate}
				\item one of type $(3,5)$ and one of type $(2,5)$
				\item one of type $(3,4)$ and one of type $(2,7)$
				\item one of type $(3,4)$, one of type $(2,5)$, and one of type $(2,3)$.
			\end{enumerate}
			then the only relatively minimal symplectic embedding of $C$ is into $\cptwo$ and this embedding is unique up to symplectic isotopy. 
			Equivalently, the corresponding contact structure $\xi_C$ has a unique minimal filling which is a rational homology ball.
		\end{prop}
		
		\begin{proof}
			For each of three cusp types, we blow up to a resolution where the proper transform of the cuspidal curve is smooth and has self-intersection $+1$. We apply McDuff's theorem to identify the $+1$-sphere with a line in $\cptwo\# N\cptwobar$.
			
			For the curve with a $(3,5)$- and $(2,5)$-cusp take the normal crossing resolution at the $(3,5)$-cusp, and the minimal smooth resolution at the $(2,5)$-singularity, blown up one additional time as in Figure~\ref{fig:3525}. Note that this resolution was obtained from the cuspidal curve by blowing up a total of $7$ times (there are $7$ exceptional divisors). Relative the identification of the $+1$-sphere with a line, the homology classes of the other curves are uniquely determined up to relabeling the $e_i$ by Lemma~\ref{l:hom} and the intersection relations as shown in Figure~\ref{fig:3525}. Note that this homological embedding uses seven $e_i$ classes, so if this resolution is minimally embedded, the ambient symplectic $4$-manifold must be $\cptwo\#7\cptwobar$. Thus if we return to the cuspidal curve by blowing down seven times to reverse the resolution, the cuspidal curve can only minimally symplectically embed in $\cptwo$. Blowing down spheres in the $e_i$ classes using Lemma~\ref{l:blowdown} gives a birational derivation from the cuspidal curve to a line arrangement of four lines with a single triple point. This line arrangement has a unique symplectic isotopy class by Proposition~\ref{p:addline}. Therefore this cuspidal quintic has a unique symplectic isotopy class in $\cptwo$ by Proposition~\ref{l:biratderiveunique}.
			
			\begin{figure}[h]
				\centering
				\labellist
				\pinlabel $+1$ at 200 87
				\pinlabel ${\scriptstyle h}$ at 200 72
				\pinlabel $-1$ at 320 63
				\pinlabel ${\scriptstyle h-e_1-e_2}$ at 345 50
				\pinlabel $-2$ at 280 63
				\pinlabel ${\scriptstyle h-e_3-e_4-e_5}$ at 255 55
				\pinlabel $-2$ at 250 30
				\pinlabel ${\scriptstyle e_5-e_6}$ at 250 10
				\pinlabel $-1$ at 142 63
				\pinlabel $\scriptstyle h-e_1-e_3$ at 110 65
				\pinlabel $-3$ at 145 45
				\pinlabel $\scriptstyle e_1-e_2-e_7$ at 145 30
				\pinlabel $-2$ at 90 30
				\pinlabel $\scriptstyle e_3-e_4$ at 90 10
				\pinlabel $-3$ at 43 30
				\pinlabel $\scriptstyle e_4-e_5-e_6$ at 43 10
				\endlabellist
				\includegraphics[scale=.85]{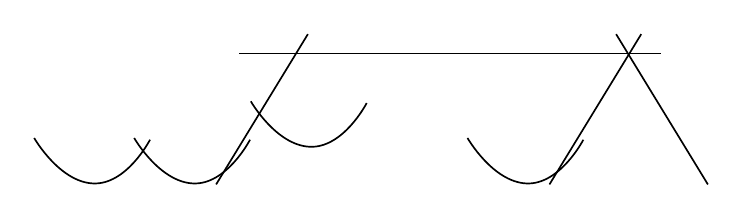}
				\caption{A resolution of a rational cuspidal quintic with one cusp of type $(3,5)$ and one of type $(2,5)$ with the only possible homological embedding.}
				\label{fig:3525}
			\end{figure}
			
			For the cuspidal curve with a $(3,4)$- and $(2,7)$-cusp, blow up two times more than the minimal resolution at the $(3,4)$-cusp, and blow up one time more than the minimal resolution at the $(2,7)$-cusp as in Figure~\ref{fig:3427}. The homology classes relative the $+1$-line are uniquely determined as shown, and use seven $e_i$ classes (the same as the number of exceptional divisors in the resolution). Therefore, the cuspidal curve embeds symplectically minimally only in $\cptwo$. Using Lemma~\ref{l:blowdown} to successively blow down exceptional spheres in the $e_i$ classes to $\cptwo$, the configuration descends to $5$ lines $\{L_i\}$ where $L_1,L_2,L_3$ intersect at a triple point and $L_3,L_4,L_5$ intersect at a triple point, and the other intersections are generic. This has a unique symplectic isotopy class by Proposition~\ref{p:addline}. Since this line arrangement is birationally derived from the embedding of the cuspidal curve, Proposition~\ref{l:biratderiveunique} implies that the cuspidal curve has a unique symplectic isotopy class in $\cptwo$.
			
			\begin{figure}[h]
				\centering
				\labellist
				\pinlabel $+1$ at 210 87
				\pinlabel ${\scriptstyle h}$ at 210 72
				\pinlabel $-3$ at 320 63
				\pinlabel ${\scriptstyle h-e_3-e_4-e_5-e_6}$ at 360 50
				\pinlabel $-1$ at 280 63
				\pinlabel ${\scriptstyle h-e_1-e_2}$ at 262 50
				\pinlabel $-2$ at 250 30
				\pinlabel ${\scriptstyle e_2-e_7}$ at 250 10
				\pinlabel $-2$ at 122 65
				\pinlabel $\scriptstyle h-e_2-e_4-e_7$ at 95 52
				\pinlabel $-1$ at 165 65
				\pinlabel $\scriptstyle h-e_1-e_3$ at 182 52
				\pinlabel $-2$ at 90 30
				\pinlabel $\scriptstyle e_4-e_5$ at 90 10
				\pinlabel $-2$ at 43 30
				\pinlabel $\scriptstyle e_5-e_6$ at 43 10
				\endlabellist
				\includegraphics[scale=.85]{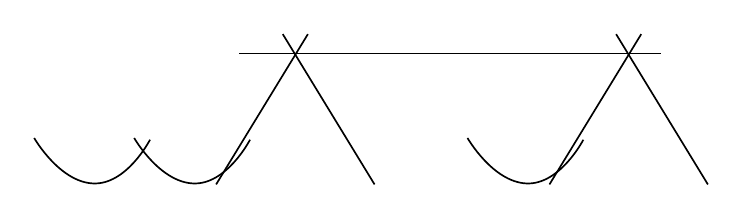}
				\caption{A resolution of a rational cuspidal quintic with one cusp of type $(3,4)$ and one of type $(2,7)$ with the only possible homological embedding.}
				\label{fig:3427}
			\end{figure}

			For the curve with cusps of type $(3,4)$, $(2,5)$, and $(2,3)$, we blow up to the minimal resolution and then once more at each of the $(3,4)$- and $(2,5)$-cusps resulting in Figure~\ref{fig:342523}. The homology classes relative the $+1$-line are uniquely determined as shown. (The uniqueness requires a little additional thought in this case, because some of the curves a priori have other options (the $(-3)$-curve and the $(-2)$-curves disjoint from the line). However, any combinations of these other than the one shown violate the intersection relations.) The unique homological embedding uses the same number of $e_i$ classes as exceptional divisors in the resolution. Therefore, the cuspidal curve embeds symplectically minimally only in $\cptwo$. 
			
			\begin{figure}[h]
				\centering
				\labellist
				\pinlabel $+1$ at 260 92
				\pinlabel ${\scriptstyle h}$ at 260 77
				\pinlabel $-3$ at 370 68
				\pinlabel ${\scriptstyle h-e_2-e_5-e_6-e_7}$ at 405 55
				\pinlabel $-1$ at 325 68
				\pinlabel ${\scriptstyle h-e_1-e_3}$ at 308 55
				\pinlabel $-2$ at 295 35
				\pinlabel ${\scriptstyle e_3-e_4}$ at 295 15
				\pinlabel $-2$ at 182 70
				\pinlabel $\scriptstyle h-e_3-e_4-e_6$ at 155 57
				\pinlabel $-1$ at 225 70
				\pinlabel $\scriptstyle h-e_1-e_2$ at 242 57
				\pinlabel $-2$ at 150 35
				\pinlabel $\scriptstyle e_6-e_7$ at 150 15
				\pinlabel $-1$ at 90 35
				\pinlabel $\scriptstyle 2h-e_1-e_2-e_3-e_4-e_5$ at 63 10
				\endlabellist
				\includegraphics[scale=.85]{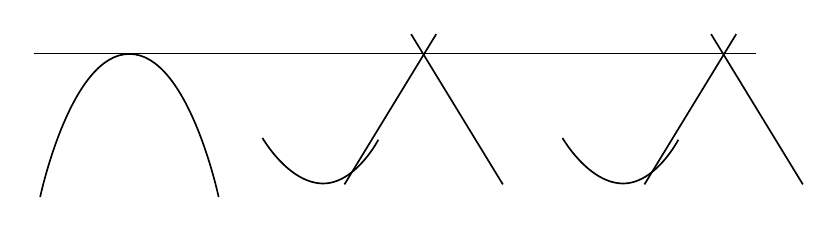}
				\caption{A resolution of a rational cuspidal quintic with one cusp of type $(3,4)$, one of type $(2,5)$, and one of type $(2,3)$ with the only possible homological embedding.}
				\label{fig:342523}
			\end{figure}
			
		We blow down exceptional spheres in the classes $e_i$, using Lemma~\ref{l:blowdown} together with blow-downs of the images of the $(-2)$-spheres disjoint from the $+1$-sphere. (For example, after blowing down a representative of $e_4$, the image of the $(-2)$-sphere in the class $e_3-e_4$ becomes an exceptional sphere representing $e_3$. The exceptional spheres representing $e_4$ and $e_3$ are sent to the same point after these two sequential blow downs.) This gives a birational derivation from the cuspidal curve to a configuration with a conic and five lines with intersections as shown in Figure~\ref{fig:F}. This configuration has a unique equisingular symplectic isotopy class by iteratively applying Proposition~\ref{p:addline} starting with $Q_0$ and adding $L_1,\dots, L_5$ one at a time in order as shown in Figure~\ref{fig:F}. Therefore this cuspidal quintic has a unique symplectic isotopy class in $\cptwo$ by Proposition~\ref{l:biratderiveunique} because it has a birational derivation to a configuration with a unique symplectic isotopy class. 
		\end{proof}

			\begin{figure}[h]
				\centering
				\labellist
				\pinlabel ${\scriptstyle e(3,4)}$ at 123 20
				\pinlabel ${\scriptstyle e(6,7)}$ at 20 35
				\pinlabel ${\scriptstyle e(1)}$ at 106 38
				\pinlabel ${\scriptstyle e(2)}$ at 156 60
				\pinlabel ${\scriptstyle e(5)}$ at 114 65
				\pinlabel $Q_0$ at 150 85
				\pinlabel $L_1$ at 120 105
				\pinlabel $L_2$ at 165 22
				\pinlabel $L_3$ at 170 81
				\pinlabel $L_4$ at 84 80
				\pinlabel $L_5$ at 50 57
				\endlabellist
				\includegraphics[scale=1.3]{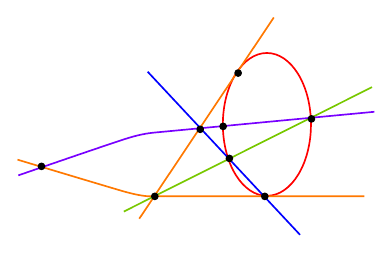}
				\caption{The configuration in $\cptwo$ resulting from blowing down $e_1,\dots, e_7$ from Figure~\ref{fig:342523}. The images of the exceptional spheres are indicated by $e(i)$. This has a unique equisingular symplectic isotopy class by iteratively applying Proposition~\ref{p:addline} starting with $Q_0$ and adding $L_1,\dots, L_5$ one at a time in order.}
				\label{fig:F}
			\end{figure}
		
		The remaining two possible cuspidal quintics with multi-sequence $[[3,2,2,2]]$ are obstructed from being realized by complex curves in $\cptwo$ by the Riemann--Hurwitz formula (see Example~\ref{ex:RH}). Because we can use $J$-holomorphic curves to recover the Riemann--Hurwitz formula for symplectic curves, the same obstruction holds in the symplectic case. However, we can actually prove the stronger statement, that these rational cuspidal curves cannot symplectically embed in any closed symplectic manifold. Moreover, reversing our argument will recover an obstruction to a symplectic Fano plane, first proven in~\cite{RuSt}.
		
			\begin{prop}\label{p:newfano}
				A rational cuspidal curves $C$ with normal Euler number $25$ and either
				\begin{enumerate}
					\item one cusp of type $(3,5)$ and two of type $(2,3)$, or
					\item one cusp of type $(3,4)$ and three of type $(2,3)$
				\end{enumerate}	
				has no symplectic embedding into any closed symplectic $4$-manifold.
				Equivalently, the associated contact structure $\xi_C$ has no symplectic filling.
			\end{prop}
			
			\begin{proof}
			We begin with the first case where $C$ has one cusp of type $(3,5)$ and two of type $(2,3)$.
				Suppose $C$ embeds symplectically minimally into $(X,\omega)$. Blow up four times to the minimal resolution and then once more at each of the three tangency points, until $C$ has self-intersection $+1$, and the configuration is given as in Figure~\ref{fig:32-2-2}.
				Apply McDuff's theorem to identify $X\# 7 \cptwobar$ with $\cptwo\#N\cptwobar$ sending $C$ to a line. The homology classes of the other curves are uniquely determined by Lemma~\ref{l:hom} and the algebraic intersection numbers as in Figure~\ref{fig:32-2-2}. There are exactly seven $e_i$ classes appearing so if $C$ was minimally symplectically embedded in $X$, we must have $X\#7\cptwobar = \cptwo\#7\cptwobar$ so $X=\cptwo$. This together with the Riemann--Hurwitz obstruction suffices to ensure $C$ does not embed into any closed symplectic manifold.
				\begin{figure}
					\centering
					\labellist
					\pinlabel $+1$ at 0 30
					\pinlabel $-2$ at 0 131
					\pinlabel $-1$ at 118 150
					\pinlabel $-2$ at 165 165
					\pinlabel $-1$ at 295 170
					\pinlabel $-2$ at 375 160
					\pinlabel $-1$ at 480 145
					\pinlabel $-3$ at 450 210
					\pinlabel $\vphantom{a}_h$ at 0 57
					\pinlabel $\vphantom{a}_{h-e_3-e_4-e_5}$ at 0 162
					\pinlabel $\vphantom{a}_{h-e_1-e_2}$ at 95 184
					\pinlabel $\vphantom{a}_{h-e_2-e_4-e_6}$ at 187 185
					\pinlabel $\vphantom{a}_{h-e_1-e_3}$ at 290 190
					\pinlabel $\vphantom{a}_{h-e_2-e_3-e_7}$ at 375 180
					\pinlabel $\vphantom{a}_{h-e_1-e_4}$ at 441 90
					\pinlabel $\vphantom{a}_{e_4-e_5-e_6}$ at 400 220
					\endlabellist
					\includegraphics[scale=.66]{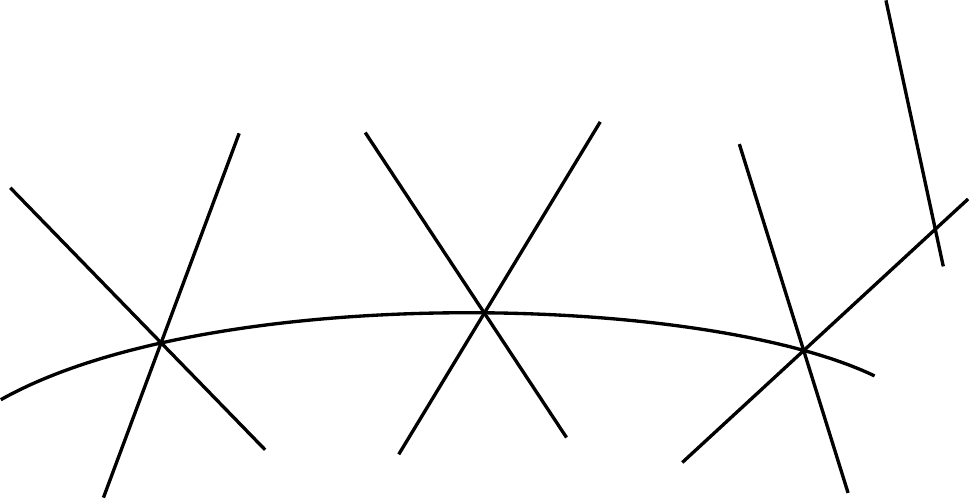}
					\caption{The unique homology classes of an embedding of the blow-up of a cuspidal curve with one cusp of type $(3,5)$ and two of type $(2,3)$.}
					\label{fig:32-2-2}
				\end{figure}
				However, going a bit further, if we apply Lemma~\ref{l:blowdown} to blow down exceptional spheres representing the $e_i$, the configuration blows down to seven lines intersecting in seven triple points: the Fano plane (Figure~\ref{fig:Fano}). (The exceptional spheres representing $e_1,e_2,e_3,e_4$ blow down to triple point intersections between lines.) Therefore the Fano line arrangement in $\cptwo$ can be birationally derived from this cuspidal quintic.
			
			Thus, we obtain a proof of the following corollary, independently of~\cite{RuSt}.
			(They prove the stronger statement in the smooth, rather than symplectic, category.)
			
			\begin{cor} \label{cor:Fano}
				The Fano configuration cannot be realised as a configuration of symplectic lines in $\cptwo$.\qed
			\end{cor}
				
				The cuspidal quintic with one $(3,4)$-cusp and three $(2,3)$-cusps is obstructed similarly. This time we blow up to the minimal resolution and then blow-up one additional time at each of the three $(2,3)$-cusp points (a total of $7$ blow-ups). The resulting configuration is shown in Figure~\ref{fig:34-2-2-2}. Identifying the $+1$-sphere with a line, the homology classes of the other curves are uniquely determined by Lemma~\ref{l:hom} and the intersection relations as shown.
				There are exactly seven $e_i$ classes used so the only possible $4$-manifold where this cuspidal curve can minimally symplectically embed is $\cptwo$. Blowing down using Lemma~\ref{l:blowdown} yields a configuration consisting of the Fano configuration together with a singular cubic. Since the Fano configuration cannot exist (by Theorem~\ref{thm:Fano} or Corollary~\ref{cor:Fano}), this cuspidal quintic cannot embed into $\cptwo$ by Proposition~\ref{l:biratexist}.
			\end{proof}				
				\begin{figure}
					\centering
					\labellist
					\pinlabel $+1$ at 0 110
					\pinlabel $h$ at 0 140
					\pinlabel $-1$ at 90 80
					\pinlabel $3h-2e_1-e_2-\dots -e_7$ at 110 170
					\pinlabel $F$ at 35 200
					\pinlabel $-1$ at 255 230
					\pinlabel $E'_1$ at 240 180
					\pinlabel $h-e_1-e_2$ at 200 220
					\pinlabel $-1$ at 399 230
					\pinlabel $E'_2$ at 384 180
					\pinlabel $h-e_1-e_3$ at 344 220
					\pinlabel $-1$ at 543 230
					\pinlabel $E'_3$ at 528 180
					\pinlabel $h-e_1-e_4$ at 488 220
					\pinlabel $-2$ at 255 30
					\pinlabel $E_1$ at 240 70
					\pinlabel $h-e_3-e_4-e_5$ at 190 20
					\pinlabel $-2$ at 399 30
					\pinlabel $E_2$ at 384 70
					\pinlabel $h-e_2-e_4-e_6$ at 334 20
					\pinlabel $-2$ at 543 30
					\pinlabel $E_3$ at 528 70
					\pinlabel $h-e_2-e_3-e_7$ at 478 20
					\endlabellist
					\includegraphics[scale=.66]{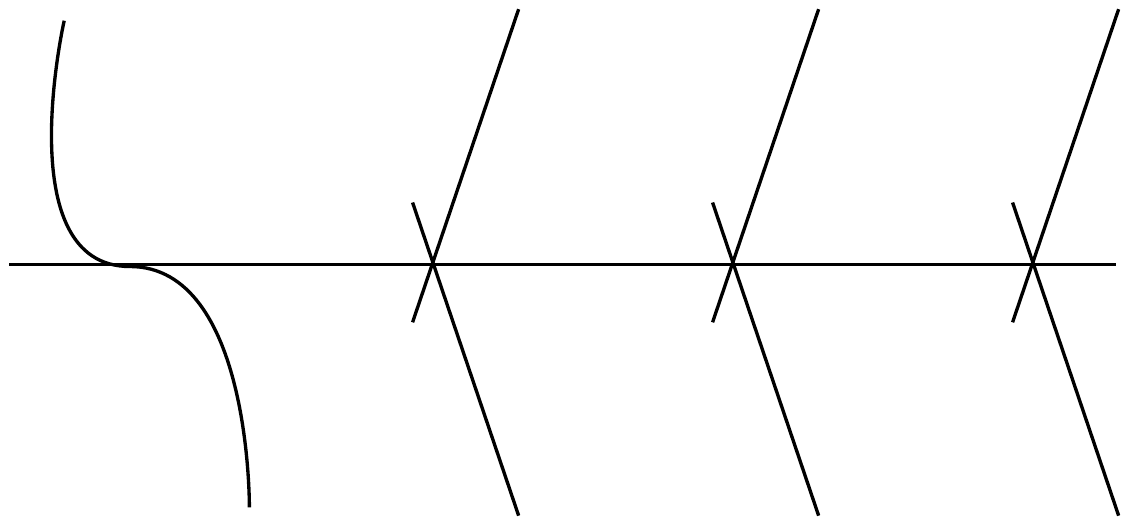}
					\caption{The unique homology classes of an embedding of the blow-up of a cuspidal curve with one cusp of type $(3,4)$ and three of type $(2,3)$.}
					\label{fig:34-2-2-2}
				\end{figure}
			
		There are two other quintics which are obstructed from realization by complex curves in $\cptwo$ using Riemann--Hurwitz (Example~\ref{ex:RH}). We can obstruct these curves from embedding symplectically into any closed symplectic $4$-manifold also.
		
			\begin{prop}\label{p:daisies}
				Let $C$ a rational cuspidal curve with normal Euler number $25$, with either 
				\begin{enumerate}
					\item one cusp of type $(2,5)$ and four of type $(2,3)$, or 
					\item six cusps of type $(2,3)$.
				\end{enumerate}
				Then $C$ does not embed symplectically into any closed symplectic $4$-manifold and contact structure $\xi_C$ associated to $C$ is not symplectically fillable.
			\end{prop}
			
			\begin{proof}
				Assume $C$ embeds in $(X,\omega)$. In both cases, we blow up $6$ times to the minimal resolution, where the proper transform of $C$ has self-intersection $+1$. We apply McDuff's theorem and classify homological embeddings of the configurations.
				
				Every exceptional curve which is simply tangent to the $+1$-curve will represent the homology class $2h-e_{i_1}-\dots - e_{i_5}$ by Lemma~\ref{l:hom}.  Any pair of such curves are disjoint, so each pair must share exactly four exceptional classes with coefficient $-1$. When there are more than three such curves, there are two ways this can happen: either all of these curves have the form $2h-e_1-e_2-e_3-e_4-e_i$ (for all different values of $i$) or all of the conics have the form $2h-e_1 - \dots -e_6+e_j$ for $j\in \{1,\dots, 6\}$ (this is possible when there are at most $6$ such curves).
				
				Keeping these generalities in mind, we now focus on the case with five cusps. There are five conics tangent to the $+1$-line, which must have homology classes fitting one of the two possibilities discussed above. The remaining $(-2)$-sphere intersects one of these conics and in each case, its homology class is uniquely determined based on its intersections with the conics as in Figure~\ref{fig:25232323}. In the first embedding where the conics all have the form $2h-e_1-e_2-e_3-e_4-e_i$, there are $10$ exceptional classes so $X\#6\cptwobar\cong \cptwo\#10\cptwobar$, so this corresponds to an embedding of the cuspidal curve in $X=\cptwo\#4\cptwobar$. 
				
				Blowing down exceptional spheres in $e_i$ classes using Lemma~\ref{l:blowdown}, the configuration in $\cptwo\#10\cptwobar$ descends to five conics all intersecting at four common fixed points (the images of $e_1,e_2,e_3,e_4$), with a line tangent to all five conics. This contains the configuration $\Gc^{\star}$ which is obstructed by Proposition~\ref{p:Gstar}. Since we birationally derived an obstructed configuration, by Proposition~\ref{l:biratexist} this homological embedding of the minimal resolution configuration into $\cptwo\#10\cptwobar$ cannot be realized by a symplectic embedding (equivalently the cuspidal curve does not have a corresponding relatively minimal embedding into $\cptwo\#4\cptwobar$).
				
				In the second embedding, only $6$ exceptional classes are used so this corresponds to a symplectic embedding of the cuspidal curve in $\cptwo$. This is obstructed by the Riemann--Hurwitz obstruction (see Example~\ref{ex:RH}).
				
				\begin{figure}
					\centering
					\labellist
					\pinlabel $+1$ at 335 95
					\pinlabel ${\scriptstyle h}$ at 335 80
					\pinlabel $-1$ at 500 70
					\pinlabel ${\color{red} \scriptstyle{2h-e_1-e_2-e_3-e_4-e_9}}$ at 492 10
					\pinlabel $\scriptstyle{2h-e_1-e_3-e_4-e_5-e_6}$ at 492 0
					\pinlabel $-1$ at 403 70
					\pinlabel ${\color{red} \scriptstyle{2h-e_1-e_2-e_3-e_4-e_8}}$ at 395 10
					\pinlabel $\scriptstyle{2h-e_1-e_2-e_4-e_5-e_6}$ at 395 0
					\pinlabel $-1$ at 307 70
					\pinlabel ${\color{red} \scriptstyle{2h-e_1-e_2-e_3-e_4-e_7}}$ at 297 10
					\pinlabel $\scriptstyle{2h-e_1-e_2-e_3-e_5-e_6}$ at 297 0
					\pinlabel $-1$ at 211 70
					\pinlabel ${\color{red} \scriptstyle {2h-e_1-e_2-e_3-e_4-e_6}}$ at 200 10
					\pinlabel $\scriptstyle {2h-e_1-e_2-e_3-e_4-e_6}$ at 202 0
					\pinlabel $-1$ at 115 70
					\pinlabel ${\color{red} \scriptstyle {2h-e_1-e_2-e_3-e_4-e_5}}$ at 104 10
					\pinlabel $\scriptstyle {2h-e_1-e_2-e_3-e_4-e_5}$ at 104 0					
					\pinlabel $-2$ at 40 30
					\pinlabel ${\scriptstyle \color{red}e_5-e_{10}}$ at 38 12
					\pinlabel $\scriptstyle e_1-e_6$ at 38 0
					\endlabellist
					\includegraphics[scale=.85]{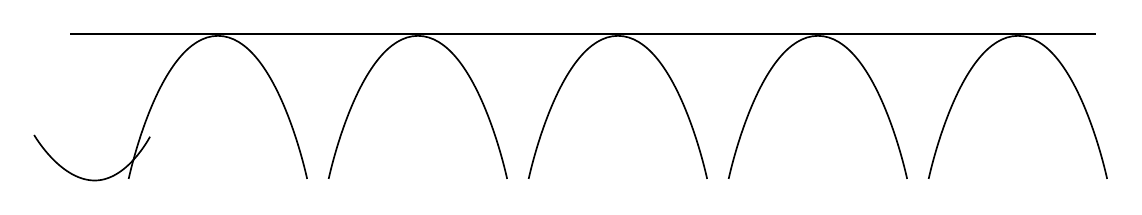}
					\caption{Minimal resolution of a rational cuspidal quintic with one cusp of type $(2,7)$, one of type $(2,5)$, and one of type $(2,3)$ with the possible homological embeddings.}
					\label{fig:25232323}
				\end{figure}
								
				In the six-cuspidal case, there are six conics tangent to the $+1$-curve $C$.
				Again, there are two possible homological embeddings, fully determined by the two options for the classes of the conics described above. The first case uses $10$ exceptional classes and thus corresponds to an embedding of the cuspidal curve into $\cptwo\#4\cptwobar$. Blowing down with Lemma~\ref{l:blowdown}, the configuration descends to six conics in a pencil with a line tangent to all six. This contains the configuration $\Gc^{\star}$ so by Propositions~\ref{p:Gstar} and~\ref{l:biratexist}, there is no relatively minimal symplectic embedding of this cuspidal curve into $\cptwo\#4\cptwobar$. The second homological embedding uses $6$ exceptional classes and thus corresponds to an embedding of the cuspidal curve into $\cptwo$, which is obstructed by Riemann--Hurwitz (Example~\ref{ex:RH}).
			\end{proof}
			
	In the previous two cases, we had potential symplectic embeddings of the minimal resolution into $\cptwo\#10\cptwobar$ with homology classes determined relative the $+1$-curve being identified with a line. Blowing down exceptional spheres in the $e_i$ classes lead to a configuration of conics and a line that we could obstruct because it contained $\Gc^\star$. Next, we give three more cases where we can obstruct all of the possible configurations of conics and lines resulting from blowing down $e_i$ classes.
	
	\begin{prop}\label{p:obstrG}
		Let $C$ a rational cuspidal curve with normal Euler number $25$, with either 
		\begin{enumerate}
			\item one cusp of type $(2,9)$ and two of type $(2,3)$,
			\item one cusp of type $(2,7)$, one of type $(2,5)$, and one of type $(2,3)$, or
			\item two cusps of type $(2,5)$, and two of type $(2,3)$.
		\end{enumerate}
		Then $C$ does not embed symplectically into any closed symplectic $4$-manifold and contact structure $\xi_C$ associated to $C$ is not symplectically fillable.
	\end{prop}
		
	\begin{proof}
		In all three cases, the minimal resolution results from $6$ blow-ups, and the proper transform of the cuspidal curve becomes a smooth $+1$-sphere $C$ which we use to apply McDuff's theorem. The six exceptional divisors are either $(-1)$-curves tangent to $C$ or $(-2)$-curves disjoint from $C$. The two choices (up to relabeling the $e_i$) of homological embedding for the tangent $(-1)$-curves each determine the homology classes of the $(-2)$-spheres. The minimal resolutions with their two possible homological embeddings are shown in Figures~\ref{fig:292323},~\ref{fig:272523}, and~\ref{fig:25252323}.
		
		For each of the three cuspidal curves, there is one homological embedding option where all of the tangent conics represent homology class of the form $2h-e_1-e_2-e_3-e_4-e_i$ (the upper option in red in the figures). If we use Lemma~\ref{l:blowdown} to blow down exceptional spheres in the $e_i$ classes, this gives a birational derivation to a configuration that contains (or equals) $\Gc^{\star}$. Since $\Gc^{\star}$ is obstructed by Proposition~\ref{p:Gstar}, there cannot be any symplectic embedding of the resolution with such homology classes in $\cptwo\#10\cptwobar$.
		
		For each of the cuspidal curves we consider the remaining homological embedding, and blow down exceptional spheres representing the $e_i$ using Lemma~\ref{l:blowdown}.
		
		\begin{figure}
			\centering
			\labellist
			\pinlabel $+1$ at 380 95
			\pinlabel $\scriptstyle h$ at 380 80
			\pinlabel $-1$ at 456 70
			\pinlabel ${\color{red}\scriptstyle 2h-e_1-e_2-e_3-e_4-e_7 }$ at 425 10
			\pinlabel $\scriptstyle 2h-e_1-e_2-e_3-e_5-e_6$ at 425 0
			\pinlabel $-1$ at 345 70
			\pinlabel ${\color{red}\scriptstyle 2h-e_1-e_2-e_3-e_4-e_6}$ at 314 10
			\pinlabel $\scriptstyle 2h-e_1-e_2-e_3-e_4-e_6$ at 314 0
			\pinlabel $-1$ at 233 70
			\pinlabel ${\color{red}\scriptstyle 2h-e_1-e_2-e_3-e_4-e_5}$ at 205 10
			\pinlabel $\scriptstyle 2h-e_1-e_2-e_3-e_4-e_5$ at 205 0
			\pinlabel $-2$ at 136 30
			\pinlabel ${\color{red}\scriptstyle e_5-e_8}$ at 134 10
			\pinlabel $\scriptstyle e_1-e_6$ at 134 0
			\pinlabel $-2$ at 88 30
			\pinlabel ${\color{red}\scriptstyle e_8-e_9}$ at 88 10
			\pinlabel $\scriptstyle e_2-e_1$ at 88 0
			\pinlabel $-2$ at 40 30
			\pinlabel ${\color{red}\scriptstyle e_9-e_{10}}$ at 40 10
			\pinlabel $\scriptstyle e_3-e_2$ at 40 0
			\endlabellist
			\includegraphics[scale=.85]{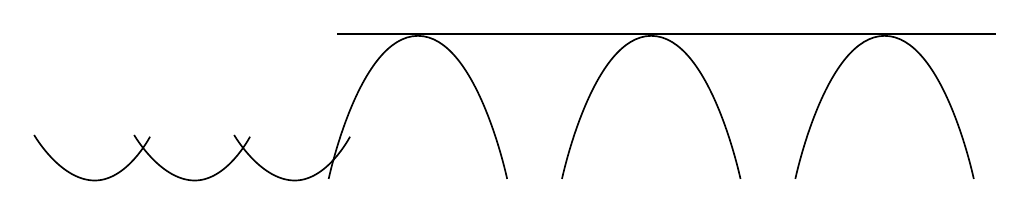}
			\caption{Minimal resolution of a rational cuspidal quintic with one cusp of type $(2,9)$ and two of type $(2,3)$ with the possible homological embeddings.}
			\label{fig:292323}
		\end{figure}
		
		For the curve with one $(2,9)$-cusp and two $(2,3)$-cusps, the minimal resolution is shown in Figure~\ref{fig:292323}. Focusing on the lower (black) homological embedding, we blow down spheres representing classes $e_6,e_1,e_2$, and $e_3$ all to the same point. Because all four of these exceptional spheres intersect the two conics on the right (which resolved the two $(2,3)$-cusps), the resulting configuration in the blow-down will contain two conics which intersect at a single point with multiplicity $4$, one more conic, and a line tangent to all three conics. Focusing on the first two conics and the tangent line, we see this configuration contains $\Gc_4$, obstructed by Proposition~\ref{p:G422}. Therefore by Proposition~\ref{l:biratexist}, the cuspidal curve with a $(2,9)$-cusp and two $(2,3)$-cusps cannot symplectically embed into $\cptwo$.

		\begin{figure}
			\centering
			\labellist
			\pinlabel $+1$ at 310 95
			\pinlabel $\scriptstyle h$ at 310 80
			\pinlabel $-1$ at 388 70
			\pinlabel ${\color{red} \scriptstyle {2h-e_1-e_2-e_3-e_4-e_7}}$ at 352 10
			\pinlabel $\scriptstyle{2h-e_1-e_2-e_3-e_5-e_6}$ at 352 0
			\pinlabel $-2$ at 414 30
			\pinlabel ${\color{red}\scriptstyle e_7-e_8}$ at 419 10
			\pinlabel $\scriptstyle e_3-e_4$ at 419 0
			\pinlabel $-1$ at 284 70
			\pinlabel ${\color{red} \scriptstyle {2h-e_1-e_2-e_3-e_4-e_6}}$ at 252 10
			\pinlabel $\scriptstyle{2h-e_1-e_2-e_3-e_4-e_6}$ at 252 0
			\pinlabel $-1$ at 182 70
			\pinlabel ${\color{red} \scriptstyle {2h-e_1-e_2-e_3-e_4-e_5}}$ at 153 10
			\pinlabel $\scriptstyle {2h-e_1-e_2-e_3-e_4-e_5}$ at 153 0
			\pinlabel $-2$ at 88 30
			\pinlabel ${\color{red}\scriptstyle e_5-e_9}$ at 88 10
			\pinlabel $\scriptstyle e_1-e_6$ at 88 0
			\pinlabel $-2$ at 40 30
			\pinlabel ${\color{red}\scriptstyle e_9-e_{10}}$ at 40 10
			\pinlabel $\scriptstyle e_2-e_1$ at 40 0
			\endlabellist
			\includegraphics[scale=.85]{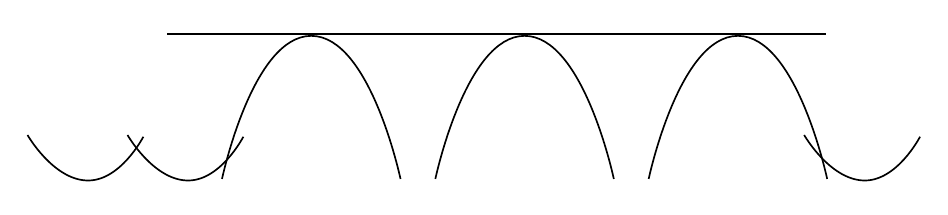}
			\caption{Minimal resolution of a rational cuspidal quintic with one cusp of type $(2,7)$, one of type $(2,5)$, and one of type $(2,3)$ with the possible homological embeddings.}
			\label{fig:272523}
		\end{figure}
		
		For the curve with one cusp each of types $(2,7)$, $(2,5)$, and $(2,3)$, the minimal resolution with the possible homological embeddings appears in Figure~\ref{fig:272523}. Again focusing on the lower (black) homological embedding, which is the only one left to rule out, we blow down exceptional spheres in $e_i$ classes using Lemma~\ref{l:blowdown}. We see that exceptional spheres in classes $e_1,e_2$, and $e_6$ all blow down to the same point, and exceptional spheres in classes $e_3$ and $e_4$ blow down to the same point. The two leftmost conics (coming from the last resolving exceptional divisors of the $(2,7)$- and $(2,3)$-cusps) each intersect $e_1,e_2,e_3$, and $e_4$. Therefore after blowing down these two conics become tangent at two different points and the $+1$-sphere descends to a line tangent to each. Therefore $\Gc_{2,2}$ appears as a subconfiguration of the result of blowing down. Since $\Gc_{2,2}$ is obstructed by Proposition~\ref{p:G422}, the cuspidal curve cannot symplectically embed because it has a birational derivation to an obstructed configuration.

		\begin{figure}
			\centering
			\labellist
			\pinlabel $+1$ at 265 95
			\pinlabel $\scriptstyle h$ at 265 80
			\pinlabel $-1$ at 437 70
			\pinlabel ${\color{red} \scriptstyle {2h-e_1-e_2-e_3-e_4-e_8}}$ at 408 10
			\pinlabel $\scriptstyle{2h-e_1-e_2-e_4-e_5-e_6}$ at 408 0
			\pinlabel $-2$ at 467 30
			\pinlabel ${\color{red}\scriptstyle e_8-e_9}$ at 473 10
			\pinlabel $\scriptstyle e_2-e_3$ at 473 0
			\pinlabel $-1$ at 340 70
			\pinlabel ${\color{red} \scriptstyle {2h-e_1-e_2-e_3-e_4-e_7}}$ at 309 10
			\pinlabel $\scriptstyle{2h-e_1-e_2-e_3-e_5-e_6}$ at 309 0
			\pinlabel $-1$ at 237 70
			\pinlabel ${\color{red} \scriptstyle {2h-e_1-e_2-e_3-e_4-e_6}}$ at 208 10
			\pinlabel $\scriptstyle {2h-e_1-e_2-e_3-e_4-e_6}$ at 208 0
			\pinlabel $-1$ at 136 70
			\pinlabel ${\color{red} \scriptstyle {2h-e_1-e_2-e_3-e_4-e_5}}$ at 107 10
			\pinlabel $\scriptstyle {2h-e_1-e_2-e_3-e_4-e_5}$ at 107 0
			\pinlabel $-2$ at 40 30
			\pinlabel ${\color{red}\scriptstyle e_5-e_{10}}$ at 40 10
			\pinlabel $\scriptstyle e_1-e_6$ at 40 0
			\endlabellist
			\includegraphics[scale=.85]{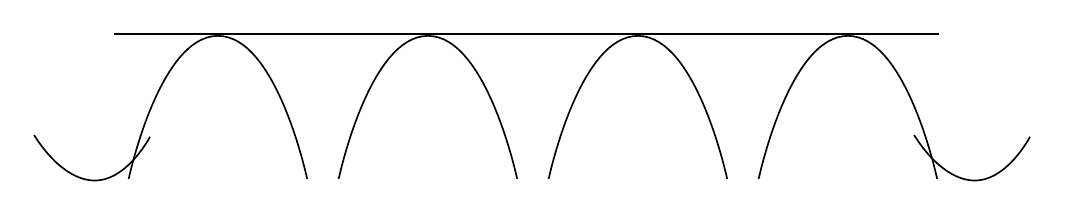}
			\caption{Minimal resolution of a rational cuspidal quintic with two cusps of type $(2,5)$ and two of type $(2,3)$ with the possible homological embeddings.}
			\label{fig:25252323}
		\end{figure}
		
		For the curve with two $(2,5)$-cusps and two $(2,3)$-cusps the minimal resolution with homological embeddings is in Figure~\ref{fig:25252323}. To rule out the second homological embedding shown in black, we blow down using Lemma~\ref{l:blowdown}. The pairs of exceptional spheres $e_1, e_6$ and $e_2, e_3$ each blow down to a single point. The two central conics coming from exceptional divisors where the $(2,3)$-cusps were blown up both intersect all four of these exceptional spheres. Therefore these two curves descend to two conics tangent at two points with a line tangent to both, which again is the obstructed configuration $\Gc_{2,2}$. Thus by Propositions~\ref{p:G422} and~\ref{l:biratexist}, this cuspidal curve cannot symplectically embed.
	\end{proof}

	In the next three cases, we look at the bicuspidal curves with multi-sequence $[[2,2,2,2,2,2]]$. This is the first case where we find relatively minimal symplectic embeddings into a nontrivial blow-up of $\cptwo$. All three cases have relatively minimal symplectic embeddings into $\cptwo\#4\cptwobar$ and one of the cases also has a symplectic embedding into $\cptwo$ (case 7 in~\cite[Theorem~2.3.10]{Namba}).
	
	\begin{prop}\label{p:4blowups}
		Let $C$ be a rational cuspidal curve with normal Euler number $25$. 
		
		If $C$ has one cusp of type $(2,9)$ and one of type $(2,5)$, then $C$ has a unique minimal symplectic embedding in $\cptwo$ and a unique minimal symplectic embedding into $\cptwo\#4\cptwobar$ up to symplectic isotopy and symplectomorphism. Equivalently the corresponding contact structure $\xi_C$ has two symplectic fillings with $b_2 = 0,4$ respectively.
			
		If $C$ has one cusp of type $(2,11)$ and one of type $(2,3)$, or two cusps of type $(2,7)$ then $C$ has a unique minimal symplectic embedding into $\cptwo\#4\cptwobar$ up to symplectic isotopy and symplectomorphism. Equivalently the corresponding contact structure $\xi_C$ has a unique symplectic filling with $b_2 = 4$.
	\end{prop}
	
	\begin{proof}
		In all three cases, we will blow up the cuspidal curve to its minimal resolution where the proper transform $C$ is a smooth $+1$-sphere that we identify with a line using McDuff's theorem. Because there are exactly two cusps, the resolution will always contain two $(-1)$-exceptional divisors tangent to the $+1$-curve and four other $(-2)$-exceptional divisors. 
		
		We determine homology classes relative the $+1$-curve. The two tangent $(-1)$-divisors necessarily must be $2h-e_1-e_2-e_3-e_4-e_5$ and $2h-e_2-e_3-e_4-e_5-e_6$ (up to relabeling the $e_i$). We then determine the options for the four $(-2)$-curves according to the intersection relations. These are shown in Figures~\ref{fig:21123},~\ref{fig:2925}, and~\ref{fig:2727}.
		
		\begin{figure}
			\centering
			\labellist
			\pinlabel $+1$ at 300 95
			\pinlabel $h$ at 300 80
			\pinlabel $-1$ at 391 70
			\pinlabel ${2h-e_2-\dots-e_6}$ at 360 20
			\pinlabel $-1$ at 280 70
			\pinlabel ${2h-e_1-\dots-e_5}$ at 250 20
			\pinlabel $-2$ at 185 30
			\pinlabel ${\color{red}e_1-e_7}$ at 185 12
			\pinlabel $e_5-e_6$ at 185 0
			\pinlabel $-2$ at 140 30
			\pinlabel ${\color{red} e_7-e_8}$ at 140 12
			\pinlabel $e_4-e_5$ at 140 0
			\pinlabel $-2$ at 90 30
			\pinlabel ${\color{red} e_8-e_9}$ at 90 12
			\pinlabel $e_3-e_4$ at 90 0
			\pinlabel $-2$ at 42 30
			\pinlabel ${\color{red}e_9-e_{10}}$ at 42 12
			\pinlabel $e_2-e_3$ at 42 0
			\endlabellist
			\includegraphics[scale=1]{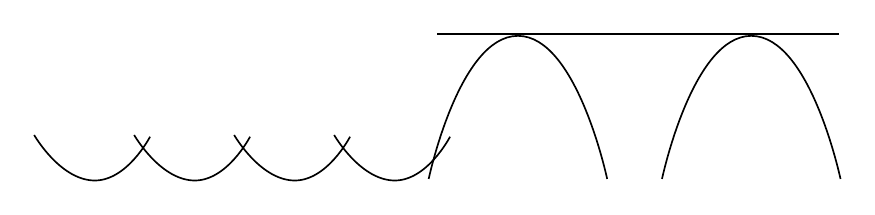}
			\caption{Minimal smooth resolution of a rational cuspidal quintic with one cusp of type $(2,11)$ and one of type $(2,3)$.}
			\label{fig:21123}
		\end{figure}
		
		\begin{figure}
			\centering
			\labellist
			\pinlabel $+1$ at 250 95
			\pinlabel $h$ at 250 80
			\pinlabel $-2$ at 380 30
			\pinlabel ${\color{red}e_6-e_{10}}$ at 380 12
			\pinlabel $e_2-e_1$ at 380 0
			\pinlabel $-1$ at 343 70
			\pinlabel ${2h-e_2-\dots-e_6}$ at 314 20
			\pinlabel $-1$ at 230 70
			\pinlabel ${2h-e_1-\dots-e_5}$ at 202 20
			\pinlabel $-2$ at 137 30
			\pinlabel ${\color{red} e_1-e_7}$ at 137 12
			\pinlabel $e_5-e_6$ at 137 0
			\pinlabel $-2$ at 90 30
			\pinlabel ${\color{red} e_7-e_8}$ at 90 12
			\pinlabel $e_4-e_5$ at 90 0
			\pinlabel $-2$ at 42 30
			\pinlabel ${\color{red}e_8-e_9}$ at 42 12
			\pinlabel $e_3-e_4$ at 42 0
			\endlabellist
			\includegraphics[scale=1]{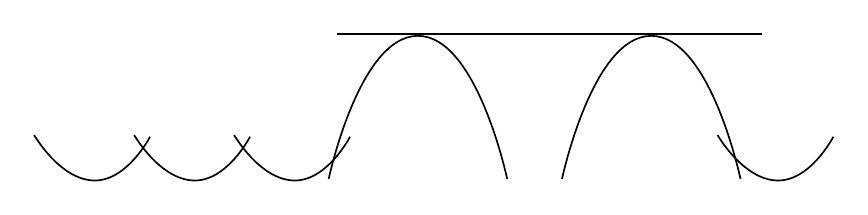}
			\caption{Minimal smooth resolution of a rational cuspidal quintic with one cusp of type $(2,9)$ and one of type $(2,5)$.}
			\label{fig:2925}
		\end{figure}

		\begin{figure}
			\centering
			\labellist
			\pinlabel $+1$ at 200 95
			\pinlabel $h$ at 200 80
			\pinlabel $-1$ at 295 70
			\pinlabel ${2h-e_2-\dots-e_6}$ at 265 20
			\pinlabel $-2$ at 325 30
			\pinlabel ${\color{red} e_6-e_7}$ at 325 12
			\pinlabel $e_2-e_1$ at 325 0
			\pinlabel $-2$ at 375 30
			\pinlabel ${\color{red}e_7-e_8}$ at 375 12
			\pinlabel $e_3-e_2$ at 375 0
			\pinlabel $-1$ at 183 70
			\pinlabel ${2h-e_1-\dots-e_5}$ at 154 20
			\pinlabel $-2$ at 88 30
			\pinlabel ${\color{red} e_1-e_9}$ at 88 12
			\pinlabel $e_5-e_6$ at 88 0
			\pinlabel $-2$ at 40 30
			\pinlabel ${\color{red}e_9-e_{10}}$ at 40 12
			\pinlabel $e_4-e_5$ at 40 0
			\endlabellist
			\includegraphics[scale=1]{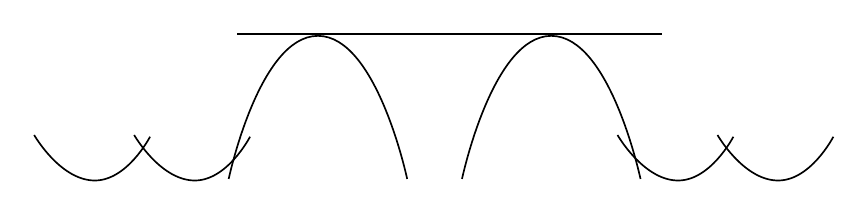}
			\caption{Minimal smooth resolution of a rational cuspidal quintic with two cusps of type $(2,7)$.}
			\label{fig:2727}
		\end{figure}
			
		In each case, there are two possible homological embeddings of the minimal resolution. One into $\cptwo\#6\cptwobar$ and the other into $\cptwo\#10\cptwobar$. If these homological embeddings can be realized symplectically, they would give symplectic embeddings of the cuspidal curve into $\cptwo$ and $\cptwo\#4\cptwobar$ respectively.
		
		In each of the three cases, if we assume we have a relatively minimal symplectic embedding of the minimal resolution into $\cptwo\#10\cptwobar$, and we blow down exceptional spheres representing the $e_i$ using Lemma~\ref{l:blowdown}, the resulting configuration in $\cptwo$ consists of two conics 
		with a line tangent to each. {The intersection points of the conics are the images of $e_2$, $e_3$, $e_4$, and $e_5$. When these exceptional spheres are realized disjointly,} this is the configuration $\Gc_1$ which has a unique symplectic isotopy class by Proposition~\ref{p:G1}. {If the images of some of the exceptional spheres representing $e_2$, $e_3$, $e_4$, and $e_5$ coincide, the resulting configuration is $\Gc_2$, $\Gc_3$, or $\Gc_4$ (though it cannot be $\Gc_4$ since this is obstructed by Proposition~\ref{p:G422}). The configurations $\Gc_2$ and $\Gc_3$ can be symplectically deformed to $\Gc_1$ (use a complex model and perform a $C^1$ small deformation, and then insert the deformation in symplectically using the fact that the change to the boundary is $C^1$ small). The symplectic isotopy uniqueness of $C$ depends on the uniqueness of the symplectic isotopy class of the resolution in $\cptwo\#6\cptwobar$. This resolution is the strict transform of the $\Gc_k$ configuration after blowing up the $e_i$. The blowing up the deformation gives a symplectic isotopy of the strict transforms which isotopes any configuration that descends to $\Gc_2$ or $\Gc_3$ to a configuration that descends to $\Gc_1$.} Therefore 
		each of the three types of cuspidal curves has a unique relatively minimal symplectic embedding in $\cptwo\#4\cptwobar$ up to symplectomorphism.
		
		For the homological embedding of the minimal resolution into $\cptwo\#6\cptwobar$, we consider each cuspidal curve separately.
		
		For the case with one cusp of type $(2,11)$ and one of type $(2,3)$, we start with a relatively minimal embedding of the resolution into $\cptwo\#6\cptwobar$ and blow down exceptional spheres in the $e_i$ classes using Lemma~\ref{l:blowdown}. The exceptional spheres in classes $e_2,e_3,e_4,e_5$, and $e_6$ all blow down to the same point. This produces a
		order-$4$ tangency between the two conics so the configuration descends to $\Gc_4$ which is obstructed (Proposition~\ref{p:G422}). Therefore this cuspidal curve does not symplectically embed in $\cptwo$.
		
		Next, we start with a symplectic embedding into $\cptwo\#6\cptwobar$ of the minimal resolution of the curve with one cusp of type $(2,9)$ and one of type $(2,5)$ with the second homological embedding given in Figure~\ref{fig:2925}. Blowing down exceptional spheres in the $e_i$ classes, we see that the spheres in classes $e_3,e_4,e_5$, and $e_6$ descend to the same point, and those in classes $e_1$ and $e_2$ descend to another point. The effect is that the conics become tangent at one point with multiplicity $3$ and intersect transversally at another point. This yields the configuration $\Gc_3$ which has a unique symplectic isotopy class by Proposition~\ref{p:G3}. Therefore by Proposition~\ref{l:biratderiveunique}, there is a unique relatively minimal symplectic embedding of this cuspidal curve into $\cptwo$.
		
		Finally, considering a potential symplectic embedding into $\cptwo\#6\cptwobar$ of the minimal resolution of the cupsidal quintic with two cusps of type $(2,7)$, the homological embedding is given in Figure~\ref{fig:2727}. Blowing down exceptional spheres using Lemma~\ref{l:blowdown} we find that exceptional spheres in classes $e_4, e_5$, and $e_6$ descend to one point and those in classes $e_1,e_2$, and $e_3$ descend to another point. This has the effect that the two conics intersect tangentially at two points. Therefore the configuration descends to $\Gc_{2,2}$ which is obstructed by Proposition~\ref{p:G422}. Thus by Proposition~\ref{l:biratexist} there is no symplectic embedding of this cuspidal quintic in $\cptwo$.
	\end{proof}
	
	\begin{remark}
		The rational cuspidal quintic with two cusps of type $(2,7)$ can alternatively be obstructed in $\cptwo$ using spectrum semicontinuity as in Example~\ref{ex:2727}.
	\end{remark}

	Finally, there are two more possible collections of cusps with multi-sequence $[[2,2,2,2,2,2]]$. There can be three cusps of type $(2,5)$ or one of type $(2,7)$ and three of type $(2,3)$. Each of these cases turns out to have a unique symplectic embedding into $\cptwo$ and no other relatively minimal symplectic embeddings (cases 8 and 6 in~\cite[Theorem~2.3.10]{Namba}, respectively). Some of the reducible configurations of conics and lines that appear in these proofs required more extensive arguments to establish their uniqueness in Section~\ref{s:rediso}, so we separate these two cases as the trickiest of the bunch.
	
	\begin{prop}\label{p:252525}
		If a rational cuspidal curve $C$ has normal Euler number $25$ and  three cusps of type $(2,5)$,
		then only relatively minimal symplectic embedding of $C$ is into $\cptwo$ and this embedding is unique up to symplectic isotopy. 
		Equivalently, the corresponding contact structure $\xi_C$ has a unique minimal filling which is a rational homology ball.
	\end{prop}
	
	\begin{proof}
		We look at the minimal smooth resolution at each singularity of $C$;
		the proper transform of $C$ is a smooth $+1$-sphere, and the total transform is given in Figure~\ref{fig:252525} with the possible homological embeddings where the $+1$-sphere is identified with a line.
		
		\begin{figure}
			\centering
			\labellist
			\pinlabel $+1$ at 265 95
			\pinlabel $\scriptstyle h$ at 265 80
			\pinlabel $-1$ at 379 70
			\pinlabel ${\color{red} \scriptstyle {2h-e_1-e_2-e_3-e_4-e_7}}$ at 346 10
			\pinlabel $\scriptstyle{2h-e_1-e_2-e_3-e_5-e_6}$ at 346 0
			\pinlabel $-2$ at 407 30
			\pinlabel ${\color{red}\scriptstyle e_7-e_{10}}$ at 412 10
			\pinlabel $\scriptstyle e_1-e_4$ at 412 0
			\pinlabel $-1$ at 240 70
			\pinlabel ${\color{red} \scriptstyle {2h-e_1-e_2-e_3-e_4-e_6}}$ at 206 10
			\pinlabel $\scriptstyle {2h-e_1-e_2-e_3-e_4-e_6}$ at 206 0
			\pinlabel $-2$ at 267 30
			\pinlabel ${\color{red}\scriptstyle e_6-e_9}$ at 272 10
			\pinlabel $\scriptstyle e_2-e_5$ at 272 0
			\pinlabel $-1$ at 91 70
			\pinlabel ${\color{red} \scriptstyle {2h-e_1-e_2-e_3-e_4-e_5}}$ at 55 10
			\pinlabel $\scriptstyle {2h-e_1-e_2-e_3-e_4-e_5}$ at 55 0
			\pinlabel $-2$ at 117 30
			\pinlabel ${\color{red}\scriptstyle e_5-e_8}$ at 122 10
			\pinlabel $\scriptstyle e_3-e_6$ at 122 0
			\endlabellist
			\includegraphics[scale=.85]{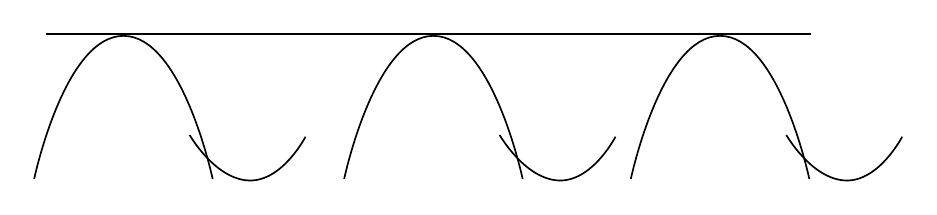}
			\caption{Minimal resolution of a rational cuspidal quintic with three cusps of type $(2,5)$ with the possible homological embeddings.}
			\label{fig:252525}
		\end{figure}

		Blowing down exceptional spheres in $e_i$ classes with Lemma~\ref{l:blowdown}, the homological embedding of the resolution into $\cptwo\#10\cptwobar$ blows down to the configuration $\Gc^\star$. 
		Due to this birational derivation, Proposition~\ref{l:biratexist} and Proposition~\ref{p:Gstar} imply this minimal embedding is not realizable .
		
		The other homological embedding is of the minimal resolution into $\cptwo\#6\cptwobar$. After blowing down $e_i$ spheres we obtain three conics with three intersection points, such that each pair of them are tangent at a distinct intersection point, plus one line tangent to each of the three conics at a different point.
		
		We are better equipped to work with configurations with a single conic and many lines than many conics and a single line. Therefore, we will find a birationally equivalent configuration. First add one line through each of the three pairs of intersection points of the conics (each intersecting the tangent line generically in double points). This does not change the symplectic isotopy classification of the configuration by Proposition~\ref{p:addline}. The resulting configuration is birationally equivalent to the configuration $\Hc$ of one conic with six lines intersecting in triple points as shown in Figure~\ref{fig:H}. The birational equivalence comes from blowing up once at each of the three intersection points of the conics and then blowing down the proper transforms of the three added lines. $\Hc$ has a unique equisingular symplectic isotopy class by Proposition~\ref{p:H}, so this cuspidal quintic also has a unique equisingular symplectic isotopy class in $\cptwo$ by Proposition~\ref{l:biratderiveunique}.
	\end{proof}

	\begin{prop}\label{p:27232323}
		If a rational cuspidal curve $C$ has normal Euler number $25$ and one cusp of type $(2,7)$ and three of type $(2,3)$,
		then only relatively minimal symplectic embedding of $C$ is into $\cptwo$ and this embedding is unique up to symplectic isotopy. 
		Equivalently, the corresponding contact structure $\xi_C$ has a unique minimal filling which is a rational homology ball.
	\end{prop}
	
	\begin{proof}
		In the minimal resolution, the proper transform of $C$ is smooth with self-intersection $+1$ so we use Theorem~\ref{thm:mcduff} to identify it with a line and classify the homological embeddings of the exceptional divisors in Figure~\ref{fig:27232323}.
		
		\begin{figure}
			\centering
			\labellist
			\pinlabel $+1$ at 225 95
			\pinlabel $\scriptstyle h$ at 225 80
			\pinlabel $-2$ at 427 30
			\pinlabel ${\color{red}\scriptstyle e_8-e_9}$ at 433 10
			\pinlabel $\scriptstyle e_2-e_3$ at 433 0
			\pinlabel $-2$ at 475 30
			\pinlabel ${\color{red}\scriptstyle e_9-e_{10}}$ at 475 10
			\pinlabel $\scriptstyle e_1-e_2$ at 475 0
			\pinlabel $-1$ at 400 70
			\pinlabel ${\color{red} \scriptstyle {2h-e_1-e_2-e_3-e_4-e_8}}$ at 368 10
			\pinlabel $\scriptstyle{2h-e_1-e_2-e_4-e_5-e_6}$ at 368 0
			\pinlabel $-1$ at 300 70
			\pinlabel ${\color{red} \scriptstyle {2h-e_1-e_2-e_3-e_4-e_7}}$ at 269 10
			\pinlabel $\scriptstyle{2h-e_1-e_2-e_3-e_5-e_6}$ at 269 0
			\pinlabel $-1$ at 199 70
			\pinlabel ${\color{red} \scriptstyle {2h-e_1-e_2-e_3-e_4-e_6}}$ at 168 10
			\pinlabel $\scriptstyle {2h-e_1-e_2-e_3-e_4-e_6}$ at 168 0
			\pinlabel $-1$ at 98 70
			\pinlabel ${\color{red} \scriptstyle {2h-e_1-e_2-e_3-e_4-e_5}}$ at 67 10
			\pinlabel $\scriptstyle {2h-e_1-e_2-e_3-e_4-e_5}$ at 67 0
			\endlabellist
			\includegraphics[scale=.85]{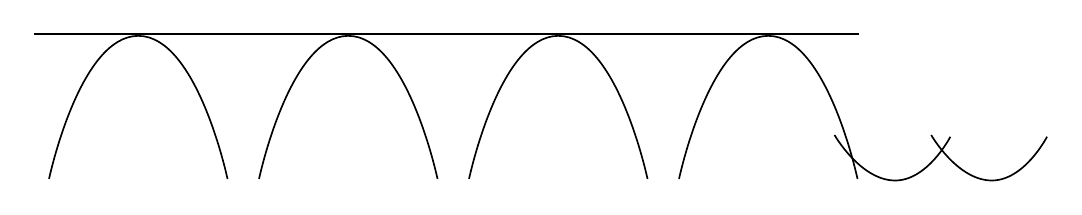}
			\caption{Minimal resolution of a rational cuspidal quintic with one cusp of type $(2,7)$ and three of type $(2,3)$ with the possible homological embeddings.}
			\label{fig:27232323}
		\end{figure}

		As above, using Lemma~\ref{l:blowdown} to blow down $e_i$ exceptional spheres, the first homological embedding blows down to a configuration that contains $\Gc^\star$, so it leads to no embedding by Proposition~\ref{p:Gstar}.
		
		The second homological embedding uses exactly six exceptional classes therefore corresponds to an embedding of the cuspidal curve in $\cptwo$. Blowing down exceptional spheres in $e_i$ classes in the second homological embedding yields a configuration of four conics and a line.
		Let $p_0$ denote the image of $e_1, e_2$, and $e_3$ after the blow-down. At $p_0$, three of the conics $Q_1,Q_2,Q_3$ intersect tangentially with multiplicity three and the fourth conic $Q_4$ is tangent to these three with multiplicity two. There are three transverse triple intersections of $Q_i,Q_j, Q_4$ for the three pairs $\{i,j\}\subset \{1,2,3\}$. The line $L$ is tangent to all four conics, but does not go through any of the intersections of the $Q_i$.
		
		We again look for a birational equivalence to a configuration with fewer conics and more lines. To ensure we get a birational equivalence, we first add a line $T$ tangent to the four conics at $p_0$. This does not change the symplectic isotopy uniqueness by Proposition~\ref{p:addline}. Next we blow up three times at $p_0$ to separate the conics at that point. The proper transform of $T$ becomes a $(-1)$-exceptional sphere which we can blow down, and then we blow down two more $(-1)$-exceptional spheres as in Figure~\ref{fig:Lbirat}. The result is the configuration $\Lc$ together with an additional line. 
		$\Lc$ is made up of two conics $P_1,P_2$ with two tangencies, together with three lines such that each line is tangent to $P_1$ and the pairwise intersections of the lines are three distinct points on $P_2$.
		The line $K$ that is added to $\Lc$ as the result of this birational equivalence is tangent to the two conics at one of their tangent intersection points. The symplectic isotopy classifications of $\Lc$ and $\Lc\cup K$ are equivalent by Proposition~\ref{p:addline}. Since $\Lc$ has a unique symplectic isotopy class by Proposition~\ref{p:L}, this cuspidal curve has a unique symplectic isotopy class in $\cptwo$ by Proposition~\ref{l:biratderiveunique}.
		\end{proof}
		
		\begin{figure}
			\centering
			\includegraphics[scale=.6]{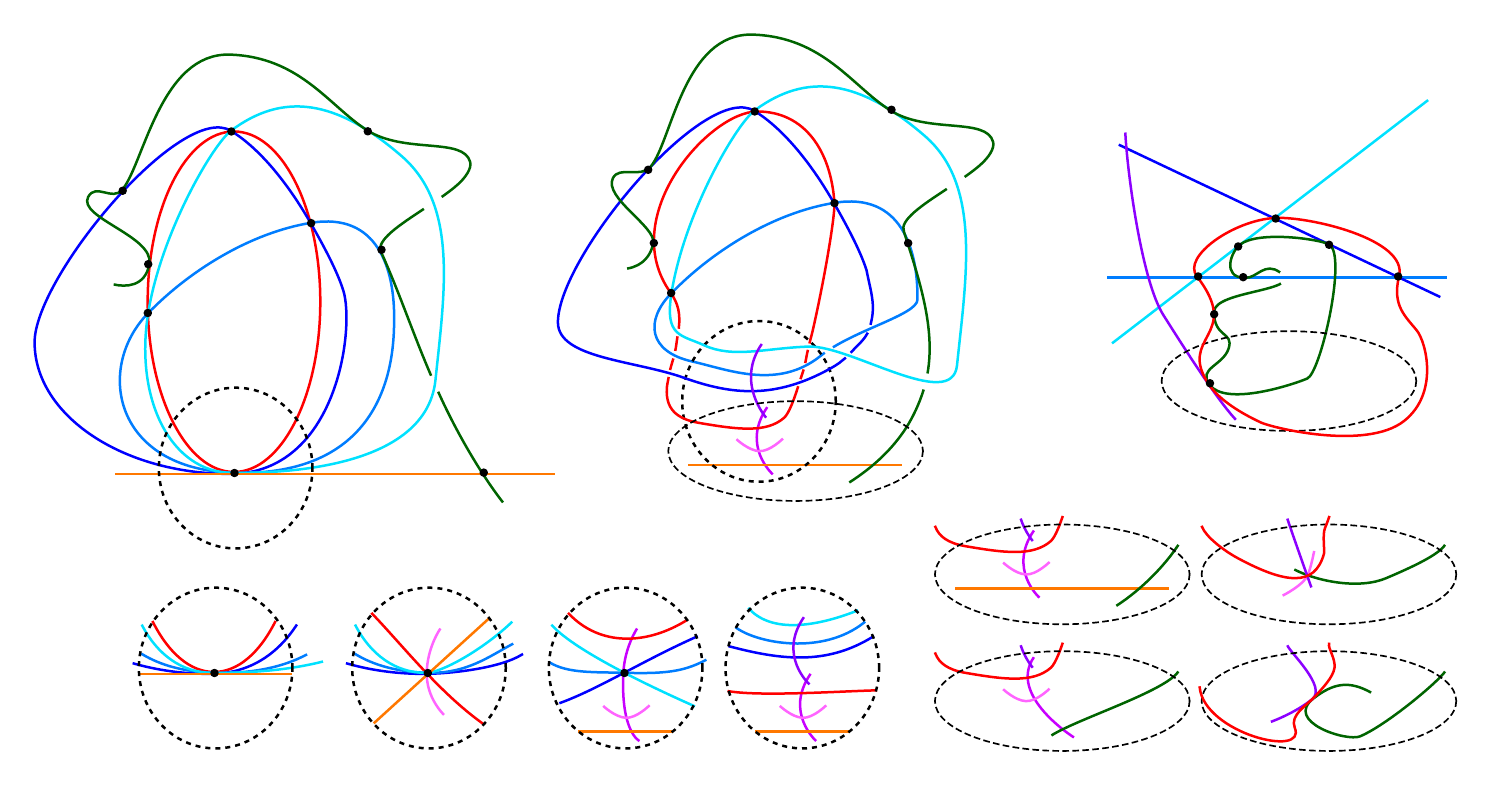}
			\caption{Birational equivalence.}
			\label{fig:Lbirat}
		\end{figure}


}
\subsection{A note on sextics}\label{ss:sextics}

In degree 6, there is only one multiplicity multi-sequence that passes the semigroup condition, but such that we cannot find a cap with a sphere of positive self-intersection; this multiplicity multi-sequence is $[[3,2^{[7]}]]$.

\begin{prop}
	There is no rational cuspidal curve in $\cptwo$ whose multiplicity multi-sequence is $[[3,2^{[7]}]]$. 
\end{prop}

Note that, in fact, this shows that the corresponding contact structures have no rational homology ball fillings (see Section~\ref{s:rational} above).

\begin{proof}
	The only allowed multiplicity sequences that are subsets of $[[3,2^{[7]}]]$ are $[2^{[k]}]$, $[3]$, and $[3,2]$; 
	these are all simple singularities of types $A_{2k}$, $E_6$, and $E_8$, respectively;
	their links are $T(2,2k+1)$, $T(3,4)$, and $T(3,5)$, respectively.
	
	Suppose such a curve exists;
	perturb it to get a smooth curve $C'$ of genus $10$, and take the double cover $\Sigma(\cptwo,C')$ of $\cptwo$ branched over $C'$;
	this is a $K3$ surface, and, by construction, it contains (as pairwise disjointly embedded) the double covers of $B^4$ branched over the Milnor fibre of all the singularities of $C$.
	
	All these Milnor fibres are negative definite, since the singularities are simple, and their second Betti numbers sum to 20, which contradicts the fact that $b_2^-(K3) = 19$.
\end{proof}

	This is where our approach for sextics differs from the previous cases we examined:
	here we are not trying to classify fillings, but rather looking specifically at embeddings in $\cptwo$, or, equivalently, at rational homology ball fillings.
	Moreover, the same argument can be used to narrow down the number of different splittings of the multiplicity multi-sequences that one needs to consider; namely, there can be no sextic whose singularities are all simple.

\begin{remark}
A similar argument to that of the proposition above, using 5-fold covers instead of 2-fold covers, can be used to obstruct the existence of the two quintics in $\cptwo$ with five and six singularities (which we excluded in Example~\ref{ex:RH}; the non-fillability of the corresponding cuspidal contact structures was established in Proposition~\ref{p:daisies}). A general result, of a more topological flavor, has been proved in~\cite[Theorem~4.7]{GKutle}.
\end{remark}


\section{Differences between the complex and symplectic categories}\label{s:orevkov}

Every complex algebraic curve in $\cptwo$ is a symplectic surface (potentially singular), but most symplectic surfaces are not complex algebraic. However, in many situations, a symplectic surface is symplectically isotopic to a complex curve. This is known for smooth symplectic surfaces of degree at most $17$ and is conjectured for smooth symplectic surfaces in $\cptwo$ in general. However, for singular curves, there need not always be an equisingular symplectic isotopy from a symplectic configuration to a complex configuration. This is easiest to verify when there is a singular symplectic configuration whose singularity types cannot be realized by any complex curve configuration. Some examples of reducible configurations with this property come from surprising and ancient theorems in projective geometry. For example, the Pappus theorem shows that given a configuration of lines as in Figure~\ref{fig:Pappus}, the points $X,Y,Z$ are necessarily collinear. Therefore there is no complex line arrangement consisting of the lines in Figure~\ref{fig:Pappus} together with an extra line passing through $X$ and $Y$ but not $Z$. 

\begin{figure}
	\labellist
	\pinlabel $X$ at 143 88
	\pinlabel $Y$ at 230 78
	\pinlabel $Z$ at 258 89
	\endlabellist	
	\includegraphics[width=0.75\textwidth]{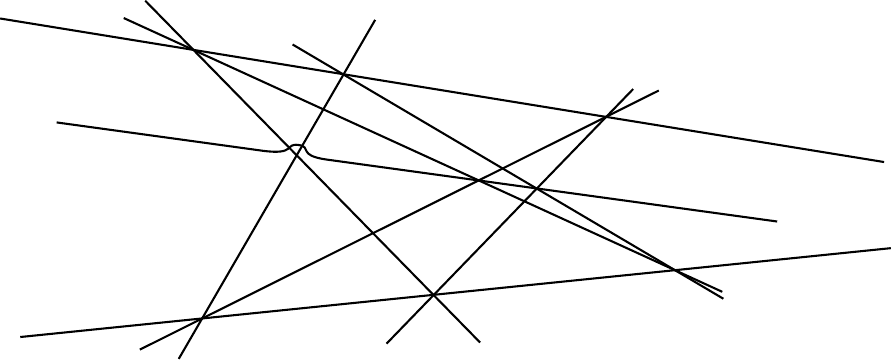}
	\caption{The fake Pappus configuration.}\label{fig:Pappus}
\end{figure}

Other examples of reducible symplectic configurations which are not realizable with complex algebraic curves come from pseudoline arrangements. A pseudoline arrangement is a collection of simple closed curves in $\rptwo$ each of which represents the homology class $[\R\P^1]$. It was proven in~\cite{RuSt} that any pseudoline arrangement in $\rptwo$ can be extended (after isotopy) to a symplectic line arrangement in $\cptwo$. Therefore any pseudoline arrangement which is not realizable by straight lines (over the complex numbers) gives another example of this phenomenon.

We have seen for rational cuspidal curves in the unicuspidal and low degree cases we looked at, every symplectic realization in $\cptwo$ is symplectically isotopic to a complex curve. This is in contrast to the differences appearing in the reducible configurations just mentioned. However, the differences between symplectic and complex singular curves are not restricted to reducible configurations. 

Using birational transformations, we provide here an example of an irreducible singular curve which is realizable symplectically but not complex algebraically. This example was given to us by Orevkov, and we are very grateful to him for explaining it to us. Although this curve is irreducible, its singularities are not cuspidal, they are locally reducible. It remains open as to whether there are any symplectic rational cuspidal curves in $\cptwo$ which are not symplectically isotopic to a complex rational cuspidal curve.

\begin{theorem}\label{t:orevkov}
There is an irreducible symplectic rational curve of degree $8$ in $\cptwo$ that is not equisingularly isotopic to a complex curve.
\end{theorem}

More precisely, we will demonstrate a specific example of a symplectic surface $C$ of degree 8, with three \emph{reducible} singularities (not cuspidal), each isomorphic to a singularity of type $(3,5)$ plus a generic line.
That is, the singularity is defined by an equation locally modeled on $x(x^3+y^5) = 0$.
We then prove that no complex curve with the same singularities can exist.

\begin{figure}
	\centering
	\includegraphics[scale=.5]{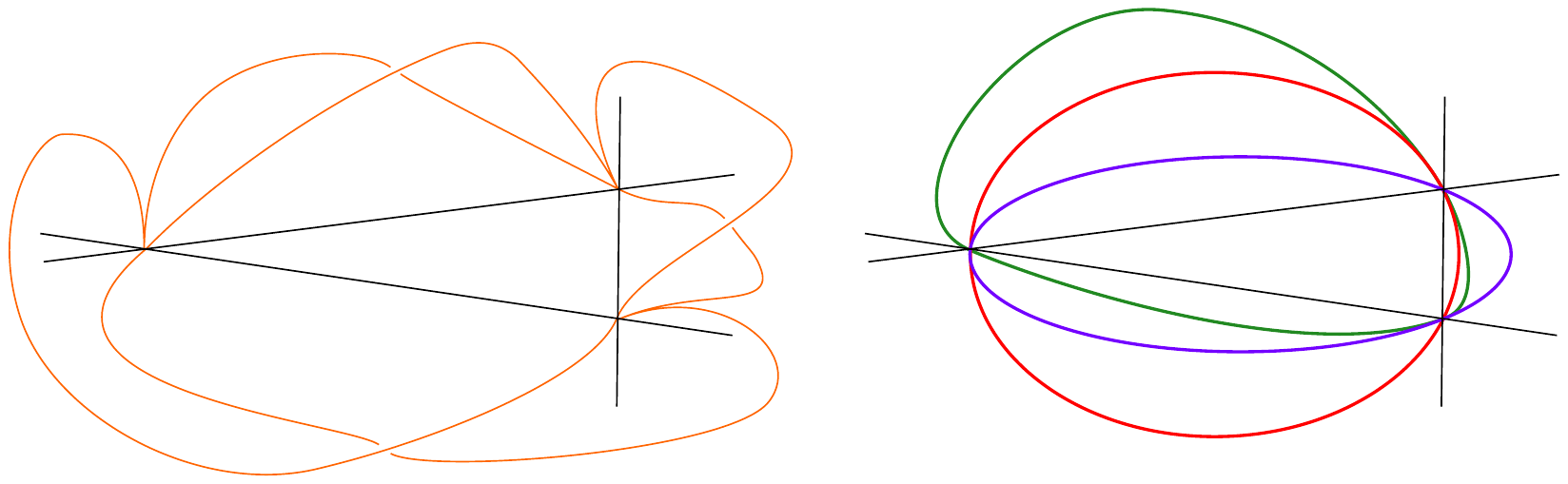}
	\caption{On the left is a degree-8 curve (orange) with three reducible singularities with one branch of type $(3,5)$ and the other smooth. A triangle of lines passes through these three singular points. On the right is the same triangle of lines with three conics. At each intersection point on the triangle, one pair of the three conics intersects tangentially and the third intersects these two transversally. Configuration $\mathcal{D}$ is built by overlaying the two sides of this configuration. The three lines on each side coincide and the tangent direction to the $(3,5)$-cusp part of each singularity of the orange curve agrees with the common tangent direction to two of the three conics at the intersection points.}
	\label{fig:configC}
\end{figure}

\begin{figure}
	\centering
	\includegraphics[scale=.3]{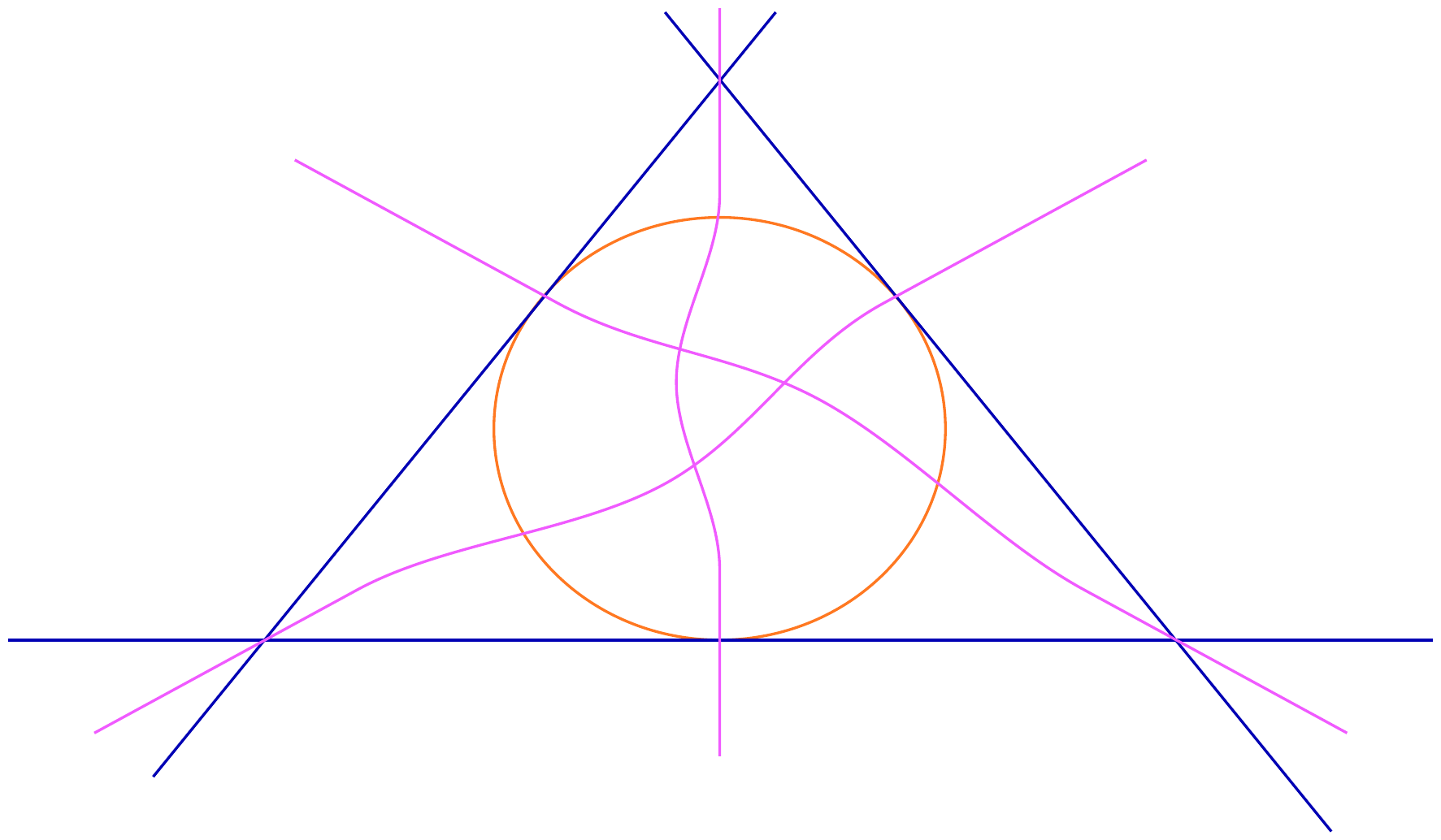}
	\caption{Configuration $\mathcal{T}$ consists of a conic with three tangent lines, together with three additional lines passing through the tangent intersections and the intersection of the other two lines, and intersecting each other in double points as shown. In a complex algebraic arrangement, these three additional lines would necessarily intersect each other all at the same point in a triple intersection.}
	\label{fig:configT}
\end{figure}

To understand the symplectic existence and complex obstruction to this rational singular curve, we relate it via a birational transformation to a reducible configuration that is easier to understand.

Define a configuration $\mathcal{D}$ by adding to the degree-$8$ rational singular curve $C$, three conics and three lines with the following intersection conditions. The three lines $\ell_1,\ell_2,\ell_3$ must pass through the three pairs of the singular points of $C$. The three conics $q_1,q_2,q_3$ must each pass through the three singular points of $C$, and each must be simply tangent to two of the cuspidal branches and transverse to the third (none should be tangent to the smooth branch of $C$ at the singularities or the lines $\ell_i$). The singular points of $C$ are the only singular points of $\mathcal{D}$ because there are no further intersections between $C,\ell_1,\ell_2,\ell_3,q_1,q_2,q_3$ for degree reasons. (See Figure~\ref{fig:configC}.)

Define another configuration $\mathcal{T}$, built from a triangle of three lines with an inscribed conic, and three \emph{non-concurrent} lines drawn from a vertex of the triangle to the opposite tangency point. (See Figure~\ref{fig:configT}.)

\begin{lemma}\label{l:CisT}
The two configurations $\mathcal{D}$ and $\mathcal{T}$ are birationally equivalent.
\end{lemma}

\begin{proof}
Blow up at the singularities of $C$, creating three exceptional divisors $e_1,e_2,e_3$. The singularities of the proper transform $\widetilde{C}$ are three simple cusps, and for each cusp, there is one exceptional divisor $e_i$ intersecting $\widetilde C$ tangentially at that cusp and transversally at one other smooth point of $\widetilde{C}$. Each $\ell_i$ has been blown up twice, so the self-intersection number of each $\widetilde{\ell}_i$ is $-1$. Therefore, we can blow down the proper transforms $\widetilde{\ell}_1, \widetilde{\ell}_2, \widetilde{\ell}_3$, returning to $\cptwo$ with a new configuration of curves. $\widetilde C$ blows down to a quartic $D$ with three simple cusp singularities.  $q_1,q_2,q_3$ blow down to three lines, each passing through two of the simple cusps. The three exceptional divisors $e_1,e_2,e_3$ blow down to three tangent lines to the cusps of $D$, and they intersect in three distinct points (i.e. the images of $\widetilde\ell_1,\widetilde\ell_2,\widetilde\ell_3$). See Figure~\ref{fig:configU}.

\begin{figure}
	\centering
	\includegraphics[scale=.5]{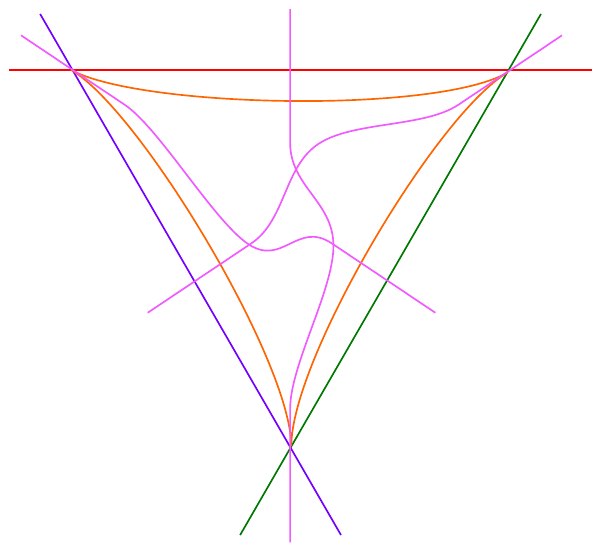}
	\caption{Intermediate configuration between configurations $\mathcal{D}$ and $\mathcal{T}$. The singular component is a degree $4$ curve with three simple cusps. There is one line through each pair of these three cusps, together with one line tangent to each cusp. The intersections of the three tangent lines are required to be three distinct double points instead of a coinciding triple point.}
	\label{fig:configU}
\end{figure}

Now blow up at the simple cusps of $D$, resolving the singularities and creating three tangent exceptional divisors $f_1,f_2,f_3$. The proper transforms of the lines $q_1,q_2,q_3$ become exceptional divisors which can then be blown down. The result is the configuration $\mathcal{T}$ where $D$ is sent to the smooth conic $Q$, $f_1,f_2,f_3$ become its tangent lines, and $e_1,e_2,e_3$ blow down to lines, each connecting a tangency of $Q$ and $f_i$ with the intersection of $f_{i+1}$ and $f_{i-1}$ (here we consider the labels modulo 3).
\end{proof}

\begin{lemma}\label{l:Texists?}
The configuration $\mathcal{T}$ is symplectically realizable, but not complex realizable.
\end{lemma}

\begin{proof}
We start by proving that $\mathcal{T}$ is not complex realizable.
In fact, this follows from Brianchon's theorem: this states that if an hexagon $ABCDEF$ is circumscribed to a conic, then $AD, BE$, and $CF$ are concurrent.
Indeed, if we call $ACE$ the vertices of the triangle in $T$, and $B,D,F$ the tangency points of the conics (so that $B$ lies on $AC$), then $ABCDEF$ is a degenerate hexagon circumscribed to a conic, and therefore $AD, BF$, and $CE$ are concurrent.

To prove the existence of a symplectic realization of $\mathcal{T}$, we just perturb the degenerate Brianchon configuration locally around the triple point of intersection of the secants. Since the symplectic condition is open, performing this operation in a $C^1$-small way preserves the fact that the curves are symplectic.
\end{proof}

\begin{proof}[Proof of Theorem~\ref{t:orevkov}]
We construct the curve $C$ starting from the configuration $\mathcal{T}$;
indeed, Lemma~\ref{l:Texists?} shows that $\mathcal{T}$ exists symplectically, and by Lemma~\ref{l:CisT}, $\mathcal{D}$ is birationally equivalent to it;
since $C$ is a component of $\mathcal{D}$, we have proved its existence as a symplectic curve.

Conversely, suppose that $C$ existed as a complex curve. We can augment any such realization of $C$ to a configuration $\mathcal{D}$ as follows. Let $p_1,p_2,p_3$ be the three singular points of $C$. There is a unique complex line through any pair of distinct points, so let $\ell_1,\ell_2,\ell_3$ be the unique lines through the three pairs of $p_1,p_2,p_3$. Each transverse intersection of $\ell_i$ with $C$ at $p_j$ contributes $+4$ to their intersection number. Therefore $p_1,p_2, p_3$ cannot be collinear since then $C$ would intersect a line with intersection number $12$ instead of $8$. Therefore the three lines $\ell_i$ are distinct. Furthermore, the intersections of $\ell_i$ with $C$ must be transverse, because a tangential intersection would add an additional positive intersection beyond the required $8$. Similarly, the $\ell_i$ cannot intersect $C$ at any other point besides the $p_j$. 

If we fix two distinct points $A,B\in \cptwo$ with complex lines in their tangent spaces $T_A$, $T_B$, and a third distinct point $C\in \cptwo$, there is a unique complex conic through $A$, $B$, and $C$ tangent to $T_A$ and $T_B$. (These five conditions give five linear constraints on the six projective coefficients of the six degree two monomials. The solution is unique because any two such conics would have intersection number $5$ instead of $4$.) In general, it is possible for this conic to be singular by degenerating into two lines. Let $T_1$, $T_2$ and $T_3$ be the complex tangent line directions to the simple cusps of $C$ at $p_1$, $p_2$, and $p_3$. Letting $A=p_i$, $T_A=T_i$, $B=p_{i+1}$, $T_B=T_{i+1}$, and $C=p_{i+2}$ with mod $3$ cyclic indices for $i=1,2,3$, we construct three conics as required in the configuration $\mathcal{D}$. To check that none of these conics is a degenerate pair of lines, we observe that the unique pair of lines through $AB$ and $AC$ is a pair of the $\ell_i$ which are not tangent to $T_A$ and thus cannot be the chosen conic.

Therefore if we had a complex realization of $C$, we could construct a complex realization of the configuration $\mathcal{D}$ and then use the birational transformation to construct a complex realization of the configuration $\mathcal{T}$, but this is impossible.
\end{proof}


\appendix

\section{Symplectic rational ball fillings of lens spaces}\label{a:lens}

In this section, we prove two results on rational homology ball fillings of lens spaces. The first has recently appeared in independent work of Fossati~\cite[Theorem 4]{Fossati}.

\begin{prop}\label{p:QHBimpliesSTD}
	If a lens space $(L(p,q),\xi)$ has a weak symplectic rational homology ball filling, then there exist coprime integers $0<k<m$ such that $(p,q) = (m^2, mk-1)$, and $\xi$ is universally tight.
\end{prop}

The family coincides with the family of cyclic quotient singularities with rational disc smoothing, which were classified by Wahl~\cite[Example~5.9.1]{Wahl} (see also \cite[Corollary~1.2]{Li}).
In the proof, we will use the Ozsv\'ath--Szab\'o contact invariant~\cite{OSz-contact} in Heegaard Floer homology~\cite{OSz-HF};
the relevant properties of the theory are the non-vanishing of the contact invariant for fillable contact structures~\cite[Theorem~1.5]{OSz-contact}, the fact that lens spaces are L-spaces (i.e. they have the smallest possible Heegaard Floer homology)~\cite[Proposition 3.1]{OSz-HFPA}, and the absolute grading on Heegaard Floer homology~\cite{OSz-correctionterms}.

\begin{proof}
	Call $L = L(p,q)$ and $W$ a weak symplectic rational homology ball filling of $(L,\xi)$.
	Since $\xi$ is fillable, it is tight~\cite{EliashbergGromov}.
	All tight contact structures on lens spaces are planar~\cite{Scho}, and it follows from Wendl's theorem that all weak symplectic fillings can be deformed to Stein~\cite{Wendl}.
	Let $J$ denote a complex structure such that $(W,J)$ is a Stein filling of $(L,\xi)$.
	
	In particular, $W$ admits a handle decompositions with no 3-handles. This implies that $\pi_1(W)$ is a quotient of $\pi_1(L)$, and hence it is cyclic, and that $H_1(W,L) \cong H^3(W) = 0$.
	Since $L$ bounds a rational ball, $p = m^2$ for some positive integer $m$, and from the long exact sequence of the pair $(W,L)$ and the universal coefficient theorem, it follows that $H_1(W) = \Z/m\Z$ (this is classical; see e.g.~\cite[Proposition~2.2]{AG}).
	
	Let $W'$ be the universal cover of $W$.
	We know that $\chi(W') = m\chi(W) = m$.
	As $W'$ is simply-connected, $H_1(W') = 0$ and $H^2(W')$ is torsion-free. Since $W'$ has a decomposition without 3-handles, lifted from that of $W$, we deduce that $H_3(W') = 0$ and that $H_2(W')$ is torsion-free.
	Adding all pieces together, it follows that $H_2(W') = \Z^{m-1}$.
	The boundary $L'$ is the $m$-fold cover of $L$, and therefore $L' = L(m,q)$.
	Let $\xi'$ be the pull back of $\xi$ to $L$, filled by the pull back $J'$ of $J$.
	The first Chern class $c_1(J')$ of $J'$ is the pull-back under the covering map of $c_1(J)$.
	Since the latter is torsion and $H^2(W')$ is torsion-free, $c_1(J') = 0$.
	
	The Ozsv\'ath--Szab\'o contact invariant $c(\xi') \in \HFplus(-L', \mathfrak{t}_{\xi'})$ lives in degree $d(-L', \mathfrak{t}_{\xi'})$, computed as the degree of the map associated to the cobordism $X = W'\setminus B^4$.
	The degree of the map $F_{X,\mathfrak{s}_{J'}}$ carrying $c(\xi)$ to $c(\xi_{\rm st})$ is
	\[
	\deg F_{X,\fs_{J'}} = \frac{c_1^2(J') - 2\chi(X) - 3\sigma(X)}4 = \frac{m-1}4
	\]
	and hence $d(-L', \mathfrak{t}_{\xi'}) = \frac{m-1}4$.
	
	It follows from~\cite[Lemma~4.5]{AG} that $q \equiv \pm 1 \pmod m$. By explicitly computing correction terms of $L(m,\pm1)$, we see that if $q \equiv 1 \pmod m$ there is no \spinc structure on $L(m,q)$ with correction term $(m-1)/4$.
	
	Therefore we reduced to the case when $q \equiv -1 \pmod m$.
	That is, $q = mk-1$ for some $0<k \le m$.
	Since $L$ bounds a smooth rational homology ball, it follows from Lisca's classification of lens spaces that bound rational homology balls that $\gcd(m,k) = 1$ or $\gcd(m,k) = 2$~\cite{Lisca-ribbon} (see~\cite[Page 247]{Lecuona} for an amendment to the statement of~\cite[Theorem~1.2]{Lisca-ribbon} that includes the case $\gcd(m,k) = 2$).
	The former case corresponds to structures in Wahl's family, and we set out to exclude the second.
	
	In this case, $m = 2m''$ is even, and we can look at the $m''$-fold cover $(W'',J'')$ of $(W,J)$, with boundary $(L'',\xi'')$.
	Observe that our assumption $\gcd(m,k) = 2$ implies that $L'' = L(2m,2m-1)$.
	The same argument as above shows that $d(-L'', \mathfrak{t}_{\xi''}) = \frac{m''-1}4 = \frac{m-2}8$.
	However, by classification of tight contact structures on lens spaces~\cite{Honda}, $L(2m,2m-1)$ admits a unique tight contact structure $\xi_0$, namely the boundary of the plumbing of $2m-1$ copies of $T^*S^2$, whose corresponding contact structure has $\deg c(\xi_0) = \frac{-2m+1}4$.
	
	To conclude the proof, we need to show that $\xi'$ is universally tight. As mentioned above, however, $\xi'$ is tight (because it is filled by $W'$), and $L' = L(m,m-1)$ admits a unique tight contact structure (by~\cite{Honda}).
	Since on each lens space there is always a universally tight contact structure (obtained by taking a quotient of $S^3$ with its standard contact structure), $\xi'$ is universally tight.
\end{proof}

Alternatively, to exclude the case $\gcd(k,m) = 2$ in the proof above, one can argue that the unique contact structure on $L(2m,2m-1)$ is also universally tight, and then so is $\xi$. Then Lisca's classification of fillings~\cite[Corollary~1.2]{Li} rules this possibility out.

The second result concerns the uniqueness of the rational homology ball filling.
It is very likely that the result is known to experts (either in the context of Milnor fibers of cyclic quotient singularities, or in the context of symplectic fillings of contact structures), but we were unable to locate the precise statement.

\begin{prop}\label{p:uniqueQHBfilling}
	The standard contact structure on the lens space $L(m^2,mk-1)$ has a unique rational homology ball symplectic filling up to symplectic deformation.
\end{prop}

The proof is, in fact, somewhat implicit in Lisca's classification paper~\cite{Li}.
We refer to the paper for notation and context. We start by recalling the following two facts about continued fractions. Both facts follow from Riemenschneier's point rule~\cite{Riemenschneier} (see~\cite[Lemma 2.6]{Li}).

If we have two continued fraction expansions $[\mathbf{a}]^-$ and $[\mathbf{b}]^-$ such that \[
1/[a_1,\dots,a_m]^- + 1/[b_1,\dots,b_\ell]^- = 1
\]
(or, equivalently, the two associated fractions are of the form $p/q$ and $p/(p-q)$ for some $q<p$ coprime positive integers), we say that $[\mathbf{a}]^-$ and $[\mathbf{b}]^-$ are \emph{dual} to each other.

If $[\mathbf{a}]^-$ and $[\mathbf{b}]^-$ are dual to each other, then $\sum (a_i-1) = \sum (b_j-1)$. We also have that
\[
[n_1,\dots,n_{j-1},1,n_{j+1},\dots,n_\ell]^- = 0
\]
if and only if $[n_{j-1},\dots,n_2,n_1]^-$ and $[n_{j+1},\dots,n_\ell]$ are dual to each other.

\begin{proof}
The (negative) continued fraction expansion of $\frac{m^2}{mk-1}$ is of the form $[n_1,\dots,n_{j-1},2,n_{j+1},\dots,n_\ell]^-$, such that $[n_1,\dots,n_{j-1},1,n_{j+1},\dots,n_\ell]^- = 0$ (see the proof of~\cite[Theorem~1.2(c)]{Li}).

Recall from the proof of~\cite[Theorem~1.1]{Li} that symplectic deformation classes of symplectic fillings of $L(p,q)$ with $\frac{p}{q} = [n_1,\dots,n_\ell]^-$ correspond to strings $\mathbf{m} = (m_1,\dots,m_\ell)$ such that $1 \le m_i \le n_i$ for each $i$ and $[m_1,\dots,m_\ell]^- = 0$.

Moreover, the Euler characteristics of the filling corresponding to $\mathbf{m}$ is $\sum (n_i-m_i)$, so that the filling is a rational homology ball if and only if $m_i = n_i$ for all indices except for one, for which $m_i = n_i-1$.
We claim that for such any $p,q$ there is at most one index $j$ such that the corresponding sequence $\mathbf{m} = (n_1,\dots,n_{j-1}, n_j-1, n_{j+1}, n_\ell)$ has $[\mathbf{m}]^- = 0$, and in that case $n_j = 2$.

If $n_j > 2$, then $[\mathbf{m}]^-$ is the continued fraction expansion of a rational number $p'/q' > 1$.
Suppose now $n_j = 2$. If $[\mathbf{m}]^- = 0$, then $[n_{j-1},\dots,n_{1}]^-$ and $[n_{j+1},\dots,n_{\ell}]^-$ are dual to each other; however, as $j$ varies between $1$ and $\ell$, the difference
\[
\sum_{h < j} (n_h-1) - \sum_{h > j} (n_h-1)
\]
is strictly decreasing, hence there is at most one value of $j$ such that it vanishes.
As we recalled above, the difference must vanish if $[n_{j-1},\dots,n_{1}]^-$ and $[n_{j+1},\dots,n_{\ell}]^-$ are dual to each other.

The proposition now follows.
\end{proof}

\bibliography{cuspidal}
\bibliographystyle{amsalpha}


\end{document}